\documentclass[12pt,reqno,a4paper,oneside]{book}
\usepackage{amsmath,amsfonts,amssymb,amsthm}
\usepackage{mathrsfs}
\usepackage{latexsym}
\usepackage[english]{babel}
\usepackage{bbm}
\usepackage{graphicx}
\usepackage{placeins}
\usepackage{hyperref}
\usepackage[usenames,dvipsnames]{xcolor}
\usepackage{frontespizio}
 
\usepackage{array}

\usepackage{tikz}
\usepackage{enumitem}
\usepackage{multicol}

\definecolor{dkred}  {rgb}{0.58,0.05,0.3}
\definecolor{pinegreen}{rgb}{0.0, 0.47, 0.44}
\definecolor{mygreen}{rgb}{0.09, 0.45, 0.27}
\def\LL{{\mathcal L}}

\renewenvironment{proof}[1][\proofname]{{\bfseries #1.} }{\qed}
\def\<{\langle}
\def\>{\rangle}

\def\Cov{{\rm Cov\,}}

\def\cl#1{{\mathscr #1}}
\newcommand{\field}[1]{\mathbb{#1}}
\newcommand{\cvd}{\hfill}
\def\sig{\mathrm{sign}}
\def\H{{\mathbb H}}
\def\b{{\beta}}
\def\K{{\mathbb{K}}}
\newcommand{\DD}{\mathbb{D}}
\newcommand{\Leb}{\mathrm{Leb}}

\newcommand{\R}{\field{R}}

\newcommand{\X}{\tilde X}

\newcommand{\N}{\field{N}}

\newcommand{\Z}{\field{Z}}

\newcommand{\Var}{{\rm Var}}
\newcommand{\Corr}{{\rm Corr}}

\newcommand{\G}{{\mathcal G}}

\def\eqlaw{\stackrel{\mathrm{law}}{=}}	
\def\Tr{{\,\mathrm{Tr}}}
\newcommand{\e}{{\rm e}}
\newcommand{\F}{{\mathscr{F}}}

\newcommand{\A}{{\mathcal A}}
\newcommand{\B}{{\mathscr B}}

\newcommand{\var}{\operatorname{Var}}

\def\cl#1{{\mathcal #1}}
\def\ep{\varepsilon}

\def\dTV{{d_{\mathrm{TV}}}}
\def\dW{{d_{\mathrm{W}}}}
\def\SSd{{\mathbb{S}^d}}
\def\SS2{{\mathbb{S}^2}}

\def\E{{\mathbb{ E}}}
\def\I{{\mathscr I}}
\def\P{{\mathbb{P}}}
\def\F{{\mathscr{F}}}

\def\K{{\field K}}

\def\a{\alpha}

\def\1{{\mathbbm{1} }}
\newtheorem{theorem}{Theorem}[section]
\newtheorem{example}[theorem]{Example}
\newtheorem{proposition}[theorem]{Proposition}
\newtheorem{lemma}[theorem]{Lemma}
\newtheorem{corollary}[theorem]{Corollary}
\newtheorem{definition}[theorem]{Definition}

\newtheorem{assumption}[theorem]{Assumption}

\newtheorem{remark}[theorem]{Remark}
\newtheorem{notation}[theorem]{Notation}
\numberwithin{equation}{section}
\newtheorem*{teo411}{Theorem 4.1.1}
\newtheorem*{teo548}{Theorem 5.4.8}
\newtheorem*{lemma5412}{Lemma 5.4.12}
\newtheorem*{teo6110}{Theorem 6.1.10}
\newtheorem*{teo628}{Theorem 6.2.8}

\begin{document}

\begin{frontespizio}
	
	\Universita{Rome Tor Vergata}
	
	\Corso[PhD]{Mathematics}
	
	\Annoaccademico{2022-2023}
	\Titolo{{Limit theorems for Gaussian fields \\ via Chaos Expansions and Applications}}
	\Candidato{Giacomo Giorgio}

\end{frontespizio}
\clearpage
\null
\thispagestyle{empty}
\clearpage
	
	\tableofcontents

	\chapter{Introduction}

In this PhD thesis we apply a combination of Malliavin calculus and Stein's method in the framework of probabilistic approximations. The specific problems we tackle with these methods are motivated by probabilistic models in cosmology (Part I: Quantitative CLTs for non linear functionals of random hyperspherical harmonics) and finance (Part II: the fractional Ornstein-Uhlenbeck process in rough volatility modelling).
	 Malliavin calculus is a celebrated tool in stochastic analysis that extends the classical calculus of variations to the setting of stochastic processes, see the monograph \cite{Nua}. One of the central objects in Malliavin calculus is the Malliavin derivative, an extension of the classical notion of derivative to the space of random variables. This derivative plays a crucial role in various applications, including stochastic differential equations, stochastic control theory and mathematical finance, where it is used, for example, to compute sensitivities of financial derivatives with respect to underlying assets.

	Stein's method is a powerful technique in probability theory for quantifying the distance between probability distributions. The main idea behind Stein's method is to approximate a target probability distribution by a simpler, known distribution while controlling the error in a quantitative manner. Unlike other approximation techniques that focus on pointwise convergence or weak convergence, Stein's method provides bounds in Total variation distance and other probability metrics.
	It is used for instance to prove quantitative limit theorems in probability theory and establish convergence rates in statistical inference.
	
   By combining Malliavin calculus techniques and Stein's method, I.Nourdin and G.Peccati \cite{NP12} obtained powerful results applicable in several aspects of stochastic calculus. One crucial and widely applied result is the Fourth Moment Theorem, which drastically simplifies the classical method of moments (and cumulants), to prove convergence of a sequence of random variables living in a fixed Wiener space to a Gaussian law. Moreover, it allows to get bounds for the convergence rate in various probability metrics. The application of the Fourth Moment Theorem requires computations of the 4th order cumulant of the random sequences, which is often based on the use of the celebrated Diagram Formula, see \cite{MPbook}. In our research, we extensively rely on the latter, after a reformulation that better suits our specific setting.
 	
   Additionally to Malliavin Calculus and Stein's method, we also apply techniques from Large Deviations Theory, see \cite{DemZei}. This is a branch of probability theory that deals with the study of rare events or extreme fluctuations in stochastic systems. It provides a framework for understanding the behavior of random variables when they deviate significantly from their typical or expected values. For this reason it is applied in various fields including statistical mechanics, information theory, finance, and more.
   
	This Thesis consists of two parts. In the first, based on \cite{CGR23}, we study the geometry of Gaussian Laplace eigenfunctions on the $d$-dimensional unit sphere, in high frequency limit. In particular we apply the Fourth Moment Theorem and its consequences to prove a quantitative Central Limit Theorem in Total Variation distance for regular statistics of these fields. In the second part we work with rough stochastic volatility models, focusing in particular on the fractional Ornstein-Uhlenbeck (fOU) process. In this part we consider two problems. the first is based on \cite{DGP}, which is in preparation. Here, we define a bivariate fractional Ornstein-Uhlenbeck (2fOU), studying its properties and the inference of its cross-correlation parameters. Then, based on \cite{GPP23}, we provide a short-time large deviation principle for Volterra-driven stochastic volatilities and apply it to short maturity pricing of options on fOU and log-fractional Brownian motion based rough volatility models.

In Chapter \ref{background} we introduce the mathematical tools that are largely used both in Part 1 and Part 2. We start by introducing the concept of an Isonormal Gaussian random field over a separable Hilbert space, and the notion of Wiener chaoses expansion. Subsequently we recall the fundamental notions of Malliavin calculus for Gaussian random fields, emphasizing the concept of multiple integrals. Moreover we introduce the main probability metrics, such as the Kolmogorov distance, the Wasserstein distance and the Total variation distance. We explore their relations, and discuss main results about the convergence in these metrics. Finally we discuss some concepts of Stein's method and introduce the Fourth Moment Theorem, that we use in several points of this thesis. We refer to \cite{NP12} for more details.

\section{Part \ref{Part1}: Quantitative CLTs for non linear functionals of random hyperspherical harmonics}	

In the first part of this Thesis, we discuss the results of our paper \cite{CGR23}. We are interested in quantitative CLTs in Total Variation distance \cite[Section C.2]{NP12} for nonlinear statistics $\lbrace \X_\ell\rbrace_{\ell\in \mathbb N}$ of random hyperspherical harmonics $\lbrace T_\ell\rbrace_{\ell\in \mathbb N}$ in the high energy limit (as $\ell\to+\infty$). Random hyperspherical harmonics $\lbrace T_\ell \rbrace_{\ell\in \mathbb N}$ are Gaussian Laplace eigenfunctions on the unit $d$-dimensional sphere $\mathbb S^d$ ($d\ge 2$). They are the Fourier components of isotropic Gaussian spherical random fields, therefore used in a wide range of disciplines; in particular, for $d=2$ they play a key role in cosmology -- in connection with the analysis of the Cosmic Microwave Background radiation data -- as well as in medical imaging and atmospheric sciences, see \cite[Chapter 1]{MPbook} for more details. For these reasons, in the last years the investigation of their geometry received a great attention, in particular the asymptotic behavior, for large eigenvalues (as $\ell\to +\infty$), of their nonlinear statistics $\lbrace \X_\ell\rbrace_{\ell\in \mathbb N}$, see \cite{MW11, MW14, ROS20, Dur16, CM18, Tod19, MRW20, MR3959555, MN24} and the references therein. 

The main goal of most of these papers is to study first and second order fluctuations for $\X_\ell$ to be some geometric functional of the excursion sets of $T_\ell$, such as the so-called Lipschitz-Killing curvatures
\cite[Section 6.3]{AT} that in dimension $2$ are the area, the boundary length and the Euler-Poincar\'e characteristic. Hence it is clear that $\X_\ell$ may be a function of the sole $T_\ell$ (in the case of the excursion measure for instance) or a function of $T_\ell$ and its derivatives. 

The above mentioned references take advantage of Wiener-It\^o theory, the random variables $\lbrace \X_\ell \rbrace_{\ell\in \mathbb N}$ being square integrable functionals of Gaussian fields. In this framework, the techniques developed allow one to establish Central Limit Theorems (CLTs) via a powerful combination of chaotic expansions and fourth moment theory by Nourdin and Peccati \cite{NP12}. It is well known that the link between Malliavin calculus and Stein's method established by these two authors permits to get estimates on the rate of convergence to the limiting Gaussian law in various probability metrics \cite[Appendix C.2]{NP12}, at least when a finite number of chaoses are involved. For general functionals instead, the so-called second order Poincar\'e inequality \cite{NPR09} may be evoked, even in its improved version \cite{Vid19}.   

However, the existing results in the literature for the above mentioned geometric functionals $\lbrace \X_\ell\rbrace_{\ell\in \mathbb N}$ (which do have an infinite chaos expansion) of random hyperspherical harmonics $\lbrace T_\ell \rbrace_{\ell\in \mathbb N}$ only deal with the Wasserstein distance, see e.g. \cite{ROS20, CM18, Ros19}. The typical situation is a single chaotic component dominating the whole series expansion, entailing the Wasserstein distance to be controlled by the square root of the fourth cumulant of this leading term plus the $L^2(\mathbb P)$-norm of the series tail. Moreover, generally there are no information on the optimal speed of convergence.

A natural question is whether or not these results could be upgraded to stronger probability metrics. We are able to solve this problem for integral functionals of the sole $T_\ell$, that are regular enough in the Malliavin sense, by taking advantage of a recent result in \cite{BCP19}. 
In this paper, the authors prove some regularization lemmas that enable one to upgrade the distance of convergence from smooth Wasserstein to Total Variation (in a quantitative way) for any sequence of random variables which are smooth and non-degenerate in some sense. The price to pay is to control the smooth Wasserstein distance between the sequence of their Malliavin covariance matrices and its limit, that however does \emph{not} need to be the Malliavin covariance  matrix of the limit. Remarkably, this technique requires neither the sequence of random variables of interest to be functionals of a Gaussian field nor the limit law to be Normal, situations that naturally occur since the underlying randomness may be not Gaussian \cite{BCPzeri, CNN20} or related functionals may show non-Normal second order fluctuations \cite{MPRW16}.

	\begin{notation}
		Let us introduce some notation that we use in what follows. For a given functions $f, g$, we say that $f=O(g)$ if $\lim_{s\to \infty} \frac{f(s)}{g(s)}=C$ or if $\lim_{s\to a} \frac{f(s)}{g(s)}=C$ with $C$ positive constant and $a\in \R$.  We say that $f=o(g)$ if $\lim_{s\to \infty} \frac{f(s)}{g(s)}=0$ or if $\lim_{s\to a} \frac{f(s)}{g(s)}=0$ with $a\in \R$. We denote with $\overset{a.s.}{\to}, \overset{p}{\to}, \overset{d}{\to}$ the convergence almost surely, in probability and in distribution respectively.  
	\end{notation}

\subsection*{Chapters \ref{chapter:geometry_of_rha} and \ref{chapter:main_result}}

In Chapter \ref{chapter:geometry_of_rha} we first recall real hyperspherical harmonics and then define the random hyperspherical harmonics $T_{\ell}$ of order $\ell$ which are isotropic Gaussian random fields on $\SSd$. Talking about random hyperspherical harmonics of order $\ell$ over $\SSd$, it is natural to introduce the Gegenbauer polynomials, that are crucial in the discussion about the moments of $T_{\ell}$ and its functionals. We prove some novel estimates on the moments of products of powers of Gegenbauer polynomials (the latter describing the covariance structure of the random hyperspherical harmonics $\lbrace T_\ell\rbrace_{\ell\in \mathbb N}$) thus extending some formulas in \cite{Mar08, Ros19} (see Lemma \ref{GauntProd} and Lemma \ref{resume}). We finally write down explicitly our functional of interest: we consider
\begin{equation*}
	\X_{\ell}=\frac{X_\ell-\E[X_\ell]}{\sqrt{\var(X_\ell)}} \quad\mbox{where}\quad X_\ell := \int_{\mathbb S^d} \varphi(T_\ell(x)) dx,
\end{equation*}
$\varphi:\mathbb R\to \mathbb R$ being square integrable w.r.t. the Gaussian density. In \cite{ROS20}, the authors prove that, under mild assumptions, the above functional $\X_{\ell}$ converges in Wasserstein distance towards a Gaussian random variable as $\ell\to +\infty$. The main goal of our work is to strengthen this result, studying the Total Variation convergence of this functional. 
In Chapter \ref{chapter:main_result} we discuss the main result of the Part \ref{Part1}. 

	\begin{teo411}
		
	Under suitable assumptions for $\varphi$, for any $0< \varepsilon < 1$, as $\ell\to +\infty$,
	\begin{equation*}
		\dTV(\X_{\ell}, Z) =O_\varepsilon\big ( \ell^{-\frac{1-\varepsilon}{2}} \big)
	\end{equation*}
	where $O_\varepsilon$ means that the constants involved in the $O$-notation depend on $\varepsilon$. 
\end{teo411}

We address this problem in light of \cite{BCP19}. We need to investigate the asymptotic behavior of the Malliavin covariance of $\X_\ell$, that we denote by $\sigma_\ell$. Under some additional regularity properties on the function $\varphi$ which are needed to ensure the existence of Malliavin derivatives of $\X_\ell$ and its integrability properties, we are able to prove the convergence in Wasserstein distance of $\sigma_\ell$ towards a non-degenerate deterministic limit, that together with the uniform boundedness of Malliavin-Sobolev norms of $\X_\ell$ guarantees the convergence in Total Variation distance for $\X_\ell$. 
To the best of our knowledge, ours is the first quantitative Limit Theorem in Total Variation distance for nonlinear functionals of random hyperspherical harmonics having an infinite chaotic expansion, generalizing in particular the work \cite{ROS20}. 
As a bonus, we gain some new results on the asymptotic behavior of Malliavin derivatives of these functionals.
For our investigation we also exploit an explicit link between the diagram formula for moments of Hermite polynomials and the graph theory, inspired by \cite{Mar08} (see Lemma \ref{STIMA}). In particular, we extrapolate a graph from each of these diagrams and use the fact that every connected graph can be covered by a tree, eventually studying only the contribution coming from these trees. 
Finally, it is worth stressing that in the context of Gaussian approximations for random variables that are functionals of an underlying Gaussian field, the second order Poincar\'e inequality by Vidotto \cite{Vid19} has led to quantitative CLTs for nonlinear functionals of stationary Gaussian fields related to the Breuer-Major theorem, with
presumably optimal rates of convergence in Total Variation distance. However, we choose to exploit the technique developed in \cite{BCP19} with a view to a subsequent generalization of our result to the interesting case of random eigenfunctions of the standard flat torus (arithmetic random waves), where the attainable limit laws include linear combinations of independent chi-square distributed random variables \cite{MPRW16, Cam19}. Moreover, it turns out that in order to obtain fruitful bounds via the Second order Poincar\'e inequality for the Gaussian approximation of our functional of interest $\tilde X_\ell$, the estimates on moments of products of powers of Gegenbauer polynomials should be much finer than those required by the approach developed in \cite{BCP19} (the one that we follow). 

It is worth stressing that our analysis does not give a complete picture. In particular, as one can see in the details of the proofs, the constant involved in $O_{\varepsilon}(\ell^{\frac{1-\varepsilon}2})$ depends on $\varepsilon$ and on the Malliavin norms of the functional up to a certain order $q_{\varepsilon}>1$ that depends on $\varepsilon$ and it grows as $\varepsilon$ approaches $0$. In order to provide the result for all $\varepsilon$, in our work we assume that the functional is in $\mathbb{D}^{k,p}$ for all $k\in \N$, $p\geq 1$. Then we can conclude that $\dTV(\tilde X_\ell, Z)=O_{\varepsilon}(\ell^{-\frac{1-\varepsilon}2})$ for any $0<\varepsilon<1$. Actually we could relax the assumption. But it is clear that, doing so, we can get the result only for \textit{some} $\varepsilon\in(0,1)$, possibly $\varepsilon$ close to $1$, and this would give the convergence but drastically ruin the rate. Furthermore, we cannot apply the result to the case of the excursion volume for small dimension $d$ because we believe that it is not even in $\mathbb{D}^{1,2}$. In this case, the excursion volume does not belong either to the domain of the Second order Poincar\'e inequality. On the other hand, for big enough $d$, being inspired by \cite{AP20}, it is reasonable to believe that this functional is regular enough to apply at least the Second order Poincar\'e inequality.

%
%

\section{Part \ref{Part2}: the fractional Ornstein-Uhlenbeck process in rough volatility modelling}

The focus of the second part of this Thesis are the rough stochastic volatility models, with particular emphasis on the fOU process. Over the past decade, the finance research has experienced significant changes. One of the most debated topics is the new kind of stochastic volatility models, called the "rough volatility models", introduced by Jim Gatheral, Thibault Jaisson, and Mathieu Rosenbaum \cite{GJR18}.

The fractional stochastic volatility models is firstly introduced to include the long memory property, that has long been always considered as a stylized fact of the volatility (\cite{DGE93}, \cite{AB97}, \cite{ABD01}). Originally, \cite{CR98} introduces the Fractional Stochastic Volatility (FSV) model in which the log-volatility is modelled as a fOU process with Hurst index $H>\frac12$, preserving the long-memory property. From \cite{CR98}, the interest for fractional volatility models increases. In this context, \cite{alos2007short, fukasawa2017short} consider a different aspect of the implied volatility, i.e. the general shape of the implied volatility surface and in particular the form of the at-the-money volatility skew. The implied volatility $\sigma_{BS}(k,\tau)$ of an option with log-moneyness $k$ and time to expiration $\tau$ is the value of the volatility parameter in the Black-Scholes formula that we need to fit the empirical options data on the market. As a function of $k$ and $\tau$, it generates a surface (for details see \cite{gatheral2006volatility}) which is not fitted by the classical diffusive volatility models. In particular, a relevant characteristic is the behaviour of the at-the-money volatility skew, i.e. $\psi(\tau)=|\frac{\partial}{\partial k} \sigma_{BS}(k,\tau)|_{k=0}$, that is well approximated by a power law function of $\tau$ \cite{gatheral2006volatility, bayer2016pricing}, near to $\tau^{-\frac 12}$, even if the topic is debated and some authors propose different considerations. For example in  \cite{GuyonSkew} the authors suggest that the skew is well approximated by a power-law for maturities above $1$ month, but it does not blow up for short maturities. In \cite{fukasawa2011asymptotic}, the author shows that the fractional stochastic volatility model driven by a fractional Brownian motion (fBm) with Hurst index $H$ generates a skew behaving approximatively as $\tau^{H-\frac 12}$. This fact leads to two observations. The paper \cite{fukasawa2011asymptotic} shows that the presence of jump is not necessary in volatility models to have the power-law explosion of the at-the-money volatility skew, contrary to what was stated earlier \cite{CW03}. The second observation follows from the analysis of the empirical skew. The fractional models driven by fBm need $H$ close to $0$ to generate a skew close to $t^{-\frac 12}$ in the volatility surface. This is in contrast to the FSV model introduced by \cite{CR98}, which had $H>\frac 12$. Taking advantage of this consideration, \cite{GJR18} study the time series of realized volatility of several assets, and show that the log-volatility behaves as fBm with Hurst index close to $0$. Then they introduce the Rough Stochastic Fractional Volatility (RSFV) model, in which the log-volatility is modelled by a fOU process with Hurst index $H<\frac 12$. In this way they add stationarity and show that the model is consistent with empirically observed time series. They validate their model estimating the smoothness parameter (i.e. the index $H$) on the realized volatility time series of several selected assets, finding values of H between to $0.08$ and $0.2$. The roughness in volatility models is also supported by many other works. In \cite{BLP21}, the authors provide a new class of models that are able to capture the roughness without giving up the slow decay of the autocorrelation function, that is instead lost in fOU based RSFV model. The authors consider the Gaussian semi-stationary (GSS) models \cite{barndorff2007ambit}, \cite{barndorff2009brownian} that allow to regulate the roughness of the process independently of the decay of the autocorrelation. Fitting this model to realized volatility time series of several assets, they find results that are consistent with both the roughness of the time series of realized volatility and the strong persistence.  
These several indications of the roughness in volatility time series motivate our work. The problem of modelling the volatility of several assets and studying the correlation of the volatility processes of these assets is largely studied in non rough-setting \cite{T06,M09}. There are also some works in rough-setting \cite{Rosenbaum_Tomas_2021, SHI2023173}, where authors propose multivariate model in which the roughness index is the same for all the components of the process. Inspired by \cite{GJR18}, we define a bivariate process whose components are fOU processes. Our starting point is \cite{multivar}. By taking advantage of the discussion related to the operator self-similar process \cite{LAHA198173, MAEJIMA1994139, HM82, DP11, lav09}, in \cite{multivar} the authors define a multivariate version of fBm, already investigated in \cite{CAA13}. The multivariate fBm given in \cite{multivar} is a $d$-variate process $(B^{H_1}, \ldots, B^{H_d})$ whose components are fBms with Hurst indexes $H_i$, $i=1,\ldots, d$. The interesting properties of this process are given by the possibility of choosing different Hurst indexes for the marginals and the rich cross-covariance structure. This structure depends on several correlation parameters: $(\rho_{ij})_{i\neq j=1}^d$, which represent the cross-correlations at time lag $0$, and $(\eta_{ij})_{i\neq j=1}^d$, which rule the property of time-(ir)reversibility of the process (see Section 5 in \cite{DP11}). It is worth to recall that for all $i\neq j$, $\rho_{ij}=\rho_{ji}$ and $\eta_{ij}=-\eta_{ji}$. We define the bivariate fractional Ornstein-Uhlenbeck process $Y=(Y^1, Y^2)$ considering a bivariate process whose components are fOU processes driven respectively by $B^{H_1}$ and $B^{H_2}$, where $(B^{H_1}, B^{H_2})$ is a bivariate fBm as in \cite{multivar}. An interesting aspect of this approach is the possibility of having different Hurst indexes in the components of the bivariate fOU, having in addition a stationary Gaussian process with a non-trivial cross-covariance structure. The latter is inherited by the bivariate fBm and it is the center of our discussion.
The presence of cross-correlations in time series is a very important issue in complex systems, as financial markets, and it is largely addressed (\cite{PDHGS01, PDHGS06,WPHS11} and many other works).
After the analysis of the covariance structure, we focus our work on statistical inference of the cross-correlation parameters. Statistical inference of parameters of univariate fOU process is a widely studied topic \cite{BIANCHI2023, KLB02, XZX11, T13, WY16, HNZ19, ES20, HH21, WXY23, Bolko_et_al}. In our work we provide estimators for the cross-correlation parameters $\rho=\rho_{12}=\rho_{21}$ and $\eta_{12}=-\eta_{21}$, under the simplifying assumption that the parameters of the marginal distributions are given. We provide two different kinds of estimators, both based on discrete-time observations of the process, that we show to be consistent. Each estimator is a suitable functional of discrete observations of the stationary Gaussian process $Y$. The asymptotic distribution of functionals of stationary Gaussian processes is widely studied field, approached in most cases with the Wiener-It\^o theory, see \cite{R60,  S63, T75, T79, DM79, M81, BM83, A94, A00} and many other works. These works have given rise to a considerable body of literature and novel techniques to approach the problem. The celebrated Fourth Moment Theorem \cite{NP12} has made it easier to handle this problem. For example remaining close to the problem of estimating the parameters of a one-dimensional fOU process, we recall \cite{HN10, HNZ19}. In this works the authors use Malliavin calculus techniques on the Wiener-It\^o multiple integrals to prove CLT theorems for their estimators. Our estimators live in the second order Wiener chaos, so we adopt a similar approach, specifically employing the Fourth Moment Theorem, to prove that, upon appropriate normalization, our estimators converge in various probability metrics to a Gaussian random variable, provided that $H=H_1+H_2$ is in a suitable interval of $(0,2)$. Following the approach in Section 7.3 in \cite{N13}, we also study cumulants of the first kind of our estimators for values of $H$ not included in the above analysis, providing the speed of convergence of the variance of the error, even if a proper CLT is still out of reach.

This second part of the Thesis also contains Chapter 7 based on \cite{GPP23}. Here, we consider short-time pricing asymptotics, i.e. pricing short maturity European options. This is a widely studied topic, as these short maturity pricing formulas provide methods for fast calibration, a quantitative understanding of the impact of 
model parameters on generated implied volatility surfaces, led to some widely used parametrizations of the volatility surface, and help in the choice of the most appropriate model to be fitted to data \cite{aitsahalia2019}.  Short maturity approximations are also used to obtain starting points for calibration procedures, which are then based on numerical evaluations. They have applications also to hedging, trading and risk management. We provide a short-time large deviations principle for stochastic volatility models, where the volatility is expressed as a function of Volterra process, not requiring the self-similarity assumption. The latter property, in approximation formulas and asymptotics, is usually key as it enables the translation of a small-noise result into a short-time one through space-time rescaling (\cite{Gu1}, \cite{Gu2}). Based on this procedure, several short-time formulas are available for rough volatility models, if the volatility process is expressed as a function of a fBM \cite{FZ17}, as a function of a Riemann-Liouville process (RLp), as in the rough Bergomi model \cite{BFGHS,FGP22,fukasawa2020}, or as a solution to a fractional SDE, as in the fractional Heston model \cite{forde_heston_H0}. However, obtaining short-time approximation formulas is more difficult  if volatility depends on a process which is not self-similar, such as the fOU process \cite{horvath_jacquier_lacombe_2019,GS17,GS18,GS20hedging,GS19,GS20}, or the log-modulated fBM (log-fBm) \cite{BHP20}. In our work we replace the self-similarity assumption by a weaker scaling property for the kernel of the process, so that it applies in particular to fOU. We prove this general result starting from \cite{CelPac},
where a pathwise Large Deviations principle (LDP) for the log-price was proved when the volatility is a function of a family of Volterra processes, and the price is solution to a scaled differential equation. Here, under suitable short-time asymptotic assumptions on the Volterra kernel, we prove a short-time LDP for the log-price process.

 We apply such an LDP to two notable examples of non-self-similar rough volatility models: models where it is given as a function of a fractional Ornstein–Uhlenbeck and models where the volatility is given as a function of a log-modulated fractional Brownian motion (see \cite{BHP20}). In both cases, we derive consequences for short-maturity European option prices implied volatility surfaces and implied volatility skew. We also provide a numerical study of the accuracy and dependence on relevant parameters of our results in the fOU case.

The second part of the Thesis is divided in three chapters. Chapters \ref{Biv_fOU_3} and \ref{Correlation_section} is based on \cite{DGP}, which is in preparation. In Chapter \ref{Biv_fOU_3}  we introduce the 2fOU process and its correlation structure, studying its properties. In \ref{Correlation_section} we define and study the estimators of the correlation parameters. In Chapter \ref{Chapter:Short-time}, based on the paper \cite{GPP23}, we provide a short-time large deviation principle for Volterra-driven stochastic volatilities, where the self similarity is replaced by suitable assumption on the kernel functions of these models. As applications of this asymptotic result we consider the fOU process and the log-modulated rough models.   

\medskip
\subsection*{Chapters \ref{Biv_fOU_3} and \ref{Correlation_section}}
Chapter \ref{Biv_fOU_3} is dedicated to the introduction of a bivariate model that generalizes the univariate fOU process. The univariate fOU process and its statistical properties are widely studied. Let $(\Omega, \mathcal F, \P)$ be a probability space and, for $H\in (0,\frac 12) \cup(\frac 12, 1)$, let $B^H$ a fBm on $\R$. Let us fix $\a, \nu\in \R_+$ and $\psi\in L^0(\Omega)$. In \cite{cheridito2003} it is proved that, for all $t$, the Langevin equation
\begin{equation*}
	Y_t^H = \psi - \alpha\int_{0}^t Y_s ds + \nu B_t^H, \quad t\geq 0
\end{equation*}
is well defined because the integral above exists as a path-wise Riemann-Stieltjes integral, and its solution is given by
$$
Y_t^{H, \psi}= e^{-\a t}\Big(\psi+\nu \int_{0}^t e^{\a u}dB_u^H \Big),\,\,\, t\geq 0.
$$
It is the unique almost surely continuous solution. In particular, the process
\begin{equation*}
	Y_t^H = \nu\int_{-\infty}^t e^{-\alpha(t-u)} dB_u^H, \quad t\in \R,
\end{equation*}
is a solution with the initial condition $\psi = Y_0^H=\nu\int_{-\infty}^0 e^{-\a u} dB^H_u$. The process $Y^H$ is the stationary fOU. The fOU process is neither Markovian nor a semimartingale for $H\neq \frac 12$, but remains Gaussian and ergodic when the mean-reverting parameter $\a$ is strictly positive (see \cite{cheridito2003}). Moreover, for $H>\frac 12$, it has long-range dependence. In recent years it assumes an important role in rough stochastic volatility models, when \cite{GJR18}, estimating the realized volatility from high frequency data, show that the log-volatility behaves as a fBm with Hurst index $H<\frac 12$, and introduce the RFSV model based on the fOU process to add the stationary property. They prove that their model exhibits consistence with financial time series data, opening to the possibility to obtain improved forecasts of realized volatility. 

In light of this paper, with the goal of modelling simultaneously more that one volaility time series, we introduce the 2fOU process. Our definition leverages a multivariate version of fractional Brownian motion outlined in \cite{multivar}, \cite{AC11}, which we recall in Section \ref{MultifBM_section}. In these works, the authors propose a multivariate fractional Brownian motion (mfBm) $(B^{H_1},\ldots, B^{H_d})$ with a rich correlation structure, allowing different Hurst indexes $(H_1,\ldots, H_d)$ for each process component while preserving inter-component correlations. This structure depends on several correlation parameters, $(\rho_{ij})_{i,j=1}^d$, that represent the cross-correlations at time lag $0$, and $(\eta_{ij})_{i,j=1}^d$, that rule the property of time-(ir)reversibility of the process. Then we naturally introduce the the $2$-fOU process as follows: for $H_1, H_2\in(0,\frac 12)\cup(\frac 12, 1)$, we consider a bivariate fBm $(B^{H_1}, B^{H_2})$ as described in \cite{multivar}, and define the 2fOU process $Y=(Y^1, Y^2)$  by the integrals 
$$
Y^i_t=\nu_i\int_{-\infty}^t e^{-\a_i(t-s)} dB_s^{H_i}.
$$

Two notable characteristics of our process are the flexibility of choosing two different Hurst indexes for its components without losing the possibility of a non-trivial correlation between the components and the stationarity. In light of \cite{GJR18}, the 2fOU process can be used to model and compare the log-volatility of two (or many)  correlated stocks.

In Section \ref{Bivariate_fouSection} we introduce the process and study the main properties of its covariance structure, that depends on the parameters of the components of the process, and on the correlation parameters $\rho=\rho_{12}=\rho_{21}$ and $\eta_{12}=-\eta_{21}$. Then the correlation parameters are not $4$, but actually they are $2$.  This process remains stationary and inherits the correlation structure of the underlying noise $(B^{H_1}, B^{H_2})$, preserving several properties observed in the univariate case. In the univariate fOU, when the time lag diverges, the auto-correlation function decades with the same rate as the auto-correlation of the increments of the underlying fBm \cite{cheridito2003}. Then we prove that it holds also for the cross-correlation function of the components of the $2$-fOU. 
	\begin{teo548}
	Let $H_1, H_2 \in (0,\frac{1}{2})\cup (\frac{1}{2},1]$, $H=H_1+H_2\neq 1$ and $N\in \N$. Then for fixed $t\in \R$, as $s\to \infty$, $i,j\in \{1,2\}$, $i\neq j$, 
	\begin{align*}
		&\Cov(Y_t^i,Y_{t+s}^j)=\notag \\
		&=\frac{\nu_1\nu_2(\rho+\eta_{ji})}{2(\a_1+\a_2)}\sum_{n=0}^N \Big(\frac{(-1)^n}{\a_j^{n+1}}+\frac{1}{\a_i^{n+1}}\Big)  \Big ( \prod_{k=0}^{n+1} (H-k) \Big)s^{H-2-n}+O(s^{H-N-3}).
	\end{align*}
	
	\noindent When $H=1$ we have
	\begin{align*}
		&\Cov(Y_t^i,Y_{t+s}^j)=\notag \\
		&=\frac{\nu_1\nu_2\eta_{ji}}{2\a_1\a_2}\frac{1}{s}+\frac{\nu_1\nu_2\eta_{ji}}{2(\a_1+\a_2)}\sum_{n=1}^N \Big(\frac{(-1)^n}{\a_j^{n+1}}+\frac{1}{\a_i^{n+1}}\Big)  \Big( \prod_{k=0}^{n-1} (-k-1) \Big)s^{-1-n}+O(s^{-N-2}).
	\end{align*}
\end{teo548}

Next, we investigate the behavior of the cross-correlation as the time-lag tends to zero. 
	\begin{lemma5412}
	For all $t\in \R$ and $s\to 0$, when $H\neq 1$ and $i\neq j$, we have that
	\begin{align*}
		\Cov(Y_t^i, Y_{t+s}^j)&=\Cov(Y_0^i, Y_0^j)-\nu_1 \nu_2 \frac{\rho-\eta_{ij}}{2}s^H+ \\&+\Big(-\a_j \Cov(Y_0^1, Y_0^2)+\a_i^{1-H} \Gamma(H+1) \nu_1\nu_2\frac{\rho-\eta_{ij}}{2}\Big) s+O(s^{\min\{1+H, 2\}}). 
	\end{align*}
	When $H=1$, we have
	\begin{align*}
		&\Cov(Y_t^i, Y_{t+s}^j)=\Cov(Y_0^i Y_0^j)-\nu_1\nu_2 \frac{\eta_{ij}}{2} s\log s +o(s^2\log s).
	\end{align*}
	
\end{lemma5412}

The asymptotic behaviour of the cross-covariance when the time-lag diverges and when the time-lag tends to $0$ is crucial in the definition of our estimators of the parameters $\rho$ and $\eta_{12}$.

In Chapter \ref{Correlation_section} we introduce two different estimators of correlation parameters $\rho$ and $\eta_{12}$. These estimators depend on the other parameters associated to the components of the process, such as $\nu_1, \nu_2, \a_1, \a_2, H_1, H_2$. For simplicity, we decide to estimate the parameters $\rho$ and $\eta_{12}$ considering the parameters of marginals as given. Even if not ideal, this is a reasonable starting point since the problem of estimating a one-dimensional fOU has been estensively studied in the literature \cite{BIANCHI2023, KLB02, XZX11, T13, WY16, HNZ19, ES20, HH21, WXY23}. In \cite{WXY23} the authors propose a two-step estimation method based on the method of moments. 
%
%
In \cite{Bolko_et_al} the authors propose a Generalized Method of Moments to joint estimate the parameters of their stochastic volatility model. They consider a general setting, where the log-volatility is a general stationary Gaussian process, without restriction on Markovian or semi-martingale properties. In this setting they study the case of fOU and employ their estimators in an empirical study, obtaining that $H$ is estimated very small for all stock indexes that they consider.

In Section \ref{Firstkind} we introduce the estimators $\hat \rho_n$ and $\hat \eta_{12,n}$ obtained inverting the expressions of the cross-correlation between the components of the process when the time lag $s$ is positive or equal to $0$, given in next Lemma \ref{crss0} and next Lemma \ref{CrossYY}, as follows
 \begin{equation*}
	\rho=a_{1}(s) \,\Cov(Y_t^1,Y_t^2)+a_{2}(s)\,\Cov(Y_{t+s}^1,Y^2_t)+a_3(s)\,\Cov(Y^1_{t}, Y_{t+s}^2)
\end{equation*}
and
\begin{equation*}
	\eta_{12}=b_1(s)\,\Cov(Y_t^1,Y_t^2)+b_2(s)\,\Cov(Y_{t+s}^1,Y^2_t)+b_3(s)\,\Cov(Y_t^1, Y_{t+s}^2)
\end{equation*}

The coefficients depend on the parameters of the marginal distributions. 
We consider a time-grid with $t_k=k$ for $k=0, \ldots, n$ on the interval $[0,n]$, and we deduce the estimators $\hat \rho_n$ and $\hat \eta_{12,n}$ taking the empirical cross-covariances as follows:
	\begin{equation*}
	\hat \rho_n=a_1(s)\,\frac{1}{n}\sum_{j=1}^n Y_j^1 Y_j^2+a_2(s)\,\frac{1}{n}\sum_{j=1}^{n-s }Y_{j+s}^1 Y_j^2+ a_3(s)\,\frac{1}{n}\sum_{j=1}^{n-s} Y_{j}^1 Y_{j+s}^2
\end{equation*}
and
\begin{equation*}
	\hat \eta_{12,n}=b_1(s) \,\frac{1}{n}\sum_{j=1}^n Y_j^1 Y_j^2+b_2(s)\,\frac{1}{n}\sum_{j=1}^{n-s} Y_{j+s}^1 Y_j^2+b_3(s)\,\frac{1}{n}\sum_{j=1}^{n-s} Y_{j}^1 Y_{j+s}^2.
\end{equation*}

 We prove that these estimators are asymptotically unbiased and consistent (in the sense that $\hat \rho_n$ and $\hat \eta_{12,n}$ converge in $L^2(\P)$ to $\rho$ and $\eta_{12}$). Then we study the asymptotic distribution of a suitable normalization of $\hat \rho_n-\rho$ and $\eta_{12, n}-\eta_{12}$. In Section \ref{Asymptotic_distribution1} we prove the result that $\sigma_\rho^2=\underset{n\to+\infty}{\lim}\Var(\sqrt{n}(\hat \rho_n-\rho))$ and $\sigma_\eta^2=\underset{n\to+\infty}{\lim}\Var(\sqrt{n}(\hat \eta_{12, n}-\eta_{12}))$ are finite and positive, and then we prove the following theorem:
 	\begin{teo6110} Let $\sigma_\rho^2=\underset{n\to+\infty}{\lim}\Var(\sqrt{n}(\hat \rho_n-\rho))$ and $\sigma_\eta^2=\underset{n\to+\infty}{\lim}\Var(\sqrt{n}(\hat \eta_{12, n}-\eta_{12}))$. Then, for $H<\frac 32$,
 	$\sqrt{n}(\hat\rho_n -\rho)$ converges to $N_{\rho}\sim\mathcal N(0,\sigma_\rho^2)$ and $\sqrt{n}(\hat\eta_{12,n} -\eta_{12})$ converges to $N_{\eta}\sim\mathcal N(0,\sigma_\eta^2)$ in Kolmogorov, Wasserstein and Total Variation distance. Additionally 
 	\begin{equation*}
 		\sqrt{n}(\hat \rho_n- \rho, \hat \eta_{12, n}-\eta_{12}) \overset{d}{\to} (N_\rho, N_\eta)
 	\end{equation*} 
 	where $\Cov(N_\rho, N_\eta)=\lim_{n\to +\infty} n\E[(\hat \rho_n-\rho)(\hat \eta_{12, n}-\eta_{12})]$.
 \end{teo6110}
 
 The proof technique is based on the Fourth Moment theorem (see \cite{NP12}). The random sequences $\hat\rho_n-\rho, \hat\eta_{12,n}-\eta_{12}$ can be written as double Wiener-It\^o integrals with respect to a Gaussian noise. Then, these sequences belong to the second order chaos. We apply the Fourth Moment theorem for $\sqrt{n}(\hat\rho_n-\rho), \sqrt{n}(\hat\eta_{12,n}-\eta_{12})$ to conclude their convergence to two Gaussian random variables in various probability metrics and then in distribution. We prove an analogous result when $H=\frac 32$, with a different normalization (see Theorem \ref{Convergence_H_32}).

In Section \ref{NON_central} we discuss the case $H_1+H_2>\frac 32$. We prove that 
\begin{align*}
	&\lim_{n\to +\infty}n^{4-2H}\Var(\hat\rho_n)=\sigma_1^2>0
	&\lim_{n\to +\infty}n^{4-2H}\Var(\hat\eta_{12,n}-\eta_{12})=\sigma_2^2>0
\end{align*}
and we provide the limits of the cumulants of all orders for the sequences $n^{2-H}(\hat\rho_n-\rho)$ and $n^{2-H}(\eta_{12,n}-\eta_{12})$ (see Theorem \ref{cumulantNC}). The approach of this Section follows Section 7.3 in \cite{N13}.
 
 In Section \ref{SecondKind} we introduce two estimators $\tilde \rho_n$ and $\tilde\eta_{12,n}$ based on the asymptotic behaviour of the cross-covariances when the time-lag tends to $0$, given in next Lemma \ref{short-time-inverse}. The first approach based on $\hat \rho_n$ and $\hat \eta_{12,n}$ is heavily affected by the procedure of firstly estimating the parameters of the marginals distribution. This is due to the fact that the coefficients that appear in $\hat \rho_n$ and $\hat \eta_{12,n}$ depend on the parameters of the marginal distributions, in particular the mean reversion parameters $\a_1, \a_2$. Many of the available univariate estimation procedures of the mean reversion parameter suffer from a bias. Together with the fact that the "real" mean reversion parameter is often very close to $0$ (see \cite{GJR18}), this makes the estimation of the mean reversion parameter problematic. For this reason we propose a method for estimating the cross-correlation parameters that does not rely on previous estimates of $\alpha$.
 Indeed, the advantage of the estimators $\tilde \rho_n$ and $\tilde\eta_{12,n}$ is that they do not depend on $\a_1, \a_2$, then they are not affected by the bias in the estimate of the mean reversion. They are constructed on a time-grid $t_k=k\Delta_n$, $k=0,\ldots, n$ and they are defined as
 \begin{equation*}
 	\tilde\rho_n=\frac{1}{\nu_1\nu_2n\Delta_n^H} \sum_{k=0}^{n-1}\Big(Y_{(k+1)\Delta_n}^1-Y_{k\Delta_n}^1\Big)\Big(Y_{(k+1)\Delta_n}^2-Y_{k\Delta_n}^2\Big)
 \end{equation*}
 and 
 \begin{equation*}
 	\tilde\eta_{12,n}=\frac{1}{\nu_1\nu_2n\Delta_n^H} \sum_{k=0}^{n-1}\Big(Y_{k\Delta_n}^1Y_{(k+1)\Delta_n}^2-Y_{(k+1)\Delta_n}^1Y_{k\Delta_n}^2\Big).
 \end{equation*} We define $\tilde \rho_n$ for all $H\in (0,2)\setminus\{1\}$ and $\tilde \eta_{12,n}$ for $H<1$ and we prove that these estimators are asymptotically unbiased and consistent estimators when $T_n=n\Delta_n\to +\infty$, $\Delta_n\to 0$ and $n\Delta_n^2\to 0$ as $n\to +\infty$. In Section \ref{theory_tilde_rho}, requiring that $n\Delta_n^{\min\{2,4-2H\}}\to 0$ and $H<\frac 32$, we prove the following result

\begin{teo628}
	Let $\sigma^2=\underset{n\to +\infty}{\lim} \Var(\sqrt{n}(\tilde \rho_n-\rho))$. Then
	\begin{equation*}
		\sqrt{n}(\tilde \rho_n-\rho)\overset{d}{\to}N(0,\sigma^2).
	\end{equation*}
	
\end{teo628}
\noindent The proof technique is again based on the Fourth moment Theorem. 
For $\hat\rho_n$, $\hat\eta_{12, n}$ and $\tilde \rho_n$ the problem of the asymptotic distribution for $H>\frac 32$ remains open. For $\hat\rho_n$, $\hat\eta_{12, n}$ we provide the limits of the cumulants of all orders. Therefore, if they converge in distribution, in view of \cite{PN12}, we know that the limit is an independent sum of a normal random variable and a double Wiener integral. For $\tilde \eta_{12,n}$, we think that for $H<1$ a similar approach can be used to prove asymptotically normality, but the particular formulation of $\hat \eta_{12, n}$ makes the computations more difficult. 

\subsection*{Chapter \ref{Chapter:Short-time}}
Chapter \ref{Chapter:Short-time} is based on the paper \cite{GPP23}. In many instances, in order to produce rough volatility dynamics, the volatility process is expressed as a function of a Volterra process, i.e. a suitable deterministic kernel integrated against a Brownian motion. In this context a very useful feature of such kernel and of the corresponding fractional process is self-similarity. 

When looking at approximation formulas and asymptotics, self-similarity is usually key as it enables the translation of a small-noise result into a short-time one through space-time rescaling. This can then be used to price short maturity options (see the discussion at the end of Section 3 in \cite{Gu2}, and \cite{Gu1}). Based on this procedure, several short-time formulas are available for rough volatility models, if the volatility process is expressed as a function of a fBm \cite{FZ17}, as a function of a Riemann-Liouville process (RLp), as in the rough Bergomi model \cite{BFGHS,FGP22,fukasawa2020}, or as a solution to a fractional SDE, as in the fractional Heston model \cite{forde_heston_H0}.

However, obtaining short-time approximation formulas is more difficult  if volatility depends on a process which is not self-similar, such as the fOU process \cite{horvath_jacquier_lacombe_2019,GS17,GS18,GS20hedging,GS19,GS20}, or the log-modulated fBM \cite{BHP20}. 

In Chapter \ref{Chapter:Short-time} we address this issue, providing a short-time large deviation principle (LDP) for Volterra-driven stochastic volatilities, where the usual self similarity assumption is replaced by a weaker scaling property for the kernel, that needs to hold only in asymptotic sense (see conditions {\bf(K1)} and {\bf (K2)} in Chapter \ref{Chapter:Short-time}). 
We prove this general result starting from \cite{CelPac}, where a pathwise LDP for the log-price was proved when the volatility is  function of a  family of Volterra processes, and the price is solution to a scaled differential equation. Here, under suitable short-time asymptotic assumptions on the Volterra kernel, we prove a short-time LDP for the log-price process.

With our general result, we analyse more in depth models with volatility given as a function of fOU or log-fBM, neither of which is self-similar. However, we note that both these processes can be seen as a perturbation of self-similar processes, so that our general result can be applied, assuming that the price process is a martingale and a moment condition on the price.

\smallskip

The first class of processes to which we apply our LDP are \emph{log-modulated rough stochastic volatility models}, introduced in \cite{BHP20} as a logarithmic perturbation of a more standard power-law Volterra stochastic volatility model, with volatility depending on a log-fBM. These models allow for the definition of a ``true'', continuous volatility process with roughness (Hurst) parameter $H\in[0,1]$ (including the ``super-rough'' case $H=0$), at the price of losing the self-similar structure of the power-law kernel. Differently from our LDP setting,  however, in \cite{BHP20} Edgeworth-type asymptotics are considered, meaning that log-moneyness is of the form $x \sqrt{t}$ ($t$ representing the time to maturity), while in order to observe a large deviations behavior we look here at a suitable log-moneyness regime (cf. equation \eqref{eq:log:moneyness}). This regime is consistent with Forde-Zhang LDP for rough volatility
\cite{FZ17} and the related large deviation results. 
When $H>0$, we obtain a short-time LDP for the log-price process and consequent short-time option pricing, implied volatility and implied skew asymptotics. For this class of processes  the rate function only depends on the self-similar power-law kernel, while the speed depends also on the modulating logarithmic function.
It is shown in \cite{BHP20} that when $H=0$ the implied volatility skew explodes as $t^{-1/2}$ (with a logarithmic correction), realising the model free bound in \cite{LEE}. Even though our proof only holds in the $H>0$ case, the expression we obtain for skew asymptotics, computed for $H=0$, is consistent with this model free bound. 
We note that \cite{BaPa} have recently proved that in the $H=0$ case, even if the log-modulated model is well defined, an LDP cannot hold. 

\medskip

The second class of models to which we apply our LDP have a stochastic volatility given by a function of a \emph{fractional Ornstein-Uhlenbeck} process, as in Chapters \ref{Biv_fOU_3} and \ref{Correlation_section}. However here we use this volatility for pricing options, while in Chapters \ref{Biv_fOU_3} and \ref{Correlation_section} we were using it for modelling time series. We find that, in short-time, such model behaves exactly as the analogous model, with volatility process given as the same function, computed on a fBM (i.e., the model studied in \cite{FZ17}). More precisely, we mean that the two models satisfy an LDP with same speed and rate function.
It follows that also the short-time implied skew (computed as a suitable finite difference) is the same for fOU and fBM-driven stochastic volatility models.
For small time scales, fOU is, in a sense, close to fBM (see equation \eqref{eq:FOU}), even though fOU is not self-similar.
It is not uncommon when dealing with rough stochastic volatility, starting from the groundbreaking work \cite{GJR18}, to consider at times fBM, at times fOU, depending on which is most convenient for the problem at hand. Our result can be seen as a justification of this type of procedure, as it shows that pricing vanilla options with one or the other volatility does not matter (too much) for short maturities. Moreover, the fOU process is the most standard choice for a \emph{stationary process} with a fractional correlation structure. This is one of the reasons why it has been used as volatility process for option pricing and related issues\footnote{note that both fBM and RLp (as in Rough Bergomi) are non-stationary and give rise to non-stationary volatility processes}  \cite{horvath_jacquier_lacombe_2019,GS17,GS18,GS20hedging,GS19,GS20}. 

From our short-time LDP we formally derive the corresponding moderate deviations result, consistent with the one holding for self-similar rough volatility \cite{BFGHS}. We provide numerical evidence for both these large and moderate deviations results and for the skew asymptotics. We investigate on simulations how the choice between fOU and fBM dynamics in the volatility affects volatility smiles and skews, how accurate are our approximations, and how they depend on the mean reversion parameter.

	\chapter{Background}\label{background}
		\section{Some notions of functional analysis}\label{functAnalysis}
	In this section, we aim to summarize certain concepts from functional analysis that are pertinent to elucidating the work. Throughout this section, we denote by $\H$ and $\K$ two separable Hilbert spaces over the real numbers. Additionally, we use $\{e_j\}_{j\in\N}$ and $\{u_\ell\}_{\ell\in\N}$ to represent their respective orthonormal bases.
	
	We introduce $\LL(\H,\K)$ to denote the collection of linear, bounded operators mapping from $\H$ to $\K$. This set forms a Banach space equipped with the norm defined as follows:
	
	$$
	\|A \|_{op}:=\sup_{\|h\|_\H \leq 1} \|Ah\|_\K.
	$$
	\subsection{Trace class operators}
	\begin{definition}
		Let $A\in \LL(\H,\H)$. $A$ is said to be: 
		\begin{itemize}
			\item[1.] positive (also denoted with $A\geq 0$) if $\<Ah,h\>_\H \geq 0$ for every $h \in \H$;
			\item[2.] symmetric if $A^*=A$, that is, $\<Ah,g\>_\H=\<h,Ag\>_\H$ for every $h,g \in \H$;
			\item[3.] compact if it maps the unit ball of $\H$ into relatively compact subset of $\H$.
		\end{itemize}
	\end{definition}
	Let $A \in \LL(\H,\H)$ be a positive operator, we define the trace of $A$ in the following way:
	\begin{equation}\label{Tr}
		\Tr (A):= \sum_{j \geq 1} \<Ag_j,g_j\>_H.
	\end{equation}
	We recall the following result, that proves the existence of the square root of a positive operator. The proof can be found in \cite{Reed}.
	\begin{proposition}
		Let $A\in\LL(\H,\H)$, $A\geq 0$. Then there is a unique $B\in \LL(\H,\H)$ with $B\geq 0$ and $B\circ B=A$, where $\circ $ denotes the composition of two operator. We denote $B$ with $\sqrt{A}$ or $A^{\frac 12}$. Furthermore, $B$ commutes with every bounded operator which commutes with $A$.
	\end{proposition}
	\begin{definition}
		Let $A\in \LL(\H,\H)$. Let $|A|=(A^*\circ A)^{\frac 12}$ denoting the absolute value of $A$, where $A^*$ is the transpose of $A$. $A$ is a trace class operator if 
		$$
		\|A\|_1=\Tr |A|:= \sum_{j \geq 1} \<(A^*\circ A)^{\frac{1}{2}}g_j,g_j\>_\H<+\infty
		$$
		is absolutely convergent.
	\end{definition}
	We define $\LL_1(\H,\H)$ the set of all trace class operators from $\H$ to $\H$. We note that if $A$ is symmetric and positive, $A=|A|$ and so $\Tr (A)=\Tr |A|$.
	\begin{remark}
		\begin{itemize}
			\item The sum in \eqref{Tr} is independent of the chosen of orthonormal basis. Indeed, let $\{g_i\}_{i\geq1}$, $\{z_j\}_{j\geq 1}$ be two orthonormal bases of $\H$. Then
			\begin{align*}
				&\sum_{i \geq 1} \<Ag_i,g_i\>_\H= \sum_{i\geq 1} \| A^{\frac{1}{2}}g_i \|^2_\H=\sum_{i\geq 1} \Big(\sum_{j\geq 1}| \<z_j, A^{\frac{1}{2}}g_i\>_\H|^2   \Big)\\
				&=\sum_{i\geq 1} \Big(\sum_{j\geq 1}| \<A^{\frac{1}{2}} z_j, g_i\>_\H|^2   \Big)=\sum_{j\geq 1} \| A^{\frac{1}{2}}z_j \|^2_\H=\sum_{j \geq 1} \<Az_j,z_j\>_\H.
			\end{align*}
			\item the space $\LL_1(\H,\H)$ endowed with the norm $\|\cdot\|_1$  is a Banach space. Moreover $\LL_1(H,H) \subset \LL(H,H)$.
			\item Any positive, symmetric and trace class operator is compact.
			\item let $A\in\LL(\H, \H)$ be compact and normal (i.e. $A^*\circ A=A\circ A^*$). As a consequence of the Spectral theorem for normal operator, there exists an orthonormal basis of $\H$ consisting of eigenvectors of $A$. Each eigenvalue is real. If $A$ is positive, each eigenvalue is non negative. Moreover, one can be reorder the sequence of eigenvalues $\{\lambda_j\}_{j\in\N}$ so that $\lim_{n\to \infty} \lambda_n=0$.  
			\item If $A\in \LL(\H,\H)$ is positive, symmetric and trace class operator, we consider the orthonormal basis $\{g_i\}_{i\geq1}$ of eigenvectors of $A$ with eigenvalues $\{\lambda_i\}_{i \geq 1}$, which are non negative and tend to $0$. Then 
			$$
			\|A\|_1=\Tr (A)=\sum_{i \geq 1} \<Ag_i, g_i\>_\H=\sum_{i \geq 1} \lambda_i \<g_i, g_i\>_\H=\sum_{i \geq 1} \lambda_i < +\infty.
			$$
			In particular
			\begin{equation*}
				Ax= \sum_{i\geq 1} \lambda_i \<x, g_i\>_H g_i.
			\end{equation*}
			\item If $A$ is a positive, symmetric and invertible operator in $\LL(\H,\H)$, the inverse $A^{-1}$ is a positive, symmetric in $\LL(\H,\H)$. If $A \in \LL_1(\H,\H)$ then $A^{-1}$ is not in $\LL_1(\H,\H)$, unless $\H$ has finite dimension. Indeed, if we consider the set $\{\lambda_i\}_{i \geq 1}$ of eigenvalues of $A$, the set of eigenvalues of $A^{-1}$ is $\{\lambda_i^{-1}\}_{i\geq 1}$ and 
			$$
			\sum_{i \geq 1} \<Ae_i, e_i\>_\H=\sum_{i \geq 1} \lambda_i^{-1} = + \infty
			$$
			where $\{e_i\}_{i \geq 1}$ is the set of eigenvectors of $A^{-1}$, unless the above sum runs on a finite number of terms.
		\end{itemize}
	\end{remark}
	
	\subsubsection{Tensor product between Hilbert spaces}\label{tensorDef}
	Let $\H$, $\K$ two Hilbert spaces. 
	\begin{definition}\label{simpleTensor}
		Let $h\in \H$ and $k\in\K$. The simple tensor product $h\otimes k$ is the rank one operator from $\H^*$ to $\K$ such that
		$$
		h\otimes k ( f^*)= f^*(h) k, \qquad{\forall f^*\in \H^*}. 
		$$
	\end{definition}
	From Definition \ref{simpleTensor} we can deduce a linear identification between the tensor product space $\H \otimes \K$ and the space of finite rank operator from $\H^*$ to $\K$.
	By Riesz representation theorem, trough a natural isomorphism we can identify $\H^*$ with $\H$, then we can see the simple tensor product as an operator from $\H$ to $\K$ such that
	\begin{equation}\label{simpletensor_id}
		h\otimes k (f)=\<h, f\>_{\H} k, \qquad{\forall f\in\H}.
	\end{equation}
	
	The tensor product space $H\otimes K$ is an Hilbert space endowed with the inner product 
	$$
	\<h_1\otimes k_1, h_2\otimes k_2\>_{\H\otimes \K}=\<h_1,h_2\>_\H \<k_1,k_2\>_{\K}, \qquad{h_1, h_2 \in \H, k_1, k_2 \in \K}.  
	$$ 
	If $\H$ and $\K$ have orthonormal bases $\{e_j\}_j$ and $\{g_i\}_i$, then the set of $\{e_j\otimes g_i\}_{i,j}$ is an orthonormal basis of $\H\otimes \K$. 
	\begin{example}
		Let $X, Y$ be two measure spaces, and $\mu, \nu$ be the respectively measures. The Hilbert space $L^2(X\times Y)$ is the space of square integrable functions on $X\times Y$ with respect to the product measure $\mu\times \nu$. We can construct a bilinear map 
		$
		L^2(X)\times L^2(Y)\to L^2(X\times Y)
		$ 
		such that $(f, g)\to f g $, where $f g(x, y)=f(x)g(y), x\in X,y\in Y$. The function $f g$ is in $L^2(X\times Y)$. It can be proved that the image of this bilinear map is dense in $L^2(X\times Y)$ when $L^2(X)$ and $L^2(Y)$ are separable. Then $L^2(X\times Y)$ is isomorphic to $L^2(X)\otimes L^2(Y)$.  
	\end{example}  
	
	Let us consider the tensor product space $\H\otimes \H$. One can notice that $h_1\otimes h_2$ is in general different from $h_2\otimes h_1$. Then we recall the definition of symmetric tensor product. 
	\begin{definition}\label{simple_symmetric}
		Let $h_1, h_2\in\H$. The simple symmetric tensor product between $H_1$ and $h_2$, denoted by $h_1\tilde \otimes h_2$, is 
		$$
		h_1\tilde \otimes h_2 =\frac12 h_1\otimes h_2+\frac 12 h_2\otimes h_1.
		$$
	\end{definition}
	
	Let us now consider Hilbert spaces $\K_1,\ldots, \K_n$, $n\in\N$. Let $k_i\in\K_i,i\in\N$. Then the simple tensor product of order $n$ is inductively given by
	$$
	k_1\otimes k_2 \qquad{n=2} 
	$$   
	$$
	(k_1\otimes k_2 \otimes \cdots \otimes k_{n-1})\otimes k_n. 
	$$    
	Then we can define the tensor product space of order $n$ denoted by $\K_1\otimes \cdots \otimes \K_n$. 
	
	Let us consider the tensor product space of order $n$ of $\H$ given by $\underset{n \text{ times}}{\H\otimes \cdots \otimes \H}$ and denoted with $\H^{\otimes n}$. A generic element of $\H^{\otimes n}$ is $h_1\otimes h_2\otimes \cdots \otimes h_n$, $h_1,\ldots, h_n\in\H$. The symmetric tensor product of $h_1\ldots, h_n$ is given by 
	$$
	h_1\tilde \otimes h_2 \tilde \otimes \cdots \tilde \otimes h_n =\frac{1}{n!}\sum_{\sigma} h_{\sigma(1)}\otimes h_{\sigma(2)}\otimes \cdots h_{\sigma(n)}, 
	$$ 
	where the sum runs over all possible permutations $\sigma$ of $\{1,\ldots, n\}$. 
	The Hilbert space of the symmetric tensor products of order $n$ of $\H$ is denoted with $\H^{\odot n}$.
	
	\begin{example}
		Let $\H=L^2(X, \F, \mu)$ where $(X, \F, \mu)$ is a non-atomic measure space. Then $\H^{\otimes n}$ can be identified with $L^2(X^n, \F^n, \mu^n )$, where $\mu^n$ is the product measure on $X^n$. The space $\H^{\otimes n}$ is given by the symmetric function of $\H^{\otimes n}$, i.e. $f$ is square integrable function on $\H^{n}$ such that $f = \tilde f$ a.e. where
		$$
		\tilde f(x_1,\ldots, x_n)=\frac{1}{n!}\sum_{\sigma}f(x_{\sigma(1)},\ldots, x_{\sigma(n)}). 
		$$  
	\end{example}
	
	\begin{example}
		Let $\H$ a generic real separable Hilbert space. Then we can fix an orthonormal basis $\{e_i\}_i$ of $\H$ and then an orthonormal basis of $\H^{\otimes n}$ given by $\{e_{i_1} \otimes \cdots \otimes e_{i_n}\}_{i_1,\ldots, i_n\geq 1}$. Let $h \in \H^{\otimes n}$, then 
		$$
		h=\sum_{i_1,\ldots, i_n\geq 1} \<h,e_{i_1}\otimes \ldots \otimes e_{i_n}\>_{\H^{\otimes n}} e_{i_1}\otimes \ldots \otimes e_{i_n}. 
		$$
		The canonical symmetrization of $h$ denoted with $\tilde h$ is 
		$$
		\tilde h=\frac{1}{n!}\sum_{\sigma}\sum_{i_1,\ldots, i_n\geq 1} \<h,e_{i_1}\otimes \ldots \otimes e_{i_n}\>_{\H^{\otimes n}} e_{i_{\sigma(1)}}\otimes \ldots \otimes e_{i_{\sigma(n)}}.
		$$
	\end{example}
	Let us now introduce the concept of contraction. 
	\begin{definition}\label{contraction}
		Let $1\leq p\leq q\in\N$, $\H$ be a real separable Hilbert space and $\{e_i\}_{i}$ be an orthonormal basis of $\H$. Let $r=0, \ldots, \min\{p,q\}=p$. The contraction of order $r$ of $e_{i_1}\otimes\cdots \otimes e_{i_p}$ and $e_{j_1}\otimes\cdots \otimes e_{j_q}$ is the element of $\H^{\otimes( p+q-2r)}$ given by
		$$
		(e_{i_1}\otimes\cdots \otimes e_{i_p})\otimes_0 (e_{j_1}\otimes\cdots \otimes e_{j_q})=e_{i_1}\otimes\cdots \otimes e_{i_p} \otimes e_{j_1}\otimes \cdots \otimes e_{j_q}
		$$
		when $r=0$ and 
		$$
		(e_{i_1}\otimes\cdots \otimes e_{i_p})\otimes_r (e_{j_1}\otimes\cdots \otimes e_{j_q}) =\Big(\prod_{k=1}^r \<e_{i_k}, e_{j_k}\>_{\H}\Big) e_{i_{r+1}}\otimes \cdots \otimes e_{i_p}\otimes e_{j_{r+1}}\otimes \cdots \otimes e_{j_q}
		$$
		when $r=1,\ldots \min\{p,q\}=p$. 
	\end{definition}
	
	When $r=p=q$ the contraction is given by
	$$
	(e_{i_1}\otimes\cdots \otimes e_{i_p})\otimes_p (e_{j_1}\otimes\cdots \otimes e_{j_p})=\<e_{i_1}\otimes\cdots \otimes e_{i_p}, e_{j_1}\otimes\cdots \otimes e_{j_p}\>_{\H^{\otimes p}}.
	$$
	
	It naturally follows the definition of contraction of order $r=0,\ldots, p$ of two elements $f, g$ where $f\in\H^{\otimes p}$ and $g\in\H^{\otimes q}$ , $p\leq q$:
	
	\begin{align*}
		&f\otimes_r g\\
		&=\sum_{i_1,\ldots, i_p\geq 1}\sum_{j_{r+1},\ldots, j_q\geq 1}f_{i_1,\ldots, i_r, i_{r+1},\ldots, i_p}g_{i_1,\ldots, i_r, j_{r+1},\ldots, j_q}  e_{i_{r+1}}\otimes\cdots \otimes e_{i_p}\otimes e_{j_{r+1}}\otimes\cdots \otimes e_{j_q}.
	\end{align*}
	where $f_{i_1,\ldots, i_r, i_{r+1},\ldots, i_p}=\<f,e_{i_1} \otimes \ldots \otimes e_{i_r}\otimes e_{i_{r+1}} \otimes \ldots \otimes e_{i_p}\>_{\H^{\otimes p}}$ and 
	$g_{i_1,\ldots, i_r, j_{r+1},\ldots, j_q}=\<g,e_{i_1} \otimes \ldots \otimes e_{i_r}\otimes e_{j_{r+1}} \otimes \ldots \otimes e_{j_q}\>_{\H^{\otimes q}}$.
	We define the symmetric contraction of order $r$ as the canonical symmetrization of the contraction of order $r$.

	\subsubsection{Hilbert-Schmidt operators}\label{HSop}
	\begin{definition}
		A linear operator $A$ from $\H$ to $\K$, where $\H$ and $\K$ are two real separable Hilbert spaces, is a Hilbert-Schmidt operator if 
		\begin{equation}
			\sum_{j\geq 1} \| Ae_i\|^2_\K < +\infty.
		\end{equation} 
		We denote $\LL_2(\H,\K)$ the set of all Hilbert-Schmidt operators from $\H$ to $\K$.
	\end{definition}
	$\LL_2(\H,\K)$ is a Hilbert space endowed with the following inner product: and 
	\begin{equation}\label{innerHSop}
		\<A,B\>_{\LL_2(\H,\K)}=\sum_{i\geq 1} \<Ae_i, Be_i\>_\K, \ \ \ A,B \in \LL_2(\H,\K),
	\end{equation}
	where $\{e_i\}_i$ is an orthonormal basis of $\H$. The associated norm is 
	\begin{equation}
		\|A\|_{\LL_2(\H,\K)}= \Big( \sum_{j \geq 1} \|Ae_j\|^2_\K\Big)^{\frac{1}{2}}= \Big(\sum_{i,j \geq 1} \<g_i,Ae_j\>^2_\K\Big)^{\frac{1}{2}},
	\end{equation}
	where $\{g_i\}_i$ is an orthonormal basis of $\K$. 
	From \eqref{innerHSop}, one can deduce an identification between $\LL_2(\H, \K)$ and $\H\otimes\K$: in particular, one can prove that, when $\{e_i\}_i$, $\{g_j\}_j$ are orthonormal bases of $\H$ and $\K$ respectively, then $\{e_i\otimes g_j\}_{i,j}$ is an orthonormal bases of $\LL_2(\H, \K)$. The action of $e_i\otimes g_j$ is given by \eqref{simpletensor_id}:
	$$
	(e_i\otimes g_j) h =\<e_i, h\>_\H g_j \qquad{h\in\H, i,j\in\N}. 
	$$
	In particular we note that $e_i \otimes g_j\in \LL(\H,\K)$ and 
	$$
	\|e_i \otimes g_j \|^2_{\LL_2(\H,\K)}=\sum_{n \geq 1} \| \<e_n, e_i\>_\H g_j\|_\K^2 =\|g_j\|^2_\K =1
	$$ 
	then $e_i \otimes g_j \in \LL_2(\H,\K)$.
	\begin{remark}
		We recall that:
		\begin{itemize}
			\item Any $A \in \LL_2(\H, \K)$ can be written as 
			$$
			A=\sum_{i,j \geq 1} a_{ij} e_i \otimes g_j
			$$
			where $\{a_{ij}\}_{i,j \geq 1} \subset \R$ is given by 
			$$
			a_{ij}=\<A,e_i \otimes g_j\>_{\LL_2(\H,\K)}=\<Ae_i, g_j\>_\K, \ \ i,j \geq 1.
			$$
			We note that $\sum_{i,j \geq 1} a_{ij}^2 <+\infty$.
			\item If $A \in \LL_2(\H, \K)$ then $A \in \LL(\H, \K)$ and 
			$$
			\|A\|_{op} \leq \|A\|_{\LL_2(\H,\K)}.
			$$
			\item $\LL_2(\H,\H)  \subset \LL(\H,\H)$: $Id \in \LL(\H,\H)$ but $ Id \not\in \LL_2(\H,\H)$.
			\item Let $\K_2$ be an Hilbert space. If $A \in \LL_2(\H, \K)$, $B\in \LL(\K_2, \H)$ and $C \in \LL(\K,\K_2)$, then $A\circ B \in\LL_2(\K_2,\K)$, $C\circ A \in \LL_2(\H, \K_2)$ and
			$$
			\| A\circ B\|_{\LL_2(\K_2,\K)}\leq \|A\|_{\LL_2(\H,\K)}\|B\|_{op}, \quad \|C\circ A\|_{\LL_2(\H,\K_2)}\leq \|A\|_{\LL_2(\H,\K)}\|C\|_{op}.
			$$
			\item If $A\in \LL_2(\H,\H)$ is invertible with $A^{-1}\in \LL(\H,\H)$ then $A^{-1} \not \in \LL_2(\H,\H)$: one would have $Id=A\circ A^{-1} \in \LL_2(\H,\H)$ and this is false.
			\item If $A\in \LL_2(\H,\K)$ then $B=A^*\circ A \in \LL(\H, \H)$ is a symmetric, positive and trace class operator. Vice versa, if $B \in \LL(\H,\H)$ is a symmetric, positive and trace class operator, then $A=B^{\frac{1}{2}}\in \LL_2(\H,\H)$
			\item If $B\in \LL(\H,\K)$ is a symmetric, positive and trace class operator, then $B \in \LL_2(\H,\K)$ and
			$$
			\|B\|_{\LL_2(\H,\K)} \leq \|B^{\frac{1}{2}}\|^2_{\LL_2(\H,\H)}=\Tr (B).
			$$
		\end{itemize}
		
	\end{remark}

	\subsubsection{The spaces $\LL_2(\K_2,\LL_2(\H,\K_1))$ and $\LL_2(\K_2\otimes\H,\K_1)$}
	Let $\H, \K_1,\K_2$ three real separable Hilbert spaces. 
	
	The space $\LL_2(\K_2,\LL_2(\H,\K_1))$ is a special case of the one considered in the section \ref{HSop} with $\K$ replaced by $\LL_2(\H,\K_1)$.
	
	Let $\{e_i\}_i$, $\{g_j\}_j$ and $\{u_\ell\}_\ell$ be orthonormal bases of $\H$, $\K_1$ and $\K_2$ respectively. Then we recall that $\{e_i\otimes g_j\}_{i,j}$ is an orthonormal basis of $\H\otimes \K_1$ that can be identified with $\LL_2(\H,\K_1)$. Then we can deduce that 
	$
	u_\ell \otimes (e_i \otimes g_j)=u_\ell \otimes e_i \otimes g_j \in \LL_2(\K_2, \LL_2(\H, \K_1)), i, j, \ell\in\N
	$ 
	and it is an orthonormal basis of $\LL_2(\K_2, \LL_2(\H, \K_1))$. The action is given by 
	$$
	u_\ell \otimes (e_i \otimes g_j) (x)=\<u_\ell, x\>_{\K_2} e_i\otimes g_j \qquad{x\in\K_2}.
	$$
	Again, one gets 
	$$
	\<e_{i_1} \otimes g_{j_1} \otimes u_{\ell_1}, e_{i_2} \otimes g_{j_2} \otimes u_{\ell_2} \>_{\LL_2(K,\LL_2(\H,H))}=\I_{i_1=i_2}\I_{j_1=j_2}\I_{\ell_1=\ell_2}.
	$$ 
	Moreover, $A \in \LL_2(\K_2,\LL_2(\H,\K_1))$ if and only if there exists $\{a_{ij\ell}\}_{i,j,\ell \geq 1}$ such that $\sum_{i,j,\ell \geq 1} a_{ij\ell}^2 <+ \infty$ and 
	$$ 
	A=\sum_{i,j,\ell \geq 1} a_{ij\ell}e_i \otimes g_j \otimes u_\ell.
	$$
	In such a case 
	$$
	a_{ij\ell}=\<A,e_i \otimes g_j \otimes u_\ell\>_{\LL_2(\K_2,\LL_2(\H,K_1))}=\<Au_\ell, e_i \otimes g_j\>_{\LL_2(\H,\K_1)}=\<(Au_\ell)(e_i), g_j\>_{\K_1}.$$
	Indeed, 
	\begin{align*}
		&a_{ij\ell}=\<A,e_i \otimes g_j \otimes u_\ell\>_{\LL_2(\K_2,\LL_2(\H,\K_1))}= \sum_{n \geq 1} \<Au_n, (e_i \otimes g_j \otimes u_\ell)(u_n)\>_{\LL_2(\H,\K_1)}\\ 
		&=\<Au_\ell, e_i \otimes g_j\>_{\LL_2(\H,\K_1)}=\sum_{m \geq 1} \<(Au_\ell)(e_m), (e_i \otimes g_j)(e_m)\>_H=\<(Au_\ell)(e_i), g_j\>_{\K_1}.
	\end{align*}
	We now introduce the space $\LL_2(\K_2\otimes\H,\K_1)$.
	Let $e_i \otimes g_j \otimes u_\ell$ be the element of $\LL^2(\K_2,\LL_2(\H, \K_1))$ that is defined previously. We can identify $e_i \otimes g_j \otimes u_\ell$ with a bilinear form from $\K_2 \times \H$ to $\K_1$ as follows: for $(x, y)\in \K_2 \times \H$,
	\begin{align*}
		(e_i \otimes g_j \otimes u_\ell)(x,y)&=((e_i \otimes g_j \otimes u_\ell)(x))(y)\\
		&=\<u_\ell,x\>_{\K_2}( e_i \otimes g_j )(y)\\
		&=\<u_\ell,x\>_{\K_2}\<e_i,y\>_{\H} g_j. 
	\end{align*}
	We define the set $\LL_2(\K_2\otimes\H,\K_1)$ as the set of the bilinear forms $B: \K_2\times \H \rightarrow \K_1$ such that
	$$
	B=\sum_{i,j,\ell \geq 1} b_{ij\ell} e_i \otimes g_j \otimes u_\ell
	$$
	for some $\{b_{ij\ell}\}_{i,j,\ell \geq 1} \subset \R$ with 
	$$
	\sum_{i,j,\ell \geq 1} b_{ij\ell}^2 < \infty.
	$$
	On $\LL_2(\K_2\otimes\H,\K_1)$ we define the scalar product 
	$$
	\<B,C\>_{\LL_2(\K_2\otimes\H,\K_1)}=\sum_{i,j,\ell \geq 1}b_{ij\ell}c_{ij\ell}  
	$$ where $\{c_{ij\ell}\}_{i,j,\ell}$ is the sequence representing $C \in  \LL_2(\K_2\otimes\H,\K_1)$. It follows that $\LL_2(\K_2\otimes\H,\K_1)$ is a Hilbert space with orthonormal basis given by $\{e_i \otimes g_j \otimes u_\ell\}_{i,j,\ell\geq 1}$.
	The identification is defined by the following map. Set $\phi:\LL_2(\K_2,\LL_2(\H,\K_1)) \rightarrow  \LL_2(\K_2\otimes\H,\K_1)$ by 
	$$
	\phi(A)=B \text{ with } B(x,y)=(Ax)(y), \quad (x,y) \in \K_2\times \H.
	$$
	$\phi$ is a well posed linear map.
	Moreover, $\phi$ is invertible with inverse  	$$
	\phi^{-1}(B)=A \text{ with } Ax=B(x,\cdot ), \quad x\in \K_2
	.$$
	Moreover one trivially has 
	$$
	\<A_1,A_2\>_{\LL_2(\K_2,\LL_2(\H,\K_1))}=\<\phi(A_1),\phi(A_2)\>_{\LL_2(\K_2\otimes\H,\K_1)}.
	$$
	Then $\phi$ is an isomorphism between Hilbert spaces. The consequence is that the spaces $\LL_2(\K_2,\LL_2(\H,\K_1))$ and $\LL_2(\K_2\otimes\H,\K_1)$ can be identified. 
	So, from now on, an element $A \in \LL_2(\K_2,\LL_2(\H,\K_1))$ can be seen as a bilinear form in $\LL_2(\K_2\otimes\H,\K_1)$: $A(x,y)=(Ax)(y)$. And vice versa: any bilinear form $B\in\LL_2(\K_2\otimes\H,\K_1)$ can be seen as an element in $\LL_2(\K_2,\LL_2(\H,\K_1))$: $(Bx)(y)=B(x,y)$. It will be clear from the context which of the two representations is used.

	\subsubsection{The space $\LL_2(\K_1\otimes \cdots \otimes \K_n, \H)$}
	Let $n\geq 1$. For $m = 1, \ldots, n,$ let $\K_m$ denote a Hilbert space spanned by the orthonormal basis $\{u^m_\ell\}_{\ell\geq 1}$. Let $\H$ be a Hilbert space with orthonormal basis $\{e_i\}_{i\geq 1}$. \newline
	For $i, \ell_1, \ldots, \ell_n \geq 1$ we define the multilinear ($n$-linear) form $u^1_{\ell_1} \otimes \cdots \otimes u^n_{\ell_n}\otimes e_i$ from $\K_1 \times \cdots \times \K_n$ to $\H$ through 
	$$
	u^1_{\ell_1} \otimes \cdots \otimes u^n_{\ell_n}\otimes e_i(x_1, \ldots, x_n)=\<u^1_{\ell_1},x_1\>_{\K_1} \cdots \<u^n_{\ell_n},x_n\>_{\K_n} e_i,
	$$
	with $(x_1, \ldots, x_n) \in \K_1 \times \cdots \times \K_n$.
	We define the set $\LL_2(\K_1\otimes \cdots \otimes \K_n, \H)$ as the set of the multilinear forms $B : \K_1 \times \cdots \times \K_n \rightarrow \H$ such that 
	$$
	B=\sum_{i,\ell_1,\ldots,\ell_n \geq 1} b_{i,\ell l_1,\ldots,l_n}u^1_{\ell_1} \otimes \cdots \otimes u^n_{\ell_n}\otimes e_i(x_1, \ldots, x_n) 
	$$ provided that $\{b_{i,\ell_1,\ldots,\ell_n}\}_{i,\ell_1,\ldots,\ell_n \geq 1}\subset \R$ satisfies
	$$
	\sum_{i,\ell_1,\ldots,\ell_n \geq 1}b^2_{i,\ell_1,\ldots,\ell_n} <\infty.
	$$
	On $\LL_2(\K_1\otimes \cdots \otimes \K_n, \H)$, we define the scalar product 
	$$
	\<B,C\>_{\LL_2(\K_1\otimes \cdots \otimes \K_n, \H)}=\sum_{i,\ell_1,\ldots,\ell_n \geq 1}b_{i,\ell_1,\ldots,\ell_n}c_{i,\ell_1,\ldots,\ell_n},
	$$
	where $\{c_{i,\ell_1,\ldots,\ell_n}\}_{i,\ell_1,\ldots,\ell_n \geq 1} \subset \R$ is the sequence representing $C\in \LL_2(\K_1\otimes \cdots \otimes \K_n, \H)$. It follows that $\LL_2(\K_1\otimes \cdots \otimes \K_n, \H)$ is a Hilbert space with orthonormal basis given by $\{ u^1_{\ell_1} \otimes \cdots \otimes u^n_{\ell_n}\otimes e_i \}_{i,\ell_1,\ldots,\ell_n \geq 1} $. \newline
	For $n=2$ we know that $\LL_2(\K_1\otimes \K_2, \H) \simeq \LL_2(\K_1, \LL_2(\K_2,\H))$. By iterating the same argument, for any $n\geq 2$ one has 
	$$
	\LL_2(\K_1 \otimes \cdots \otimes \K_n, \H) \simeq \LL_2(\K_1, \LL_2(\K_2 \otimes \cdots \otimes \K_n, \H)).
	$$
	Equivalently, set 
	$$
	\H_0=\H \text{ and for } m \geq 1,\ \H_m=\LL_2(\K_{n-m+1}, \H_{m-1}).
	$$
	Then, 
	$$
	\LL_2(\K_1\otimes \cdots \otimes \K_n, \H)=\H_n.
	$$
	\begin{remark}
		Let $A\in \LL_2(\K_1\otimes \cdots \otimes \K_n,\H)$. For $m \in \{1, \ldots, n-1\}$ we consider the projection $\Pi_mA$ on the first $m$ coordinates, that is, 
		\begin{equation}\label{proiezione}
			\Pi_mA(x_1, \ldots, x_m)=A(x_1, \ldots, x_m, \cdot), \quad (x_1, \ldots, x_m) \in \K_1\otimes \cdots \otimes \K_m.
		\end{equation}
		So $\Pi_mA$ is  a multilinear form in $\K_1\times \cdots \times \K_m$ taking values in the set of the multilinear forms from $\K_{m+1} \times \cdots \times \K_n$ with values in $\H$. Then one has
		$$
		\Pi_mA \in \LL_2(\K_1 \otimes \cdots \otimes \K_m, \LL_2(\K_{m+1} \otimes \cdots \otimes \K_n, \H)).
		$$
		Indeed, if $\mathbb{S}=\LL_2(\K_1 \otimes \cdots \otimes \K_m, \LL_2(\K_{m+1} \otimes \cdots \otimes \K_n, \H))$,
		\begin{align*}
			\|\Pi_mA\|^2_{\mathbb{S}}&=\sum_{i_1, \ldots, i_m \geq 1} \| \Pi_m A(u^1_{i_1}, \ldots, u^m_{i_m})\|^2_{\LL_2(\K_{m+1} \otimes \cdots \otimes \K_n, \H)}\\
			&=\sum_{i_1, \ldots, i_m \geq 1}\sum_{i_{m+1}, \ldots, i_n \geq 1} \| \Pi_m A(u^1_{i_1}, \ldots, u^m_{i_m})(u^{m+1}_{i_{m+1}}, \ldots,u^n_{i_n}) \|^2_\H\\
			&=\sum_{\scriptsize
				\begin{array}{c}
					i_1, \ldots, i_m\geq 1\\ 
					i_{m+1}, \ldots, i_n \geq 1
				\end{array}
			} \|A(u^1_{i_1}, \ldots, u^m_{i_m}, u^{m+1}_{i_{m+1}}, \ldots,u^n_{i_n}) \|^2_\H\\
			&=\|A\|^2_{\LL_2(\K_1\otimes \cdots \otimes \K_n,\H)} <+\infty.
		\end{align*}
	\end{remark}
	
	\section{Isonormal Gaussian fields}\label{ISON_Gau}
	
	This Section is devoted to introduce background material related to isonormal Gaussian random fields (see \cite{NP12} for many details).

	Let us consider a complete probability space $(\Omega, \F, \P)$ and a real separable Hilbert space $\mathbb{H}$ with inner product denoted by $\<\cdot, \cdot\>_\mathbb{H}$. From now on we denote with $L^2(\Omega):=L^2(\Omega, \F, \P)$.	 
	\begin{definition}\label{IsonormalGaussian}
		An \textbf{isonormal Gaussian field} over $\mathbb{H}$ is a family $X=\{\, X(h) \,:\, h\in \mathbb{H}\}$ of centered jointly Gaussian random variables on $(\Omega, \F, \P)$, whose covariance structure is given by 
		$$
		\E[X(h)X(h' )] = \< h, h' \>_\mathbb{H}, \quad{\forall h, h'\in \mathbb{H}}.
		$$ 
	\end{definition}
	From now on we assume that $\F$ is the $\sigma$-algebra generated by $X$. Let us recall some results related to isonormal Gaussian field, which are proven in \cite{NP12}. 
	\begin{proposition}\label{exis_GF}
		Let $(\H, \<\cdot, \cdot\>_{\H})$ be a real separable Hilbert space. Then there exists an isonormal Gaussian field over $\H$. Moreover, when $X, X'$ are isonormal Gaussian field over $\H$, they are equal in law. 
	\end{proposition}
	The previous result ensures the existence of an isonormal Gaussian field for a given real separable Hilbert space $\H$ and the uniqueness (in law).  
	
	Let us show some examples of isonormal Gaussian fields (see \cite{NP12}). 
	
	\begin{example}\label{General_Polish_isG}
		Let us consider a Polish space $\mathcal A$ and its Borel $\sigma$-algebra $\B(\mathcal A)$, endowed with a positive, $\sigma$-finite and non-atomic measure $\mu$. We denote with $W$ a random Gaussian  noise on $(\mathcal A, \B(\mathcal A))$, i.e. a centered Gaussian family  $W=\{W(A)\, :\, A\in \B(\mathcal A),\, \mu(A)<\infty\}$ such that 
		$\E[W(A)W(B)]=\mu(A\cap B)$, $A, B \in \B(\mathcal A)$. Then we build our Gaussian field on the Hilbert space $\mathbb{H}=L^2(\mathcal A, \B(\mathcal A), \mu)$, that is a real separable Hilbert space with the classical inner product $\< f, g \>_\mathbb{H}:=\int_{\mathcal A} f(x)g(x)\mu(dx), \,f,g \in \mathbb{H}$. There exists an isonormal Gaussian field on it, i.e. the centered jointly Gaussian family $X(h),\, h\in L^2(\mathcal A, \B(\mathcal A), \mu)$ such that 
		\begin{equation}\label{CovISON}
			\Cov(X(h), X(h'))=\E[X(h) X(h')]=\int_{\mathcal A} h(x) h'(x)dx, \quad{h, h' \in L^2(\mathcal A, \B(\mathcal A), \mu) }.
		\end{equation}
		We can explicit the random variables of the family $X$ such as the Wiener-It\^o integrals with respect to the Gaussian random noise $W$, i.e.
		\begin{equation}\label{IsonOnA}
			X(h)=\int_{\mathcal A} h(x)W(dx), \quad{h \in L^2(\mathcal A, \B(\mathcal A), \mu)}. 
		\end{equation} 
	\end{example}
	
	Example \ref{General_Polish_isG} is general: by varying the Polish space $\mathcal A$ we can obtain different Gaussian field. For example, by taking $\A=[0,+\infty)$ and $\mu$ as the Lebesgue measure on $\A$, the Gaussian random field $X$ on $L^2([0,+\infty), \B([0,+\infty), \mu)$ is the closure in $L^2(\Omega)$ of the linear space generated by the standard Brownian motion $W_t=X([0,t))$. 
	Again as a particular case of Example \ref{General_Polish_isG}, we propose the following example. 
	
	\begin{example}\label{Bi_isonormal}
		Let $W_1, W_2$ be two independent Brownian motions on $\R$ and denote with $W=(W^1, W^2)$. Let us consider $\H=L^2(\R; \R^2)$ (that are function $f$ from $\R$ to $\R^2$ such that $\int_{\R} \|f(t)\|^2 dt <\infty$). We have the isonormal Gaussian field $X$ on $\H$ given by the $L^2(\Omega)$-closure of the linear space generated by the $W$. $X$ is the family of Ito's integrals with respect to $W$. 
		
		Now, let $\{f^i(t, \cdot)\}_{t\in \R, i=1,2}$ be a family of functions in $L^2(\R^2; \R^2)$. We define a bivariate Gaussian process $Y=(Y^1, Y^2)$ in such a way:		 
		\begin{equation}\label{BivPr}
			Y^i_t=\int_{\R} \< f^i(t,s),W(ds)\>_{\R^2}=\int_{\R} f^i_1(t,s) W^1(ds)+\int_{\R}f_2^i(t,s)W^2(ds). 
		\end{equation}
		We can look at $(Y^i_t)$, $i=1,2$, $t\in\R$ as a particular expression of the field $X$: for $i=1,2$ and $t\in \R$, we consider $g_t^i\in L^2(\R;\R^2)$ such that $g_t^i=f^i(t,\cdot)$. Then $Y_t^i=X(g_t^i)$.

		\begin{lemma}
			The process $Y=(Y^1, Y^2)$ defined in \eqref{BivPr} is a centered Gaussian process.
		\end{lemma}
		\begin{proof}
			Let us take $t \in\R$ and $\lambda_1, \lambda_2\in \R$. Since 
			\begin{align*}
				\lambda_1 Y_t^1+\lambda_2 Y_t^2=\int_{\R} \< \lambda_1 f^1(t,s)+\lambda_2 f^2(t,s),W(ds)\>_{\R^2}
			\end{align*} 
			and $\lambda_1 f^1(t,s)+\lambda_2f^2(t,s)$ is in $L^2(\R^2;\R^2)$, then $(Y_t^1, Y_t^2)$ is Gaussian for all $t\in\R$. Moreover, for $t_1, \ldots, t_k\in \R$, $\lambda_1,\ldots, \lambda_{2k}\in\R$, we have
			\begin{align*}
				&\lambda_1 Y_{t_1}^1+\lambda_2 Y_{t_1}^2+\ldots \lambda_{2k-1}Y_{t_k}^1 +\lambda_{2k} Y_{t_k}^2\\
				&=\int_{\R} \<\lambda_1 f^1(t_1,s)+\lambda_2 f^2(t_1,s)+\ldots \lambda_{2k-1}f^1(t_k, s)+ +\lambda_{2k} f^2(t_k, s), W(ds)\>_{\R^2}
			\end{align*} 
			and being $\lambda_1 f^1(t_1,s)+\lambda_2 f^2(t_1,s)+\ldots \lambda_{2k-1}f^1(t_k, s)+ +\lambda_{2k} f^2(t_k, s) \in L^2(\R^2;\R^2)$, then each finite distribution of $Y$ is Gaussian. 
		\end{proof}

	\end{example}
	
	\section{Wiener chaos}
	
	Let us recall the notion of Wiener chaos. The entire work is strictly linked to the Wiener-Chaos expansion of a $L^2(\Omega)$ random variable. 
	
	Let us start recalling the definition of Hermite polynomials (\cite[\S 5.5]{Sze39}, \cite[\S 1.4]{NP12}).  
	\begin{definition}\label{DefHermitePolynomial}
		Let $H_q(x)$ denote the Hermite polynomial of degree $q\in\N$, defined by
		$$
		H_q(x)=
		(-1)^q e^{\frac{x^2}{2}}\frac{d^q}{dx^q}(e^{-\frac{x^2}{2}}),\quad q\geq 1
		$$
		and $H_0(x)=1$ . 
		
	\end{definition}
	From now on $H_q$ will denote the Hermite polynomial of degree $q$.
	Hermite polynomials play a central role in the Wiener chaos expansion, because of some properties, that are recalled in the following lemma:
	\begin{lemma}\label{HermiteProperties}
		Let $\{H_q\}_{q\in \N}$ the family of Hermite polynomials. Then
		\begin{itemize}
			\item[(1)] for any $q\geq 0$, $H_q'(x)=q H_{q-1}(x)$;
			\item[(2)] let $\nu$ be the standard Gaussian measure on $\R$ (i.e. $\nu(dx)=\frac{1}{\sqrt{2\pi}}e^{-\frac{x^2}{2}}\lambda (dx)$ where $\lambda$ is the Lebesgue measure on $\R$). Then 
			$$
			\int_{\R} H_q(x) H_p(x) \nu(dx)=\begin{cases} q!\qquad{q=p}\\
				0\,\qquad{q\neq p}\end{cases}
			$$
			\item[(3)] the family $\Big\{\frac{1}{\sqrt{q}}H_q\Big\}_{q\in\N}$ is an orthonormal basis of $L^2(\R, \B(\R),\nu)$, where $\B(\R)$ is the Borel $\sigma$-algebra of $\R$.
			\item[(4)] let $Z_1, Z_2$ two jointly Gaussian random variables $\mathcal N(0,1)$ and $q_1, q_2\geq0$, then
			\begin{equation}\label{hermiteCov}
				\E[H_{q_1}(Z_1)H_{q_2}(Z_2)] = q_1 !\, \E[Z_1 Z_2]^{q_1} \1_{q_1=q_2}.
			\end{equation}
		\end{itemize}
	\end{lemma} 
	Relation \eqref{hermiteCov} is a particular case of the more general and useful diagram Formula (\cite[Proposition 4.15]{MPbook}). Let us start recalling some notions (see \cite[\S 4.3.1]{MPbook}, in particular the figure at page 97). Let $n\in\N$ and $q_1,\ldots, q_n \in \N$. A diagram $G$ of order $(q_1,\ldots, q_n)$ is a set of points $\{(i,h): 1\leq i\leq n, 1\leq h\leq q_i\}$ called $\mathit{vertices}$ and a partition of these points into pairs
	$$
	\{ ((i, h),(j, k)) : 1\leq i\leq j\leq n;\,\,\, 1\leq h \leq q_i,\,\,\, 1\leq k \leq q_j        \}
	$$
	called $\mathit{edges}$, such that $(i,h)\neq(j,k)$ (self loops are not allowed) and moreover, every vertex of the diagram is linked to one and only one vertex trough an edge.  One can graphically represent $G$ by a set of $n$ rows, where the $i$th row contains $q_i$ dots. The $h$th dot (from left to right) of the $i$th row represents the point $(i,h)$. The edges of the diagram are represented as lines connecting the two corresponding dots. Let us denote with $\Gamma_{\overline{F}}(q_1,\ldots,q_n)$ the set of no-flat diagrams of order $(q_1,\ldots, q_n)$, $\Gamma_{C}(q_1,\ldots,q_n)$ the set of connected diagrams and $k_{ij}(G)$ is the number of edges from row $i$ to row $j$ of the diagram. A diagram is \textit{no-flat} if for all edges $((i,h),(j,k))$ we have $i\neq j$. It graphically means that we can connect only dots that are in two different rows. A diagram is \textit{connected} if it is not possible to partition the rows of the diagram in two non-connected sub-diagram. 
	
	\begin{proposition}\label{Diagram_Formula}
		Let $(Z_1, \ldots, Z_n)$ be a centered Gaussian vector. Let $H_{l_1},\ldots, H_{l_n}$ be the Hermite polynomials of degrees $l_1,\ldots, l_n$ respectively. Then 
		\begin{equation*}
			\E\Big[\prod_{i=1}^n H_{l_i}(Z_i)\Big]=\sum_{G\in\Gamma_{\overline F}(q_1,\ldots, q_n)} \prod_{1\leq i\leq j\leq n} \Cov(Z_i,Z_j)^{\eta_{ij}}
		\end{equation*}
		\begin{equation*}
			\text{Cum}\Big(\prod_{i=1}^n H_{l_i}(Z_i)\Big)=\sum_{G\in\Gamma_{C}(q_1,\ldots, q_n)} \prod_{1\leq i\leq j\leq n} \Cov(Z_i,Z_j)^{\eta_{ij(G) }},
		\end{equation*}
		where $\text{Cum}\Big(\prod_{i=1}^n H_{l_i}(Z_i)\Big)$ is the cumulant of $\prod_{i=1}^n H_{l_i}(Z_i)$ and $\eta_{ij}(G)$ indicates the number of edges between the $i$th row and the $j$th row of $G$. 
	\end{proposition}
	In our treatment we use largely the following corollary of Proposition \ref{Diagram_Formula}. As a directly consequence of Proposition \ref{Diagram_Formula}, we recall the well-known Isserlis' Theorem.
	\begin{corollary}[Isserlis' Theorem]\label{corDiagFor}
		Let $(Y_1, Y_2, Y_3, Y_4)$ be a Gaussian vector such that $\E[Y_i]=0$. Then 
		\begin{align*}
			&\E[Y_1 Y_2 Y_3 Y_4]=\E[Y_1 Y_2]\E[Y_3Y_4]+\E[Y_1 Y_3]\E[Y_2Y_4]+\E[Y_1 Y_4]\E[Y_2Y_4].
		\end{align*}  
	\end{corollary}
	\begin{proof}
		We apply Proposition \ref{Diagram_Formula} with $(Z_, Z_2, Z_3, Z_4)=(Y_1, Y_2, Y_3, Y_4)$ and $H_{\ell_i}=H_1$ for $i=1,2,3,4$. 
	\end{proof}
	
	The above corollary is very specific because in Chapter \ref{Biv_fOU_3} we largely compute the expectation of product of $4$ Gaussian random variables. 
	
	Now we explain how Hermite polynomials are linked to isonormal Gaussian field. Let $X$ be an isonormal Gaussian random field on a real separable Hilbert space $\mathbb{H}$. 
	
	\begin{definition}\label{Chaos-order-n}
		Let $q\in \N$. Then
		\begin{equation}
			\mathcal C_q=\overline{\{H_q(X(h)), h\in \mathbb{H}, \|h\|_{\mathbb{H}}=1\}}^{L^2(\Omega)}
		\end{equation}
		is called the $q$th Wiener chaos of $X$.
	\end{definition}	
	
	By Definition \ref{Chaos-order-n}, it follows that $\mathcal C_0=\R$ and $\mathcal C_1=X$, and by orthogonality of Hermite polynomials recalled in Lemma \ref{HermiteProperties}, it follows that $\mathcal C_q$ and $\mathcal C_p$ are orthogonal subspace of $L^2(\Omega)$ when $q\neq p$. Orthogonality property of the Wiener chaoses implies the following proposition.
	\begin{proposition}[Wiener-Chaos Decomposition]\label{WienerChaoticDecomposition}
		Let $q\in\N$. The linear subspace generated by $\{H_q(X(h)), h\in \mathbb{H}, \|h\|_{\mathbb{H}}=1\}$ is dense in $L^r(\Omega)$ for all $r\in[1,+\infty)$, and 
		$$
		L^2(\Omega)=\bigoplus_{n=0}^{\infty} \mathcal C_q.
		$$
	\end{proposition}
	
	Proposition \ref{WienerChaoticDecomposition} ensures us that for a given $F\in L^2(\Omega)$, there exists a sequence $\{F_q\}_{q\in\N}$, $F_q\in\mathcal C_q$, such that $F=\sum_{q=0}^{+\infty} F_q$, where the series converges in $L^2(\Omega)$. We indicate with $J_q(F)$ the projection on $\mathcal C_q$ of $F$.

	Once the Wiener chaos decomposition is defined, we can introduce the Ornstein-Uhlenbeck operator $L$, which will play an important role in our approach: for $F\in L^ 2(\Omega, \F,\P)$, one says that $F\in Dom(L)$ if and only if
	$$
	\sum_{q \geq 1} q^2\E[J_q(F)^ 2]<\infty
	$$
	and in such a case,
	\begin{equation}\label{defL}
		LF = -\sum_{q \geq 1} q J_q(F).
	\end{equation}
	
	Let us conclude the subsection recalling the definition of \textit{Hermite rank}. 
	\begin{definition}\label{HermiteRank}
		Let $F\in L^2(\Omega, \F,\P)$, where $\F$ is generated by a Gaussian field $T$. The Hermite rank of $F$ is the smallest $q\geq 2$ such that $J_q(F)\neq 0$. 
	\end{definition} 
	\section{Malliavin calculus for Gaussian random fields}\label{MDGF}
	
	In this section we recall some definitions and properties of Malliavin calculus for Gaussian random fields --  details and results can be found fully explained in \cite[\S 2.3]{NP12} or \cite{Nua}. For $k\in\N$, we denote with $C^ k_p(\R^ m)$, respectively $C^ k_b(\R^ m)$, the set of functions $f:\R^ m\to\R$ that are continuously differentiable up to order $k$ and with polynomial growth, respectively bounded. And as usual,  $C^ \infty_p(\R^ m)=\cap_{k\geq 0}C^ k_p(\R^ m)$ and $C^ \infty _b(\R^ m)=\cap_{k\geq 0}C^ k_b(\R^ m)$.
	
	Let $X$ be an isonormal Gaussian random field based on a real valued Hilbert space $\mathbb H$. Let $\mathcal{S}$ denote the set of the simple functionals, being defined as the variables of the form 
	\begin{equation}\label{smoothF}
		f(X(g_1),\ldots, X(g_m))
	\end{equation}
	where $m\geq 1$, $f\in C^ \infty_p(\R^m)$ and $g_1,\ldots, g_m\in \mathbb{H}$. 
	It is well known that $\mathcal{S}$ is dense in $L^p(\Omega)$.
	
	Given $k\in \N$, we denote with $\mathbb{H}^{\otimes k}$ and $\mathbb{H}^{\odot k}$, respectively, the $k$-th tensor product and the $k$-th symmetric tensor product of $\mathbb{H}$.

	Let $F\in \mathcal{S}$ be given by \eqref{smoothF} and $k \in \N$. The $k$-th \textit{Malliavin derivative}  is the element of $L^2(\Omega; \mathbb{H}^{\odot k})$ defined by
	$$
	D^{(k)} F=\sum_{i_1,\ldots, i_k=1}^m \frac{\partial^k f}{\partial x_{i_1}\cdots \partial x_{i_k}}(X(g_1),\ldots,X(g_m))\, g_{i_1} \otimes \cdots \otimes g_{i_k}.
	$$
	Being $f$ is $C^{\infty}(\R^m)$, the Schwarz theorem guarantees that the $k$th derivative of $F$ is a random variable whose values are in $\H^{\odot k}$, i.e. it is symmetric.  
	We recall that $D^{(k)}: \mathcal{S} \subset L^p(\Omega;\R)\hookrightarrow L^p(\Omega;\mathbb{\H}^{\odot k})$ is closable.
	So, for $k\in \N$ and $p\geq 1$, one defines the space $\DD^{k,p}$ as the closure of $\mathcal{S}$ with respect to the norm 
	$$
	\|F \|_{\DD^{k,p}} =\big(\E[|F|^p]+\E[\|F\|_{\H}^p]+\ldots +\E[\|D^{(k)} F\|_{\H^{\otimes k}}^p]\big)^{\frac{1}{p}}
	$$
	and the Malliavin derivative can be extended to the set $\DD^{k,p}$, being the \textit{domain} of $D^{(k)}$ in $L^p(\Omega;\R)$. 
	In particular, the space $\DD^{k,2}$ is a Hilbert space with respect to the inner product
	$$
	\<F,G\>_{\DD^{k,2}}=\E[FG]+\sum_{r=1}^k \E[\<D^{(r)}F,D^{(r)}G\>_{\H^{\odot r}}].
	$$
	Moreover, the chain rule property does hold: for every $\phi\in C^ 1_b(\R^ m)$ and $F=(F_1,\ldots, F_m)$ with $F_i\in \DD^{1,p}$, $i=1,\ldots,m$, for some $p\geq 1$, then $\phi(F)\in\DD^{1,p}$ and 
	\begin{equation}\label{chainRule}
		D\phi(F)=\sum_{r=1}^k \frac{\partial \phi}{\partial x_r}(F) D F_r.
	\end{equation}
	\begin{lemma}
		Let $F, G\in\DD^{1,2}$. Then
		\begin{equation}\label{Relatprod}
			\E[G\<DF, h\>_{\H}]=-\E[F\<DG,h\>_{\H}]+\E[X(h)FG]\qquad{h\in\H}.
		\end{equation}
	\end{lemma}
	
	\subsubsection{The divergence operator}
	Let us recall that, for $p\geq 1$, the Malliavin derivative operator of order $p$ is a linear operator 
	$$
	D^{(p)}:\DD^{p,2}\to L^2(\Omega;\H^{\odot p}).
	$$
	Let us denote with $\delta^p$ the adjoint operator of $D^{(p)}$, the so-called divergence operator of order $p$. 
	The domain of the divergence operator of order $p$ is the set $Dom(\delta^p)$ given by the element $h\in\H^{\odot p}$ for which there exists $c>0$ such that 
	$$
	\E[\<D^{(p)} F, h\>_{\H^{\otimes p}}]\leq c \|F\|_{L^2(\Omega)}, \qquad{\forall F\in \DD^{p,2}}.
	$$
	
	\begin{definition}\label{integrationBypart}
		Let $h$ be in $Dom(\delta^p)$. Then $\delta^p(h)$ is the unique element of $L^2(\Omega)$ such that 
		\begin{equation}\label{IBP}
			\E[\<D^{(p)} F, h\>_{\H^{\otimes p}}]=\E[F\delta^p(h)], \qquad{\forall F\in\DD^{p,2}}.
		\end{equation}
	\end{definition}

	Let us consider a random field $X$ on $\H$. Then, by choosing $G\equiv 1$ in \eqref{Relatprod}, we have
	$$
	\E[\<DF, h\>_\H]=\E[X(h)F], 
	$$ 
	then $h\in Dom(\delta^p)$ and $\delta(h)=X(h)$.  In general, $\H^{\otimes p}$ is in $Dom (\delta^p)$. 
	\section{Multiple integrals}
	
	\begin{definition}\label{multipleintegral}
		Let $h\in\H^{\odot p}$, $p\geq 1$. The $p$th multiple integral of $h$ with respect to $X$ is 
		$$
		I_p(h)=\delta^p(h). 
		$$
	\end{definition}
	Let us recall some properties of $p$th multiple integral. 
	\begin{itemize}
		\item[1)] let $p\geq 1$ and $h\in\H^{\odot p}$. Then for all $q\in[1,\infty)$, $I_p(h)\in\DD^{\infty, q}$. Moreover 
		$$
		D^r(I_p(h))=\frac{p!}{(p-r)!}I_{p-r}(h) \qquad{r\leq p}.
		$$
		In above derivative, we write $I_{p-r}(h)=\delta^{p-r}(h)$, where $h$ is an element of $\H^{\odot p}$.  The meaning of applying the divergence operator of order $p-r$ to a given element of $\H^{\odot p}$ is the following: we recall that 
		$$
		h=\sum_{i_1,\ldots, i_p=1}^\infty \<h, e_{i_1}\otimes \cdots \otimes e_{i_p}\>_{\H^{\otimes p} }   e_{i_1}\otimes \cdots \otimes e_{i_p}
		$$
		where $\{e_i\}_i$ is an orthonormal basis of $\H$. Then 
		$$
		I_{p-r}(h)=\sum_{i_1,\ldots, i_p=1}^{\infty}  \<h, e_{i_1}\otimes \cdots \otimes e_{i_p}\>_{\H^{\otimes p}}  e_{i_1}\otimes \cdots \otimes e_{i_r} \otimes I_{p-r}(e_{i_{r+1}} \otimes \cdots\otimes e_{i_p}).
		$$
		\item[2)]
		Let $p,q\geq 1$ and $h\in\H^{\odot p}$, $g\in\H^{\odot q}$. Then
		\begin{equation}\label{isometry}	    	\E[I_q(h)I_p(g)]=\begin{cases}
				p!	\<h, g\>_{\H^{\otimes p}} \qquad p=q,\\
				0 \hspace{40pt} \qquad{otherwise}
			\end{cases}
		\end{equation}
		and 
		$$
		\E[I_q(h)^2]=p!\|h\|_{\H^{\otimes p}}^2.
		$$

	\end{itemize}

	\begin{theorem}\label{product_formula}
		Let $p,q\geq 1$. Let $f\in \H^{\odot p}$ and $g\in \H^{\odot q}$. Then
		$$
		I_p(f)I_q(g)=\sum_{r=0}^{p\wedge q} r!\binom{p}{r}\binom{q}{r} I_{p+q-2r}(f\tilde \otimes_r g).
		$$
	\end{theorem}
	
	The last result that we recall shows that, if we consider a random vector sequence, whose components are Wiener-It\^o integrals, then the componentwise convergence to Gaussian always implies the joint convergence. 
	
	\begin{theorem}[Theorem 6.2.3 in \cite{NP12}]\label{jointconvergence}
		Let $d\geq 2$ and $q_1, \ldots , q_d \geq 1$ be some fixed integers. Consider vectors $F_n= (F_{1,n}, \ldots, F_{d,n}) = (I_{q_1}( f_{1,n}),\ldots , I_{q_d} ( f_{d,n}))$, $n \geq 1$, with $f_{i,n}\in \H^{\odot q_i}$. Let $C \in \mathcal M_d(\R)$ be a symmetric non-negative definite matrix, and let $N \sim N_d (0,C)$. Assume that
		$$
		\lim_{n\to \infty} \E[F_{i,n}F_{j,n}]=C(i,j), \quad{1\leq i,j\leq d}.
		$$
		Then, as $n\to \infty$, the following two conditions are equivalent:
		\begin{itemize}
			\item[a)] $F_n$ converges in law to $N$.
			\item[b)]  For every $i=0,\ldots, d$, $F_{i,n}$ converges in law to $\mathcal N(0, C(i,i))$.
		\end{itemize}
	\end{theorem}

	\section{Distances of probability measures}\label{ProbDist}
	
	In this section, we provide brief definitions of three distances between probability measures that will be central to our discussion: the Kolmogorov distance, the Total Variation distance, and the Wasserstein distance. For further details, we refer to \cite{NP12}.
	
	\begin{definition}
		Let $X,Y$ be two real random variables. The Kolmogorov distance between $X$ and $Y$ is defined as 
		\begin{equation}\label{Kdistance}
			d_{K}(X,Y):=\sup_{z\in\R} |\P(X\leq z)-\P(Y\leq z)|.
		\end{equation}
	\end{definition}
	
	\begin{definition}
		Let $X,Y$ be two integrable random variables. The Wasserstein distance between $X$ and $Y$ is defined as 
		\begin{equation}\label{Wdistance}
			\dW(X,Y) := \sup_{h\in \text{Lip}(1)} \left |\mathbb E[h(X)] - \mathbb E[h(Y)] \right |,
		\end{equation}
		where $\text{Lip}(1)$ denotes the space of functions $h:\mathbb R\to \mathbb R$ which are Lipschitz continuous with Lipschitz constant $\le 1$. 
	\end{definition}
	
	\begin{definition}
		Let $X, Y$ be two random variables. The Total Variation distance between $X$ and $Y$ is defined as
		\begin{equation}\label{TVdistance}
			\dTV(X,Y) := \sup_{A\in \B(\mathbb R)}\left | \mathbb P(X\in A) - \mathbb P(Y\in A)\right |.
		\end{equation}
		
	\end{definition}
	
	From now on we will denote the Kolmogorov distance, the Wasserstein distance and the Total Variation distance  between $X$ and $Y$ with $d_K$, $\dW(X,Y)$ and $\dTV(X, Y)$ respectively. 
	
	\begin{remark}
		The topologies induced by the Kolmogorov, Wasserstein, and Total Variation distances are stronger than the topology of convergence in law. Specifically, when a sequence of integrable random variables $F_n$ converges in Wasserstein distance to a random variable $X$, we have that
		\begin{align*}
			&\sup_{\lambda \in \R^d} |\E[e^{i\<\lambda, F_n\>_{d}}]-\E[e^{i\<\lambda, F\>_{d}}]|\\&\leq \sup_{\lambda \in \R^d} |\E[\cos(\<\lambda, F_n\>_{d})]-\E[\cos(\<\lambda, F\>_{d})]|+\sup_{\lambda \in \R^d} |\E[\sin(\<\lambda, F_n\>_{d})]-\E[\sin(\<\lambda, F\>_{d})]| \\
			&\leq 2\dW(F_n, F)\to 0
		\end{align*}
		being $\sin, \cos \in \text{Lip}(1)$. This implies convergence in law.
		Regarding the Total Variation and the Kolmogorov distances, we observe that for all $z\in\R$, $(-\infty,z] \in \B(\R)$ then 
		$$
		d_K(X,Y)\leq d_{TV}(X,Y).
		$$  
		If $F_n$ converges to $F$ in Kolmogorov distance, then
		$$
		|\P(F_n\leq z)-\P(F\leq z)|\to 0
		$$
		thus we have convergence in law. 
		Convergence in Total Variation distance implies convergence in Kolmogorov distance and then in law.
		
		We can also observe that the convergence in these probability metrics is not equivalent to convergence in law. 
		Let us denote with $U\sim Unif([0,1])$ and 
		$$
		F_n=\begin{cases}
			n \quad{U\in[0,\frac 1n]}\\
			0 \quad{otherwise}
		\end{cases}
		$$
		then $\lim_{n\to \infty}\P(|F_n|\leq \delta)=\frac{1}n$ definitely. Then $F_n\overset{p}{\to} 0$ and then in law. But 
		\begin{align*}
			\dW(F_n, 0)\geq |\E[F_n]-\E[0]|=1,
		\end{align*} 
		then $F_n\overset{d}{\to} 0$ but $F_n$ doesn't converge in Wasserstein distance. 
		
		When we consider $F_n=\frac 1n$ almost surely (a.s.), we have that $F_n\to 0$ a.s. and then in law. But 
		$$
		\dTV(F_n, 0)\geq d_K(F_n, 0)=\sup_{z\in \R}|\P(F_n\leq z)-\P(0\leq z)|\geq \Big|\P\Big(\frac 1n\leq 0\Big)-\P(0\leq 0)\Big|=1,
		$$
		then $F_n$ doesn't converge to $0$ in Kolmogorov and Total Variation distance. 
	\end{remark}
	
	We additionally recall the following proposition concerning the probability distances between two centered Gaussian random variables (see Proposition 3.6.1 in \cite{NP12}).
	
	\begin{proposition}\label{dWTVKgaussian}
		Let $N_1\sim\mathcal N(0,\sigma_1^2)$ and $N_2\sim \mathcal N(0,\sigma_2^2)$. Then
		\begin{align*}
			&d_K(N_1,N_2)\leq \frac{1}{\sigma_1^2\vee\sigma_2^2}|\sigma_1^2-\sigma_2^2|\\
			&\dW(N_1,N_2)\leq \frac{\sqrt{\frac{2}{\pi}}}{\sigma_1\vee\sigma_2}|\sigma_1^2-\sigma_2^2|\\
			&\dTV(N_1,N_2)\leq \frac{2}{\sigma_1^2\vee\sigma_2^2}|\sigma_1^2-\sigma_2^2|.
		\end{align*}
	\end{proposition}
	
	We finally recall the improved version of Second-order Poincarè inequality in \cite{Vid19}.
	\begin{proposition}[Theorem 2.1 in \cite{Vid19}]
		Let $\H=L^2(A,\B(A), \mu)$, where $(A,\B(A))$ is a Polish space endowed with its Borel $\sigma$-field and $\mu$ is a positive, $\sigma$-finite, and non-atomic measure. Let $X$ an isonormal Gaussian field on $\H$.   $F\in \mathbb{D}^{2,4}$ (in the sense of Section \ref{MDGF}) be such that  $\E[f(N)]=0$ and $\Var(f(N))=\sigma^2$ and let $N\sim \mathcal N(0,\sigma^2)$; then 
		\begin{align*}
			d_{M}(F, N) \leq c_M \Big(\int_{A\times A} \E\Big[\Big((D^{(2)}F\otimes_1 D^{(2)}F)(x,y)\Big)^2\Big]^{\frac 12}\E\Big[\Big(DF(x) DF(y)\Big)^2\Big]^{\frac 12} d\mu(x)\mu(y)
			\Big)^{\frac 12}
		\end{align*}
		where $M\in\{TV, W, K\}$ and $C_{TV}=\frac 4{\sigma^2}$, $C_W=\sqrt{\frac8{\sigma^2 \pi}}$ and $C_K=\frac 2{\sigma^2}$. 
	\end{proposition}
	
	\section{Fourth Moment Theorem}
	Theorem 5.13 in \cite{NP12} demonstrates a direct connection between stochastic calculus and probability metrics.
	
	\begin{theorem}\label{principal_0}
		Let $F\in\mathbb D^{1,2}$ such that $\E[F]=0$ and $\E[F^2]=\sigma^2<+\infty$. Then we have for $N\sim \mathcal N(0,1)$  
		$$
		\dW(F, N)\leq\sqrt{ \frac2{\pi \sigma^2}}\E[|\sigma^2-\<DF, -DL^{-1}F\>_\H|].
		$$
		Also, assuming that $F$ has density, we have
		\begin{align*}
			&\dTV(F,  N)\leq \frac 2{\sigma^2}\E[|\sigma^2-\<DF, -DL^{-1}F\>_\H|] \\
			&d_{K}(F,  N)\leq \frac 1{\sigma^2}\E[|\sigma^2-\<DF, -DL^{-1}F\>_\H|] 
		\end{align*}
		Moreover, if $F\in\mathbb D^{1,4}$, we have 
		$$
		\E[|\sigma^2-\<DF, -DL^{-1}F\>_\H|] \leq \sqrt{\Var(\<DF, -DL^{-1}F\>)}.
		$$
	\end{theorem}
	
	When $F=I_q(f)$ for $f\in \H^{\odot q}$, $q\geq 2$, Theorem 5.2.6 in \cite{NP12} ensures us that
	$$
	\E[|\sigma^2-\<DF, -DL^{-1}F\>_\H|]\leq \sqrt{\Var\Big(\frac 1q\|DF\|_\H^2\Big)}.
	$$
	and from Lemma 5.2.4 
	\begin{align}\label{v_4_est}
		\Var\Big(\frac 1q\|DF\|_\H^2\Big)&=\frac1{q^2}\sum_{r=1}^{q-1}r^2 r!^2\binom{q}{r}^4 (2q-2r)!\|f\tilde \otimes_r f\|^2_{\H^{\otimes 2q-2r}}\notag \\
		&\leq \frac1{q^2}\sum_{r=1}^{q-1}r^2 r!^2\binom{q}{r}^4 (2q-2r)!\|f \otimes_r f\|^2_{\H^{\otimes 2q-2r}}.
	\end{align}
	Finally, we conclude that
	\begin{align}\label{c_4_est}
		\Var\Big(\frac 1q\|DF\|_\H^2\Big)\leq \frac{q-1}{3q} \kappa_4(F)\leq (q-1)\Var\Big(\frac 1q\|DF\|_\H^2\Big)
	\end{align}
	where 
	$$
	\kappa_4(F)=\E[F^4]-3\E[F^2]^2
	$$
	is the fourth cumulant of $F$. Then 
	\begin{theorem}\label{Principal}
		Let $\{F_n\}_{n\in\N}$ a sequence of random variables belonging to a fixed $q$-th Wiener chaos, for fixed integer $q\geq 2$. Then 
		\begin{align*}
			&\dW\Big(\frac{F_n}{\sqrt{\Var(F_n)}},N\Big) \leq C_{W}(q)\sqrt{\frac{\kappa_4(F_n)}{\Var(F_n)^2}} \\
			&\dTV\Big(\frac{F_n}{\sqrt{\Var(F_n)}},N\Big) \leq C_{TV}(q)\sqrt{\frac{\kappa_4(F_n)}{\Var(F_n)^2}}\\
			&d_{K}\Big(\frac{F_n}{\sqrt{\Var(F_n)}},N\Big) \leq C_{K}(q)\sqrt{\frac{\kappa_4(F_n)}{\Var(F_n)^2}}
		\end{align*}
		where $N\sim \mathcal N(0,1)$. In particular, when 
		$
		\frac{\kappa_4(F_n)}{\Var(F_n)^2}\to 0
		$ 
		then 
		$$
		\frac{F_n}{\sqrt{\Var(F_n)}}\overset{d}{\to} N .
		$$
	\end{theorem}
     A fundamental consequence of the above theorem is the Fourth-Moment Theorem (Theorem 5.2.7 in \cite{NP12}).
     \begin{theorem}[Fourth-Moment Theorem]\label{FOURTH}
     	Let $F_n=I_q(f_n)$, $n\geq 1$, be a sequence of random variables belonging to the $q$th chaos of $X$, for some fixed integer $q\geq 2$ (so that $f_n\in \H^{\odot q}$). Assume, moreover, that $\E[F_n^2]\to \sigma^2>0$ as $n\to+\infty$. Then, as $n\to +\infty$ the following assertions are equivalent:
     	\begin{enumerate}
     	\item $F_n$ converges in distribution to $N\sim \mathcal N(0,\sigma^2)$;
     	\item $\E[F_n^4]	\to 3\sigma^4$ or equivalently $\kappa_4(F_n)\to 0$;
     	\item $\Var(\|DF_n\|^2_\H)\to 0$;
     	\item $\|f_n\tilde \otimes_r f_n\|_{\H^{\otimes (2q-2r)}}\to 0$, for all $r=1,\ldots, q-1$;
     	\item $\|f_n \otimes_r f_n\|_{\H^{\otimes (2q-2r)}}\to 0$, for all $r=1,\ldots, q-1$.
     	\end{enumerate}
     \end{theorem}

	\part{Quantitative CLTs for nonlinear functionals of random hyperspherical harmonics}\label{Part1}
	\chapter{Geometry of random hyperspherical harmonics}\label{chapter:geometry_of_rha}

	In this chapter we introduce the setting of our problem. We first recall the definition of real spherical harmonics and their properties. Then we define the Gaussian random field which we will be dealing with. Finally we introduce the functional of the Gaussian field for which we prove the main result of this chapter. 

	Let $d\geq 2 $  be a positive integer. We denote by $ \mathbb S^d\subset \R^{d+1}$ the $d$-dimensional unit sphere. Accordingly, we set $\B(\mathbb S^d)$ to be its Borel $\sigma$-field and we write $\Leb(dx)=dx$ for the Lebesgue measure on $(\mathbb S^d, \B(\mathbb S^d))$. It is known that 
	$
	\mu_d:=\int_{\mathbb S^d} dx=\frac{2\pi^{\frac{d+1}{2}}}{\Gamma(\frac{d+1}{2})},
	$
	$\Gamma$ being the Gamma function. 
	\section{Real (hyper)-spherical Harmonics}\label{REAL_spher}
	
	Let $\Delta_{\mathbb R^{d+1}}$ denote the Laplace operator on 
	$\mathbb \R^{d+1}$, i.e. 
	$$
	\Delta_{\R^{d+1}}=\sum_{k=1}^{d+1}\frac{\partial^2}{\partial x^2_k}.
	$$ 
	Let $f\in C^2(\R^{d+1}; \R)$; $f$ is an \textit{harmonic function} if $\Delta_{\R^{d+1}} f=0$. Let us now introduce the (hyper)-spherical coordinates on $\R^{d+1}$ (further details can be found in \cite{EF14}):
	\begin{align*}
		&\rho=\sqrt{x_1^2+\ldots +x_{d+1}^2}\,\, \quad \qquad \qquad{r\in[0,+\infty)} ,\\
		&\phi=\tan^{-1}\left(\frac{x_2}{x_1}\right)\,\, \quad \qquad \qquad \qquad{\phi \in[0,2\pi)},\\
		&\theta_1=\tan^{-1}\left(\frac{\sqrt{x^2_1+x_2^2}}{x_3}\right)\,\,\,\quad \qquad{\theta_1 \in\Big[-\frac{\pi}{2},\frac{\pi}2\Big]},\\
		&\vdots \\
		&\theta_{d-1}=\tan^{-1}\left(\frac{\sqrt{x_1^2+\cdots x_{d}^2}}{x_{d+1}}\right)\quad{\theta_{d-1} \in\Big[-\frac{\pi}{2},\frac{\pi}2\Big]}.
	\end{align*}
	Then we can write the Cartesian coordinates as functions of spherical coordinates:
	\begin{align*}
		&x_1=\rho \sin \theta_{d-1} \sin \theta_{d-2}\cdots\sin \theta_2 \sin \theta_1 \cos \phi,\\
		&x_2=\rho \sin \theta_{d-1} \sin \theta_{d-2}\cdots\sin \theta_2 \sin \theta_1 \sin \phi,\\
		&x_3=\rho \sin \theta_{d-1} \sin \theta_{d-2}\cdots\sin \theta_3 \sin \theta_2 \cos \theta_1,\\
		&x_4=\rho \sin \theta_{d-1} \sin \theta_{d-2}\cdots\sin \theta_4 \sin \theta_3 \cos \theta_2,\\
		&\vdots\\
		&x_{d}=\rho \sin \theta_{d-1} \cos \theta_{d-2},\\
		&x_{d+1}=\rho \cos\theta_{d-1}.
	\end{align*} 
	Then we consider the Laplace-Beltrami operator $\Delta_{\SSd}$, i.e. the restriction of $\Delta_{\R^{d+1}}$ to the hypersphere $\SSd$. It is well known (see \cite{EF14}) that $\Delta_{\SSd}$ satisfies 
	\begin{equation}\label{RelationLaplace}
		\Delta_{\R^{d+1}}=\frac{\partial^2}{\partial \rho^2}+\frac{d}{\rho}\frac{\partial}{\partial \rho} +\frac1{\rho^2}\Delta_{\SSd}.
	\end{equation} 
	The relation \eqref{RelationLaplace} can be derived using the chain rule and the (hyper)-spherical coordinates.

	We are ready to define the real hyperspherical harmonics.
	
	\begin{definition}\label{HypSph}
		Let $\ell \in \N$. An hyperspherical harmonic of degree $\ell$, denoted with $Y_{\ell;d} (x)$, is an homogeneous harmonic polynomial of degree $\ell$ in $d+1$ variables restricted to the hypersphere $\SSd$. It means that 
		$$
		Y_{\ell;d} :\SSd\to \R, \qquad{Y_\ell(x)=P_{\ell; d}(x), x\in \SSd}, 
		$$ 
		for some $P_{\ell;d}$ homogeneous harmonic polynomial of degree $\ell$.
	\end{definition}
	
	An immediate consequence of the definition is the following proposition (see \cite{EF14}).
	
	\begin{proposition}\label{Helmotz}
		Let $Y_\ell$ be an hyperspherical harmonic of degree $\ell$. Then 
		\begin{equation}\label{helmholtz}
			\Delta_{\mathbb S^d} Y_{\ell; d} = \lambda_{\ell;d} Y_{\ell;d},
		\end{equation}
		where 
		$\lambda_{\ell;d}=-\ell(\ell+d-1)$.
	\end{proposition}
	\begin{proof}
		Recalling Definition \ref{HypSph}, there exists an homogeneous polynomial $P_{\ell;d}$ of degree $\ell$ in $d+1$ such that for all $x\in \SSd$, $Y_{\ell;d}(x)=P_{\ell;d}(x)$. Let us take $y \in \R^{d+1}$ such that $\|y\|_{\R^{d+1}}=\rho$, then there exists $x\in \SSd$ such that $y=\rho\cdot x$. We also notice that $P_{\ell;d }(y)=\rho^\ell P_{\ell;d}(x)=\rho^{\ell} Y_{\ell; d}(x)$ from homogeneity of $P_{\ell;d}$. Then, being $P_{\ell;d}$ harmonic, we have
		\begin{align*}
			0=\Delta_{\R^{d+1}} P_{\ell; d}	=\Delta_{\R^{d+1}} (\rho^{\ell} Y_{\ell;d})=\ell(\ell-1)\rho^{\ell-2} Y_{\ell;d} +\frac d{\rho} \ell \rho^{\ell-1} Y_{\ell;d}+\frac{\rho^{\ell}}{\rho^2} \Delta_{\SSd} Y_{\ell;d}.
		\end{align*}
		Then the statement holds.
	\end{proof}

	Equation \eqref{helmholtz} in Proposition \ref{Helmotz} is known as Helmholtz equation. Proposition \ref{Helmotz} states that the hyperspherical harmonics of order $\ell\in \N^*$ are the eigenfunction of the Laplace-Beltrami operator $\Delta_{\SSd}$ with eigenvalues $\lambda_{\ell;d}$. The dimension of the $\ell$-th eigenspace is 
	$$
	n_{\ell;d} = \frac{2\ell+d-1}{\ell} \binom{\ell+d-2}{\ell-1}
	$$ 
	($n_{0;d} = 1$). It can be useful to notice that 
	$
	n_{\ell;2}=2\ell + 1$ and 
	\begin{equation}\label{n-ell}
		n_{\ell;d} \sim \frac{2}{(d-1)!} \ell^{d-1}\quad \text{as}\quad \ell\to +\infty.
	\end{equation}
	From now on, for fixed $\ell \in \N^*$, we will refer to the family of hyperspherical harmonics $(Y_{\ell,m;d})_{m=1}^{n_{\ell;d}}$ of order $\ell$ to indicate an orthonormal basis of the eigenspace. It means that every eigenfunction $f=f_{\ell}$ in (\ref{helmholtz}) of eigenvalues $\lambda_{\ell;d}$ can be written in the form
	\begin{equation}\label{eigenfunction}
		f_\ell(x)=\sqrt{\frac{\mu_d}{n_{\ell;d}}}\sum_{m=1}^{n_{\ell;d}} c_{\ell,m} Y_{\ell,m;d}(x),\quad x\in \mathbb{S}^d,
	\end{equation}
	where $c_{\ell,m}$ are real numbers, and $\sqrt{\mu_d}/\sqrt{n_{\ell;d}}$ is a normalizing factor whose role will be clarified just below. 
	
	\subsection{Properties of hyperspherical harmonics}
	
	We here summarize some useful properties of real hyperspherical harmonics. 
	
	\begin{proposition}\label{OrthHyp}
		Let $\ell, \ell'\in \N^*$ such that $\ell\neq \ell'$. Then
		$$
		\int_{\SSd} Y_{\ell; d}(x)Y_{\ell';d}(x)dx =0,
		$$
		that is the hyperspherical harmonics of different degrees are orthogonal (in the sense of $L^2(\SSd;\R, \Leb)$ ). 
	\end{proposition}
	Proof of Proposition \ref{OrthHyp} follows from some classical calculus arguments and from the properties of homogeneous harmonic polynomials (see \cite{EF14}). We also recall that the space of real hyperspherical harmonics are stable with respect the action of the special orthogonal group $SO(d)$. 
	\begin{proposition}\label{invariance_rot}
		Let $Y_{\ell;d}$ be an hyperspherical harmonic of degree $\ell$ and $R$ the representative orthogonal matrix of an element of $SO(d)$. Then $Y_{\ell; d} \circ R$ is an hyperspherical harmonic of degree $\ell$.
		
	\end{proposition}
	As a consequence of Proposition \ref{invariance_rot}, for a given special orthogonal matrix $R$, the function $Y_{\ell, j;d}\circ R$ belongs to the $\ell$-th eigenspace of $\Delta_{\SSd}$, then we can write
	$$
	Y_{\ell, j;d}(Rx)=\sum_{m=1}^{n_{\ell;d}} c_{m j} Y_{\ell,m;d}(x),\qquad{j=1,\ldots,n_{\ell;d}}
	$$
	and using the orthogonality, we have
	\begin{align*}
		0=\int_{\SSd} \sum_{m=1}^{n_{\ell;d}} c_{m j}Y_{\ell,m;d}(x) \sum_{m=1}^{n_{\ell;d}} c_{m' j'} Y_{\ell,m';d}(x)dx=\sum_{m=1}^{n_{\ell;d}} c_{m j} c_{m j'}. 
	\end{align*}
	The function 
	$$
	F_{\ell;d}(x,y)=\frac{\mu_d}{n_{\ell;d}}\sum_{m=1}^{n_{\ell;d}} Y_{\ell,m;d}(x) Y_{\ell,m;d}(y) 
	$$
	is invariant under rotation of coordinates. This implies that $F_{\ell; d}(x,y)=G_{\ell;d}(\<x,y\>)$, where $G_{\ell;d}$ is a polynomial defined in $[-1,1]$ and $\<x, y\>=\cos d(x,y)$, $d(x,y)$ denoting the geodesic distance on $\SSd$ between $x,y\in \SSd$.   
	
	\begin{definition}\label{GegPoly_}
		For $\ell \in \N^*$, the polynomial $G_{\ell;d}(\<x,y\>)$ defined as
		\begin{equation}\label{covT_2}
			G_{\ell;d }(\<x,y\>)=\frac{\mu_d}{n_{\ell;d}} \sum_{m=1}^{n_{\ell;d}} Y_{\ell,m;d}(x)Y_{\ell,m;d}(y)\qquad{x,y\in\SSd},
		\end{equation}
		corresponds to the so called Gegenbauer polynomial of degree $\ell$.  
	\end{definition}
	Relation in \eqref{covT_2} follows from the well known addition Theorem for Gegenbauer polynomials. 
	
	\begin{remark}
		Gegenbauer polynomial is a particular case of Jacobi polynomial. Referring to \cite[\S 4.7]{Sze39}, Jacobi polynomials $J_{\ell}^{\alpha, \beta}$, $\ell\in \N$, with parameters $\alpha, \beta >-1$, are a family of polynomials in $[-1,1]$, that are orthogonal with respect the weight $\omega(t)=(1-t)^\alpha (1+t)^{\beta}$. Here we consider the Jacobi polynomials normalized in such a way 
		$$
		J_{\ell}^{\alpha, \beta} (1)=1
		$$  
		(differently from \cite{Sze39}, in which $J_{\ell}^{\alpha, \beta}(1)=\binom{n+\alpha}{n}$). Gegenbauer polynomials correspond to Jacobi polynomials with $\alpha=\beta=\frac{d}2-1$. Then they are a family of polynomials in $[-1,1]$ that are orthogonal with respect the weight $\tilde \omega(t)=(1-t^2)^{\frac d2 -1}$. We choose to consider 
		\begin{equation}\label{Normal_}
			G_{\ell;d}(1)=1. 
		\end{equation}
		When $d=2$, $\{G_{\ell;2}\}_\ell\equiv \{P_{\ell}\}_{\ell}$ that is the family of the so called Legendre polynomials.
	\end{remark}
	
	Let us recall some properties of Gegenbauer polynomials that we will use later in our work. First one is recalled in the following proposition. 
	\begin{proposition}\label{bound}
		Let $\ell\in \N^*$. Then for all $t\in [-1,1]$,  
		$$
		|G_{\ell;d}(t)|\leq 1
		$$
		and 
		$$
		\int_{\SSd} G_{\ell;d}(\<x,y\>)^2 dy=\frac{\mu_d}{n_{\ell;d}}\qquad{x\in\SSd}. 	
		$$
	\end{proposition}
	
	\begin{proof}
		We have that 
		\begin{align*}
			&G_{\ell;d}(\<x,y\>)^2=\Big(\frac{\mu_d}{n_{\ell;d}}\sum_{m=1}^{n_{\ell;d}} Y_{\ell,m;d}(x) Y_{\ell,m;d}(y) \Big)^2\\
			&\leq \frac{\mu^2_d}{n^2_{\ell;d}}\sum_{m=1}^{n_{\ell;d}} Y^2_{\ell,m;d}(x) \sum_{m'=1}^{n_{\ell;d}} Y^2_{\ell,m';d}(y) = G_{\ell;d}(\<x,x\>)G_{\ell;d}(\<y,y\>)=G^2_{\ell;d}(1)=1.
		\end{align*}
		Moreover
		\begin{align*}
			\int_{\SSd}G_{\ell;d}(\<x,y\>)^2 dy &= \frac{\mu^2_d}{n^2_{\ell;d}}\sum_{m, m'=1}^{n_{\ell;d}} \int_{\SSd} Y_{\ell,m;d}(x) Y_{\ell,m;d}(y)  Y_{\ell,m';d}(x) Y_{\ell,m';d}(y) dy \\
			&=\frac{\mu^2_d}{n^2_{\ell;d}}\sum_{m=1}^{n_{\ell;d}} Y^2_{\ell,m;d}(x) =\frac{\mu_d}{n_{\ell;d}}.
		\end{align*}
	\end{proof}
	
	We finally recall the following reproducing formula for Gegenbauer polynomials. 
	\begin{proposition}\label{ReprForm_}
		Let $\ell \in \N^*$ and $x,y \in \SSd$. Then
		\begin{equation}\label{G1}
			\int_{\SSd} G_{\ell;d}(\<x,z\>)G_{\ell;d}(\<z,y\>)dz=\frac{\mu_d}{n_{\ell;d}}G_{\ell; d}(\<x,y\>).
		\end{equation}
	\end{proposition}
	
	\begin{proof}
		Let us compute the r.h.s. of \eqref{G1} inserting the representation \eqref{covT_2} and exploiting the orthonormality of the hyperspherical harmonics:
		\begin{align*}
			\int_{\SSd} G_{\ell;d}(\<x,z\>)G_{\ell;d}(\<z,y\>)dz&=\frac{\mu^2_d}{n^2_{\ell;d}}\sum_{m,m'=1}^{n_{\ell;d}}\int_{\SSd} Y_{\ell,m;d}(x)Y_{\ell,m;d}(z)Y_{\ell,m';d}(y)Y_{\ell,m';d}(z) dz\\
			&=\frac{\mu^2_d}{n^2_{\ell;d}}\sum_{m,m'=1}^{n_{\ell;d}}Y_{\ell,m;d}(x)Y_{\ell,m';d}(y)       \int_{\SSd} Y_{\ell,m;d}(z)Y_{\ell,m';d}(z) dz\\
			&=\frac{\mu^2_d}{n^2_{\ell;d}}\sum_{m=1}^{n_{\ell;d}}Y_{\ell,m;d}(x)Y_{\ell,m;d}(y)  =G_{\ell;d}(\<x,y\>).  
		\end{align*}
	\end{proof}
	
	For later use, we resume here a result can be founded in \cite[Proposition 1.1]{ROS20}, in which the asymptotic behavior as $\ell\to\infty$ of the moments of the Gegenbauer polynomials is described.

	\begin{proposition}\label{GegProp}
		
		For $q\in\N$, $q\geq 2$, set
		\begin{equation}\label{cqd}
			c_{q;d}= 
			\begin{cases}
				\Big(2^ {\frac d2 -1}\Big(\frac d2-1\Big)! \Big)^ q\int_0^ \infty J_{\frac d2 -1}(u)^ q u^ {-q(\frac d2-1)+d-1}du& \mbox{ if } q\geq 3,\cr
				\frac{(d-1)! \mu_d}{4\mu_{d-1}}&\mbox{                                     if } q=2,\\
			\end{cases}
		\end{equation}
		where $J_{\frac d2 -1}$ is the Bessel function of order $\frac d2 -1$. 
		
		For $q\geq 2$ and $d\geq 2$, the function $\SSd\ni y \mapsto \int_{\SSd} G_{\ell; d}(\<x,y\>)^q dx$ is constant. Moreover, the following properties hold.
		
		For $d\geq 2$ one has $\int_{\SSd} G_{\ell; d}(\<x,y\>)^2 dx =\frac{\mu_d}{ n_{\ell;d} }$ and, as  $\ell\to\infty$,
		\begin{equation}\label{asGeg2}
			\int_{\SSd} G_{\ell; d}(\<x,y\>)^2 dx =
			2\mu_{d-1} \frac{c_{2;d}}{\ell^{d-1}}(1+o_{2;d}(1)).
		\end{equation}
		
		Set now $q\geq 3$.  Then 
		\begin{itemize}
			\item if $d\geq 3$, then
			\begin{equation}\label{asGegq}
				\int_\SSd G_{\ell;d} (\<x,y\>)^q dx = 2 \mu_{d-1}\frac{c_{q;d}}{\ell^d} (1+o_{q;d}(1));
			\end{equation}

			\item
			if $d=2$, the behavior differs according to $q\neq 4$ (being as in \eqref{asGegq}) and $q= 4$: 
			\begin{equation}\label{asLegq}
				\int_\SS2 G_{\ell; 2}(\<x,y\>)^ qdx\equiv 
				\int_\SS2
				P_\ell(\<x,y\>)^q dx = 
				\begin{cases}
					\frac{12\log \ell}{\pi\ell^2}(1+o_{4;2}(1)) & q=4\\
					\frac{4\pi c_{q;2}}{\ell^2}(1+o_{q;2}(1)) & q=3\mbox{ or } q\geq 5.
				\end{cases}
			\end{equation}
			
		\end{itemize}
		
	\end{proposition}
	Proof details can be found in \cite[Proposition 1.1]{ROS20}, see also \cite{MW11}, \cite{MW14}. It is worth noticing that constants $c_{q;d}$ are strictly positive for every $q, d$; in particular for odd $q$ this result is highly non trivial, see \cite{GMT23}. 
	
	\subsection{Gaunt integrals}
	In Proposition \ref{GegProp} we have recalled the asymptotic behaviour of moments of Gegenbauer polynomials in high energy limit. For the purpose of the discussion, we are also interested in the analysis of integrals involving products of Gegenbauer polynomials. Then, we introduce the notion of Gaunt integrals, which is crucial in this part of our discussion. 
	
	\begin{definition}\label{GauntIntegralsDef}
		For $d\geq 2$, $q\in \N$ and $n_1,\ldots, n_q \in \{0, \ldots, n_{\ell;d}\}$, the generalized Gaunt integral (of order $q$) on $\SSd$ is:
		$$
		\G_{\ell n_1, \ldots, \ell n_q}=\int_{\SSd} \prod_{i=1}^q Y_{\ell n_i}(x) dx.
		$$
	\end{definition} 
	Let us recall in \cite[Lemma 1.5]{Ros19}, that linked a sum of products of Gaunt integrals of order $3$ with a particular moment of Gegenbauer polynomial.  
	\begin{lemma}
		For every $\ell \in \N$, $M, M'\in \{1,2,\ldots, n_{\ell;d}\}$ and $d\geq 2$ 
		\begin{align*}
			\sum_{m_1, m_2=1}^{n_{\ell;d}} \G_{\ell m_1, \ell m_1, \ell M}\G_{\ell m_1, \ell m_1, \ell M'}=\delta_M^{M'}\frac{n_{\ell;d}^2 \mu_{d-1}}{\mu_d^2}\int_{-1}^1 G_{\ell;d}(t)^3(\sqrt{1-t^2})^{d-2} dt.
		\end{align*}
	\end{lemma}
	We provide a generalization of above result, giving a sort of product formula for Gaunt integrals. This is of interest in itself, in particular for the analysis of some random functionals on $\SSd$ which are beyond the scopes of this work.

	\begin{lemma}\label{GauntProd}
		For $q\in \N$ and $n, n_1, \ldots, n_q \in \{0,\ldots, n_{\ell;d}\}$ one has
		$$
		\sum_{m_1, \ldots, m_r=1}^{n_{\ell;d}} \G_{\ell m_1, \ldots, \ell m_r, \ell n}\, \G_{\ell m_1, \ldots, \ell m_r, \ell n_1, \ldots,\ell n_q}= \Big(\frac{n_{\ell;d}}{\mu_d} \Big)^{r-1}\hat \gamma_{\ell ;r} \,\,\G_{\ell n, \ell n_1,\ldots, \ell n_q}
		$$  
		where
		\begin{equation}\label{gammal}
			\hat\gamma_{\ell; r}=n_{\ell;d}\frac{\mu_{d-1}}{\mu_d}\int_{-1}^1 G_{\ell;d}(t)^{r+1} \big(\sqrt{1-t^2}\big)^{d-2} dt.
		\end{equation}
	\end{lemma}
	
	\begin{proof}
		From \eqref{covT} we have
		\begin{align*}
			&\sum_{m_1, \ldots, m_r} \G_{\ell m_1, \ldots, \ell m_r, \ell n}\, \G_{\ell m_1, \ldots, \ell m_r, \ell n_1, \ldots, n_q}=\\
			&=\sum_{ m_1, \ldots,  m_r} \int_{(\SSd)^2} \prod_{i=1}^r Y_{\ell m_i}(x)Y_{\ell m_i}(y)  Y_{\ell n}(x) \prod_{j=1}^q Y_{\ell n_j}(y) dxdy\\
			&=\Big(\frac{n_{\ell;d}}{\mu_d}\Big)^r \int_{(\SSd)^2} G_{\ell;d}(\<x,y\>)^r \,\,Y_{\ell n}(x) \prod_{j=1}^q Y_{\ell n_j}(y) dxdy.
		\end{align*}
		Since $\Big( \Big( \frac{\mu_{d-1} n_{\ell;d}}{\mu_d}\Big)^\frac{1}{2} G_{j;d}  \Big)_{j=0}^{r \ell}$ is an orthonormal system on $[-1,1]$ with the weight function $(1-t^2)^{d/2-1}$ (see \cite{Sze39}), we can write
		\begin{equation}\label{sviluppo}
			G_{\ell;d}(t)^r =\sum_{j=0}^{r \ell} \gamma_{j,\ell; r} G_{j;d}(t)
		\end{equation}
		where, for $j=0,1,\ldots,r\ell$
		\begin{equation}\label{gammalr}
			\gamma_{j,\ell; r}=n_{j;d}\frac{\mu_{d-1}}{\mu_d}\int_{-1}^1 G_{\ell;d}(t)^r G_{j;d}(t)\big(\sqrt{1-t^2}\big)^{d-2} dt.
		\end{equation}
		Substituting \eqref{sviluppo}, we have
		\begin{align*}
			&\sum_{m_1, \ldots, m_r} \G_{\ell m_1, \ldots, \ell m_r, \ell n}\, \G_{\ell m_1, \ldots, \ell m_r, \ell n_1, \ldots, n_q}\\
			&=\Big(\frac{n_{\ell;d}}{\mu_d}\Big)^r \int_{(\SSd)^2} \sum_{j=0}^{r\ell} \gamma_{j, \ell; r} G_{j;d}(\<x,y\>)\,\,Y_{\ell n}(x) 
			\prod_{i=1}^q Y_{\ell n_i}(y)dxdy\\
			&=\Big(\frac{n_{\ell;d}}{\mu_d}\Big)^{r-1}\sum_{j=0}^{r\ell}  \gamma_{j,\ell; r} \int_{(\SSd)^2} \sum_{h=0}^{n_{\ell;d}} Y_{j h}(x) Y_{j h}(y)\,\,Y_{\ell n}(x) 
			\prod_{i=1}^q Y_{\ell n_i}(y)dxdy\\
			&=\Big(\frac{n_{\ell;d}}{\mu_d}\Big)^{r-1}\sum_{j=0}^{r\ell}  \gamma_{j,\ell; r}  \sum_{h=0}^{n_{\ell;d}} \underbrace {\int_{\SSd} Y_{j h}(x) Y_{\ell n}(x) dx }_{=\1_{j=\ell}\1_{h=n}}\int_{\SSd}Y_{j h}(y) 
			\prod_{i=1}^q Y_{\ell n_i}(y)dy\\
			&=\Big(\frac{n_{\ell;d}}{\mu_d}\Big)^{r-1}   \gamma_{\ell,\ell ; r}   \, \int_{\SSd}Y_{\ell n}(y) 
			\prod_{i=1}^q Y_{\ell n_i}(y)dy\\
			&=\Big(\frac{n_{\ell;d}}{\mu_d}\Big)^{r-1}   \gamma_{\ell,\ell ; r}  \,\G_{\ell n, \ell n_1, \ldots, \ell n_q}.
		\end{align*}
		Since $\hat{\gamma}_{\ell;r}= \gamma_{\ell,\ell ; r} $, the statement follows.
	\end{proof}
	
	As a consequence of Lemma  \ref{GauntProd}, we state the following properties for cross moments of Gegenbauer polynomials (which is of independent interest). 
	\begin{lemma}\label{resume}
		For $d\geq 2$, the following statements holds:
		\begin{enumerate}
			\item for $q_1, q_2 \geq 2$ there exists a positive constant $c_{d,1}$ such that
			\begin{equation}\label{A_1}
				\begin{array}{l}
					\int_{(\SSd)^3} G_{\ell;d}(\<x_1,x_2\>) G_{\ell;d}(\<x_1,x_3\>)^{q_1} G_{\ell;d}(\<x_2,x_3\>)^{q_2} dx_1dx_2dx_3 \smallskip\\
					\quad = c_{d,1} \int_{(\SSd)^2 }G_{\ell;d} (\<x,y\>)^{q_1+1} dxdy\int_{(\SSd)^2}G_{\ell;d} (\<x,y\>)^{q_2+1} dxdy;
				\end{array}
			\end{equation}
			\item for $q_1,q_2 \geq 2$ and $q_3\geq 1$ there exists a positive constant $c_{d,2}$ such that
			\begin{equation}\label{B}
				\begin{array}{l}
					\int_{(\SSd)^4} G_{\ell;d}(\<x_1, x_2\>)G_{\ell;d}(\<x_1, x_4\>)^{q_1} G_{\ell;d}(\<x_2, x_3\>)^{q_2}G_{\ell;d}(\<x_3, x_4\>)^{q_3}dx_1dx_2dx_3dx_4\smallskip\\
					\quad=c_{d,2 }\prod_{i=1}^3\int_{(\SSd)^2} G_{\ell; d}(\<x,y\>)^{q_i+1}dxdy;
				\end{array}
			\end{equation} 
			\item for $q_1, q_2 \geq 2$ and $q_3\geq 0$ there exists a positive constant $c_{d,3}$ such that
			\begin{equation}\label{C}
				\begin{array}{l}
					\int_{(\SSd)^4} G_{\ell;d}(\<x_1,x_2\>)G_{\ell;d}(\<x_1,x_4\>)^{q_1} G_{\ell;d}(\<x_2, x_3\>)^{q_2} G_{\ell;d}(\<x_2,x_4\>)^{q_3}
					G_{\ell;d}(\<x_3,x_4\>)dx\smallskip\\
					\quad = c_{d,3} \prod_{i=1}^3 \int_{(\SSd)^2} G_{\ell;d}(\<x,y\>)^{q_i +1} dxdy.
					%
					%
				\end{array}
			\end{equation} 
		\end{enumerate} 
	\end{lemma}

	\begin{proof}
		By \eqref{covT_2} we have
		\begin{align*}
			&\int_{(\SSd)^3} G_{\ell;d}(\<x_1,x_2\>) G_{\ell;d}(\<x_1,x_3\>)^{q_1} G_{\ell;d}(\<x_2,x_3\>)^ {q _2}dx_1dx_2dx_3\\
			&=	\Big(\frac{\mu_d}{n_{\ell;d}}\Big)^{q_1+q_2+1}\sum_{r, n_1, \ldots, n_{q_1}=1}^{n_{\ell;d}} \sum_{m_1,\ldots, m_{q_2}=1}^{n_{\ell;d}}  \\
			&\int_{(\SSd)^3} Y_{\ell,r;d}(x_1) Y_{\ell, r;d}(x_2)\prod_{i=1}^{q_1} Y_{\ell, n_i; d}(x_1) Y_{\ell, n_i; d}(x_3)\prod_{j=1}^{q_2} Y_{\ell, m_j; d}(x_2) Y_{\ell, m_j; d}(x_3) dx_1dx_2dx_3.
		\end{align*}
		and by using Definition \ref{GauntIntegralsDef} and Lemma \ref{GauntProd}, when $q_1,q_2\geq 2$, we have
		\begin{align*}
			&\int_{(\SSd)^3} G_{\ell;d}(\<x_1,x_2\>) G_{\ell;d}(\<x_1,x_4\>)^{q_1} G_{\ell;d}(\<x_2,x_4\>)^ {q _2}dx_1dx_2dx_4\\
			&=\Big(\frac{\mu_d}{n_{\ell;d}}\Big)^{q_1+q_2+1} \sum_{r, n_1,\ldots, n_{q_1}} \sum_{m_1,\ldots, m_{q_2}} \G_{\ell r , \ell n_1, \ldots, \ell n_{q_1}} \G_{\ell r, \ell m_1,\ldots, \ell m_{q_2}} \G_{\ell n_1,\ldots, n_p,\ell m_1, \ldots, \ell m_{q_2}}\\
			&=\Big(\frac{\mu_d}{n_{\ell;d}}\Big)^{q_2+2} \hat \gamma_{\ell; q_1}   \sum_{r, m_1,\ldots, m_{q_2}} \G_{\ell r , \ell m_1, \ldots, \ell m_{q_2}}^2 \\
			&=c_{d,1 }\int_{(\SSd)^2}G_{\ell;d} (\<x,y\>)^{q_2+1} dxdy \int_{(\SSd)^2 }G_{\ell;d} (\<x,y\>)^{q_1+1} dxdy.
		\end{align*}
		
		By the same procedure we have
		\begin{align*}
			&\int_{(\SSd)^4} G_{\ell;d}(\<x_1, x_2\>)G_{\ell;d}(\<x_1, x_4\>)^{q_1} G_{\ell;d}(\<x_2, x_3\>)^{q_2}G_{\ell;d}(\<x_3, x_4\>)^{q_3}dx_1dx_2dx_3dx_4\\
			&=\Big(\frac{\mu_d}{n_{\ell;d}}\Big)^{\sum_{i=1}^3 q_i+1}\sum_{\underset{r, n_1,\ldots n_{q_1}}{m_1,\ldots, m_{q_2},t_1,\ldots, t_{q_3}}} \G_{\ell r, \ell n_1, \ldots, \ell n_{q_1}}\G_{\ell r, \ell m_1, \ldots, \ell m_{q_2}}\G_{ \ell n_1, \ldots, \ell n_{q_1}, \ell t_1,\ldots, \ell t_{q_3}} \G_{ \ell m_1, \ldots, \ell m_{q_2},\ell t_1, \ldots, \ell t_{q_3}}\\
			&=\Big(\frac{\mu_d}{n_{\ell;d}}\Big)^{q_2+q_3+2}\hat \gamma_{\ell;q_1} \sum_{r,m_1,\ldots, m_{q_2},t_1,\ldots, t_{q_3}}\G_{\ell r, \ell m_1, \ldots, \ell m_{q_2}}\G_{ \ell m_1, \ldots, \ell m_{q_2},\ell t_1, \ldots, \ell t_{q_3}}\G_{\ell r, \ell t_1,\ldots, \ell t_{q_3}}\\
			&=\Big(\frac{\mu_d}{n_{\ell;d}}\Big)^{q_3+3} \hat\gamma_{\ell; q_1} \hat \gamma_{\ell; q_2} \sum_{r,t_1,\ldots, t_{q_3}}\G_{\ell r, \ell t_1,\ldots, \ell t_{q_3}}^2\\
			&=c_{d,2} \int_{(\SSd)^2} G_{\ell; d}(\<x,y\>)^{q_1+1}dxdy\int_{(\SSd)^2} G_{\ell; d}(\<x,y\>)^{q_2+1}dxdy\int_{(\SSd)^2} G_{\ell; d}(\<x,y\>)^{q_3+1}dxdy,
		\end{align*}
		with $q_1, q_2 \geq 2$ and $q_3\geq 1$. Finally
		\begin{align*}
			& \int_{(\SSd)^4} G_{\ell;d}(\<x_1,x_2\>)G_{\ell;d}(\<x_3,x_4\>)G_{\ell;d}(\<x_1,x_4\>)^{q_1} G_{\ell;d}(\<x_2,x_3\>)^{q_2} G_{\ell;d}(\<x_2,x_4\>)^{q_3}dx\\
			&=\Big(\frac{\mu_d}{n_{\ell;d}}\Big)^{q_1+q_2+q_3+2}\sum_{r_1,r_2}\sum_{m_1,\ldots, m_{q_1}}\sum_{n_1,\ldots, n_{q_2}} \times \\
			&\times \sum_{s_1,\ldots, s_{q_3}}\G_{\ell r_1, \ell m_1,\ldots, \ell m_{q_1}} \G_{\ell r_1, \ell n_1, \ldots, \ell n_{q_2}, \ell s_1, \ldots, \ell s_{q_3}}\G_{\ell r_2, \ell n_1, \ldots, \ell n_{q_2}}\G_{\ell r_2,  \ell m_1 ,\ldots, \ell m_{q_1},\ell s_1, \ldots, \ell s_{q_3}}.
		\end{align*}
		Again from Lemma \ref{GauntProd} we have
		\begin{align*}
			& \sum_{m_1,\ldots, m_{q_1}}  \G_{\ell r_1, \ell m_1, \ldots, \ell m_{q_1}}\G_{\ell r_2,  \ell m_1 ,\ldots, \ell m_{q_1},\ell s_1, \ldots, \ell s_{q_3}}=\Big(\frac{n_{\ell;d}}{\mu_d}\Big)^{q_1-1} \gamma_{\ell; q_1 } \G_{\ell r_1,\ell r_2, \ell s_1,\ldots, \ell s_{q_3}}
		\end{align*}
		and 
		\begin{align*}
			\sum_{n_1,\ldots,  n_{q_2}}  \G_{\ell r_1, \ell n_1, \ldots, \ell n_{q_2}, \ell s_1, \ldots, \ell s_{q_3}}\G_{\ell r_2, \ell n_1, \ldots, \ell n_{q_2}}  =\Big(\frac{n_{\ell;d}}{\mu_d} \Big)^{q_2-1}\hat\gamma_{\ell;q_2}\,   \G_{\ell r_1,\ell r_2, \ell s_1,\ldots, \ell s_{q_3}}.
		\end{align*}
		Then, by inserting,
		\begin{align*}
			& \int_{(\SSd)^4} G_{\ell;d}(\<x_1,x_2\>)G_{\ell;d}(\<x_3,x_4\>)G_{\ell;d}(\<x_1,x_4\>)^{q_1} G_{\ell;d}(\<x_2,x_3\>)^{q_2} G_{\ell;d}(\<x_2,x_4\>)^{q_3}dx\\
			&=\Big(\frac{\mu_d}{n_{\ell;d}}\Big)^{q_3+4}\hat\gamma_{\ell; q_1}\hat  \gamma_{\ell;q_2}\,\sum_{r_1, r_2}\sum_{s_1,\ldots, s_{q_3}}\G_{\ell r_1, \ell r_2, \ell s_1,\ldots, \ell s_{q_3}}^2\\
			&=\Big(\frac{\mu_d}{n_{\ell;d}}\Big)^2 \hat\gamma_{\ell; q_1}  \hat\gamma_{\ell;q_2} \int_{(\SSd)^2} G_{\ell;d}(\<x,y\>)^{q_3+2} dxdy\\
			&=c_{d,3} \int_{(\SSd)^2} G_{\ell;d}(\<x,y\>)^{q_1 +1} dxdy.  \int_{(\SSd)^2} G_{\ell;d}(\<x,y\>)^{q_2+1} dxdy.  \int_{(\SSd)^2} G_{\ell;d}(\<x,y\>)^{q_3+2} dxdy.
		\end{align*}
		
	\end{proof}

	\section{Random hyperspherical harmonics}
	
	Section \ref{ISON_Gau}, Section \ref{ProbDist} and  Section \ref{REAL_spher} provide the general setting in which we will move. Now we introduce the particular Gaussian field which we consider and its main properties. Then we define the main topic of our work, defining the statistics of our Gaussian field.     
	\begin{definition}
		Let us fix $\ell \in \mathbb N^*$. The $\ell$th random hyperspherical harmonic $T_\ell$ on $\mathbb S^d$ is 
		\begin{equation}\label{Tl}
			T_\ell(x)=\sqrt{\frac{\mu_d}{n_{\ell;d}}}\sum_{m=1}^{n_{\ell;d}} a_{\ell,m} Y_{\ell,m;d}(x),\quad x\in \mathbb{S}^d,
		\end{equation} 
		where $(a_{\ell,m})_{m=1}^{n_{\ell;d}}$ are standard Gaussian i.i.d. random variables in $\R$.
	\end{definition}
	\begin{remark}
		Let us recall that 
		$$
		T_\ell : \Omega \times \mathbb S^d \longrightarrow \R; \quad (\omega, x)\mapsto T_\ell(\omega, x)
		$$
		is an isotropic and centered Gaussian random field on $\mathbb S^d$ (we will omit the dependence of $\omega$ in $T_\ell(\omega,x)$, as usual). The isotropy property for $T_\ell$  means that the random fields $T_\ell(\cdot)$ and $T_\ell(g \cdot)$ are equal in law for every $g\in SO(d+1)$, the special orthogonal group of $(d+1)\times (d+1)$-matrices. 
		Indeed, by the addition formula for hyperspherical harmonics in \eqref{covT_2}), the covariance kernel of $T_\ell$ is given by
		\begin{equation}\label{covT}
			\Cov(T_\ell(x),T_\ell(y)) = \frac{\mu_d}{n_{\ell;d}}\sum_{m=1}^{n_{\ell;d}} Y_{\ell,m;d}(x) Y_{\ell,m;d}(y)= G_{\ell;d}(\<x,y\>), \quad x,y\in \mathbb S^d,
		\end{equation}
		where $G_{\ell;d}$ denotes the $\ell$th Gegenbauer polynomial of degree $\ell$. Then the covariance function of $T_\ell$ is independent from trasformation given by $SO(d+1)$. From normalization choice in \eqref{Normal_} follows that $\Var(T_\ell(x)) = 1$ for every $x\in \mathbb S^d$. Note also that, by construction (\ref{Tl}), $T_\ell$ a.s. satisfies the Helmholtz equation (\ref{helmholtz}) with eigenvalue $-\lambda_\ell = -\ell(\ell+d-1)$ (hence the name "random hyperspherical harmonic"). 
	\end{remark}
	Let us notice that our field $T_\ell$ can be seen as an isonormal Gaussian field. Recalling notions introduced in Subsection \ref{ISON_Gau}, we denote $H=L^2(\mathbb{S}^d,\B(\mathbb{S}^d), \Leb)$ the real separable Hilbert space of square integrable functions on $\mathbb{S}^d$ w.r.t. the Lebesgue measure, with inner product
	$
	\<f,g\>_H= \int_{\mathbb{S}^d} f(x) g(x) dx.
	$
	We define a Gaussian white noise on $\SSd$, that is the centered Gaussian family
	$
	W=\{W(A), A\in \B(\SSd), \Leb(A)<+\infty\}
	$
	such that for $A,B \in \B(\SSd)$, we have
	$$
	\E[W(A)W(B)]=\int_{\SSd} 1_{A\cap B} (x) dx.
	$$
	Once the noise $W$ has been introduced, an isonormal Gaussian random field $T$ on $H$ can be defined as in \eqref{IsonOnA}: for $f\in H$,
	\begin{equation}\label{gausField}
		T(f)=\int_{\SSd} f(x) W(dx)
	\end{equation}
	that is the Wiener-Ito integral of $f$ with respect to $W$. It is an isonormal Gaussian field on $H$ whose covariance is given by
	\begin{equation}\label{CovTf}
		\Cov(T(f), T(g))=\<f,g\>_H
	\end{equation}
	(see \eqref{CovISON}).
	
	\begin{proposition}\label{ReprIson}
		For $\ell\in\N^*$, let $f_{x;\ell}, x\in\SSd$ the family of functions defined as
		\begin{equation}\label{kernel}
			f_{x;\ell}(y) = \sqrt{\frac{n_{\ell;d}}{\mu_d}} G_{\ell; d}(\<x,y\>), \quad y\in\SSd.
		\end{equation} 
		Then 
		\begin{equation}\label{spR}
			T_\ell(x)=\int_{\mathbb{S}^d} f_{x;\ell}(y) W(dy), \quad x\in \mathbb{S}^d.
		\end{equation}
		Hence the representation in \eqref{spR} holds in law. 
	\end{proposition} 
	
	\begin{proof}
		We know that $T_\ell$ is a centered Gaussian random field on $\SSd$. The field 
		$$
		\int_{\mathbb{S}^d} f_{x;\ell}(y) W(dy), \quad x\in \mathbb{S}^d
		$$ 
		is also a centered Gaussian random field (see \eqref{IsonOnA}). We have only prove that the covariance structures are equal. 
		We have
		\begin{align*}
			&\E\Big[\int_{\mathbb{S}^d} f_{x;\ell}(y) W(dy)\int_{\mathbb{S}^d} f_{x';\ell}(y) W(dy)\Big]=\int_{\SSd}f_{x;\ell}(y) f_{x';\ell}(y)dy\\
			&=\frac{n_{\ell;d}}{\mu_d}\int_{\SSd} G_{\ell;d}(\<x,y\>)G_{\ell;d}(\<x',y\>)dy =G_{\ell; d}(\<x,x'\>),
		\end{align*}
		and it is the covariance $\E[T_\ell(x)T_{\ell}(x')]$. It concludes the proof.
		
	\end{proof}

	To conclude the section, we give some formulas that will be used in the sequel. They involve the Malliavin derivatives of $T_\ell$, following the setting described in Subsection \ref{MDGF}. Let $x\in \SSd$ and $T_\ell(x)$ be the Gaussian random field in \eqref{spR}. Then
	$$
	D_y T_\ell(x)=D_y\Big(\int_{\SSd}f_{x;\ell}(z)W(dz) \Big)=f_{x;\ell}(y).
	$$ 
	As an immediate consequence of the chain rule \eqref{chainRule}, for  $q\in \N$ and $p\geq 1$ then $H_q(T_\ell(x))\in \DD^{1,p}$ and, from \eqref{smoothF} and \eqref{spR},
	\begin{align*}
		D_y H_q(T_{\ell}(x))&=H_q'(T_\ell (x))D_yT_\ell (x)=q H_{q-1}(T_{\ell}(x)) D_yT_\ell(x)\\&=q H_{q-1}(T_{\ell}(x)) f_{x;\ell}(y)=q H_{q-1}(T_{\ell}(x)) \sqrt{\frac{n_{\ell;d}}{\mu_d}} G_{\ell;d}(\<x, y \>).
	\end{align*}
	Iterating the argument, $H_q(T_\ell(x))\in \DD^{k,p}$ for every $k\in\N$ and
	\begin{equation}\label{DkHp}
		D^{(k)}_{y_1,\ldots, y_k } H_q(T_\ell (x)) =  \Big(\frac{n_{\ell;d}}{\mu_d}\Big)^{\frac{k}{2}}\frac{q!}{(q-k)!} H_{q-k}(T_\ell(x)) \prod_{r=1}^k G_{\ell;d}(\<x, y_r\>).
	\end{equation}
	Moreover, by developing standard density arguments, \eqref{DkHp} gives
	$\int_{\SSd} H_q(T_\ell (x)) dx\in \DD^{k,p}$ 
	and
	\begin{equation}\label{DkIntHp}
		D^{(k)}_{y_1,\ldots, y_k } \int_\SSd H_q(T_\ell (x)) dx=  
		\Big(\frac{n_{\ell;d}}{\mu_d}\Big)^{\frac{k}{2}}\frac{q!}{(q-k)!} \int_{\SSd}H_{q-k}(T_\ell(x)) \prod_{r=1}^k G_{\ell;d}(\<x, y_r\>) dx.
	\end{equation}

	\subsection{Statistics of random hyperspherical harmonics}
	
	We are interested in functionals of random hyperspherical harmonics of the type 
	\begin{equation}\label{Xl}
		X_\ell := \int_{\mathbb S^d} \varphi(T_\ell(x)) dx,
	\end{equation}
	where $\varphi\in L^2(\nu)$. In particular, we study the asymptotic behavior of the sequence of random variables $\lbrace X_\ell\rbrace_{\ell\in \mathbb N}$ as $\ell\to +\infty$ by means of chaotic decompositions \cite[\S 2.2]{NP12}: if $Z\sim \mathcal N(0,1)$, then $\varphi(Z)$ can be written as an orthogonal series in $L^2(\mathbb P)$ as follows 
	\begin{equation}\label{chphi}
		\varphi(Z)=\sum_{q\geq 0} \frac{b_q}{q!} H_q(Z),\quad \text{ where }\quad b_q:=\E[\varphi(Z)H_q(Z)],
	\end{equation}
	(properties of Hermite polynomials and the Wiener chaos expansion will be recalled in Section \ref{ISON_Gau}).  Substituting (\ref{chphi}) into (\ref{Xl}) gives the chaotic expansion for $X_\ell$:
	\begin{equation}\label{chXl}
		X_\ell = \sum_{q\ge 0} X_\ell[q],\quad \text{ where }\quad  X_\ell[q]:= \frac{b_q}{q!} \int_{\mathbb S^d} H_q(T_\ell(x))\,dx,
	\end{equation}
	see Section \ref{Conv0} for more details. Note that the term corresponding to $q=1$ in the series (\ref{chXl}) is null. Indeed, being $H_1(x)=x$,
	\begin{align*}
		\int_{\SSd} H_1(T_\ell(x))dx=\int_{\SSd} T_\ell(x)dx=\sqrt{\frac{\mu_d}{n_{\ell;d}}}\sum_{m=1}^{n_{\ell;d}}\int_{\SSd} Y_{\ell,m;d}(x)dx, 
	\end{align*}
	By orthogonality properties of hyperspherical harmonics, we have $\int_{\mathbb S^d} Y_{\ell,m;d}(x)\,dx = 0$ for $\ell\in \mathbb N^*$.
	
	By standard properties of Hermite polynomials \cite[Section 2.2]{NP12} it is immediate to check that 
	\begin{equation}\label{first_moments}
		\mathbb E[X_\ell] = \mathbb E [X_\ell[0]]= \mathbb E[\varphi(Z)] \mu_d,\quad \Var(X_\ell) = \sum_{q\ge 2}\frac{b_q^2}{q!} \int_{(\mathbb S^d)^2} G_{\ell;d}(\<x,y\>)^q\,dx dy. 
	\end{equation}
	\begin{remark}\rm
		By the symmetry property of Gegenbauer polynomials, i.e., 
		$$
		G_{\ell;d}(t)=(-1)^\ell G_{\ell; d}(-t),
		$$
		if both $\ell$ and $q$ are odd the $q$th moment of $G_{\ell;d}$ vanishes. From now on, we take only even $\ell$, hence by $\ell\to +\infty$ we mean \emph{as $\ell$ goes to infinity along even $\ell$}. 
	\end{remark}

	It is known  that the $q$th moment of Gegenbauer polynomials $G_{\ell;d}$ behaves, as $\ell\to +\infty$, as $1/n_{\ell;d}$ up to positive constants for $q=2$ while for $q\ge 3$ it is $o(1/n_{\ell;d})$. 
	
	We need to introduce some more notation: we define the standardized statistic 
	\begin{equation}\label{tildeXl}
		\X_\ell=\frac{ X_\ell-\E[X_\ell]}{\sqrt{\var(X_\ell)}},
	\end{equation} 
	
	\cite{ROS20} provides the following result for Hermite-$2$ rank statistics.
	
	\begin{theorem} [Theorem 1.7 in \cite{ROS20}]\label{ROS}
		Let $\varphi$ be as in \eqref{chphi} such that $b_2\neq 0$. Then, as $\ell\to +\infty$,
		\begin{equation}\label{varXl}
			\Var(X_\ell) \sim \frac{b_2^2}{2} \frac{(\mu_d)^2}{n_{\ell;d}}, 
		\end{equation}
		and moreover 
		\begin{equation}\label{dW-th}
			\dW(\X_{\ell}, Z)=O\big(\ell^{-\frac{1}{2}}\big).
		\end{equation}
	\end{theorem}
	Theorem \ref{ROS} only deals with Wasserstein distance, and gives no information on the speed of convergence in stronger probability metrics such as Total Variation.
	
	The main goal of this work is to strenghten and upgrade Theorem \ref{ROS} from Wasserstein to Total Variation distance for suitably regular nonlinear functionals of $T_\ell$ of the form (\ref{Xl}).

	\chapter{Convergence in Total Variation for smooth statistics}\label{chapter:main_result}
	
	The assumptions on $\varphi$ in Theorem \ref{ROS} are rather weak: it suffices that $\varphi$ is a square integrable function w.r.t. the Gaussian measure $\nu$ and $b_2\neq 0$.  In order to investigate the convergence for $X_\ell$ towards the Gaussian law in Total Variation distance, we need $\varphi(Z)$ to satisfy some additional regularity properties in the Malliavin sense.  These are summarized in the following condition.

	\begin{assumption}\label{ASSUMPTION}
		Let $\varphi(Z)$ fulfill  \eqref{chphi}. We assume that $b_2\neq 0$. Moreover,  
		$\varphi(Z)\in Dom(L)$ and 	
		$\varphi(Z), L\varphi(Z)\in \cap_{k\geq 0}\cap_ {p\geq 2}\DD^{k,p}$, that is,  for every $k\in \N$ and $p\geq 2$ the $k$th order Malliavin derivative of $\varphi(Z)$ and of $L\varphi(Z)$, given by
		\begin{equation}\label{Dkphi}
			D^ k\varphi(Z)=\sum_{q\geq k} \frac{b_{q}}{(q-k)!} H_{q-k}(Z)
			\quad \text{ and } \quad D^ kL\varphi(Z)=-\sum_{q\geq k} q\,\frac{b_{q}}{(q-k)!} H_{q-k}(Z),
		\end{equation}
		exist and belong to $L^p(\mathbb P)$.
		Furthermore, the same properties are satisfied by the function $\phi\in L^2(\nu)$ defined by
		\begin{equation}\label{Tpsi}
			\phi(z)= \sum_{q\geq 2} \frac{|b_q|}{q!} H_{q}(z),
		\end{equation}
		that is, $\phi(Z)\in Dom(L)$ and 	 $\phi(Z), L\phi(Z)\in  \cap_{k\geq 0}\cap_ {p\geq 2}\DD^{k,p}$: for $k\in\N$ and $p\geq 2$, 
		\begin{equation}
			\label{DkTphi}
			D^ k\phi(Z)=\sum_{q\geq 2\vee k} \frac{|b_{q}|}{(q-k)!} H_{q-k}(Z)\quad
			\text{ and }\quad
			D^ kL\phi(Z)=-\sum_{q\geq 2\vee k} q\frac{|b_{q}|}{(q-k)!} H_{q-k}(Z)
		\end{equation}
		both belonging to $L^p(\mathbb P)$. 
		%
		%
	\end{assumption}
	From now on we assume that Assumption \ref{ASSUMPTION} holds. The requested Malliavin regularity will not be really surprising once the mathematical tools we are going to use will become clear (namely, the use of Proposition \ref{Main}).
	Let us  give right away a condition allowing $\varphi$ to satisfy Assumption \ref{ASSUMPTION}.
	\begin{proposition}\label{exponential}
		Suppose  that there exist $C, R>0$ such that $|b_q|\leq CR^ q$ for every $q\geq 0$ in (\ref{chphi}). Then Assumption \ref{ASSUMPTION} holds.
	\end{proposition}
	
	\begin{proof}
		It suffices to prove that if there exist $C,R>0$ such that $|a_q|\leq CR^ q$ for every $q\geq 0$ then $\sum_{q\geq 0} \frac{|a_{q}|}{q!} |H_{q}(Z)|$ converges in $L^2(\R,\B(\R),\nu)$ to a r.v. having  all moments. 
		
		We recall that
		$$
		H_q(z)=\sum_{n=0}^ {\lfloor q/2\rfloor} q!\, \frac{(-1)^ n}{2^ nn! (q-2n)!}\, z^ {q-2n},
		$$
		hence
		$$
		|H_q(z)|\leq \sum_{n=0}^ {\lfloor q/2\rfloor} q!\, \frac{1}{2^ nn! (q-2n)!}\, |z|^ {q-2n}.
		$$
		By splitting the cases $q$ even and $q$ odd and by inserting the above estimate, we have
		\begin{align*}
			\sum_{q\geq 0} \frac{|a_{q}|}{q!} |H_{q}(Z)|
			&
			\leq \sum_{q\geq 0} \sum_{n=0}^ {q} \frac{CR^{2q}}{2^ nn! (2q-2n)!}\, |Z|^ {2q-2n}
			+\sum_{q\geq 0} \sum_{n=0}^ {q} \frac{CR^{2q+1}}{2^ nn! (2q+1-2n)!}\, |Z|^ {2q+1-2n}\\
			&
			\leq C\sum_{n\geq 0}\frac{R^ {2n}}{2^ nn!}\sum_{q\geq n}\Big(\frac{(R|Z|)^ {2q-2n}}{ (2q-2n)!}+\frac{(R|Z|)^ {2q+1-2n}}{ (2q+1-2n)!}\Big)\\
			&
			= C\sum_{n\geq 0}\frac{R^ {2n}}{2^ nn!}\sum_{m\geq 0}\frac{(R|Z|)^ {m}}{ m!}
			=C \e^ {R^2+R|Z|}.
		\end{align*} 
		Since the above r.v. has got any moment, the statement follows.
		
	\end{proof}
	
	As a meaningful example, let $t\in\R$ denote any parameter and set $\varphi(z)=\e^{tz}$. Then $\varphi$ satisfies Assumption \ref{ASSUMPTION}, as a consequence of Proposition \ref{exponential} and of the well known representation
	$$
	\e^{tz}=\e^ {\frac{t^2}2}\sum_{q\geq 0} \frac{t^ q}{q!}\, H_q(z), \quad z\in\R.
	$$
	Notice that the above function provides a nonlinear functional that does not have a finite chaos expansion.

	\section{Main results}
	
	We are now in a position to state the main result of this work.
	
	\begin{theorem}\label{mainThm}
		Let $\varphi$ satisfy Assumption \ref{ASSUMPTION}, then, for any $0< \varepsilon < 1$, as $\ell\to +\infty$,
		\begin{equation}\label{TVrate}
			\dTV(\X_{\ell}, Z) =O_\varepsilon\big ( \ell^{-\frac{1-\varepsilon}{2}} \big)
		\end{equation}
		where $O_\varepsilon$ means that the constants involved in the $O$-notation depend on $\varepsilon$. 
	\end{theorem}
	It is worth noticing that, \emph{conditionally on Theorem \ref{ROS} and the use of Proposition \ref{Main} below}, 
	our result, i.e.,   the upper bound in (\ref{TVrate}) for the Total Variation distance, cannot be improved. Indeed, the upper bound given in Proposition \ref{Main} for $\dTV(\X_{\ell}, Z)$ cannot be smaller than $d_W(\X_{\ell}, Z)^{1-\varepsilon}$, and $d_W(\X_{\ell}, Z) = O(\ell^{-\frac12})$.
	It is worth noticing that, for any dimension $d$, the upper bound in (\ref{TVrate}) for the Total Variation distance coincides (morally) with the quantitative bound of $\dW(\X_{\ell}, Z)$ in \eqref{dW-th}. So, conditionally on Theorem \ref{ROS} (and the use of Proposition \ref{Main} below), our result cannot be improved.
	
	To the best of our knowledge, Theorem \ref{mainThm} is the first result on the convergence of statistics of random hyperspherical harmonics (in particular having an infinite chaos expansion) in Total Variation distance. For $\varphi = H_q$ the $q$th Hermite polynomial with $q\ge 2$, or for $\varphi$ equal to a linear combination of  such Hermite polynomials, bounding from above $\dTV(\X_{\ell}, Z)$ is an application of the fourth moment theorem by Nourdin and Peccati, see \cite{MW14} for results in the two-dimensional case and \cite{ROS20, Ros19} for higher dimensions. 
	
	An intermediate key step to prove our main result relies on the investigation of the asymptotic behavior of the sequence of Malliavin derivatives of $X_\ell$; we stress that this analysis is new and leads to some results of independent interest, see Proposition \ref{Conv0} and Proposition \ref{UnifLim} for more details. 
	\begin{remark}
		Besides the case of higher Hermite rank functionals, that we do believe it can be dealt with by using the same approach as the one developed for the proof of Theorem \ref{mainThm} though involving heavier computations, we leave as a topic for future research the interesting case of the indicator function: for $u\in \mathbb R$, 
		\begin{equation*}
			\varphi(z) = \1_{[u,+\infty)}(z),\quad z\in \mathbb R,
		\end{equation*}
		thus $X_\ell$ is the so-called excursion area at level $u$, see \cite{MW14}. Indeed, it can be verified that for small dimension d, this random variable is not in $\mathbb D^{1,2}$. On the other hand, for big enough d, being inspired by \cite{AP20}, it is reasonable to believe that this functional is regular enough to apply at least the Second order Poincar\'e inequality. 
	\end{remark}
	\section{Proofs of the main results}
	
	In this Section we explain the main ideas behind our argument, eventually giving the proof of our main result.
	
	\subsection{On the proofs}\label{sect:proofs}
	
	To show the main ideas of the proof of Theorem \ref{mainThm} and of the results that we are going to use, we need to introduce some properties associated with the Malliavin regularity of the random variables at hands. We give here a result developed in \cite{BCP19}, holding in a \textit{purely abstract} Malliavin calculus setting (see \cite[Section 2.1]{BCP19}), that is, based on a random noise that does not need to be Gaussian (see e.g. the one used in \cite{BCPzeri}). Let us resume it here briefly.	First of all, it is assumed that the following ingredients are given:
	\begin{itemize}
		\item a set $\mathcal{E}\subset \cap_{p\geq 2}L^ p(\Omega)$ such that for every $n\in\N^*$, $f\in C_p^ \infty(\R^ n)$ and $F=(F_1,\ldots, F_n)\in\mathcal{E}^ n$ then $f(F)\in\mathcal{E}$ (so, $\mathcal{E}$ is an algebra);
		\item a Hilbert space $\mathcal{H}$, whose inner product and associated norm will be denoted by $\<\cdot,\cdot\>_{\mathcal{H}}$ and $|\cdot|_{\mathcal{H}}$ respectively; we let $L^ p(\Omega;\mathcal{H}) $ stand for the set of the r.v.'s taking values in $\mathcal{H}$ whose norm has moment of order $p$. 
	\end{itemize}
	%
	%
	In this environment, it is assumed that there exist two linear operators 
	$$
	D\,:\,\mathcal{E}\to \cap_{p\geq 2}L^ p(\Omega;\mathcal{H}) 
	\quad\mbox{and}\quad
	L\,:\,\mathcal{E}\to \mathcal{E}
	$$
	such that
	\begin{itemize}
		\item[\textbf{(M1)}] for every $F\in\mathcal{E}$ and $h\in\mathcal{H}$, $D_hF:=\<DF,h\>_{\mathcal{H}}\in \mathcal{E}$;
		\item[\textbf{(M2)}] 
		for every $n\in\N^*$, $f\in C_p^ \infty(\R^ n)$ and $F=(F_1,\ldots, F_n)\in\mathcal{E}^ n$ one has 
		$$
		Df(F)=\sum_{i=1}^n \partial_{x_i}f(F)DF_i\in\mathcal{E};
		$$
		\item [\textbf{(M3)}] for every $F,G\in \mathcal{E}$ one has 
		$\E[LF\, G]=-\E[\<DF,DG\>_\mathcal{H}]=\E[F\, LG]$.
	\end{itemize}
	Thus, we recognize that these are settings and properties typically fulfilled in Malliavin calculus (but not in \textit{any} Malliavin calculus framework - for example this is not in the case of jump processes, where the chain rule \textbf{(M2)} does not hold in general, see e.g. the discussion and the references quoted  in \cite[Section1]{BCP19}). Hence, we call $D$ the Malliavin derivative and $L$ the Ornstein-Uhlenbeck operator. The higher order Malliavin derivatives can be defined straightforwardly: for $k\geq 2$, 
	$$
	D^ k\,:\, \mathcal{E}\to \cap_ {p\geq 2}L^p(\Omega;\mathcal{H}^ {\otimes k})
	$$ 
	is the multilinear operator such that for every $h_1,\ldots,h_k\in \mathcal{H}$ and $F\in\mathcal{E}$,
	$$
	D^ k_{h_1,\ldots,h_k}F:= \<D^ kF, h_1\otimes\cdots\otimes h_k\>_{\mathcal{H}^ {\otimes k}}=
	D_{h_k}D^ {k-1}_{h_1,\ldots,h_{k-1}}F.
	$$ 
	Notice that, when dealing with a concrete Malliavin calculus, one can choose $\mathcal{E}$ either the set of the simple functionals or  the set $\DD^ \infty$ of the r.v.'s  whose Malliavin derivative of any order does exist and has finite moment of any power.
	
	In order to introduce the result in \cite{BCP19} that we are going to use, we first need to define the involved Malliavin-Sobolev norms: for $F=(F_1,\ldots,F_n)\in \mathcal{E}^ n$, we set
	\begin{equation}\label{mall0}
		|F|_{1,q}=\sum_{k=1}^q\sum_{i=1}^ n|D^ kF_i|_{\mathcal{H}^{\otimes k}},\quad |F|_q=|F|+|F|_{1,q},\quad \|F\|_{k,p}=\| |F|_k\|_p,
	\end{equation}
	where $\|\cdot \|_p$ is the standard norm in $L^p(\Omega)$. Then, for $k\in\N^ *$ and $p\geq 2$, we set
	$$
	\DD^ {k,p}=\overline{\mathcal{E}}^ {\|\cdot\|_{k,p}}\quad \mbox{and}\quad \DD^ {k,\infty}=\cap_ {p\geq 2}\DD^{k,p}.
	$$
	We also extend the operator $L$ in the usual way: for $F=(F_1,\ldots,F_n)\in \mathcal{E}^ n$, we set $LF=(LF_1,\ldots,LF_n)$ and $\|F\|_{\mathrm{OU}}=\|F\|_2+\|LF\|_2$. And we define $Dom(L)=\overline{\mathcal{E}}^ {\|\cdot\|_{{\mathrm{OU}}}}$.

	Let us recall the Malliavin calculus described in Subsection \ref{MDGF}, built on a Gaussian random field. It is well known that such a Malliavin calculus framework (in \S \ref{MDGF}) satisfies the abstract hypotheses required in \cite[Section 2.1]{BCP19} and resumed here in \S \ref{sect:proofs} (see e.g. \cite{Nua} or \cite{NP12}): just take $\mathcal{E}=\mathcal{S}$, 
	$\mathcal{H}=H=L^ 2(\SSd,\B(\SSd), \Leb)$ and $L$ as the Ornstein-Uhlenbeck operator defined in \eqref{defL}. In particular, the duality relationship \textbf{(M3)} does hold for $F,G\in\DD^ {2,2}$ and therefore, it holds true on $\mathcal{E}$.
	
	Now,  fix $q\in\N$ and $F=(F_1,\ldots,F_n)\in (\DD^ {q+1,\infty})^n$. If $F=(F_1,\ldots,F_n)\in (Dom(L))^n$ and $LF=(LF_1,\ldots,LF_n)\in (\DD^ {q,\infty})^n$, the following quantities are well posed:
	\begin{equation}\label{mall}
		\begin{array}{ll}
			&\mathcal{C}_q(F)=\big(|F|_{1,q+1}+|LF|_{q}\big)^ q\big(1+|F|_{1,q+1}\big)^ {4nq},\\
			&\mathcal{C}_{q,p}(F)=\|\mathcal{C}_q(F)\|_ p,\\
			&\mathcal{Q}_{q}(F)=\mathcal{C}_{q,2}(F)\|(\det \sigma_F)^ {-1}\|_{2q}^ q,
			%
		\end{array}
	\end{equation}
	in which $p\geq 2$ and $\sigma_F$ is the Malliavin covariance matrix of $F$, that is,
	\begin{equation}\label{MallCov}
		(\sigma_F)_{i,j}=\<DF_i,DF_j\>_{\mathcal{H}},\quad i,j=1,\ldots, n.
	\end{equation}
	Notice that the quantity   $\mathcal{C}_{q,p}(F)$, respectively $\mathcal{Q}_{q}(F)$, in \eqref{mall} is in principle well posed whenever $F_i\in\DD^ {q+1,\bar p}$ for a suitable $\bar p\geq p$, respectively $\bar p\geq 2$.
	
	We are now ready to state the result in \cite{BCP19} on which our asymptotic analysis will be based:
	
	\begin{proposition}\label{Main}
		Let $F$ and $G$ be random vectors in $\R^ n$ such that
		$$
		M_q(F,G):=\mathcal{C}_{q,1}(F)+\mathcal{Q}_{q}(G)<\infty, 
		$$
		for every $q\geq 1$. Let $U>0$ be a real random variable such that $\|U^{-1}\|_ q<\infty$ for every $q\geq 1$.
		Then for every $\varepsilon>0$ there exist $C_\varepsilon>0$ and $q_\varepsilon>1$ such that
		$$
		\dTV(F,G)\leq C_\varepsilon\big(M_{q_\varepsilon}(F,G)+\|U^ {-1}\|_{2/\varepsilon}\big)\big(\dW(F, G)+\dW(\det \sigma_{F}, U)\big)^{1-\varepsilon}.
		$$
		
	\end{proposition}
	
	This is actually \cite[Proposition 3.12]{BCP19}, see in particular (3.30), with the choice $p=p'=1$ (remark that, as it immediately and clearly follows from the proof, there is a misprint in the requests therein: it is erroneously asked that $\mathcal{C}_{q,1}(G), \mathcal{Q}_{q}(F)<\infty$ instead of $\mathcal{C}_{q,1}(F), \mathcal{Q}_{q}(G)<\infty$).

	Our plan is to use Proposition \ref{Main} with $F=\X_{\ell}$ and $G=Z\sim \mathcal N(0,1)$.  Indeed, in our framework, the underlying Hilbert space is $H=L^2(\mathbb S^d, \B(\mathbb S^d), \Leb)$, the random eigenfunction $T_\ell$ admitting the isonormal representation (\ref{spR}).  Thus
	\begin{equation}\label{sigmal}
		\sigma_\ell = \sigma_{\tilde X_\ell} =\int_{\mathbb S^d} |D_y \tilde X_\ell |^2\,dy.
	\end{equation}
	First, Assumption \ref{ASSUMPTION} will guarantee that all the involved Malliavin functionals are well defined (we will give more details about Malliavin calculus for Gaussian random fields in \S \ref{MDGF}). 
	Theorem \ref{ROS} already ensures that $\dW(\X_{\ell}, Z)\to 0$ (giving also an estimation of the speed of convergence). Therefore, we obtain the stronger convergence $\dTV(\X_{\ell}, Z)\to 0$ (together with a useful upper bound on the rate), once we prove that:
	\begin{enumerate}
		\item[\textbf{(H1)}] there exists a \textit{deterministic} $U>0$ such that $\dW(\sigma_{\ell}, U)\to 0$ with some speed, 
		\item[\textbf{(H2)}] for every $q\geq 1$, $\sup_{\ell }M_q(\X_{\ell},Z)<\infty$, where $M_q(\tilde X_\ell, Z)$ is defined in Proposition \ref{Main}.
	\end{enumerate}
	
	\subsection{Proof of Theorem \ref{mainThm}}
	
	Concerning \textbf{(H1)}, we will prove the following key result. 
	
	\begin{theorem}\label{Conv0}
		Let $\sigma_\ell$ be the Malliavin covariance of $\X_\ell$. Under Assumption \ref{ASSUMPTION}, we have 
		$$
		|\E[\sigma_\ell ]- 2 | = O\left (\eta_{\ell;d}\right )
		\quad \mbox{and}\quad
		\Var(\sigma_\ell )= O\left (\ell^{-1}\1_{d=2}+ \ell^{-(d-1)/2}\1_{d\geq 3}\right ),
		$$
		where 
		\begin{equation*}
			\eta_{\ell;d} =  \1_{d=2}\left ( \1_{b_4\ne 0}\frac{\log \ell}{\ell} +\1_{b_4=0}\frac{1}{ \ell}\right )  + \frac{1}{\ell} \1_{d\ge 3}.
		\end{equation*}
	\end{theorem}
	
	As for \textbf{(H2)}, it is enough to prove that, uniformly in $\ell$, all the moments of the main Malliavin operators involved in $M_q(\X_{\ell},Z)$ are bounded.  This is why we will prove the following result.
	\begin{proposition}\label{UnifLim}
		Under Assumption \ref{ASSUMPTION}, for every $k \in \N$ and $n \geq 1$, there exists $\tilde C_{n,k,d}>0$ such that
		$$
		\sup_{\ell \mathrm{\ even}}\E[|D^{(k)} \X_\ell|^n_{\mathcal{H}^{\otimes k}}]\leq \tilde C_{n,k, d}
		\quad\mbox{and}\quad
		\sup_{\ell \mathrm{\ even} }\E[|D^{(k)} L\X_\ell|^n_{\mathcal{H}^{\otimes k}}]\leq \tilde C_{n,k, d}.
		$$ 
	\end{proposition}
	
	We postpone the proofs of Proposition \ref{Conv0} and of Proposition \ref{UnifLim} to Section \ref{sect:Conv0} and Section \ref{sect:UnifLim} respectively. Based on such results, the proof of the CLT in Total Variation distance (Theorem \ref{mainThm}) follows. 
	
	\medskip
	
	\begin{proof}[Proof of Theorem \ref{mainThm} assuming Propositions \ref{Conv0} and \ref{UnifLim}]
		We use Proposition \ref{Main} with $F=\X_{\ell}$, $G=Z$ and $U=2$. We have
		$$
		\dW(\sigma_\ell, 2)\leq \|\sigma_\ell- 2\|_1\leq \|\sigma_\ell- 2\|_2\leq \Var(\sigma_\ell)^{1/2}+|\E[\sigma_{\ell}]-2|\to 0
		$$
		and then, recalling the asymptotic behavior of $\sigma_\ell$ in Proposition \ref{Conv0} we obtain
		\begin{equation}
			\label{dWsigma}
			\dW(\sigma_\ell, 2)=
			\begin{cases}
				O(\ell^{-1/2}) & d=2,3 \\
				O(\ell^{-3/4}) & d=4 \\
				O(\ell^{-1}) & d\geq 5
			\end{cases}
		\end{equation}
		Since $G=Z\sim \mathcal N(0,1)$, $DG=1$, that gives $\sigma_G=1$, $D^ kG=0$ for every $k\geq 2$ and $LG=-G$, so that (see \eqref{mall0}-\eqref{mall}) $\mathcal{Q}_q(G) = \mathcal{Q}_1(G)<\infty$ for every $q\geq 1$. As for $\mathcal{C}_{q,1}(\X_{\ell})$, we 
		have
		\begin{align*}
			\mathcal{C}_{q,1}(\X_\ell)&=\|\mathcal{C}_q(\X_\ell)\|_1=\E[\big(|\X_\ell|_{1,q+1}+|L\X_\ell|_{q}\big)^ q\big(1+|\X_\ell|_{1,q+1}\big)^ {4 q}]\\
			&\leq \E[\big(|\X_\ell|_{1,q+1}+|L\X_\ell|_{q}\big)^{2q}]^{\frac{1}{2}}\E[\big(1+|\X_\ell|_{1,q+1}\big)^ {8 q}]^{\frac{1}{2}}\\
			&\leq (\E[|\X_\ell|_{1,q+1}^{2q}]^{\frac{1}{2}}+\E[|L\X_\ell|^{2q}]^{\frac{1}{2}})(1+\E[|\X_\ell|_{1,q+1}^{8q}]^{\frac{1}{2}})
		\end{align*}
		and Proposition \ref{UnifLim} allows one to check that $\sup_{\ell \mathrm{\ even} }\mathcal{C}_{q,1}(\X_\ell)<\infty$ for every $q\geq 1$. 
		Then, Theorem \ref{mainThm} ensures that, for $\varepsilon>0$, 
		$$
		\dTV(\X_\ell, Z)\leq \textrm{C}_\varepsilon\big(\dW(\X_\ell, Z)+\dW(\det \sigma_{\ell}, 2)\big)^{1-\varepsilon}.
		$$
		Now, combining the above estimate on $\dW(\sigma_\ell, 2)$ and the result on $\dW(\X_\ell, Z)$ in Theorem \ref{ROS}, we conclude the proof. 
		
	\end{proof}

	Comparing \eqref{dW-th} and \eqref{dWsigma}, when applying Proposition \ref{Main} the presence of $\dW(\sigma_\ell, 2)$  does not worsen the quantitative convergence rate for $\dTV(\X_\ell, Z)$: in fact, whenever $d\geq 2$ we obtain that $\dTV(\X_\ell, Z)= O_\varepsilon(\dW(\X_\ell, Z)^ {1-\varepsilon}) $, for any $\varepsilon>0$ close to 0. In other words, the term coming from the Malliavin covariance does not slow down the convergence speed.

	\section{Convergence of Malliavin covariances}\label{sect:Conv0} 
	
	In this section we prove Lemma \ref{Conv0}. Let us anticipate that the proof requires a finer different method for the case $d=2$ than $d\ge 3$. Therefore, as it will be clear from reading the proof, we will have to split in two different approaches.  In Section \ref{SubSecGraph} we prove Lemma \ref{DIAGRAMMA} and Lemma \ref{STIMA}. The first one gives us a different way to write the well known diagram formula \cite{MPbook}, the second one provides estimates for cross moments of Gegenbauer polynomials. These estimates are used for the proof of Lemma \ref{Conv0} when $d\geq 3$ and for the proof of Lemma \ref{UnifLim}. In Section \ref{Case_d2} we explain the different technique we are used to prove Lemma \ref{Conv0} when $d=2$.  
	
	\subsection{Diagram formula and graphs}\label{SubSecGraph}
	
	Now we introduce some notation and results that are useful to prove Lemma \ref{Conv0}. We first study a different way to write down the well known diagram formula \cite[Proposition 4.15]{MPbook} which we recall in Proposition \ref{Diagram_Formula}. To this purpose, we need to introduce a set that we will use several times.
	
	\begin{definition}\label{defA}
		For $q_1,\ldots,q_n\in\N$, we define $\mathcal{A}_{q_1,\ldots,q_n}$ as the set given by the indexes $\{k_{ij}\}^n_{{ i,j=1}}$ such that for every $i,j=1,\ldots,n$, 
		\begin{equation}\label{A}
			\mbox{$k_{i,j}\in\N$, $k_{ii}=0$, $k_{ij}=k_{ji}$ \text{ and } $\sum_{j=1}^n k_{ij}=q_i$;}
		\end{equation}
		
	\end{definition}
	
	We provide a reformulation of the well known diagram formula for Hermite polynomials \cite[Proposition 4.15]{MPbook}, equivalently of a particular case of the standard Feynman diagram representation of moments of Wick products \cite[Theorem 3.12]{Ja97}.
	
	\begin{lemma}\label{DIAGRAMMA}
		Let $n\geq 2$ and let $(Z_1,\ldots, Z_n)$ be a $n$-dimensional centered Gaussian vector. For $q_1,\ldots,q_n\in\N$, consider  $\mathcal{A}_{q_1,\ldots,q_n}$ as in Definition \ref{defA}.  Then, 
		\begin{equation}\label{eqDIAGRAMMA}
			\E[\prod_{r=1}^n H_{q_r}(Z_r)]=\prod_{r=1}^n q_r! \times \sum_{\{k_{i,j}\}^n_{{ i,j=1}} \in \mathcal{A}_{q_1,\ldots,q_n}  } \prod_{\underset{i<j}{i,j=1}}^n\frac{\E[Z_iZ_j]^{k_{ij}}}{k_{ij}!}. 
		\end{equation}
		In particular, taking $Z_1=\cdots=Z_n=Z\sim \mathcal N(0,1)$, one has
		\begin{equation}\label{PassDel}
			\E[\prod_{r=1}^n H_{q_r}(Z)]=\prod_{r=1}^n q_r! \times \sum_{\{k_{i,j}\}^n_{{ i,j=1}} \in \mathcal{A}_{q_1,\ldots,q_n}  } \prod_{\underset{i<j}{i,j=1}}^n\frac{1}{k_{ij}!}. 
		\end{equation}

	\end{lemma}

	\begin{proof}
		We start from the diagram formula (see Proposition \ref{Diagram_Formula}). For a diagram in  $\Gamma_{\overline{F}}(q_1,\ldots,q_n)$, let $R_i$ denote its $i$th row, $i=1,\ldots, n$. Consider the the first row $R_1$. In $R_1$ we have $q_1$ dots; we fix a partition of $q_1$ dots in $n-1$ groups of dots. We order the groups and denote them $R_{1j}$, $j=2,\ldots, n$: $R_{1j}$ is the group of dots in $R_1$ that are linked with dots in the $j$th row. We denote with $k_{1j}$ the number of dots in $R_{1j}$, that coincides with the number of edges connecting row $1$ with row $j$. We fix $k_{12}\in\{0,\ldots, q_1 \}$. There are $\binom{q_1}{k_{12}}$ choices for $k_{12}$ dots in the first row. In general for $j=3,\ldots n$, we fix $k_{1j}=0,\ldots, (q_1-\sum_{h=2}^{j-1} k_{1h}) $ to have that 
		$$
		\sum_{j=2}^n k_{1j}=q_1.
		$$
		For $j=3,\ldots,n$ there are $\binom{q_1-\sum_{h=2}^{j-1} k_{1h}}{k_{1j}}$ choices for $k_{1j}$. Then, the number of choices of $\{ k_{1,j} \}_{j=2}^n$ according to the above condition is
		$$
		\prod_{j=1}^n \binom{q_1-\sum_{r=1}^{j-1} k_{1r}}{k_{1j}}= \frac{q_1!}{\prod_{j=1}^n k_{1j}!}.
		$$ 
		We recall that $k_{ii}=0$ for $i=1,\ldots, n$ because we are considering no-flat diagrams. In practice we have computed the number of partitions of $q_1$ dots in $n-1$ groups. We can do the same for the other rows. And so we have that the number of partition of $q_i$ that is 
		$$
		\prod_{j=1}^n \binom{q_i-\sum_{r=1}^{j-1} k_{ir}}{k_{ij}}= \frac{q_i!}{\prod_{j=1}^n k_{ij}!}.
		$$ 
		Notice that $k_{ij}=k_{ji}$. Now we are able to compute the number of diagrams for fixed $\{k_{ij}\}_{i,j =1}^n$. We recall that $k_{ij}$ represent the number of dots of the $i$th row and of the $j$th row that are linked. There are $k_{ij}!$ way to match the dots.
		Then the number of no-flat diagrams for a fixed $\{k_{ij}\}_{i,j =1}^n$ is 
		$$
		\prod_{i=1}^n \frac{q_i!}{\prod_{j=1}^n k_{ij}!} \prod_{\underset{r<s}{r,s=1}}^n k_{rs}!=\prod_{r=1}^n q_r! \prod_{\underset{i<j}{i,j=1}}^n \frac{1}{k_{ij}}.
		$$
		In order to conclude, it remains to determine the set of all admissible $\{k_{ij}\}_{i,j =1}^n$. Recalling that, for a fixed no-flat  diagram, $k_{ij}$ is the number of edges connecting row $i$ with row $j$, then of course $k_{ij}=k_{ji}$. Moreover, $k_{ii}=0$ because the diagram is no-flat  and $\sum_{j=1}^n k_{ij}=q_i$ for every $i$, as every vertex belongs to a unique edge. This means that $\{k_{ij}\}_{i,j=1}^ n\in \mathcal{A}_{q_1,\ldots,q_n}$ (see Definition \ref{defA}). The statement now follows.

	\end{proof}
	
	Let us see how we apply such result. For fixed $n\geq 2$ and $x_1,\ldots,x_n \in \SSd $ , the random vector $(T_\ell(x_1),\ldots , T_\ell( x_n ))$ is a centered Gaussian random vector whose covariances are given by (see \eqref{covT})
	$$
	\E[T_\ell(x_i)T_\ell(x_j)]=G_{\ell;d}(\<x_i,x_j\>), \quad i,j=1,\ldots,n.
	$$
	When dealing with our proofs, we often need to compute and/or estimate quantities of the type 
	$$
	\int_{(\SSd)^ n}\E[\prod_{r=1}^n H_{q_r}(T_\ell(x_r))]dx.
	$$
	By using \eqref{eqDIAGRAMMA}, we have
	\begin{equation}\label{eqDIAGRAMMA2}
		\int_{(\SSd)^ n}\E[\prod_{r=1}^n H_{q_r}(T_\ell(x_r))]dx=\prod_{r=1}^n q_r! \times \sum_{\{k_{i,j}\}^n_{{ i,j=1}} \in \mathcal{A}_{q_1,\ldots,q_n}  } \int_{(\SSd)^ n}\prod_{\underset{i<j}{i,j=1}}^n\frac{G_{\ell;d}(\<x_i,x_j\>)^{k_{ij}}}{k_{ij}!} dx. 
	\end{equation}
	Therefore, it would be very useful to get a good estimate for the integrals appearing in the r.h.s of \eqref{eqDIAGRAMMA2}, that is, when the integrand function is the product of a number of powers of Gegenbauer polynomials. We provide this estimate in the following Lemma.

	\begin{lemma}\label{STIMA}
		For $n\in \N^*$, let $\kappa=\{k_{ij}\}^{n}_{{ i,j=1}}\in \mathcal{A}_{q_1,\ldots,q_{n}}$ be fixed. Let $\mathfrak{G}_\kappa$ denote the extrapolated graph from $\kappa$ and $N_\kappa$ denote the number of connected components of $\mathfrak{G}_\kappa$. Then,
		\begin{equation}\label{stimaGRAFI}
			\int_{(\SSd)^{n}} \prod_{\underset{i<j}{i,j= 1}}^{n} G_{\ell;d}(\<x_i,x_j\>)^{2k_{ij}} dx \leq \frac{C_{d}(N_\kappa) }{\ell^{(d-1)(n-N_\kappa) }}
		\end{equation}
		where $C_{d}(N_\kappa)=(8 \mu_d \mu_{d-1} c_{2;d})^{n-N_\kappa} \mu_d^{N_\kappa}$, $c_{2;d}$ being given in \eqref{cqd}. As a consequence, for $n=2p$, 
		\begin{equation}\label{connessiSTIMA}
			\int_{(\SSd)^{2p}} \prod_{\underset{i<j}{i,j= 1}}^{2p} G_{\ell;d}(\<x_i,x_j\>)^{2k_{ij}} dx\leq \frac{C_{d;p}}{\ell^{(d-1)p}},
		\end{equation}
		where  $C_{d;p}=(2 (d-1)!\mu_d^2)^{2p} \mu_d^p$. 
	\end{lemma}
	Before presenting proof of Lemma \ref{STIMA}, we need to recall some elementary concepts of graph theory  \cite{Va17}. 
	
\smallskip

		A \textit{graph} is a set of point called \textit{nodes} linked together by lines called \textit{edges}. Formally, a graph is a pair $\mathfrak{G}=(V,E)$ of sets, where $V$ is the set of nodes and $E$ is the set of edges. We can identify $E$ with a subset of $V\times V$. Precisely, if $V=\{x_1,\ldots,x_n\}$ and there exists an edge between $x_i$ and $x_j$, then the pair $(x_i,x_j)\in E$.    
		A \textit{subgraph} of $\mathfrak{G}=(V,E)$ is a graph $\mathfrak{G}'=(V',E')$ where $V'\subset V$ and $E'$ is the set of all the edges of $E$ that link only nodes in $V'$. We say that a node $x$ has degree $m$ if there are $m$ edges that are incident to $x$, the case $m=0$ meaning that the node is isolated.

\smallskip
		
		A \textit{path} between two nodes $x,y$ of $\mathfrak{G}$ is a sequence of edges connecting $x$ with $y$ and joining a sequence of distinct nodes, so, in particular, all edges of the path are distinct. 
		We say that two nodes $x,y$ of a graph $\mathfrak{G}$ are \textit{connected} if $\mathfrak{G}$ contains a path between $x$ and $y$. A graph is said to be \textit{connected} if every pair of nodes in the graph is connected. A \textit{connected component} of a graph $\mathfrak{G}$ is  connected subgraph of the graph that is maximal. We can consider a graph as the union of its connected components.

\smallskip
		
		In our treatment we are interested in a particular class of connected graphs: the trees. A \textit{tree} is a connected graph where each pair of nodes is connected by exactly one path. We first observe that in a tree there exists a non-empty subset of nodes with degree $1$. In fact, equivalently, a tree  is a connected graph in which every subgraph (and in particular the graph itself) contains at least one node with degree $1$. 
		Hence when we delete some of $1$ degree nodes, the subgraph that we obtain is also a tree, that has again a subset of new $1$ degree nodes. If we progressively delete the $1$ degree nodes, we finally obtain a empty graph.

\smallskip
		
		The last and most important property (for our treatment) of connected graphs is the following: a connected graph $\mathfrak{G}$ always contains a \textit{spanning tree}, i.e. a subgraph of $\mathfrak{G}$ that is a tree and contains all nodes of $\mathfrak{G}$.

	\medskip

	\begin{proof}[Proof of Lemma \ref{STIMA}] 
		The proof is based on the concept of \textit{extrapolated graph} from  a given $\kappa=\{k_{ij}\}^n_{{ i,j=1}}\in \mathcal{A}_{q_1,\ldots,q_n}$. Such graph is defined as the pair $\mathfrak{G}_\kappa=(V,E_\kappa)$ in which the set of the nodes is given by $V=\{1,\ldots,n\}$ and the set of the edges is given as follows: the edge $(i,j)$ does exist iff $k_{ij}\neq 0$ (notice that, since $k_{ii}=0$, there are no self-loops).
		
		Let $\kappa=\{k_{ij}\}_{i,j=1}^{n}\in \mathcal{A}_{q_1-1,\ldots, q_{n}-1}$ (see \eqref{eqDIAGRAMMA}). We extrapolate from $\kappa$ the graph $\mathfrak{G}=(V,E)$ with $ V=\{x_1,\ldots, x_{n}\}$ and $(x_i,x_j)\in E$ iff	$k_{ij}\neq0$ and for every $i=1,\ldots,n$, $k_{ii}=0$ (no self loops in $\mathfrak{G}$). 
		
		Let $N$ be the number of the connected components of $\mathfrak{G}$ and we denote with $\mathfrak{G}_h$, $h=1,\ldots, N$ these components. We denote with $m_h$ the number of nodes in $\mathfrak{G}_h$. Then $\sum_{h=1}^N m_h=n$.
		We observe that if $x_i$ is a node of $\mathfrak{G}_{h_1}$, $x_j$ is a node of $\mathfrak{G}_{h_2}$ and $h_1\neq h_2$ then $k_{ij}=0$. This justifies the following equality:
		\begin{equation}\label{integrale}
			\begin{array}{l}
				\int_{(\SSd)^{n}} \prod_{\underset{i<j}{i,j= 1}}^{n} G_{\ell;d}(\<x_i,x_j\>)^{2k_{ij}} dx_1\dots dx_{n}\\
				=\prod_{h=1}^N \int_{(\SSd)^{m_h}} \prod_{x_{i_r}, x_{i_s} \in \mathfrak{G}_h} G_{\ell ; d}(\<x_{i_r}, x_{i_s})^{2k_{i_r i_s}} dx_{i_1}\ldots dx_{i_{m_h}}.
			\end{array}
		\end{equation}
		Now we observe that if $\mathfrak{G}_h$ is a tree, the $1$ degree nodes of $\mathfrak{G}_h$ are the variables $x_{i_r}$ 
		for which there exists one and only one $i_s$  such that $k_{i_r i_s} \neq 0$. Hence there is one and only one polynomial $G_{\ell;d}$ in the variable $x_{i_r}$ in the integral \eqref{integrale}.
		We identify the action of deleting  $1$ degree nodes with that of integrating the polynomial $G_{\ell;d}(\<x_{i_r}, x_{i_s}\>)$ in the variable $x_{i_r}$. 
		Our connected components are not always trees, but we know that there always exists the spanning tree. So for all $\mathfrak{G}_h$, we consider the spanning tree $\tilde{ \mathfrak{G}}_h$, and delete the edges of $\mathfrak{G}_h$ that are not in $\tilde{\mathfrak{G}}_h$. This deleting operation corresponds, when studying the integral in the r.h.s. of \eqref{integrale}, with the estimate $|G_{\ell;d}(\<x_{i_r}, x_{i_s}\>)|\leq 1 $ for each pair $(x_{i_r}, x_{i_s})$ giving the deleted edge.
		Since in a tree with $m_h$ nodes there are $m_h-1$ edges, the resulting estimate consists in integrating $m_h-1$ polynomials.
		
		It follows that
		\begin{eqnarray*}
			\int_{(\SSd)^{n}} \prod_{\underset{i<j}{i,j= 1}}^{n} G_{\ell;d}(\<x_i,x_j\>)^{2k_{ij}} dx
			&=&\prod_{h=1}^N \int_{(\SSd)^{m_h}} \prod_{x_{i_r}, x_{i_s} \in \mathfrak{G}_h} G_{\ell ; d}(\<x_{i_r}, x_{i_s})^{2k_{i_r i_s}} dx_{i_1}\ldots dx_{i_{m_h}}\nonumber\cr
			&\leq&
			\prod_{h=1}^N \int_{(\SSd)^{m_h}} \prod_{x_{i_r}, x_{i_s} \in \tilde{ \mathfrak{G}}_h} G_{\ell ; d}(\<x_{i_r}, x_{i_s})^{2k_{i_r i_s}} dx_{i_1}\ldots dx_{i_{m_h}}\nonumber\cr
			&\leq&
			\prod_{h=1}^N \int_{(\SSd)^{m_h}} \prod_{x_{i_r}, x_{i_s} \in \tilde {\mathfrak{G}}_h} G_{\ell ; d}(\<x_{i_r}, x_{i_s})^{2} dx_{i_1}\ldots dx_{i_{m_h}}\nonumber\cr
			&
			\leq&
			\prod_{h=1}^N  \frac{(8 \mu_d \mu_{d-1} c_{2;d} )^{m_h}\mu_d^{N}}{\ell^{(d-1)(m_h-1)}}
			=\frac{(8 \mu_d \mu_{d-1} c_{2;d})^{n-N} \mu_d^{N}}{\ell^{(d-1)(\sum_{h=1}^Nm_h-N)}}\nonumber\cr
			&=&\frac{(8 \mu_d \mu_{d-1} c_{2;d})^{n-N} \mu_d^{N}}{\ell^{(d-1)(n-N)}}.
		\end{eqnarray*}
		We end by observing that the maximum number $N$ of connected components in a graph that contains $n=2p$ nodes is $p$, when there aren't $0$ degree nodes. Moreover there are exactly $p$ connected components when all subgraph contains exactly $2$ nodes. Then, being $8 \mu_d \mu_{d-1} c_{2;d}>1$, we have
		\begin{equation}
			\int_{(\SSd)^{2p}} \prod_{\underset{i<j}{i,j= 1}}^{2p} G_{\ell;d}(\<x_i,x_j\>)^{2k_{ij}} dx\leq \frac{C_{d;p}}{\ell^{(d-1)p}}
		\end{equation}
		where $C_{d;p}=(2 (d-1)!\mu_d^2)^{2p} \mu_d^p$, thus concluding the proof.
		
	\end{proof}

	Let us remark that, in principle, (\ref{connessiSTIMA}) might be useful to get some estimates on concatenated sums of products of so-called Clebsch-Gordan coefficients $\lbrace C^{L,M}_{\ell_1 ,m_1, \ell_2, m_2}\rbrace$ that we define by 
	\begin{equation*}\label{CG}
		Y_{\ell_1, m_1;d}(x) Y_{\ell_2,m_2;d}(x) = \sum_{L=0}^{\ell_1+\ell_2}\sum_{M=1}^{n_{L;d}} C^{L,M}_{\ell_1, m_1, \ell_2, m_2;d} Y_{L,M;d}(x),\quad x\in 	\mathbb S^d.
	\end{equation*}
	This is because there exists a precise link between such quantities and moments of Gegenbauer polynomials which can be established via the addition formula \eqref{covT}. 
	However, it is not clear whether it is actually possible to obtain optimal or novel estimates, even if in dimension $ d> 2 $ a little is known about these coefficients. (See \cite[Section 3.5]{MPbook} for a complete discussion in the case of the $2$-sphere).
	
	\subsection{Case $d=2$}\label{Case_d2}
	We devote this subsection to describe the proof technique of the convergence of Malliavin covariance in dimension $d=2$. Let us first introduce notation.
	For $q_1,\ldots, q_4 \in \N$ and $\kappa=\{k_{ij}\}_{i,j=1}^ 4\in \mathcal{A}_{q_1-1, \ldots, q_4-1}$ (see Definition \ref{defA}), we define
	\begin{align}
		&\mathfrak{I}_{q_1,\ldots, q_4, \kappa}(\ell)=\int_{(\SS2)^4} P_{\ell}(\<x_1,x_2\>)^{k_{12}+1}\prod_{i<j, i<3}P_{\ell}(\<x_i,x_j\>)^{k_{ij}} \, P_{\ell}(\<x_3,x_4\>)^{k_{34}+1}dx, \label{I}\\
		&R_{q_1,\ldots,q_4,\kappa}=\#\{k_{ij}: k_{ij}\neq 0, i<j \}. \label{R}
	\end{align}
	We recall that
	$\mathcal{N}_{q_1-1, q_2-1, q_3-1, q_4-1}$ is defined as  the set of $\kappa=\{k_{ij}\}_{i,j=1}^ 4\in \mathcal{A}_{q_1-1, \ldots, q_4-1}$ such that
	\begin{equation}\label{setN}
		\mbox{$k_{12}=q_1-1=q_2-1,$ $k_{34}=q_3-1=q_4-1$, $ k_{13}=k_{14}=k_{23}=k_{24}=0$ }
	\end{equation}
	and we define the set $C_{q_1-1,\ldots, q_4-1}$ as 
	\begin{equation}\label{setC}
		\mathcal{C}_{q_1-1,q_2-1,q_3-1, q_4-1}=\mathcal{A}_{q_1-1,\ldots,q_4-1}\setminus \mathcal{N}_{q_1-1, q_2-1, q_3-1, q_4-1}.
	\end{equation}
	The main result is the following proposition.
	\begin{proposition} \label{prop-I}
		There exists $C>0$ such that for every $q_1,\ldots, q_4 \geq 2$ and $\kappa=\{k_{ij}\}_{i,j=1}^ 4\in \mathcal{C}_{q_1-1, \ldots, q_4-1}$ one has
		\begin{equation}\label{estd2}
			\left |\int_{(\SS2)^4} P_{\ell}(\<x_1,x_2\>)^{k_{12}+1}\prod_{i<j, i<3}P_{\ell}(\<x_i,x_j\>)^{k_{ij}} \, P_{\ell}(\<x_3,x_4\>)^{k_{34}+1}dx \right |\leq \frac{C}{\ell^3}.
		\end{equation}
	\end{proposition}
	
	The proof of Proposition \ref{prop-I} is quite long and heavy, so we split our analysis in five different lemmas built according to the number of the $k_{ij}$'s in $\kappa$ that are not null, taking advantage of some symmetry properties. As an immediate consequence, the proof Proposition \ref{prop-I} will follow from the auxiliary lemmas \ref{lemmaI-1}--\ref{lemmaI-5} we are going to introduce.

	%
	%
	We will split the cases in accordance with $R_{q_1,\ldots,q_4,\kappa}$. Clearly, $R_{q_1,\ldots,q_4,\kappa}\neq 0$ and $R_{q_1,\ldots,q_4,\kappa}\neq 1$, the latter because the 4 equations $\sum_{l=1}^4 k_{il}=q_i-1$, $i=1,\ldots,4$, must be satisfied. So $2\leq \R_{q_1,\ldots,q_4,\kappa}\leq 6$. Notice that if $\kappa\in \mathcal{N}_{q_1-1,\ldots,q_4-1}$ then $R_{q_1,\ldots,q_4,\kappa}=2$.
	
	In the following lemmas, we take $q_1,\ldots,q_4\geq 2$, $\kappa\in \mathcal{A}_{q_1-1,\ldots,q_4-1}$ and we  consider $\mathfrak{I}_{q_1,\ldots, q_4, \kappa}(\ell)$ and $R_{q_1,\ldots,q_4,\kappa}$ defined in \eqref{I} and \eqref{R} respectively.
	
	\begin{lemma}\label{lemmaI-1}
		If $R_{q_1,\ldots,q_4,\kappa}=2$ and $\kappa\notin \mathcal{N}_{q_1-1,\ldots,q_4-1}$,
		one has
		\begin{equation}\label{asymp-R=2}
			|\mathfrak{I}_{q_1,\ldots, q_4, \kappa}(\ell)|\leq \frac{C}{\ell^ {3}},
		\end{equation}
		$C>0$ denoting a constant independent of $q_1,\ldots,q_4$ and of $\kappa$.
		
	\end{lemma}

	\begin{proof}
		If $\kappa\notin \mathcal{N}_{q_1-1,\ldots,q_4-1}$ then \eqref{setN} is false, so either $k_{23},k_{14}\neq 0$ or $k_{13},k_{24}\neq 0$. The estimates in these cases are similar, so we only deal with the former one.

		According to  $\kappa=\{k_{ij}\}_{i,j=1}^ 4\in \mathcal{A}_{q_1-1, \ldots, q_4-1}$ in Definition \ref{defA}, if $k_{23}, k_{14} \neq 0$ then $k_{14}=q_1-1=q_4-1$ and $k_{23}=q_2-1=q_3-1$. This implies $q_1=q_4$ and $q_2=q_3$ and we prove that
		\begin{equation*}
			|\mathfrak{I}_{q_1,\ldots, q_4, \kappa}(\ell)|\leq  \frac{C}{\ell^{5}}\times \ell^ {\#\{i\in\{1,2\}\,:\,q_i=2\}}\times (\log \ell) ^ {\#\{i\in\{1,2\}\,:\,q_i=4\}}.
		\end{equation*}
		Hereafter, throughout this section, $C$ denotes a positive constant that may vary from a line to another but is independent of $q_1,\ldots,q_4$ and $\kappa$.
		In this case, $\mathfrak{I}_{q_1,\ldots, q_4, \kappa}(\ell)$ becomes
		\begin{align*}
			\mathfrak{I}_{q_1,\ldots, q_4, \kappa}(\ell)= &\int_{(\SS2)^4} P_\ell(\<x_1,x_2\>) P_\ell(\<x_1,x_4\>)^{q_1-1} P_\ell(\<x_2,x_3\>)^{q_2-1}P_\ell(\<x_3,x_4\>)dx.\\
		\end{align*}
		We start to consider the case $q_1=2$. Using twice the reproducing formula \eqref{G1}, we have
		\begin{align*}
			|\mathfrak{I}_{q_1,\ldots, q_4, \kappa}(\ell)|&=\Big|\int_{(\SS2)^4} P_\ell(\<x_1,x_2\>) P_\ell(\<x_1,x_4\>) P_\ell(\<x_2,x_3\>)^{q_2-1} P_\ell(\<x_3,x_4\>)dx\Big|\\
			&=\Big(\frac{4\pi}{2\ell+1}\Big)^2 \Big|\int_{(\SS2)^4}  P_\ell(\<x_2,x_3\>)^{q_2} dx_2 dx_3\Big|\\
			&\leq \frac{C}{(2\ell+1)^2}\Big( \1_{q_2=2}\frac{1}{2\ell+1}+\1_{q_2=4}\frac{\log \ell}{\ell^2} +\1_{q_2=3, q_2\geq 5} \frac{1}{\ell^2}     \Big)
			\leq \frac C{\ell^3}.
		\end{align*}
		If $q_3=2$ one proceeds similarly. Finally consider the case $q_1, q_2 \geq 3$. From Lemma \ref{B} and Proposition \ref{GegProp}, 		we have
		\begin{align*}
			|\mathfrak{I}_{q_1,\ldots, q_4, \kappa}(\ell)|&=\Big|\int_{(\SS2)^4} P_\ell(\<x_1,x_2\>) P_\ell(\<x_1,x_4\>)^{q_1-1} P_\ell(\<x_2,x_3\>)^{q_2-1} P_\ell(\<x_3,x_4\>)dx\Big|\\
			&= C \int_{(\SS2)^4}P_\ell(\<x_1,x_4\>)^{q_1} dx \int_{(\SS2)^4}  P_\ell(\<x_2,x_3\>)^{q_2}  dx \int_{\SS2} P_\ell(\<x_3, x_4\>)^2 dx\\ 
			&\leq \frac{C}{(2\ell+1)^5}(\log \ell)^{\#\{i=1,2:\, q_i=4\}}.
		\end{align*}		
		The statement now holds.
	\end{proof}

	\begin{lemma}\label{lemmaI-2}
		If $R_{q_1,\ldots,q_4,\kappa}=3$, one has
		\begin{equation}\label{asymp-R=3}
			|\mathfrak{I}_{q_1,\ldots, q_4, \kappa}(\ell)|\leq \frac{C}{\ell^ {3}},
		\end{equation}
		$C>0$ denoting a constant independent of $q_1,\ldots,q_4$ and of $\kappa$.
	\end{lemma}

	\begin{proof}
		$R_{q_1,\ldots,q_4,\kappa}=3$ gives 20 different cases, but thanks to some symmetry properties, we can group them in 5 main cases.
		
		First, let us observe that there are instances that are not in accordance with the condition $\kappa= \{k_{ij}\}_{i,j=1}^ 4\in \mathcal{A}_{q_1-1, \ldots, q_4-1}$:
		\begin{multicols}{2}
			\begin{enumerate}[label=(\alph*)]
				\item $k_{23},k_{24},k_{34}\neq 0$;
				\item $ k_{13},k_{14},k_{34}\neq 0$; 
				\item $k_{12}, k_{14}, k_{24}\neq 0$;  
				\item $k_{12},k_{13},k_{23}\neq 0$.
			\end{enumerate}
		\end{multicols}
		So, we proceed with the cases that can be really verified.

		\smallskip
		
		\noindent
		\textbf{Case 1:} we assume here that there exists $i=1,2,3,4$ where $k_{ij}\neq 0$ for every $j\neq i$, 
		that is,
		\begin{multicols}{2}
			\begin{enumerate}[label=(\alph*)]
				\item $k_{12}, k_{13}, k_{14}\neq 0$;
				\item $k_{12}, k_{23}, k_{24}\neq 0$; 
				\item $k_{13}, k_{23}, k_{34}\neq 0$;
				\item $k_{14}, k_{24}, k_{34}\neq 0$.
			\end{enumerate}		 
		\end{multicols}
		
		Figure \ref{R=3-case1} shows all non flat diagrams with exactly 3 non null edges satisfying the above conditions.
		
		\begin{figure}[h]
				{\tiny 	\begin{center}
						\begin{tikzpicture}
							\node[shape=circle,draw](A) at (1,3) {$1$};
							\node[shape=circle,draw](B) at (3,3) {$2$};
							\node[shape=circle,draw](C) at (3,1) {$3$};
							\node[shape=circle,draw](D) at (1,1) {$4$};
							\draw (A) -- (B);
							\draw (A) -- (D);
							\draw (A) -- (C);
							\draw (2,0) node{(a)};
							\node[shape=circle,draw](A) at (4.5,3) {$1$};
							\node[shape=circle,draw](B) at (6.5,3) {$2$};
							\node[shape=circle,draw](C) at (6.5,1) {$3$};
							\node[shape=circle,draw](D) at (4.5,1) {$4$};
							\draw (B) -- (A);
							\draw (B) -- (C);
							\draw (B) -- (D);
							\draw (5.5,0) node{(b)};
							\node[shape=circle,draw](A) at (8,3) {$1$};
							\node[shape=circle,draw](B) at (10,3) {$2$};
							\node[shape=circle,draw](C) at (10,1) {$3$};
							\node[shape=circle,draw](D) at (8,1) {$4$};
							\draw (C) -- (A);
							\draw (C) -- (B);
							\draw (C) -- (D);
							\draw (9,0) node{(c)};
							\node[shape=circle,draw](A) at (11.5,3) {$1$};
							\node[shape=circle,draw](B) at (13.5,3) {$2$};
							\node[shape=circle,draw](C) at (13.5,1) {$3$};
							\node[shape=circle,draw](D) at (11.5,1) {$4$};
							\draw (D) -- (A);
							\draw (D) -- (B);
							\draw (D) -- (C);
							\draw (12.5,0) node{(d)};
						\end{tikzpicture}
					\end{center}
				}
				\caption{\small $R_{q_1,\ldots,q_4,\kappa}=3$, Case 1}\label{R=3-case1}
			\end{figure}
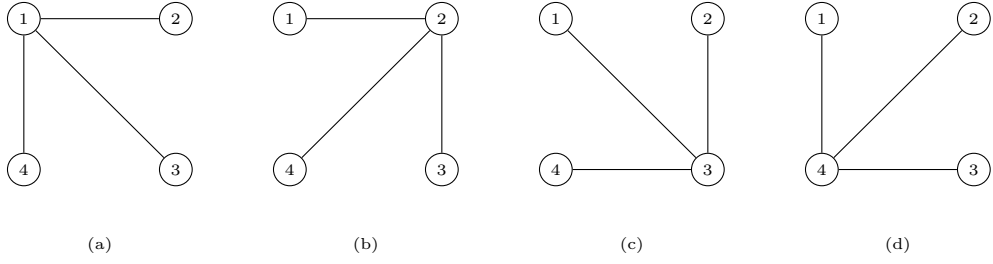

			Instances (a)--(d) can all be studied in a similar way thanks to symmetry properties, so we consider the latter, that is,  $k_{14}, k_{24}, k_{34}\neq 0$, and we prove that
			\begin{equation}\label{asymp2}
				|\mathfrak{I}_{q_1,\ldots, q_4, \kappa}(\ell)|\leq \frac{C}{\ell^{6}}\times\frac{\ell^{\#\{i=1,2,3\,:\,q_i =2\}}}{\ell^{\1_{q_3=2, q_1\neq 2,q_2\neq 2}} }\times (\log \ell)^{\#\{1=1,2\,:\, q_i=4\} +\1_{q_3=3}}          .
			\end{equation}

			The constraints on the $k_{ij}$'s imply that $k_{14}=q_1-1$, $k_{24}= q_2-1$ and $k_{34}=q_3-1$, hence
			$$
			\mathfrak{I}_{q_1,\ldots, q_4, \kappa}(\ell)=\int_{(\SS2)^4} P_\ell(\<x_1, x_2\>)P_\ell(\<x_1, x_4\>)^{q_1-1}P_\ell(\<x_2, x_4\>)^{q_2-1}P_\ell(\<x_3, x_4\>)^{q_3}dx.
			$$
			Applying Lemma \ref{B} and Proposition \ref{GegProp}, for $q_1,q_2\geq 3$ and $q_3\geq 2$ we obtain, 
			\begin{align*}
				|\mathfrak{I}_{q_1,\ldots, q_4, \kappa}(\ell)|&\leq C\, \big|\int_{(\SS2)^2} P_\ell(\<x,y\>)^{q_1} dxdy\, \int_{(\SS2)^2}P_\ell(\<x,y\>)^{q_2} dxdy\,\int_{(\SS2)^2} P_\ell(\<x,y\>)^{q_3+1} dxdy\big| \\
				& \leq  \frac{C}{\ell^{6}}(\log \ell)^{\#\{i=1,2 : q_i=4\}+\1_{q_3=3}}
			\end{align*} 
			and \eqref{asymp2} holds.		
			If $q_1=q_2=q_3=2$ we have
			\begin{align*}
				|\mathfrak{I}_{q_1,\ldots, q_4, \kappa}(\ell)|&=\big|\int_{(\SS2)^3}P_\ell(\<x_2, x_4\>)P_\ell(\<x_3, x_4\>)^2dx_2 dx_3 dx_4 \int_{\SS2}P_\ell(\<x_1, x_2\>)P_\ell(\<x_1, x_4\>)dx_1 \big|\\
				&=\frac{C}{(2\ell+1)^3}
			\end{align*}
			and \eqref{asymp2} follows.
			If $q_1=q_3=2$ and $q_2\geq 3$ (or $q_2=q_3=2$ and $q_1\geq 3$)	we have	
			\begin{align*}
				|\mathfrak{I}_{q_1,\ldots, q_4, \kappa}(\ell)|&=\big|\int_{(\SS2)^3}P_\ell(\<x_2, x_4\>)^{q_2-1}P_\ell(\<x_3, x_4\>)^2 dx_2 dx_3 dx_4 \int_{\SS2}P_\ell(\<x_1, x_2\>)P_\ell(\<x_1, x_4\>)dx_1 \big|\\
				&=\frac{C}{2\ell+1}\big|\int_{(\SS2)^3}P_\ell(\<x_2, x_4\>)^{q_2}P_\ell(\<x_3, x_4\>)^2 dx_2 dx_3 dx_4  \big|
				\leq \frac{C}{\ell^{4}}(\log \ell)^{\1_{q_2=4}} 
			\end{align*}		
			and \eqref{asymp2} again holds.	
			If $q_1=2$ and $q_2,q_3\geq 3$ (or $q_2=2$ and $q_1,q_3\geq 3$)		
			\begin{align*}
				|\mathfrak{I}_{q_1,\ldots, q_4, \kappa}(\ell)|&=\big|\int_{(\SS2)^3}P_\ell(\<x_2, x_4\>)^{q_2-1}P_\ell(\<x_3, x_4\>)^{q_3}dx_2 dx_3 dx_4 \int_{\SS2}P_\ell(\<x_1, x_2\>)P_\ell(\<x_1, x_4\>)dx_1 \big|\\
				&=\frac{C}{2\ell+1}\big|\int_{(\SS2)^3}P_\ell(\<x_2, x_4\>)^{q_2}P_\ell(\<x_3, x_4\>)^{q_3} dx_2 dx_3 dx_4  \big|\\
				&\leq \frac{C}{\ell^{5}} (\log \ell)^{\#\{i=2,3\,:\, q_i=4\}},
			\end{align*}	
			which gives \eqref{asymp2}.		
			Finally, \eqref{asymp2} holds also if $q_3=2$ and $q_1,q_2\geq 3$: applying Lemma \ref{B} we have
			\begin{align*}
				|\mathfrak{I}_{q_1,\ldots, q_4, \kappa}(\ell)|&=\big|\int_{(\SSd)^4} P_\ell(\<x_1, x_2\>)P_\ell(\<x_1, x_4\>)^{q_1-1}P_\ell(\<x_2, x_4\>)^{q_2-1}P_\ell(\<x_3, x_4\>)^2dx\big|\\
				&\leq \frac{C}{\ell^{6}}(\log \ell)^{\#\{i=1,2\,:\, q_i=4\}}  .
			\end{align*}	
			
			The remaining cases (a)--(c) produce an estimate similar to \eqref{asymp2}. As a consequence, \eqref{asymp-R=2} holds in \textbf{Case 1}.

			\smallskip
			
			\noindent
			\textbf{Case 2:} we deal now with the case when there exists just one connection between the pairs $(1,2)$ and $(3,4)$, that is,
			\begin{multicols}{2}
				\begin{enumerate}[label=(\alph*)]
					\item $k_{12}, k_{14}, k_{23}\neq 0$;
					\item $k_{12}, k_{13}, k_{24}\neq 0$;
					\item $k_{34}, k_{13}, k_{24}\neq 0$;
					\item $k_{34}, k_{14}, k_{23}\neq 0$.
				\end{enumerate}  
			\end{multicols}
			This is graphically shown in Figure 	\ref{R=3-case2}.
			\begin{figure}[h]
				{\tiny 	\begin{center}
						\begin{tikzpicture}
							\node[shape=circle,draw](A) at (1,3) {$1$};
							\node[shape=circle,draw](B) at (3,3) {$2$};
							\node[shape=circle,draw](C) at (3,1) {$3$};
							\node[shape=circle,draw](D) at (1,1) {$4$};
							\draw (A) -- (B);
							\draw (A) -- (D);
							\draw (B) -- (C);
							\draw (2,0) node{(a)};
							\node[shape=circle,draw](A) at (4.5,3) {$1$};
							\node[shape=circle,draw](B) at (6.5,3) {$2$};
							\node[shape=circle,draw](C) at (6.5,1) {$3$};
							\node[shape=circle,draw](D) at (4.5,1) {$4$};
							\draw (A) -- (B);
							\draw (A) -- (C);
							\draw (B) -- (D);
							\draw (5.5,0) node{(b)};
							\node[shape=circle,draw](A) at (8,3) {$1$};
							\node[shape=circle,draw](B) at (10,3) {$2$};
							\node[shape=circle,draw](C) at (10,1) {$3$};
							\node[shape=circle,draw](D) at (8,1) {$4$};
							\draw (C) -- (A);
							\draw (B) -- (D);
							\draw (C) -- (D);
							\draw (9,0) node{(c)};
							\node[shape=circle,draw](A) at (11.5,3) {$1$};
							\node[shape=circle,draw](B) at (13.5,3) {$2$};
							\node[shape=circle,draw](C) at (13.5,1) {$3$};
							\node[shape=circle,draw](D) at (11.5,1) {$4$};
							\draw (D) -- (A);
							\draw (C) -- (B);
							\draw (D) -- (C);
							\draw (12.5,0) node{(d)};
						\end{tikzpicture}
					\end{center}
				}
				\caption{\small $R_{q_,\ldots,q_4,\kappa}=3$, Case 2}\label{R=3-case2}
			\end{figure}
			We study the instance (d), that is $k_{34}, k_{14}, k_{23}\neq 0$.  As $\kappa= \{k_{ij}\}_{i,j=1}^ 4\in \mathcal{A}_{q_1-1, \ldots, q_4-1}$, it follows that $k_{14}=q_1-1$ and $k_{23}=q_2-1$ and we e notice that $k_{34}$ can be equal to $1$.  We have
			\begin{align*}
				&\mathfrak{I}_{q_1,\ldots, q_4, \kappa}(\ell)=\int_{(\SS2)^4} P_\ell(\<x_1, x_2\>)  P_\ell(\<x_1, x_4\>)^{q_1-1}P_\ell(\<x_2, x_3\>)^{q_2-1}P_\ell(\<x_3, x_4\>)^{k_{34}+1}dx
			\end{align*}
			and we prove that
			\begin{equation}\label{asymp3}
				|\mathfrak{I}_{q_1,\ldots, q_4, \kappa}(\ell)|\leq \frac{C_2}{\ell^{6}} \times \ell^{\#\{i=1,2\,:\, q_i=2\}} \,(\log \ell)^{\#\{i=1,2\,:\, q_i=4\}+\1_{k_{34}=2}}.
			\end{equation}
			For $q_1, q_2\geq 3$, applying Lemma \ref{B}, we have immediately
			\begin{align*}
				|\mathfrak{I}_{q_1,\ldots, q_4, \kappa}(\ell)|\leq \frac{C}{\ell^{6}} \times \ell^{\#\{i=1,2\,:\, q_i=2\}} \, (\log \ell)^{\#\{i=1,2\,:\, q_i=4\}+\1_{k_{34}=2}} .
			\end{align*}		
			
			If $q_1=q_2=2$ we have
			\begin{align*}
				|\mathfrak{I}_{q_1,\ldots, q_4, \kappa}(\ell)|&=\big|\int_{(\SS2)^3} P_\ell(\<x_2, x_3\>) P_\ell(\<x_3, x_4\>)^{k_{34}+1}dx_2 dx_3 dx_4 \int_{\SS2}P_\ell(\<x_1, x_2\>)  P_\ell(\<x_1, x_4\>)dx_1\big|\\
				&=\Big(\frac{C}{2\ell+1}\Big)^2\big|\int_{(\SSd)^2} P_\ell(\<x_3, x_4\>)^{k_{34}+2} dx_3 dx_4 \big|
				\leq \frac{C}{\ell^{4}}\, (\log \ell )^{\1_{k_{34}=2}},
			\end{align*}
			and if $q_1=2$ and $q_2\geq 3$ (or $q_2=2$ and $q_1\geq 3$) we can apply Lemma \ref{I} and we have
			\begin{align*}
				|\mathfrak{I}_{q_1,\ldots, q_4, \kappa}(\ell)|&=\big|\int_{(\SSd)^3} P_\ell(\<x_2, x_3\>)^{q_2-1} P_\ell(\<x_3, x_4\>)^{k_{34}+1}dx_2 dx_3 dx_4 \int_{\SSd}P_\ell(\<x_1, x_2\>)  P_\ell(\<x_1, x_4\>)dx_1\big|\\
				&=\frac{C}{2\ell+1}\big|\int_{(\SS2)^3} P_\ell(\<x_2, x_3\>)^{q_2-1}P_\ell(\<x_2, x_4\>) P_\ell(\<x_3, x_4\>)^{k_{34}+1}dx_2 dx_3 dx_4\big|\\
				&\leq \frac{C}{\ell^{5}}\,(\log \ell)^{\1_{q_2=4}+\1_{k_{34}=2}}.
			\end{align*}
			Then \eqref{asymp3} holds.
			
			By symmetry, the cases (a)--(c) give a similar estimate. By resuming, \eqref{asymp-R=2} is true also in \textbf{Case 2}.

			\smallskip
			
			\noindent
			\textbf{Case 3:} we assume now that either $k_{12}=0$ and $k_{34}=0$, as displayed in Figure \ref{R=3-case3}. 
			\begin{figure}[h]
				{\tiny 	\begin{center}
						\begin{tikzpicture}
							\node[shape=circle,draw](A) at (1,3) {$1$};
							\node[shape=circle,draw](B) at (3,3) {$2$};
							\node[shape=circle,draw](C) at (3,1) {$3$};
							\node[shape=circle,draw](D) at (1,1) {$4$};
							\draw (A) -- (D);
							\draw (B) -- (C);
							\draw (B) -- (D);
							\draw (2,0) node{(a)};
							\node[shape=circle,draw](A) at (4.5,3) {$1$};
							\node[shape=circle,draw](B) at (6.5,3) {$2$};
							\node[shape=circle,draw](C) at (6.5,1) {$3$};
							\node[shape=circle,draw](D) at (4.5,1) {$4$};
							\draw (A) -- (D);
							\draw (B) -- (C);
							\draw (A) -- (C);
							\draw (5.5,0) node{(b)};
							\node[shape=circle,draw](A) at (8,3) {$1$};
							\node[shape=circle,draw](B) at (10,3) {$2$};
							\node[shape=circle,draw](C) at (10,1) {$3$};
							\node[shape=circle,draw](D) at (8,1) {$4$};
							\draw (A) -- (C);
							\draw (B) -- (D);
							\draw (A) -- (D);
							\draw (9,0) node{(c)};
							\node[shape=circle,draw](A) at (11.5,3) {$1$};
							\node[shape=circle,draw](B) at (13.5,3) {$2$};
							\node[shape=circle,draw](C) at (13.5,1) {$3$};
							\node[shape=circle,draw](D) at (11.5,1) {$4$};
							\draw (A) -- (C);
							\draw (B) -- (D);
							\draw (B) -- (C);
							\draw (12.5,0) node{(d)};
						\end{tikzpicture}
					\end{center}
				}
				\caption{\small $R_{q_1,\ldots,q_4,\kappa}=3$, Case 3}\label{R=3-case3}
			\end{figure}
			
			\noindent
			This means that
			\begin{multicols}{2}
				\begin{enumerate}[label=(\alph*)]
					\item $k_{14}, k_{23}, k_{24}\neq 0$;
					\item $k_{13}, k_{14}, k_{23}\neq 0$;
					\item $k_{13}, k_{14}, k_{24}\neq 0$;
					\item $k_{13}, k_{23}, k_{24}\neq 0$.
				\end{enumerate}
			\end{multicols}	
			We study the instance (a), that is $k_{14}, k_{23}, k_{24}\neq 0$. Since $\kappa\in \mathcal{A}_{q_1-1,\ldots,q_4-1}$, one gets $k_{14}=q_1-1$ and $k_{23}=q_3-1$. So
			\begin{align*}
				&\mathfrak{I}_{q_1,\ldots, q_4, \kappa}(\ell)=\int_{(\SS2)^4} P_\ell(\<x_1, x_2\>)P_\ell(\<x_1, x_4\>)^{q_1-1}P_\ell(\<x_2, x_3\>)^{q_3-1}P_\ell(\<x_2, x_4\>)^{k_{24}}P_\ell(\<x_3, x_4\>) dx.
			\end{align*}
			We prove that
			\begin{equation}\label{asymp4}
				|\mathfrak{I}_{q_1,\ldots, q_4, \kappa}(\ell)|\leq \frac{C}{\ell^{6}} \times \ell^{\#\{i=1,3\,:\, q_i=2\}} \, (\log \ell)^{\#\{i=1,3\,:\, q_i=4\}+\1_{k_{24}=2}}.
			\end{equation}
			For $q_1, q_3 \geq 3$, applying Lemma \ref{C}, we immediately have
			\begin{align*}
				|\mathfrak{I}_{q_1,\ldots, q_4, \kappa}(\ell)|\leq \frac{C}{\ell^{6}}\,(\log \ell )^{\#\{i=1,3\,:\, q_i=4 \}+\1_{k_{24}=2}}.
			\end{align*}
			If $q_1=q_3 =2$ we have 
			\begin{align*}
				|\mathfrak{I}_{q_1,\ldots, q_4, \kappa}(\ell)|&=\Big|\int_{(\SS2)^2} \int_{\SSd} P_\ell(\<x_1, x_2\>)P_\ell(\<x_1, x_4\>) dx_1 \int_{\SS2} P_\ell(\<x_2, x_3\>) P_\ell(\<x_3, x_4\>) dx_3 \times\\
				&\times P_\ell(\<x_2, x_4\>)^{k_{24}}  dx_2 dx_4\Big|\leq \frac{C}{\ell^{4}} \,(\log \ell )^{\1_{k_{24}=2}}
			\end{align*}
			and \eqref{asymp4} holds.
			If $q_1=2$ and $q_3\geq 3$ (or $q_3=2$ and $q_1\geq 3$), applying Lemma \ref{I}, we have
			\begin{align*}
				|\mathfrak{I}_{q_1,\ldots, q_4, \kappa}(\ell)|=&\Big|\int_{(\SS2)^3} \int_{\SSd} P_\ell(\<x_1, x_2\>)P_\ell(\<x_1, x_4\>) dx_1 P_\ell(\<x_2, x_3\>)^{q_3-1} P_\ell(\<x_3, x_4\>)\\
				&\times P_\ell(\<x_2, x_4\>)^{k_{24}}  dx_2 dx_3  dx_4\Big|\\
				\leq &\frac{4\pi}{2\ell+1}\Big|\int_{(\SS2)^3}  P_\ell(\<x_2, x_3\>)^{q_3-1}  P_\ell(\<x_2, x_4\>)^{k_{24}+1} P_\ell(\<x_3, x_4\>) dx_2 dx_3 dx_4\Big|\\
				\leq &\frac{C}{\ell^{5}} \,(\log \ell )^{\1_{q_3=4}+\1_{k_{24}=2}}
			\end{align*}
			and \eqref{asymp4} holds. Again, similar estimates can be produced if (b)--(d) are true. As a consequence, the estimate 
			\eqref{asymp-R=2} holds in \textbf{Case 3} as well.

			\smallskip
			
			\noindent
			\textbf{Case 4:}  we finally assume that $k_{12},k_{34}\neq 0$, as in Figure \ref{R=3-case4}, that is,
			\begin{multicols}{2}
				\begin{enumerate}[label=(\alph*)]
					\item $k_{12}, k_{24}, k_{34}\neq 0$;
					\item $k_{12}, k_{13}, k_{34}\neq 0$;
					\item $k_{12}, k_{23}, k_{34}\neq 0$;
					\item $k_{12}, k_{14}, k_{34}\neq 0$.
				\end{enumerate}
			\end{multicols}
			\begin{figure}[h]
				{\tiny 	\begin{center}
						\begin{tikzpicture}
							\node[shape=circle,draw](A) at (1,3) {$1$};
							\node[shape=circle,draw](B) at (3,3) {$2$};
							\node[shape=circle,draw](C) at (3,1) {$3$};
							\node[shape=circle,draw](D) at (1,1) {$4$};
							\draw (A) -- (B);
							\draw (C) -- (D);
							\draw (B) -- (D);
							\draw (2,0) node{(a)};
							\node[shape=circle,draw](A) at (4.5,3) {$1$};
							\node[shape=circle,draw](B) at (6.5,3) {$2$};
							\node[shape=circle,draw](C) at (6.5,1) {$3$};
							\node[shape=circle,draw](D) at (4.5,1) {$4$};
							\draw (A) -- (B);
							\draw (C) -- (D);
							\draw (A) -- (C);
							\draw (5.5,0) node{(b)};
							\node[shape=circle,draw](A) at (8,3) {$1$};
							\node[shape=circle,draw](B) at (10,3) {$2$};
							\node[shape=circle,draw](C) at (10,1) {$3$};
							\node[shape=circle,draw](D) at (8,1) {$4$};
							\draw (A) -- (B);
							\draw (C) -- (D);
							\draw (B) -- (C);
							\draw (9,0) node{(c)};
							\node[shape=circle,draw](A) at (11.5,3) {$1$};
							\node[shape=circle,draw](B) at (13.5,3) {$2$};
							\node[shape=circle,draw](C) at (13.5,1) {$3$};
							\node[shape=circle,draw](D) at (11.5,1) {$4$};
							\draw (A) -- (B);
							\draw (C) -- (D);
							\draw (A) -- (D);
							\draw (12.5,0) node{(d)};
						\end{tikzpicture}
					\end{center}
				}
				\caption{\small  $R=3$, Case 4}\label{R=3-case4}
			\end{figure}
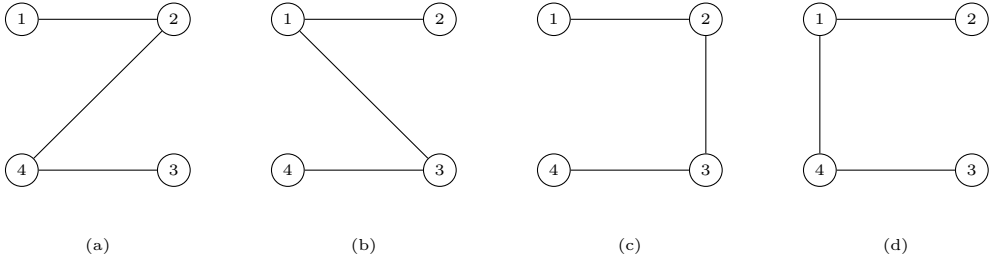

			We study  $k_{12}, k_{24}, k_{34}\neq 0$. As $\kappa=\{k_{ij}\}_{i,j=1}^ 4\in\mathcal{A}_{q_1-1,\ldots,q_4-1}$, we have $k_{12}=q_1-1$, $k_{34}=q_3-1$. Then
			\begin{align*}
				&\mathfrak{I}_{q_1,\ldots, q_4, \kappa}(\ell)=\int_{(\SS2)^4} P_\ell(\<x_1, x_2\>)^{q_1}P_\ell(\<x_3, x_4\>)^{q_3} P_\ell(\<x_2, x_4\>)^{k_{24}}dx\\
				&= \int_{(\SS2)^2} P_\ell(\<x, y\>)^{q_1}dxdy \int_{(\SS2)^2} P_\ell(\<x, y\>)^{q_3}dxdy \int_{(\SS2)^2} P_\ell(\<x, y\>)^{k_{24}}dxdy. 
			\end{align*}
			If $k_{24}=1$, $I$ is equal to $0$. By using Proposition \ref{GegProp}, we easily obtain 
			\begin{align*}
				|\mathfrak{I}_{q_1,\ldots, q_4, \kappa}(\ell)|\leq \frac{C}{\ell^{6}} \,\1_{k_{24}\neq 1} \times \ell^{\#\{i=1,3\,:\, q_i=2\}+\1_{k_{24}=2}} (\log \ell)^{\#\{ i=1,3\,:\, q_i=4\}+\1_{k_{24}=4}} .
			\end{align*}
			The other instances are similar and the following general estimate holds in \textbf{Case 3}:
			\begin{equation}\label{asymp5-1}
				\mathfrak{I}_{q_1,\ldots, q_4, \kappa}(\ell)\leq 
				\frac{C}{\ell^ {3}},
			\end{equation}
			and the proof is concluded.    
			
		\end{proof}

		\begin{lemma}\label{lemmaI-3}
			
			If $R_{q_1,\ldots,q_4,\kappa}=4$, one has
			\begin{equation}\label{asymp-R=4}
				|\mathfrak{I}_{q_1,\ldots, q_4, \kappa}(\ell)|\leq \frac{C}{\ell^ {3}},
			\end{equation}
			$C>0$ denoting a constant independent of $q_1,\ldots,q_4$ and of $\kappa$.
			
		\end{lemma}
		
		\begin{proof}
			We split the proof in 3 different cases according to the properties of $\kappa\in \mathcal{A}_{q_1-1, \ldots, q_4-1}$.
			
			\smallskip
			
			\noindent
			\textbf{Case 1:} there exists just one connection either the pair $(1,2)$ or $(3,4)$:
			\begin{multicols}{2}
				\begin{enumerate}[label=(\alph*)]
					\item $k_{14},k_{23}, k_{24}, k_{34}\neq 0$;
					\item $k_{13},k_{23}, k_{24}, k_{34}\neq 0$;
					\item $k_{13},k_{14}, k_{24}, k_{34}\neq 0$;
					\item $k_{13},k_{14}, k_{23}, k_{34}\neq 0$;
					\item $k_{12},k_{14}, k_{23}, k_{24}\neq 0$;
					\item $k_{12},k_{13}, k_{23}, k_{24}\neq 0$;
					\item $k_{12},k_{13}, k_{14}, k_{24}\neq 0$;
					\item $k_{12},k_{13}, k_{14}, k_{23}\neq 0$.
				\end{enumerate}
			\end{multicols}
			Figure \ref{R=4-case1} shows the cases (a)--(d) where $k_{34}\neq 0$, the latter (e)--(h) turning out by changing the role of the pairs $(3,4)$ and $(1,2)$.  
			
			\begin{center}
				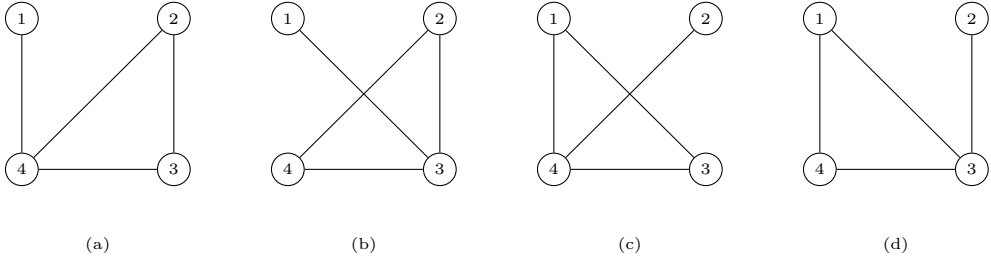
\begin{figure}[h]
					{\tiny 	\begin{center}
							\begin{tikzpicture}
								\node[shape=circle,draw](A) at (1,3) {$1$};
								\node[shape=circle,draw](B) at (3,3) {$2$};
								\node[shape=circle,draw](C) at (3,1) {$3$};
								\node[shape=circle,draw](D) at (1,1) {$4$};
								\draw (C) -- (D);
								\draw (A) -- (D);
								\draw (B) -- (D);
								\draw (B) -- (C);
								\draw (2,0) node{(a)};
								\node[shape=circle,draw](A) at (4.5,3) {$1$};
								\node[shape=circle,draw](B) at (6.5,3) {$2$};
								\node[shape=circle,draw](C) at (6.5,1) {$3$};
								\node[shape=circle,draw](D) at (4.5,1) {$4$};
								\draw (C) -- (D);
								\draw (A) -- (C);
								\draw (B) -- (D);
								\draw (B) -- (C);
								\draw (5.5,0) node{(b)};
								\node[shape=circle,draw](A) at (8,3) {$1$};
								\node[shape=circle,draw](B) at (10,3) {$2$};
								\node[shape=circle,draw](C) at (10,1) {$3$};
								\node[shape=circle,draw](D) at (8,1) {$4$};
								\draw (C) -- (D);
								\draw (A) -- (D);
								\draw (B) -- (D);
								\draw (A) -- (C);
								\draw (9,0) node{(c)};
								\node[shape=circle,draw](A) at (11.5,3) {$1$};
								\node[shape=circle,draw](B) at (13.5,3) {$2$};
								\node[shape=circle,draw](C) at (13.5,1) {$3$};
								\node[shape=circle,draw](D) at (11.5,1) {$4$};
								\draw (C) -- (D);
								\draw (A) -- (D);
								\draw (A) -- (C);
								\draw (B) -- (C);
								\draw (12.5,0) node{(d)};
							\end{tikzpicture}
						\end{center}
					}
					\caption{\small  $R_{q_1,\ldots,q_4,\kappa}=4$, Case 1, (a)--(d); for items (e)--(h), just inverting the roles between the pair of nodes $(1,2)$ and $(3,4)$. }\label{R=4-case1}
				\end{figure}
			\end{center}
			
			We suppose that $k_{14},k_{23}, k_{24}, k_{34} \neq 0$, the other instances being similar. We notice that, in according to $\kappa=\{k_{ij}\}_{i,j=1}^ 4\in \mathcal{A}_{q_1-1,\ldots,q_4-1}$, $k_{14}=q_1-1$
			and we can have $k_{23},k_{24},k_{34}\geq 1$. 
			
			When $q_1=2$ we have
			\begin{align*}
				|\mathfrak{I}_{q_1,\ldots, q_4, \kappa}(\ell)|&=\Big|\int_{(\SS2)^4}P_\ell(\<x_1, x_2\>) P_\ell(\<x_1, x_4\>) P_\ell(\<x_2, x_3\>)^{k_{23}}   P_\ell(\<x_2, x_4\>)^{k_{24}}P_\ell(\<x_3, x_4\>)^{k_{34}+1} dx\Big|\\
				&=\frac{4\pi}{2\ell+1}\Big|\int_{(\SS2)^4} P_\ell(\<x_2, x_3\>)^{k_{23}}   P_\ell(\<x_2, x_4\>)^{k_{24}+1} P_\ell(\<x_3, x_4\>)^{k_{34}+1} dx_2 dx_3 dx_4\Big|.
			\end{align*}
			If $k_{23 }=1$, from Lemma \ref{A}, $|\mathfrak{I}_{q_1,\ldots, q_4, \kappa}(\ell)|\leq \frac{C}{\ell^{5}}(\log \ell )^{\1_{k_{24}=2}+\1_{k_{34}=2}} $, whereas for $k_{23}>1$, we have $|\mathfrak{I}_{q_1,\ldots, q_4, \kappa}(\ell)|\leq \frac{C}{\ell^4}(\sqrt{\log \ell})^{\1_{k_{23}=2}+\1_{k_{24}=1}+\1_{k_{34}=1}}$.
			When $q_1\geq 3$, by applying the Cauchy Schwarz inequality, we have
			\begin{align*}
				&|\mathfrak{I}_{q_1,\ldots, q_4, \kappa}(\ell)|\\
				&=\Big|\int_{(\SS2)^4}P_\ell(\<x_1, x_2\>) P_\ell(\<x_1, x_4\>)^{q_1-1} P_\ell(\<x_2, x_3\>)^{k_{23}}   P_\ell(\<x_2, x_4\>)^{k_{24}}P_\ell(\<x_3, x_4\>)^{k_{34}+1} dx\Big|\\
				&\leq \int_{(\SS2)^4}| P_\ell(\<x_1, x_2\>) P_\ell(\<x_1, x_4\>)^2 P_\ell(\<x_2, x_3\>)   P_\ell(\<x_2, x_4\>) P_\ell(\<x_3, x_4\>)^{2}| dx\\
				&\leq \Big(\int_{(\SS2)^4}P_\ell(\<x_1, x_2\>)^2 P_\ell(\<x_1, x_4\>)^4 P_\ell(\<x_2, x_3\>)^2 dx\Big)^{\frac{1}{2}}\Big( \int_{(\SS2)^4} P_\ell(\<x_2, x_4\>)^2 P_\ell(\<x_3, x_4\>)^{4}dx\Big)^{\frac{1}{2}}\\
				&\leq \frac{C_2}{\ell^3 \sqrt{\ell}}\log \ell.
			\end{align*}
			
			We can resume by stating the following estimate:
			\begin{equation}\label{asymp7-0}
				|\mathfrak{I}_{q_1,\ldots, q_4, \kappa}(\ell)|\leq \frac{C}{\ell^3 \sqrt{\ell}}\log \ell.
			\end{equation}

			\noindent	
			\textbf{Case 2:} we assume that $k_{12}, k_{34}\neq 0$ as in Figure \ref{R=4-case2}, that is,
			\begin{multicols}{2}
				\begin{enumerate}[label=(\alph*)]
					\item $k_{12},k_{23}, k_{24}, k_{34}\neq 0$;	
					\item $k_{12},k_{13}, k_{14}, k_{34}\neq 0$;
					\item $k_{12},k_{14}, k_{24}, k_{34}\neq 0$;
					\item $k_{12},k_{13}, k_{23}, k_{34}\neq 0$;
					\item $k_{12},k_{13}, k_{24}, k_{34}\neq 0$;
					\item $k_{12},k_{14}, k_{23}, k_{34}\neq 0$.
				\end{enumerate}
			\end{multicols}

			\begin{center}
				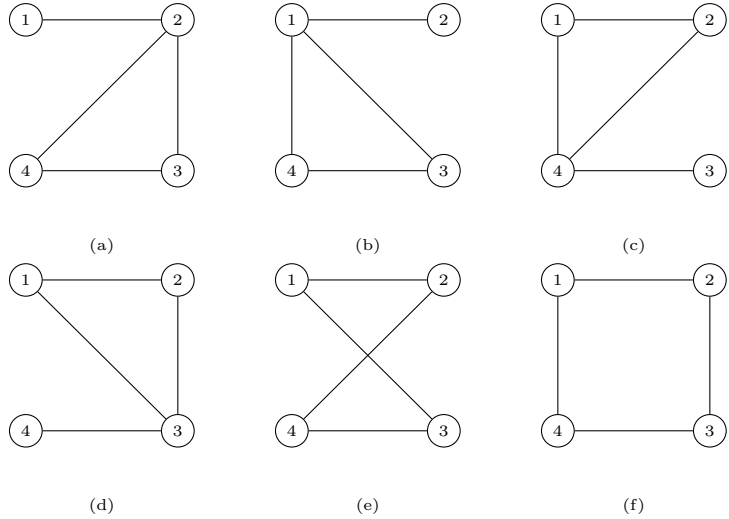
\begin{figure}[h]
					{\tiny 	\begin{center}
							\begin{tikzpicture}
								\node[shape=circle,draw](A) at (1,3) {$1$};
								\node[shape=circle,draw](B) at (3,3) {$2$};
								\node[shape=circle,draw](C) at (3,1) {$3$};
								\node[shape=circle,draw](D) at (1,1) {$4$};
								\draw (A) -- (B);
								\draw (C) -- (D);
								\draw (B) -- (D);
								\draw (B) -- (C);
								\draw (2,0) node{(a)};
								\node[shape=circle,draw](A) at (4.5,3) {$1$};
								\node[shape=circle,draw](B) at (6.5,3) {$2$};
								\node[shape=circle,draw](C) at (6.5,1) {$3$};
								\node[shape=circle,draw](D) at (4.5,1) {$4$};
								\draw (A) -- (B);
								\draw (C) -- (D);
								\draw (A) -- (C);
								\draw (A) -- (D);
								\draw (5.5,0) node{(b)};
								\node[shape=circle,draw](A) at (8,3) {$1$};
								\node[shape=circle,draw](B) at (10,3) {$2$};
								\node[shape=circle,draw](C) at (10,1) {$3$};
								\node[shape=circle,draw](D) at (8,1) {$4$};
								\draw (A) -- (B);
								\draw (C) -- (D);
								\draw (A) -- (D);
								\draw (B) -- (D);
								\draw (9,0) node{(c)};
							\end{tikzpicture}
							
							\begin{tikzpicture}
								\node[shape=circle,draw](A) at (1,3) {$1$};
								\node[shape=circle,draw](B) at (3,3) {$2$};
								\node[shape=circle,draw](C) at (3,1) {$3$};
								\node[shape=circle,draw](D) at (1,1) {$4$};
								\draw (A) -- (B);
								\draw (C) -- (D);
								\draw (A) -- (C);
								\draw (B) -- (C);
								\draw (2,0) node{(d)};
								\node[shape=circle,draw](A) at (4.5,3) {$1$};
								\node[shape=circle,draw](B) at (6.5,3) {$2$};
								\node[shape=circle,draw](C) at (6.5,1) {$3$};
								\node[shape=circle,draw](D) at (4.5,1) {$4$};
								\draw (A) -- (B);
								\draw (C) -- (D);
								\draw (A) -- (C);
								\draw (B) -- (D);
								\draw (5.5,0) node{(e)};
								\node[shape=circle,draw](A) at (8,3) {$1$};
								\node[shape=circle,draw](B) at (10,3) {$2$};
								\node[shape=circle,draw](C) at (10,1) {$3$};
								\node[shape=circle,draw](D) at (8,1) {$4$};
								\draw (A) -- (B);
								\draw (C) -- (D);
								\draw (A) -- (D);
								\draw (B) -- (C);
								\draw (9,0) node{(f)};
							\end{tikzpicture}
						\end{center}
					}
					\caption{\small  $R_{q_1,\ldots,q_4,\kappa}=4$, Case 2. }\label{R=4-case2}
				\end{figure}
			\end{center}
			
			We study the case $k_{12},k_{23}, k_{24}, k_{34}\neq 0$. 
			Then
			\begin{align*}
				&\mathfrak{I}_{q_1,\ldots, q_4, \kappa}(\ell)=\int_{(\SSd)^4}P_\ell(\<x_1, x_2\>)^{q_1}  P_\ell(\<x_2, x_3\>)^{k_{23}}   P_\ell(\<x_2, x_4\>)^{k_{24}} P_\ell(\<x_3, x_4\>)^{k_{34}+1}dx\\
				& =\int_{\SSd}P_\ell(\<x_1, x_2\>)^{q_1}dx_1 \int_{(\SSd)^3} P_\ell(\<x_2, x_3\>)^{k_{23}} P_\ell(\<x_2, x_4\>)^{k_{24}} P_\ell(\<x_3, x_4\>)^{k_{34}+1}dx_2dx_3dx_4
			\end{align*}
			If $k_{23}=1$, from Lemma \ref{A} we have $|\mathfrak{I}_{q_1,\ldots, q_4, \kappa}(\ell)|\leq \frac{C}{\ell^{4}}(\log \ell)^{\1_{q_1=4}+\1_{k_{24}=3}+\1_{k_{34}=2}}$. Changing the role of $k_{23}$ and $k_{24}$ we obtain the same estimate. 
			If $k_{23}, k_{24}>1$, we have $|\mathfrak{I}_{q_1,\ldots, q_4, \kappa}(\ell)|\leq \frac{C}{\ell^4}(\log \ell)^{\1_{q_1=4}+\1_{k_{23}=2}+\1_{k_{24}=2}+\1_{k_{34}=1}}$.  
			
			Resuming, we have 
			\begin{equation}\label{asymp7-1}
				|\mathfrak{I}_{q_1,\ldots, q_4, \kappa}(\ell)|\leq \frac{C}{\ell^4}(\log \ell)^4.
			\end{equation}
			
			\noindent
			\textbf{Case 3:} we consider when $k_{12}=k_{34}=0$, giving that $k_{13},k_{14}, k_{23}, k_{24}\neq 0$. We have that
			\begin{align*}
				&|\mathfrak{I}_{q_1,\ldots, q_4, \kappa}(\ell)|\\
				&\leq \mbox{\small $\int_{(\SS2)^4} |P_\ell(\<x_1,x_2\>)P_\ell(\<x_1,x_3\>)^{k_{13}}P_\ell(\<x_1,x_4\>)^{k_{14}}P_\ell(\<x_2,x_3\>)^{k_{23}}P_\ell(\<x_2,x_4\>)^{k_{24}}P_\ell(\<x_3,x_4\>)| dx$}\\
				&\leq  \mbox{\small $\Big(\int_{(\SS2)^4} P_\ell(\<x_1,x_2\>)^2 P_\ell(\<x_1,x_3\>)^2 P_\ell(\<x_2,x_4\>)^2 dx  \int_{(\SS2)^4} P_\ell(\<x_1,x_4\>)^2 P_\ell(\<x_2,x_3\>)^2 P_\ell(\<x_3,x_4\>)^2 dx\Big)^{\frac{1}{2}}$}\\
				&\leq \frac{C}{\ell^3}.
			\end{align*}
			Then, by considering also \eqref{asymp7-0} and \eqref{asymp7-1}, the estimate in \eqref{asymp-R=4} holds.

		\end{proof}
		
		\begin{lemma}\label{lemmaI-4}
			If $R_{q_1,\ldots,q_4,\kappa}=5$, one has
			\begin{equation}\label{asymp-R=5}
				|\mathfrak{I}_{q_1,\ldots, q_4, \kappa}(\ell)|\leq \frac{C}{\ell^3 \sqrt{\ell}}\log \ell,
			\end{equation}
			$C>0$ denoting a constant independent of $q_1,\ldots,q_4$ and of $\kappa$.
			
		\end{lemma}
		\begin{proof}
			First we notice that we can distinguish two different instances: the first one is when $k_{12}, k_{34}\neq 0$, the second one is when $k_{12}=0$ or $k_{34}=0$ .
			We start from the first case. Without loss of generality, we can suppose that $k_{13}=0$. Hence, by using the Cauchy-Schwarz inequality and Proposition \ref{GegProp},  
			\begin{align*}
				&|\mathfrak{I}_{q_1,\ldots, q_4, \kappa}(\ell)|\\
				&=\Big|\int_{(\SS2)^4} P_\ell(\<x_1, x_2\>)^{k_{12}+1} P_\ell(\<x_1, x_4\>)^{k_{14}}P_\ell(\<x_2, x_3\>)^{k_{23}}P_\ell(\<x_2, x_4\>)^{k_{24}}P_\ell(\<x_3, x_4\>)^{k_{34}+1} \Big|dx\\
				&\leq \mbox{\small $\Big(\int_{(\SS2)^4} P_\ell(\<x_1, x_2\>)^{4} P_\ell(\<x_1, x_4\>)^{2}P_\ell(\<x_2, x_3\>)^{2} dx\Big)^{\frac 12}\Big( \int_{(\SS2)^4}P_\ell(\<x_2, x_4\>)^{2}P_\ell(\<x_3, x_4\>)^{4}dx \Big)^{\frac{1}{2}}$}\\
				&\leq \frac{C}{\ell^3 \sqrt{\ell}}\log \ell.
			\end{align*}

			Now we study $k_{12}=0$.  If $q_1=q_2=3$, again by the Cauchy Schwarz inequality and Proposition \ref{GegProp},
			\begin{align*}
				&|\mathfrak{I}_{q_1,\ldots, q_4, \kappa}(\ell)|\\
				&\leq \int_{(\SS2)^4}| P_\ell(\<x_1, x_2\>) P_\ell(\<x_1, x_3\>) P_\ell(\<x_1, x_4\>) P_\ell(\<x_2, x_3\>)  P_\ell(\<x_2, x_4\>)| P_\ell(\<x_3, x_4\>)^2  dx\\
				&\leq \mbox{\small $\Big(\int_{(\SS2)^4} P_\ell(\<x_1, x_2\>)^2 P_\ell(\<x_1, x_3\>)^2 P_\ell(\<x_3, x_4\>)^4 dx\Big)^{\frac 12} \Big(\int_{(\SS2)^4} P_\ell(\<x_1, x_4\>)^2 P_\ell(\<x_2, x_3\>)^2  P_\ell(\<x_2, x_4\>)^2 dx\Big)^{\frac{1}{2}}$}\\
				&\leq \frac{C}{\ell^3 \sqrt{\ell}}\log \ell.
			\end{align*}
			The other cases can be handled in the same way and bring to a faster decay, and \eqref{asymp-R=5} follows.
			
		\end{proof}
		\begin{lemma}\label{lemmaI-5}
			If $R_{q_1,\ldots,q_4,\kappa}=6$, one has
			\begin{equation}\label{asymp-R=6}
				|\mathfrak{I}_{q_1,\ldots, q_4, \kappa}(\ell)|\leq \frac{C}{\ell^4}\log \ell,
			\end{equation}
			$C>0$ denoting a constant independent of $q_1,\ldots,q_4$ and of $\kappa$.
		\end{lemma}
		\begin{proof}
			Suppose that $k_{ij}\neq 0$ for every $i\neq j$. From the Cauchy Schwarz inequality and Proposition \ref{GegProp}, it follows that
			\begin{align*}
				&|\mathfrak{I}_{q_1,\ldots, q_4, \kappa}(\ell)|\\
				&\leq \int_{(\SS2)^4}| P_\ell(\<x_1, x_2\>)^2 P_\ell(\<x_1, x_3\>) P_\ell(\<x_1, x_4\>) P_\ell(\<x_2, x_3\>)  P_\ell(\<x_2, x_4\>)| P_\ell(\<x_3, x_4\>)^2 |dx \\
				&\leq \mbox{\small $\Big(\int_{(\SS2)^4}P_\ell(\<x_1, x_2\>)^4 P_\ell(\<x_1, x_3\>)^2  P_\ell(\<x_2, x_4\>)^2 dx\Big)^{\frac{1}{2}}\Big( \int_{(\SS2)^4} P_\ell(\<x_1, x_4\>)^2P_\ell(\<x_2, x_3\>)^2 P_\ell(\<x_3, x_4\>)^4 dx\Big)^{\frac{1}{2}}$ }\\
				&\leq \frac{C}{\ell^4}\log \ell.
			\end{align*}
			The other cases are similar, giving also a faster decay, so \eqref{asymp-R=6} holds.
		\end{proof}

		\subsection{Proof of Theorem \ref{Conv0}}\label{3.2}
		
		We are now in a position to prove Theorem \ref{Conv0}, that is the main result on the convergence in Wasserstein distance for the Malliavin covariances of $\tilde X_\ell$, as $\ell\to +\infty$. 
		
		Recalling the (finite dimensional) chaos expansion for $\varphi(Z)$ in \eqref{chphi} and substituting it in \eqref{Xl}, we obtain the following expansion for $X_\ell$:
		\begin{equation}\label{chaosXl}
			X_\ell=m_{\ell;d}+\sum_{q \geq 2} \frac{b_q}{q!} \int_{\mathbb S^d}  H_q(T_\ell(x)) dx,
		\end{equation}
		where
		$$
		m_{\ell;d}=\E[X_\ell]=\E[\varphi(Z)] \mu_d .
		$$
		Equation \eqref{chaosXl} is really the chaos expansion of $X_\ell$. One has in fact, for every $f\in H$ and $p\neq q$, 
		$
		\E[ H_{p}(T(f))\int_{\mathbb S^d}  H_q(T_\ell(x)) dx]
		=\int_{\mathbb S^d} \E[ H_{p}(T(f)) H_q(T_\ell(x))]dx=0.
		$
		
		Notice that \eqref{chaosXl} says that the projection on the chaos of order 1 is null, as briefly mentioned above. In fact, recalling \eqref{chphi}, by \eqref{Tl}
		$$
		J_1(X_\ell)=b_1\int_{\SSd}T_\ell(x)dx=b_1\sqrt{\frac{\mu_d}{n_{\ell;d}}}\sum_{m=1}^{n_{\ell;d}} a_{\ell,m} \int_{\SSd} Y_{\ell,m;d}(x) dx=0.
		$$
		Following \eqref{chaosXl}, the chaos expansion of the normalized r.v. $\X_\ell$ is given by
		\begin{equation}\label{chaos-tildeX}
			\X_\ell=\frac 1{v_{\ell;d}}\,\sum_{q \geq 2} \frac{b_q}{q!} \int_{\mathbb S^d}  H_q(T_\ell(x)) dx,\mbox{ where }v^2_{\ell;d}=\Var(X_{\ell}).
		\end{equation} 
		Notice that, by \eqref{varXl} from Theorem \ref{ROS} and \eqref{n-ell}, 
		\begin{equation}\label{v-ell}
			v_{\ell;d}\sim b_2c_d \ell^{- \frac{d-1}{2}}.
		\end{equation}

		
		\smallskip
		
		\begin{proof}[Proof of Theorem \ref{Conv0}]
			By using \eqref{DkIntHp} (with $k=1$) and classical density arguments, the Malliavin derivative $D\tilde X_\ell: \Omega \to H$ is given by
			$$
			D_y \X_\ell
			=\frac{1}{v_{\ell;d}} \sqrt{\frac{n_{\ell;d}}{\mu_d}} \sum_{q\geq 2} \frac{b_q}{(q-1)!}\int_{\mathbb{S}^d} H_{q-1}(T_\ell(x))  G_{\ell; d}(\<x,y\>) dx.
			$$
			Following \eqref{MallCov}, with $n=1$ and $\mathcal{H}=H=L^ 2(\SSd,\B(\SSd), \Leb)$, we can write down  the Malliavin covariance $\sigma_\ell$ of $\X_\ell$:
			\begin{align*}
				\sigma_\ell
				=&\int_{\mathbb{S}^d} |D_y \X_\ell|^2 dy= \frac{1}{v_{\ell;d}^2}\frac{n_{\ell;d}}{\mu_d}\sum_{q, p\geq 2} \frac{b_q b_p}{(q-1)!(p-1)!}\times \\
				&\times  \int_{\mathbb{S}^d} \int_{(\mathbb{S}^d)^2}H_{q-1}(T_\ell(x))  H_{p-1}(T_\ell(z))  G_{\ell; d}(\<x,y\>)G_{\ell; d}(\<z,y\>) dx dz dy.
			\end{align*}
			By using the duplication formula \eqref{G1}, we obtain
			
			\begin{equation}\label{sigma}
				\sigma_\ell
				=\frac{1}{v_{\ell;d}^2}\sum_{q, p\geq 2} \frac{b_q b_p}{(q-1)!(p-1)!}  \int_{(\mathbb{S}^d)^2} H_{q-1}(T_\ell(x))  H_{p-1}(T_\ell(z)) G_{\ell; d}(\<x,z\>) dxdz.
			\end{equation}
			Therefore, by \eqref{covT},
			\begin{align*}
				\E[\sigma_\ell]&=\frac{1}{v_{\ell;d}^2}\sum_{q, p\geq 2} \frac{b_q b_p}{(q-1)!(p-1)!}  \int_{(\mathbb{S}^d)^2} \E[H_{q-1}(T_\ell(x))  H_{p-1}(T_\ell(z)) ] G_{\ell; d}(\<x,z\>) dxdz\\
				&=\frac{1}{v_{\ell;d}^2}\sum_{q\geq 2} \frac{b_q^2}{(q-1)!}  \int_{(\mathbb{S}^d)^2}  G_{\ell; d}(\<x,z\>)^q dxdz.
			\end{align*}
			Then, from the asymptotics for moment of Gegenbauer polynomials in Proposition \ref{GegProp} and from \eqref{v-ell}, we have that 
			$$
			\E[\sigma_\ell]-2= O\left(  \1_{d\geq 3} \frac{1}{\ell}+ \1_{d=2}\left( \1_{b_4 \neq 0}\frac{\log \ell}{\ell}+ \1_{b_4=0} \frac{1}{\ell}      \right )\right)
			$$ as $\ell \to \infty$. In the above result we underline that the difference between $d=2$ and $d\geq 3$ changes the asymptotic behavior when $b_4\neq 0$.

			Now we study the variance of $\sigma_\ell$.
			Denoting with $dx:=dx_1 dx_2 dx_3 dx_4$, we have that 
			\begin{align*}
				&\E[\sigma_\ell^2]=\\
				&\frac{1}{v_{\ell;d}^4} \sum_{q_1,q_2, q_3, q_4\geq 2} \Big(\prod_{j=1}^4 \frac{b_{q_j}}{(q_j-1)!} \Big)\int_{(\mathbb S^d)^4 } \E[\prod_{i=1}^4 H_{q_i-1}(T_\ell(x_i))] G_{\ell;d}(\<x_1,x_2\>)G_{\ell;d}(\<x_3,x_4\>)dx.
			\end{align*} 
			By using Lemma \ref{DIAGRAMMA}, we have
			\begin{align*}
				&\E[\sigma_\ell^2]=\\
				&\frac{1}{v_{\ell;d}^4} \sum_{q_1,q_2, q_3, q_4\geq 2} \Big(\prod_{j=1}^4 \frac{b_{q_j}}{(q_j-1)!} \Big)\prod_{r=1}^4 (q_r-1)!  \times \\
				&\sum_{\{k_{i,j}\}^4_{{ i,j=1}} \in \mathcal{A}_{q_1-1,\ldots,q_4-1}  }
				\prod_{\underset{i<j}{i,j=1}}^4\frac{1}{k_{ij}!} \int_{(\mathbb S^d)^4 }\prod_{\underset{i<j}{i,j=1}}^4 G_{\ell;d}(\<x_i,x_j\>)^{k_{ij}} G_{\ell;d}(\<x_1,x_2\>)G_{\ell;d}(\<x_3,x_4\>)dx.
			\end{align*} 
			First of all, when $q_1=q_2$ and $q_3=q_4$ we have
			\begin{align*}
				&\sum_{q_1 q_3\geq 2} b_{q_1}^2 b_{q_3}^2  \sum_{\{k_{i,j}\}^4_{{ i,j=1}} \in \mathcal{A}_{q_1-1,\ldots, q_3-1}  } \prod_{\underset{i<j}{i,j=1}}^4\frac{1}{k_{ij}!}\times \\
				&\times  \int_{(\mathbb S^d)^4 }\prod_{\underset{i<j}{i,j=1}}^4 G_{\ell;d}(\<x_i,x_j\>)^{k_{ij}} G_{\ell;d}(\<x_1,x_2\>)G_{\ell;d}(\<x_3,x_4\>)dx.
			\end{align*}
			Now, when $k_{13}=k_{14}=k_{23}=k_{24}=0$, one gets $k_{12}=q_1-1$ and $k_{34}=q_3-1$. So it remains
			\begin{align*}
				&\sum_{q_1 q_3\geq 2} \frac{b_{q_1}^2 b_{q_3}^2}{(q_1-1)!(q_3-1)!}\int_{(\mathbb S^d)^4 } G_{\ell;d}(\<x_1,x_2\>)^{q_1}G_{\ell;d}(\<x_3,x_4\>)^{q_3}dx.
			\end{align*}  
			This is exactly $v_{\ell;d}^4\, \E[\sigma_\ell]^2$. Recalling sets  
			$\mathcal{N}_{q_1-1, q_2-1, q_3-1, q_4-1}$ and $\mathcal{C}_{q_1-1,q_2-1,q_3-1, q_4-1}$ in \eqref{setN} and \eqref{setC}, we obtain
			\begin{align*}
				&\Var(\sigma_\ell )=\E[\sigma_\ell^ 2]-\E[\sigma_\ell]^ 2
				=\frac{1}{v_{\ell;d}^4} \sum_{q_1,q_2, q_3, q_4\geq 2} \prod_{i=1}^ 4\frac{b_{q_i}}{(q_i-1)!} \prod_{r=1}^4 (q_r-1)! \\
				&\times  \sum_{\{k_{i,j}\}^4_{{ i,j=1}} \in \mathcal{C}_{q_1-1,\ldots,q_4-1}  } \prod_{\underset{i<j}{i,j=1}}^4\frac{1}{k_{ij}!} \int_{(\mathbb S^d)^4 }\prod_{\underset{i<j}{i,j=1}}^4 G_{\ell;d}(\<x_i,x_j\>)^{k_{ij}} G_{\ell;d}(\<x_1,x_2\>)G_{\ell;d}(\<x_3,x_4\>)dx,
			\end{align*}
			so that
			\begin{align}
				&\Var(\sigma_\ell )
				\leq \frac{1}{v_{\ell;d}^4} \sum_{q_1,q_2, q_3, q_4\geq 2} \prod_{i=1}^ 4\frac{|b_{q_i}|}{(q_i-1)!} \prod_{r=1}^4 (q_r-1)! \times\nonumber\\
				&\times  \sum_{\{k_{i,j}\}^4_{{ i,j=1}} \in \mathcal{C}_{q_1-1,\ldots,q_4-1}  } \prod_{\underset{i<j}{i,j=1}}^4\frac{1}{k_{ij}!} \Big|\int_{(\mathbb S^d)^4 }\prod_{\underset{i<j}{i,j=1}}^4 G_{\ell;d}(\<x_i,x_j\>)^{k_{ij}} G_{\ell;d}(\<x_1,x_2\>)G_{\ell;d}(\<x_3,x_4\>)dx\Big|\label{appoggio}
			\end{align}
			We now prove that there exists $c>0$ such that for every $\{k_{i,j}\}^n_{{ i,j=1}} \in \mathcal{C}_{q_1-1,\ldots,q_4-1} $,
			\begin{equation}\label{app1}
				\Big|\int_{(\SSd)^4} \prod_{\underset{i<j}{i,j=1}}^4 G_{\ell;d}(\<x_i,x_j\>)^{k_{ij}} G_{\ell;d}(\<x_1,x_2\>)G_{\ell;d}(\<x_3,x_4\>)dx\Big|
				\leq \frac{c_d}{\ell^{2d-2+\frac{d-1}{2}}}.
			\end{equation}  
			\eqref{app1} will follow by applying Lemma \ref{STIMA}. For a fixed  $\kappa=\{k_{ij}\}_{i,j=1}^ 4\in  \mathcal{C}_{q_1-1,\ldots,q_4-1} $,let $N_\kappa$ be the number of the connected components of the extrapolated graph $\mathfrak{G}_\kappa$. We observe that $N_\kappa\in\{1,2\}$, recall in fact that by \eqref{defA} for any $i$ there exists at least an index $j\neq i$ such that $k_{ij}>0$ (here, $q_i-1\geq 1$ for every $i$). So, we split our reasoning according to $N_\kappa=1$ and $N_\kappa=2$.
			
			\smallskip
			
			\textbf{Case 1: $N_\kappa=1$.} By using the Cauchy-Schwarz inequality we have
			\begin{align*}
				&\Big|\int_{(\SSd)^4} \prod_{\underset{i<j}{i,j=1}}^4 G_{\ell;d}(\<x_i,x_j\>)^{k_{ij}} G_{\ell;d}(\<x_1,x_2\>)G_{\ell;d}(\<x_3,x_4\>)dx\Big|\leq \\
				&\Big(\int_{(\SSd)^4} \prod_{\underset{i<j}{i,j=1}}^4 G_{\ell;d}(\<x_i,x_j\>)^{2k_{ij}} dx\Big)^ {1/2}
				\Big(\int_{(\SSd)^4} G_{\ell;d}(\<x_1,x_2\>)^2G_{\ell;d}(\<x_3,x_4\>)^2dx\Big)^{1/2}
			\end{align*}  
			Estimating the first factor by means of \eqref{stimaGRAFI} with $N_\kappa=1$ and computing the second factor by means of \eqref{asGeg2}, straightforward computations give \eqref{app1}
			
			\smallskip
			
			\textbf{Case 2: $N_\kappa=2$.}  Figure \ref{N=2} shows all possible extrapolated graphs having $2$ connected components. We 
			notice that the graph in (a) is extrapolated by an element $\kappa=\{k_{ij}\}_{i,j=1}^ 4$ belonging to $\mathcal{N}_{q_1-1,\ldots,q_4-1} $. Such indexes have been already deleted, so we study the cases shown in (b) and in (c).
			\begin{center}
				\begin{figure}[h]
					{\footnotesize 	\begin{center}
							\begin{tikzpicture}
								\node[shape=circle,draw](A) at (1,3) {$1$};
								\node[shape=circle,draw](B) at (3,3) {$2$};
								\node[shape=circle,draw](C) at (3,1) {$3$};
								\node[shape=circle,draw](D) at (1,1) {$4$};
								\draw (A) -- (B);
								\draw (C) -- (D);
								\draw (2,0) node{(a)};
								\node[shape=circle,draw](A) at (4.5,3) {$1$};
								\node[shape=circle,draw](B) at (6.5,3) {$2$};
								\node[shape=circle,draw](C) at (6.5,1) {$3$};
								\node[shape=circle,draw](D) at (4.5,1) {$4$};
								\draw (A) -- (C);
								\draw (B) -- (D);
								
								\draw (5.5,0) node{(b)};
								\node[shape=circle,draw](A) at (8,3) {$1$};
								\node[shape=circle,draw](B) at (10,3) {$2$};
								\node[shape=circle,draw](C) at (10,1) {$3$};
								\node[shape=circle,draw](D) at (8,1) {$4$};
								\draw (A) -- (D);
								\draw (B) -- (C);
								\draw (9,0) node{(c)};
							\end{tikzpicture}
						\end{center}
					}
					\caption{\small All possible extrapolated graphs $\mathfrak{G}_\kappa$ from $\kappa=\{k_{ij}\}_{i,j=1}^4$  having exactly two connected components. }\label{N=2}
				\end{figure}
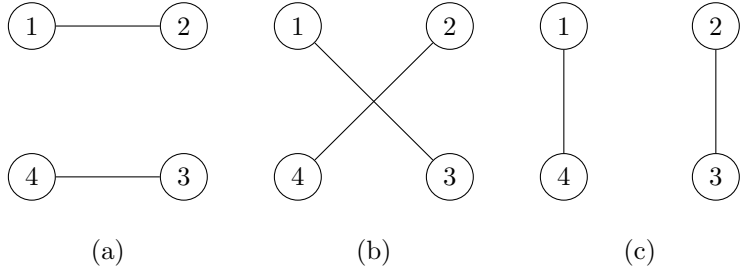
			\end{center}
			As for case (b), we have
			\begin{align*}
				&\Big|\int_{(\SSd)^4} \prod_{\underset{i<j}{i,j=1}}^4 G_{\ell;d}(\<x_i,x_j\>)^{k_{ij}} G_{\ell;d}(\<x_1,x_2\>)G_{\ell;d}(\<x_3,x_4\>)dx\Big|\\
				&=\Big|\int_{(\SSd)^4} G_{\ell;d}(\<x_1,x_2\>)G_{\ell;d}(\<x_1,x_3\>)^{q_1-1}G_{\ell;d}(\<x_2,x_4\>)^{q_2-1} G_{\ell;d}(\<x_3,x_4\>) dx\Big|.
			\end{align*}
			Assume that $q_1=2$ or $q_2=2$. W.l.g we set $q_1=2$. Then, using \eqref{G1}, we have
			\begin{align*}
				&\Big|\int_{(\SSd)^4} G_{\ell;d}(\<x_1,x_2\>)G_{\ell;d}(\<x_1,x_3\>)G_{\ell;d}(\<x_2,x_4\>)^{q_2-1} G_{\ell;d}(\<x_3,x_4\>) dx_1dx_2dx_3dx_4\Big|\\
				&=\Big|  \frac{\mu_d}{n_{\ell;d}} \int_{(\SSd)^3} G_{\ell;d}(\<x_2,x_3\>)G_{\ell;d}(\<x_2,x_4\>)^{q_2-1} G_{\ell;d}(\<x_3,x_4\>) dx_2dx_3dx_4 \Big|\\
				&=\Big|  \frac{\mu_d^2}{n_{\ell;d}^2} \int_{(\SSd)^2}G_{\ell;d}(\<x_2,x_4\>)^{q_2}  dx_2dx_4 \Big|\\
				&\leq \frac{\mu_d^2}{n_{\ell;d}^2}\int_{(\SSd)^2}G_{\ell;d}(\<x_2,x_4\>)^{2}  dx_2dx_4 \leq \frac{c_d}{\ell^{3d-3}},
			\end{align*}
			the last inequality following from \eqref{asGeg2}. If instead $q_1\geq 3$ and $q_2 \geq 3$,
			\begin{align*}
				&\Big|\int_{(\SSd)^4} G_{\ell;d}(\<x_1,x_2\>)G_{\ell;d}(\<x_1,x_3\>)^{q_1-1}G_{\ell;d}(\<x_2,x_4\>)^{q_2-1} G_{\ell;d}(\<x_3,x_4\>) dx\Big|\\
				&\leq \int_{(\SSd)^4} G_{\ell;d}(\<x_1,x_3\>)^2 G_{\ell;d}(\<x_2,x_4\>)^2 |G_{\ell;d}(\<x_3,x_4\>)| dx
			\end{align*}
			By integrating first w.r.t $x_1$, then w.r.t. $x_2$ and by using \eqref{asGeg2}, we get
			\begin{align*}
				&\Big|\int_{(\SSd)^4} G_{\ell;d}(\<x_1,x_2\>)G_{\ell;d}(\<x_1,x_3\>)^{q_1-1}G_{\ell;d}(\<x_2,x_4\>)^{q_2-1} G_{\ell;d}(\<x_3,x_4\>) dx\Big|\\
				&
				\leq \frac{c_d}{\ell^{2d-2}} \int_{(\SSd)^2} |G_{\ell;d}(\<x_3,x_4\>)| dx_3dx_4.
			\end{align*}
			Now we use the Cauchy-Schwarz inequality and again apply  \eqref{asGeg2}. We finally obtain \eqref{app1}.
			
			Case (c) in Figure \ref{N=2} can be treated analogously, so \eqref{app1} finally holds.

			Coming back to the study of $\Var(\sigma_\ell)$, we use \eqref{v-ell}, we insert the estimates \eqref{app1} in \eqref{appoggio} and we have 
			$$
			\Var(\sigma_\ell )
			\leq \frac{c_d \, \ell^{2d-2}}{\ell^{2d-2+\frac{d-1}{2}}} \sum_{q_1,q_2, q_3, q_4\geq 2} \prod_{i=1}^ 4\frac{|b_{q_i}|}{(q_i-1)!}
			\prod_{r=1}^4 (q_r-1)! \sum_{\{k_{i,j}\}^4_{{ i,j=1}} \in \mathcal{A}_{q_1-1,\ldots,q_4-1}  } \prod_{\underset{i<j}{i,j=1}}^4\frac{1}{k_{ij}!} 
			$$
			We now use \eqref{PassDel}: for $Z\sim \mathcal N(0,1)$,
			\begin{align*}
				\Var(\sigma_\ell )
				&\leq \frac{c_d \, \ell^{2d-2}}{\ell^{2d-2+\frac{d-1}{2}}} \sum_{q_1,q_2, q_3, q_4\geq 2} \prod_{i=1}^ 4\frac{|b_{q_i}|}{(q_i-1)!}
				\E\Big[\prod_{r=1}^4 H_{q_r-1}(Z)\Big]\\
				&=  \frac{c_d }{\ell^{\frac{d-1}{2}}} \E\Big[\Big(\sum_{q \geq 2}\frac{|b_q|}{(q-1)!}H_{q-1}(Z)\Big)^ 4\Big].
			\end{align*}
			By Assumption \ref{ASSUMPTION},
			\begin{align*}
				\Var(\sigma_\ell )
				&\leq \frac{c_d }{\ell^{\frac{d-1}{2}}}\E[|D \phi(Z)|^4]\stackrel{\ell\to\infty}{\longrightarrow}0.
			\end{align*}
			
			This concludes the proof for the case $d\geq 3$. If $d=2$, the above estimate gives $\Var(\sigma_\ell )=O(\ell^{-\frac{1}{2}})$, which is not enough as it would give, in Theorem \ref{mainThm}, $O_\varepsilon(\ell^ {-\frac{1-\varepsilon}4})$ instead of $O_\varepsilon( \ell^{-\frac{1-\varepsilon}{2}})$. This, in turn, would imply that the convergence speed in Total Variation distance depends on the dimension. However, in the case $d=2$ we can actually prove that
			$\Var(\sigma_\ell )=O(\ell^ {-1})$, allowing us to reach the optimal bound $O_\varepsilon( \ell^{-\frac{1-\varepsilon}{2}})$ in Theorem \ref{mainThm}.	So, when $d=2$ we need to improve the estimates of the integrals in \ref{appoggio}.
			The proof strategy is described in Subsection \ref{Case_d2}, which ends with Proposition \ref{prop-I} . \\
			Let us come back to \eqref{appoggio} for $d=2$, in particular $G_{\ell;2} =P_\ell$ the $\ell$-th Legendre polynomial. By recalling \eqref{v-ell}, by applying the estimate in Proposition \ref{prop-I} below, by using \eqref{PassDel} and Assumption \ref{ASSUMPTION}, we get
			\begin{align*}
				\Var(\sigma_\ell )
				&\leq C\ell^ 2 \times  \frac{1}{\ell^3} 
				\sum_{q_1,q_2, q_3, q_4\geq 2} \prod_{i=1}^ 4\frac{|b_{q_i}|}{(q_i-1)!} \E\Big[\prod_{r=1}^4 H_{q_r-1}(Z)\Big]\\
				&\leq \frac C\ell \, \E\Big[\Big(\sum_{q \geq 2}\frac{|b_q|}{(q-1)!}H_{q-1}(Z)\Big)^ 4\Big]
				= \frac C\ell \, \E[|D \phi(Z)|^4]=O(\ell^{-1}),
			\end{align*} 
			hence the proof is concluded.
		\end{proof} 
		
		\section{Uniform boundedness of Malliavin-Sobolev norms}\label{sect:UnifLim}
		
		This section is devoted to the proof of Proposition \ref{UnifLim}, that is, all moments of $|D^{(k)} \X_\ell|_{\mathcal{H}^{\otimes k}}$ and of $|D^{(k)} L\X_\ell|_{\mathcal{H}^{\otimes k}}$ are uniformly bounded in $\ell$.

		\medskip

		\begin{proof}[Proof of Proposition \ref{UnifLim}] Without loss of generality, we can assume that $n$ is even. So, we fix $k\in\N$ and $n=2p$, $p\in\N$.  We first prove that $\sup_{\ell \mathrm{\ even}}\E[|D^{(k)} \X_\ell|^{2p}_{\mathcal{H}^{\otimes k}}]$ is finite.
			
			Recall the chaos expansion \eqref{chaos-tildeX} and by using \eqref{DkIntHp}, it easily follows that
			$$
			D_{y_1,\ldots y_k}^{(k)} \X_\ell= \frac{1}{v_{\ell;d}} \left(\frac{n_{\ell;d}}{\mu_d}\right)^{\frac k2} \sum_{q\geq 2\vee k} \frac{b_q}{(q-k)!}\int_{\mathbb{S}^d} H_{q-k}(T_\ell(x)) \prod_{i=1}^k G_{\ell; d}(\<x,y_i\>) dx. 
			$$
			Thus,
			\begin{align*}
				&\E[|D^{(k)} \X_\ell|^{2 p}_{\mathcal{H}^{\otimes k}}]=\E\Big[\Big(\int_{(\mathbb{S}^d)^k} |D^{(k)}_{y_1,\ldots,y_k} \X_\ell|^2 dy_1 dy_2 \ldots dy_k\Big)^{p}\Big]\\
				&= \frac{1}{v_{\ell;d}^{2 p}} \E\Big[\Big(\sum_{q_1,q_2 \geq 2 \vee k} \frac{b_{q_1} b_{q_2}}{(q_1-k)!(q_2-k)!} \int_{(\mathbb{S}^d)^2}H_{q_1-k}(T_\ell(y)) H_{q_2-k}(T_\ell(z)) G_{\ell; d}(\<y, z\>)^k dydz \Big)^p  \Big]\\
				&=\frac{1}{v_{\ell;d}^{2p}}  \sum_{  q_i \geq 2 \vee k, i=1,\ldots,2p} \prod_{i=1}^{2p} \frac{b_{q_i}}{(q_i-k)!} \prod_{r=1}^{2p} (q_r-k)! \sum_{\{k_{i,j}\}^{2p}_{\underset{i<j}{ i,j=1}} \in A_{q_1-k,\ldots,q_{2p}-k} } \prod_{\underset{i<j}{i,j=1}}^{2p}\frac{1}{k_{i,j}!}\\
				&\times \int_{(\mathbb{S}^d)^{2p}} \prod_{\underset{i<j}{i,j=1}}^{2p} G_{\ell;d}(\<x_i,x_j\>)^{k_{i,j}} \prod_{s=1}^p G_{\ell; d}(\<x_s, x_{s+p}\>)^k dx,
			\end{align*}  
			in which we have used \eqref{eqDIAGRAMMA}. We start to estimate the integrals in the above r.h.s.
			Using the Cauchy-Schwarz inequality and \eqref{connessiSTIMA}, it follows that
			\begin{align*}
				&\Big|\int_{(\mathbb{S}^d)^{2p}} \prod_{\underset{i<j}{i,j=1}}^{2p} G_{\ell;d}(\<x_i,x_j\>)^{k_{i,j}} \prod_{s=1}^p G_{\ell; d}(\<x_s, x_{s+p}\>)^k dx\Big|\\
				&\leq \Big(\int_{(\mathbb{S}^d)^{2p}} \prod_{\underset{i<j}{i,j=1}}^{2p} G_{\ell;d}(\<x_i,x_j\>)^{2k_{i,j}} dx \int_{(\SSd)^{2p}}\prod_{s=1}^p G_{\ell; d}(\<x_s, x_{s+p}\>)^{2k} dx\Big)^{\frac{1}{2}}\leq \frac{C_{d;p}}{\ell^{(d-1)p}}.
			\end{align*}
			We insert the above estimate and, by using the asymptotics of $v_{\ell;d}$ in \eqref{v-ell} and the representation \eqref{PassDel}, it follows that
			\begin{align*}
				&\E[|D^{(k)} \X_\ell|^{2 p}_{\mathcal{H}^{\otimes k}}] \leq\frac{1}{v_{\ell;d}^{2p}}\frac{C_{d;p}}{\ell^{(d-1)p}} \\
				&\qquad \times \sum_{  q_i \geq 2 \vee k, i=1,\ldots,2p} \left( \prod_{i=1}^{2p} \frac{|b_{q_i}|}{(q_i-k)!}\right) \prod_{r=1}^{2p} (q_r-k)! \sum_{\{k_{i,j}\}^{2p}_{\underset{i<j}{ i,j=1}} \in A_{q_1-k,\ldots,q_{2p}-k} } \prod_{\underset{i<j}{i,j=1}}^{2p}\frac{1}{k_{i,j}!}\\
				&\leq \frac{\mathrm{Const} \big(\ell^{(d-1)p}+o(\frac{1}{\ell^{(d-1)p}})\big)}{\ell^{(d-1)p}} \times\sum_{  q_i \geq 2 \vee k, i=1,\ldots,2p} \left( \prod_{i=1}^{2p} \frac{|b_{q_i}|}{(q_i-k)!}\right) \E[\prod_{r=1}^{2p} H_{q_r-k} (Z)]\\
				&= \mathrm{Const}(1+o(1))\E\Big[\Big|
				\sum_{q \geq 2 \vee k}\frac{|b_q|}{(q-k)!} H_{q-k}(Z)\Big|^{2p}\Big].
			\end{align*}
			Hereafter $\mathrm{Const}$ denotes a positive constant, possibly changing from a line to another and possibly depending on $d$ and $p$ but  independent of $\ell$. Now, by \eqref{DkTphi} in Assumption \ref{ASSUMPTION}, we obtain 
			$$
			\sup_{\ell }\E[|D^{(k)} \X_\ell|^{2 p}_{\mathcal{H}^{\otimes k}}]\leq \mathrm{Const}\, \E[|D^ k\phi(Z)|^ {2p}]<\infty.
			$$
			
			Concerning the study of $\E[|D^{(k)} L\X_\ell|^{2p}_{\mathcal{H}^{\otimes k}}]$, by using \eqref{defL} we have
			$$
			L\X_\ell=-\frac{1}{v_{\ell;d}} \sum_{q\geq 2} \frac{b_q}{(q-1)!} \int_{\SSd} H_q(T_\ell(x))dx.
			$$
			By comparing this expansion with \eqref{chaos-tildeX}, one deduces that one can repeat the same computations with $b_q$ replaced by $-qb_q$. Hence, 
			$$
			\sup_{\ell}\E[|D^{(k)}L \X_\ell|^{2 p}_{\mathcal{H}^{\otimes k}}]\leq \mathrm{Const}\, 
			\E\Big[\Big|
			\sum_{q \geq 2 \vee k}\frac{q|b_q|}{(q-k)!} H_{q-k}(Z)\Big|^{2p}\Big]
			= \mathrm{Const}\, \E[|D^ kL\phi(Z)|^ {2p}]
			$$
			and this is finite again because of \eqref{DkTphi} in Assumption \ref{ASSUMPTION}. This concludes the proof.
		\end{proof}

		\part{The fractional Ornstein-Uhlenbeck process in rough volatility modelling}\label{Part2}

		\chapter{A bivariate fractional Ornstein Uhlenbeck process}\label{Biv_fOU_3}
		\section{Introduction}
	In this Chapter we introduce a bivariate process that generalizes the univariate fOU process. In Section \ref{Univariate_section} we recall the definition of univariate fOU and its properties. In Section \ref{MultifBM_section}, following \cite{multivar}, we introduce the mfBm, which is the building block of our model. Finally, in Section \ref{Bivariate_fouSection} we define the 2fOU process and we study its properties.

			\section{The univariate Ornstein-Uhlenbeck process}\label{Univariate_section}
		In this section we recall the definition of the univariate fOU process and its main properties. In this discussion, we primarily follow the work of Cheridito et al. \cite{cheridito2003}. To begin, let us recall that a fBm with Hurst parameter $H\in (0,1]$ is an almost surely continuous, centered Gaussian process $(B^H_t)_{t\in \R}$ with covariance function given by
		\begin{equation}\label{cfBM}
			\Cov(B_t^H, B_s^H) = \frac{1}{2}\left(|t|^{2H} + |s|^{2H} - |t-s|^{2H}\right).
		\end{equation}
		This process is self-similar with exponent
		$H$ and admits a Volterra representation with kernel (see e.g. \cite{Nua})
		\begin{equation}\label{eqn:kernelfbm}
			K_H(t,s) = c_H \left[ \left( \frac{t}{s} \right) ^{H-1/2}(t-s)^{H-1/2} - \left( H-\frac{1}{2} \right) s^{1/2 - H} \int_s^t \!u^{H-3/2}(u-s)^{H-1/2} du \right],
		\end{equation}
		where
		\[
		c_H = \Big(  \frac{2H \, \Gamma(3/2-H)}{\Gamma(H+1/2) \, \Gamma(2-2H)}\Big)^{1/2}.
		\]
		This means that there exists a Brownian motion $W$ such that for all $t\geq 0$ 
		\begin{equation}\label{Volterra_FBM}
			B_t^H=\int_0^t K_H(t,s) dW_s.
		\end{equation}
		Let $(\Omega, \mathcal F, \P)$ be a probability space. Let us fix $\a\in \R_+$. Then for all $t>b$, $t,b\in\R$, the random variable
		$$
		X^b_t=\int_{b}^t e^{\a u} dB_u^H(\omega).
		$$ 
		exists as a pathwise Riemann-Stieltjes integral, as per Proposition A.1 in \cite{cheridito2003}. Moreover, it can be expressed as:
		\begin{equation}\label{FormX}
			X_t = e^{\alpha t} B_t^H(\omega) - e^{\alpha b} B_b^H(\omega) - \alpha \int_{b}^t B_u^H(\omega) e^{\alpha u} du.
		\end{equation}
		Now, consider $\alpha, \nu > 0$, and $\psi \in L^0(\Omega)$. The solution to the Langevin equation
		\begin{equation}\label{Langeq}
			Y_t^H = \psi - \alpha\int_{0}^t Y_s ds + \nu B_t^H, \quad t\geq 0
		\end{equation}
		exists as a path-wise Riemann-Stieltjes integral, and it is given by
		$$
		Y_t^{H, \psi}= e^{-\a t}\Big(\psi+\nu \int_{0}^t e^{\a u}dB_u^H \Big),\,\,\, t\geq 0.
		$$
		It is the unique almost surely continuous process that solves \eqref{Langeq}. In particular, the process
		\begin{equation}\label{fOU1}
			Y_t^H = \nu\int_{-\infty}^t e^{-\alpha(t-u)} dB_u^H, \quad t\in \R,
		\end{equation}
		solves \eqref{Langeq} with the initial condition $\psi = Y_0^H$. From \eqref{FormX}, it follows that $Y_t^H$ has the following almost surely representation:
		\begin{equation}\label{repY}
			Y_t^H = \nu \left( B_t^H - \alpha\int_{-\infty}^t B_u^H e^{-\alpha(t- u)} du \right).
		\end{equation}
		The process $Y^H$ is the stationary fOU process. 
		
		If on the other hand one considers $Y_0^H=0$, the solution to \eqref{Langeq} is given by
		\begin{align*}
			V_t=\nu \int_0^t e^{-\a(t-u)} dB_u^H
		\end{align*}
		and, by integration by parts and stochastic Fubini theorem, we can represent $V_t$ as follows:
		\begin{equation}\label{eq:FOU}
			V_t=\nu B_t^H-\a \nu \int_0^t e^{-\a(t-u)} B_u^H du.
		\end{equation}
		
		The process $V$ inherits a Volterra representation from \eqref{Volterra_FBM}:
		\begin{align*}
			V_t&=\int_0^t K^H(t,s) dB_s-\a \nu \int_0^t e^{-\a(t-u)} \int_{0}^u K^H(u,r) dB_r du\\
			&=\int_0^t \Big(K^H(t,s)-\a \nu \int_s^t e^{-\a(t-u)} K^H(u,s) du  \Big)dB_s\\
			&=\int_0^t K(t,s) dB_s
		\end{align*}
		with
		$$
		K(t,s)=K_H(t,s)-\a \int_s^t e^{-\a(t-u)} K^H(u,s) du.  
		$$
		and $K_{H}$ as above (see, e.g., Section 2 in \cite{CelPac}). We will refer to $V$ as the fOU process with $V_0=0$. We use the notation $V$ to differentiate it from the stationary counterpart $Y^H$ (see Section \ref{sec:pricing}). 
		
		In the rest of this section we recall some properties of the stationary Ornstein-Uhlenbeck process $Y^H$. Let us recall Lemma 2.1 in \cite{cheridito2003}.

		\begin{lemma}\label{repLemma21}
			Let $H\in (0,\frac 12)\cup (\frac 12, 1]$, $\a>0$ and $-\infty\leq a<b\leq c<d<+\infty$. Then
			$$
			\E\Big[\int_a^b e^{\a u} dB_u^H \int_c^de^{\a v}dB_v^H\Big]=H(2H-1)\int_a^b e^{\a u}\Big(\int_c^d e^{\a v}(v-u)^{2H-2}dv\Big) du.
			$$
		\end{lemma}
		Now, let us provide an explicit expression for the autocovariance function and the variance of $Y_t^H$, where $t\in \R$. Due to the stationarity of $Y^H$, we recall that $\Var(Y_t^H) = \Var(Y_0^H)$ for all $t\in \R$. From \cite{PipirasTaqqu}, we have:
		\begin{equation}\label{PTformula}
			\Cov(Y^H_t, Y^H_{t+s}) = \nu^2\frac{\Gamma(2H+1)\sin\pi H}{2\pi}\int_{-\infty}^{+\infty} e^{\mathbf{i}sx} \frac{|x|^{1-2H}}{\alpha^2+x^2} dx.
		\end{equation}
		The variance of the process is computed from the definition. 
		\begin{lemma}\label{varfOU1}
			Let $Y^H$ be the fOU process in \eqref{fOU1}. Then
			$$
			\Var(Y_t^H)=\Var(Y^H_0)=\nu^2\frac{\Gamma(2H+1)}{2\a^{2H}}.
			$$
		\end{lemma}
		\begin{proof}
			From \eqref{repY},
			\begin{align*}
				&\Var(Y_t^H)=\E\Big[\Big(\nu\int_{-\infty}^t e^{-\a( t-u)}dB^H_u\Big)^2\Big]=\nu^2 \E\Big[\Big(B_t^{H}-\a\int_{-\infty}^t  e^{-\a(t-u ) } B_u^{H} du\Big)^2  \Big]\\
				&=\nu^2  |t|^{2H} -2\nu^2 \a \int_{-\infty}^t e^{-\a (t-u)} \E[B_t^{H} B_u^{H}] du
				+\nu^2 \a^2 e^{-2\a t} \int_{-\infty}^t \int_{-\infty}^t e^{\a(u+v)} \E[B_u^{H}B_v^{H}] dvdu\\
				&=\nu^2  |t|^{2H}-\nu^2 \a \int_{-\infty}^t e^{-\a(t-u) } (|t|^{2H}+|u|^{2H}-(t-u)^{2H})du +\\
				&+ \frac{\nu^2 \a^2 e^{-2\a t}}{2} \int_{-\infty}^t \int_{-\infty}^t e^{\a(u+v)} (|u|^{2H}+|v|^{2H}-|v-u|^{2H} ) dvdu\\
				&=\nu^2\frac{\Gamma(2H+1)}{\a^{2H}}-\frac{\nu^2 \a^2 e^{-2\a t}}{2} \int_{-\infty}^t \int_{-\infty}^t e^{\a(u+v)} |v-u|^{2H}  dvdu\\
				&=\nu^2\frac{\Gamma(2H+1)}{2\a^{2H}}.
			\end{align*}
			
		\end{proof}
		
		Next, we recall Theorem 2.3 in \cite{cheridito2003}, which provides the asymptotic behavior of the autocovariance function of $Y^H$. Additionally, we present a result regarding the regularity of the covariance function.
		
		\begin{theorem}\label{Decay1}[Theorem 2.3 in \cite{cheridito2003}]
			Let $H\in (0,\frac 12)\cup (\frac 12,1]$ and $N\in\N$. Let $Y^H$ the fOU in \eqref{fOU1}. Then for $t\in \R$ and $s\to \infty$, 
			\begin{equation}\label{cov1dim}
				\Cov(Y_t^H,Y_{t+s}^H)=\frac{1}{2}\nu^2 \sum_{n=1}^N \a^{-2n}\Big(\prod_{k=0}^{2n-1}(2H-k)   \Big)s^{2H-2n} +O(s^{2H-2N-2}).
			\end{equation}
			
		\end{theorem}
		The autocovariance decays as a power-law, particularly illustrating long-range dependence for $H\in(\frac{1}{2}, 1]$. Since Theorem \ref{Decay1} establishes that the auto-covariance of the fOU process tends to $0$, we deduce the ergodicity of the fOU process (see \cite{maruyama1949harmonic}).
		We can also deduce that, when $|t-s|\to 0$, $\Cov(Y_t^H, Y_s^H)\to \Var(Y_0^H)$, because of the stationarity of $Y^H$. Let us show that $\Cov(Y_t^H, Y_s^H)-\Var(Y_0^H)=O(|t-s|^{2H})$.
		\begin{lemma}\label{shorttime1dim}
			For $t\to s$ and $H\neq \frac 12$ 
			\begin{align*}
				&		\Cov(Y^H_{t}, Y^H_s)=\\
				&=\Var(Y^H_0)-\frac{\nu^2}{2}|t-s|^{2H}+\frac{\a^2}{2}\Var(Y^H_0)|t-s|^2 - \frac{\a^2\nu^2|t-s|^{2H+2}}{4(H+1)(1+2H)} +o(|t-s|^{4}).
			\end{align*}
		\end{lemma}
		\noindent\begin{proof}
			By stationarity, we have $\Cov(Y^H_t, Y^H_s)=\Cov(Y^H_{t-s}, Y^H_0)$, then we study $\Cov(Y^H_s, Y^H_0)$ when $s\to 0$. Lemma 2.1 in \cite{cheridito2003} ensures us that
			\begin{align*}
				\Cov(Y^H_{s}, Y^H_0)&=\nu\E\Big[\Big(Y^H_0e^{-\a s}+\nu e^{-\a s}\int_0^{s} e^{\a  u}dB_u^H\Big)\int_{-\infty}^0e^{\a v}dB_v\Big] \\
				&=e^{-\a s}\Big(\Var(Y^H_0)-\nu^2 H(1-2H) \int_0^{s} du \int_{-\infty}^0 e^{\a(u+v)} (u-v)^{2H-2} dv\Big).
			\end{align*}
			Let us compute the integral. Taking 
			$$
			\begin{cases}
				x=u+v\\
				y=u-v
			\end{cases}
			$$
			we have
			\begin{align*}	
				&\Cov(Y^H_s, Y^H_0)\\
				&=e^{-\a s}\Big(\Var(Y^H_0) -\frac{\nu^2 H(1-2H)  }2 \Big( \int_{0}^{s} dy \int_{-y}^y e^{\a  x} y^{2H-2} dx+ \int_{s}^{+\infty} dy \int_{-y}^{-y+2s }e^{\a  x} y^{2H-2} dx\Big)\\
				&= e^{-\a  s}\Big(\Var(Y^H_0)-\nu^2 H(1-2H)  \Big(\frac{1}{2\a }\int_{0}^{s} (e^{\a  y}-e^{-\a  y}) y^{2H-2}     dy+\frac{e^{2\a  s}-1}{2\a } \int_{s}^{+\infty} e^{-\a  y} y^{2H-2} dy\Big)\Big)
			\end{align*}
			The uniform convergence (on compact sets) of the Taylor series of the exponential function gives us 
			\begin{align*}
				&\frac{1}{2\a}\int_{0}^{s} (e^{\a y}-e^{-\a y}) y^{2H-2}   dy =\frac{1}{2\a}\int_{0}^{s} \sum_{n=0}^{+\infty} \frac{\a^n(1-(-1)^n)}{n!} y^{n-2+2H}dy\\
				&=\frac1{2\a}\sum_{n=0}^{+\infty} \int_0^s \frac{\a^n(1-(-1)^n)}{n!} y^{n-2+2H}dy=\frac 1{2\a}\sum_{n=1}^{+\infty}\frac{\a^n(1-(-1)^n)}{(n-1+2H)n!} s^{n-1+2H}\\
				&=\frac{s^{2H}}{2H}+\frac{\a^2}{6}\frac{s^{2H+2}}{2H+2}+o(s^{2H+2})\\
			\end{align*}
			and
			\begin{align*}
				&\frac{e^{2\a s}-1}{2\a }\int_{s}^{+\infty} e^{-\a y} y^{2H-2}   dy =\frac{e^{2\a s}-1}{2\a} \Big( \frac{e^{-\a s}s^{2H-1}}{1-2H}-\frac{\a}{1-2H}\int_{s}^{+\infty}e^{-\a y}y^{2H-1}dy\Big)\\
				&=\frac{e^{\a s}-e^{-\a s}}{2\a} \frac{s^{2H-1}}{1-2H}-\frac{e^{2\a s}-1}{2(1-2H)}\Big(\int_0^{+\infty} e^{-\a y} y^{2H-1} dy-\int_0^s e^{-\a y }y^{2H-1} dy\Big)\\
				&=\big(s+\frac{\a ^2}{6} s^3+o(s^3)\big)\frac{s^{2H-1}}{1-2H}-\frac{\a ^{1-2H}\Gamma(2H)}{1-2H}\Big(s+\a s^2+\frac 23 \a ^2s^3+o(s^3)\Big)+\\
				&+\frac{\a  s+\a ^2 s^2 +o(s^2)}{1-2H}\sum_{n=0}^{+\infty} \frac{(-1)^n \a ^n}{n!(2H+n)} s^{2H+n}\\
				&=\frac{s^{2H}}{1-2H}+\frac{\a  }{2H(1-2H)} s^{2H+1}+\frac{2H^2+H+3}{6H(1-2H)(1+2H)}\a^2 s^{2H+2}+ o(s^{2H+2})-\\&-\frac{\a ^{1-2H}\Gamma(2H)}{1-2H}\Big(s+\a s^2+\frac 23 \a ^2s^3+o(s^3)\Big).
			\end{align*}
			Then, inserting the above equations, we have
			\begin{align*}
				&\int_0^{s} du \int_{-\infty}^0 e^{\a(u+v)} (u-v)^{2H-2} dv\\&=\frac{s^{2H}}{2H(1-2H)}+\frac{\a s^{2H+1}}{2H(1-2H)}+\frac{2H^2+3H+2}{4H(1-2H)(1+2H)(1+H)}s^{2H+2}\\
				&-\frac{\a^{1-2H}\Gamma(2H)}{1-2H}\Big(s+\a s^2+\frac 23 \a^2s^3+o(s^3)\Big)+o(s^{2H+2}).
			\end{align*}
			Then
			\begin{align*}
				&\Cov(Y^H_s,Y^H_0)=e^{-\a s}\Big(\Var(Y^H_0)-\nu^2H(1-2H)\Big(\frac{1}{2H(1-2H)}s^{2H} + \frac{\a }{2H(1-2H)} s^{2H+1} \\
				&+\frac{2H^2+H+3}{6H(1-2H)(H+1)(1+2H)}\a ^2 s^{2H+2}-\frac{\a ^{1-2H}}{1-2H}\Gamma(2H)\Big(s+\a s^2+\frac 23 \a ^2s^3+o(s^3)\Big)\Big)\Big)\\
				&=		\Var(Y^H_0)-\frac{\nu^2}{2}s^{2H}+\frac{\a ^2}{2}\Var(Y^H_0)s^2 -\frac{\a ^2 \nu^2}{4(H+1)(1+2H)} s^{2H+2}+o(s^{4}).
			\end{align*}	
			since by Lemma \ref{varfOU1} we can deduce that the coefficients of $s^{1+2H}$ and $s^3$ are $0$.
		\end{proof}
		
		\section{Multivariate fractional Brownian motion}\label{MultifBM_section}
		
		The purpose of this chapter is to introduce and examine the 2fOU. As a building block, we introduce the mfBm, as defined in \cite{AC11}, \cite{multivar}.
		\subsection{Definition of multivariate fractional Brownian motion}
		Let us recall the definition of mfBm given in \cite{multivar}. 
		\begin{definition}\label{fBm2}
			Fixed $d\in\N$, the d-variate fractional Brownian motion (dfBm) $(B^{H_1}_t, \ldots, B^{H_d}_t)_{t\in \R}$ with $H_i \in (0,1)$ for $i=1,\ldots, d$, is a centered Gaussian process taking values in $\R^d$ such that:
			\begin{itemize}
				\item $B^{H_i}$, for $i=1,\ldots, n$, is a fBm with Hurst index $H_i \in (0,1)$;
				\item it is a self-similar process with parameter $(H_1, \ldots, H_d)$, i.e. 
				$$
				(B^{H_1}_{\lambda t},\ldots,  B^{H_d}_{\lambda t})= (\lambda^{H_1} B^{H_1}_t, \ldots, \lambda^{H_n} B^{H_d}_t)_{t\in \R}
				$$ 
				in the sense of finite-dimensional distribution. 
				\item the increments are stationary.
			\end{itemize}
		\end{definition}
		
		The cross-covariance functions of a $d$-fBm have a precise form that is described in \cite{lav09} (and also reported in \cite{multivar}):
		\begin{theorem}\label{cov_nfBm}
			Let $(B^{H_1}_t, \ldots, B^{H_d}_t)_{t\in \R}$, $H_i \in (0,1)$ for $i=1,\ldots, d$, be the $d$-fBm in Definition \ref{fBm2}. The cross-covariance functions have the following representations: 
			\begin{itemize}
				\item[1)] for $i\neq j$, if $H_{ij}=H_i+H_j \neq 1$, there exist $\rho_{ij}=\rho_{ji}\in [-1,1]$ and $\eta_{ij}=-\eta_{ji}\in \R$ such that $\rho_{ij}=\mathrm{Corr}(B^{H_i}_1, B^{H_j}_1)$ and 
				\begin{align}\label{cH}
					\Cov(B_t^{H_i}, B_s^{H_j})=\frac{\sigma_i \sigma_j }{2}\big((\rho_{ij}+&\sig(t)\eta_{ij})|t|^{H_{ij}}+(\rho_{ij}-\sig(s)\eta_{ij})|s|^{H_{ij}} \notag\\
					&-(\rho_{ij}-\sig(s-t)\eta_{ij})|s-t|^{H_{ij}}\big)
				\end{align}
				where $\sigma_i^2 =\Var(B^{H_i}_1),\, \sigma_j^2 =\Var(B^{H_j}_1)$ and $\sig:\R \rightarrow \{-1, 1\}$ is given by $\sig(x)=1$ for $x\geq 0$ and $\sig(x)=-1$ for $x<0$;
				\item[2)] for  $i\neq j$, if $H_i+H_j = 1$ there exist $\rho_{ij}=\rho_{ji}\in [-1,1]$ and $\eta_{ij}=-\eta_{ji}\in \R$ such that $\rho_{ij}=\mathrm{Corr}(B^{H_i}_1, B^{H_j}_1)$ and
				\begin{align}\label{cH1}
					\Cov(B_t^{H_i}, B_s^{H_j})=\frac{\sigma_i \sigma_j }{2}\big(\rho_{ij}&(|s|+|t|-|s-t|)+ \notag\\
					&+\eta_{ij}(s\log|s|-t\log|t|- (s-t)\log|s-t|)\big) 
				\end{align}
				where $\sigma_i^2 =\Var(B^{H_i}_1),\, \sigma_j^2 =\Var(B^{H_j}_1)$.
			\end{itemize}
		\end{theorem}
		
		The covariance structure of the mfBm Brownian motion is subject to numerous constraints due to its joint self-similarity property. This characteristic has been thoroughly examined in a broader context in studies such as \cite{lav09}, followed by more specific investigations in \cite{multivar} and \cite{AC11}. As demonstrated in Theorem \ref{cov_nfBm}, the covariance structure relies on $d^2$ parameters: $(\rho_{ij})_{i,j=1, i\neq j}^d \in [-1,1]$, $(\eta_{ij})_{i,j=1,i\neq j}^d\in \R$ and $(\sigma_i)_{i=1}^d>0$. Here, $\rho_{ij}$ represents the correlation between $B^{H_i}_1$ and $B^{H_j}_1$, forming a symmetric parameter ($\rho_{ij}=\rho_{ji}$). The parameter $\sigma_i$ denotes the standard deviation of $B_1^{H_i}$ while $\eta_{ij}$ is antisymmetric ($\eta_{ij}=-\eta_{ji}$) and linked to the time reversibility of the process. 
		
		\begin{remark}
			Time-reversibility amounts to temporal symmetry in the probabilistic structure of a strictly stationary time series process. 
			A process $Z_t$ is said to be time-reversible if the joint distributions of 
			$$
			(Z_t, Z_{t+\tau_1},\ldots, Z_{t+\tau_k})
			$$ 
			and 
			$$(Z_t, Z_{t-\tau_1},\ldots,Z_{t-\tau_k})$$ are equal for all $k\in \N$ and $\tau_1,\ldots, \tau_k \in \R$.

		\end{remark}

		In \cite{multivar}, the authors investigated specific parameter choices such as $(\eta_{ij})_{i,j}$ depending on $(\rho_{ij})_{i,j}$ or when $\eta_{ij}=0$ (the time reversible case). In the  general scenario, $(\eta_{ij})_{i,j}$ are unconstrained. 
		
		Moving forward, let us focus on a bivariate fractional Brownian motion (2fBm) $(B^{H_1}, B^{H_2})$, where we denote $\rho=\rho_{12}=\rho_{21}$ and $H=H_1+H_2$. Here, $H_1$ and $H_2$ denote the Hurst indexes of $B^{H_1}$ and $B^{H_2}$ respectively. Additionaly, without loss of generality, we set $\sigma_1=\sigma_2=1$.

		\begin{remark}
			
			The functions in \eqref{cH} and \eqref{cH1} serve as covariance functions if and only if certain conditions on $\rho$ and $\eta_{12}$ are satisfied, as discussed in Section 3.4 of \cite{multivar}. 
			We define the coherence functions $C_{12}=C_{21}$ and the relative constraint on the parameters as follows: for $H\neq 1$ 
			\begin{equation}\label{domain}
				C_{12}=\frac{\Gamma(H+1)^2}{\Gamma(2H_1+1)\Gamma(2H_2+1)} \frac{\rho^2 \sin^2 \big(\frac{\pi}{2}H\big)+\eta_{12}^2 \cos^2(\frac{\pi}{2}H)}{\sin \pi H_1 \sin \pi H_2}\leq 1
			\end{equation}  
			and for $H=1$
			\begin{equation}\label{domainH1}
				C_{12}=\frac{1}{\Gamma(2H_1+1)\Gamma(2H_2+1)} \frac{\rho^2 +\frac{\pi^2}{4}\eta_{12}^2 }{\sin \pi H_1 \sin \pi H_2}\leq 1.
			\end{equation}
			Proposition 9 in \cite{multivar} establishes that \eqref{cH} and \eqref{cH1} indeed function as covariance functions when $C_{12}\leq 1$. Therefore, for given $H_1, H_2$, the parameter space of $\rho$ and $\eta_{12}$  is constrained by \eqref{domain} and \eqref{domainH1}. This parameter space forms the interior of the ellipse $\frac{\rho^2}{a^2}+\frac{\eta_{12}^2}{b^2}=1$  centered at the origin, with semi-axes length given by 
			$$
			a=\sqrt{\frac{\Gamma(2H_1+1)\Gamma(2H_2+1)\sin\pi H_1 \sin \pi H_2}{\Gamma(H+1)^2 \sin^2(\frac{\pi}2 H)}}
			$$ 
			and 
			$$
			b=\sqrt{\frac{\Gamma(2H_1+1)\Gamma(2H_2+1)\sin\pi H_1 \sin \pi H_2}{\Gamma(H+1)^2 \cos^2(\frac{\pi}2 H)}}
			$$
		\end{remark}

		Let us recall that for a fixed $h\in\R$, the stationary property of the increments of the fBm ensures that the covariance of the increment process $(\overline B^{h, H_i}_t)_{t\in\R}$ with $\overline{B}^{h , H_i}_t=B^{H_i}_{t+h}-B^{H_i}_h$ is  
		\begin{align*}
			\Cov(\overline{B}^{h , H_i}_t, \overline{B}^{h , H_i}_s)=\text{Cov}(B^{H_i}_{t+h}-B^{H_i}_h, B^{H_i}_{s+h}-B^{H_i}_h)=\text{Cov}(B^{H_i}_t, B^{H_i}_s).
		\end{align*}
		
		Therefore $\overline{B}^{h, H_i}$ is a fBm with Hurst index $H_i$. This property can be extended to the mfBm. 
		\begin{lemma}\label{increments}
			Let $(B^{H_1}, B^{H_2})$ be the 2fBm defined in \ref{fBm2}. For fixed $h\in \R$, let $(\overline{B}^{h, H_1},\overline{B}^{h, H_2})$ be the process defined as
			$$
			(\overline{B}_t^{h, H_1},\overline{B}_t^{h, H_2})=(B_{t+h}^{H_1}-B_h^{H_1},B_{t+h}^{H_2}-B_h^{H_2}), \,\,\, t\geq 0.
			$$
			Then $(\overline{B}^{h, H_1},\overline{B}^{h, H_2})$ is a 2fBm as in Definition \ref{fBm2} with Hurst indexes $(H_1, H_2)$. 
		\end{lemma} 
		\begin{proof}
			From the stationarity of the increments, we have that 
			\begin{align*}
				&\E[\overline{B}^{h , H_1}_t \overline{B}^{h, H_2}_s]=\E[  (B_{t+h}^{H_1}-B_h^{H_1})(B_{s+h}^{H_2}-B^{H_2}_h)  ]=\E[B_t^{H_1} B_s^{H_2}].
			\end{align*}
			It follows that, for every $h\in \R$, the process $(\overline{B}^{h,H_1}, \overline{B}^{h,H_2})$ is a 2fBm as in Definition \ref{fBm2}.
		\end{proof}

		Let us finally recall the following theorem, which provides a moving average representation of $(B^{H_1}, B^{H_2})$.
		\begin{theorem}\label{MovAverageFBM}[Theorem 8 in \cite{multivar}]
			Let $(B^{H_1}, B^{H_2})$ be the 2fBm in Definition \ref{fBm2}. For $(H_1, H_2) \in (0,1)^2$ and $H_1,H_2\neq \frac 12$ there exists $M^+$, $M^{-}$ two $2\times 2$ real matrices such that, for $i=1,2$, 
			\begin{align}\label{movinRep}
				&B^{H_i}_t=\sum_{j=1}^2\int_{\R} M_{i,j}^+((t-x)_+^{H_i-\frac 12}-(-x)_{+}^{H_i-\frac 12})+M_{i,j}^-((t-x)_-^{H_i-\frac 12}-(-x)_{-}^{H_i-\frac 12}) W_j(dx)
			\end{align}
			where $W=(W_1, W_2)$ is a Gaussian white noise with zero mean, independent components and covariance $\E[W_i(dx)W_j(dx)]=\delta_i^j dx$. 
		\end{theorem}

		\section{The bivariate fractional Ornstein-Uhlenbeck process}\label{Bivariate_fouSection}
		Now we are ready to define the 2fOU process. Our building blocks are the univariate stationary process recalled in Section \ref{Univariate_section} and the mfBm defined in Section \ref{MultifBM_section}. 
		\begin{definition}\label{fou}
			Let $\a_1,\a_2>0$, $\nu_1,\nu_2 >0$ and $H_1, H_2 \in (0,1)\setminus \{\frac 12\}$. A bivariate fOU process  $Y=(Y^1_t, Y^2_t)_{t\in \R}$ is a centered Gaussian process such that, for all $t\in \R$, 
			\begin{equation*}
				Y_t^1=\nu_1\int_{-\infty}^t e^{-\a_1(t-s)}dB_s^{H_1}
			\end{equation*}
			and 
			\begin{equation*}
				Y_t^2= \nu_2\int_{-\infty}^t e^{-\a_2(t-s)}dB_s^{H_2}
			\end{equation*}
			where $(B^{H_1}_t, B^{H_2}_t)_{t\in \R}$ is a 2fBm defined in \ref{fBm2}.
		\end{definition}
		We start with a bivariate Gaussian noise represented by $(B^{H_1, H_2}_t)_{t\in\R}=(B_t^{H_1}, B_t^{H_2})_{t\in\R}$, and, using two univariate fOU processes, we formulate $Y_t$ as an integral based on $B^{H_1, H_2}$ with appropriate kernel matrix. We can express $Y_t=(Y_t^1, Y_t^2)$ as: 
		\begin{align*}
			Y_t=\int_0^t K(t,s)\cdot dB^{H_1, H_2}_{s},  
		\end{align*} 
		where $K(t,s)$ is  $2\times 2$ matrix, with 
		$$
		K_{ij}(t,s)=\begin{cases}\nu_i e^{-\a_i(t-s)}\quad{i=j}\\
			0 \,\,\,\,\,\,\,\,\,\,\,\,\,\,\,\, \qquad{i\neq j}\end{cases}.
		$$ 
		The bivariate process in Definition \ref{fou} is a Gaussian process. 
		
		\begin{remark}
			We can give a moving average representation for $Y$, derived from Theorem \ref{MovAverageFBM}. Let $W$ be the bidimensional white noise defined in Theorem \ref{MovAverageFBM}. We observe that, by \eqref{repY}, we have
			$$
			Y^i=\nu_i\Big(B_t^{H_i}-\a_i\int_{-\infty}^t B_u^{H_i} e^{-\a_i(t-u)} du\Big).
			$$
			Using the representation of $B^i$ in \eqref{movinRep} we have
			\begin{align}\label{movAverY}
				&Y^i=\nu_i\Big( \sum_{j=1}^2\int_{\R} M_{i,j}^+((t-s)_+^{H_i-\frac 12}-(-s)_{+}^{H_i-\frac 12})+M_{i,j}^-((t-s)_-^{H_i-\frac 12}-(-s)_{-}^{H_i-\frac 12}) W_j(ds)-\notag\\
				&-\a_i\int_{-\infty}^t e^{-\a_i(t-u)}\Big(\sum_{j=1}^2\int_{\R} M_{i,j}^+((u-r)_+^{H_i-\frac 12}-(-r)_{+}^{H_i-\frac 12})+M_{i,j}^-((u-r)_-^{H_i-\frac 12}-(-r)_{-}^{H_i-\frac 12}) W_j(dr)\Big) du\Big)\notag\\
				&=\sum_{j=1}^2\int_{\R} K^j_i(t,s) W_j(ds),
			\end{align}
			where
			\begin{align*}
				&K^j_i(t,s)= \nu_i\Big( M_{i,j}^+((t-s)_+^{H_i-\frac 12}-(-s)_{+}^{H_i-\frac 12})+M_{i,j}^-((t-s)_-^{H_i-\frac 12}-(-s)_{-}^{H_i-\frac 12}) -\\
				&-\a_i\int_{s}^t e^{-\a_i(t-u)} \Big(M_{i,j}^+((u-s)_+^{H_i-\frac 12}-(-s)_{+}^{H_i-\frac 12})+M_{i,j}^-((u-s)_-^{H_i-\frac 12}-(-s)_{-}^{H_i-\frac 12}) \Big)du.
			\end{align*}
			
			We notice that 
			\begin{align*}
				\Var(Y^i_t)=\int_{\R} K_i^1(t,s)^2 ds+\int_{\R} K_i^2(t,s)^2ds <\infty,
			\end{align*}
			then $K_i^j(t, \cdot) \in L^2(\R)$ for all $t\in\R$, and $K_i(t, \cdot)=(K_i^1(t, \cdot), K_i^2(t, \cdot)) \in L^2(\R;\R^2)$.

			The relation in \eqref{movAverY} shows that the process $Y$ is a Gaussian process. 
		\end{remark}

		\subsection{Properties of the covariance structure}\label{prop_2fou_section}
		The univariate fOU that we consider in Section \ref{Univariate_section} is a stationary process. Our process maintains this important property. The components of $Y$ are univariate fOU, therefore their own structure is stationary. We are now interested in their cross-covariance structure.  In the following lemma we prove that also the cross-covariance function only depends on the time-lag.
		\begin{lemma}\label{stat}
			For $t,s\in \R$ we have
			$$
			\Cov(Y_t^{1}, Y_{t+s}^{2} )=\Cov (Y_0^{1}, Y_s^{2}). 
			$$
			Then, the process $Y$ is a stationary Gaussian process. 
		\end{lemma}
		
		\begin{proof}
			Let us define $\overline{B}^{h , H_i}_t=B^{H_i}_{t+h}-B^{H_i}_h  $, for $i=1,2$, for all $h>0$. Lemma \ref{increments} ensures us that $(\overline{B}^{h , H_1}\overline{B}^{h , H_2})$ is a 2fBm.
			Then  
			$
			d\overline{B}^{h , H_i}_t=dB^{H_i}_{t+h},
			$ 
			and so
			\begin{align*}
				\Cov(Y_t^1, Y_{t+s}^{2} )&=\E\Big[\nu_1 \int_{-\infty}^t e^{-\a_1(t-u)} dB_u^{H_1}\, \nu_2 \int_{-\infty}^{t+s} e^{-\a_2(t+s-v)} dB_v^{H_2}\Big]\\
				&=\E\Big[\nu_1 \int_{-\infty}^t e^{-\a_1(t-u)} d\overline{B}_{u-t}^{t, H_1}\, \nu_2 \int_{-\infty}^{t+s} e^{-\a_2(t+s-v)} d\overline{B}_{v-t}^{t, H_2}\Big]\\
				&=\E\Big[\nu_1 \int_{-\infty}^0 e^{\a_1 u} d\overline{B}_{u}^{t, H_1}\, \nu_2 \int_{-\infty}^{s} e^{-\a_2(s-v)} d\overline{B}_{v}^{t, H_2}\Big]\\
				&=\E\Big[\nu_1 \int_{-\infty}^0 e^{\a_1 u} dB_{u}^{ H_1}\, \nu_2 \int_{-\infty}^{s} e^{-\a_2(s-v)} dB_{v}^{ H_2}\Big]\\
				&=\Cov(Y_0^{1}, Y_s^{2}).
			\end{align*}
		\end{proof}
		
		Now we explore the properties of $(Y^1, Y^2)$, focusing on the study of the cross-covariance functions. The auto-covariances have the properties recalled in Section \ref{Univariate_section}.

		Let us start with a technical lemma which extends Lemma \ref{repLemma21}.

		\begin{lemma}\label{1}
			Let $(B^{H_1},B^{H_2})$ be 2fBm, with $H_1, H_2 \in (0,1)$.  
			Let $\a_1,\a_2 >0$ and $-\infty\leq a <b \leq c < d<+\infty$. Then, for $i,j\in \{1,2\}, i\neq j$, 
			\begin{itemize}
				\item[1)] if $H_1+H_2\neq 1$ we have
				\begin{align}\label{2}
					\E\Big[\int_a^b  e^{\a_i u}dB_u^{H_i}   &\int_c^d  e^{\a_j v}dB_v^{H_j} \Big]\notag\\&=H(H-1)\frac {\rho+\eta_{ji}}{2}\int_c^d e^{\a_j v} \Big( \int_a^b  e^{\a_i u}(v-u)^{H-2} du \Big) dv;
				\end{align}
				
				\item[2)]
				if $H_1+H_2=1$ we have
				\begin{align}
					\E\Big[\int_a^b  e^{\a_i u}dB_u^{H_i}   &\int_c^d  e^{\a_j v}dB_v^{H_j} \Big]\notag \\ &=\frac{\eta_{ij}}{2}\int_c^d e^{\a_j v} \Big( \int_a^b  e^{\a_i u}(v-u)^{-1} du\Big) dv.
				\end{align}
				
			\end{itemize}
		\end{lemma}
		
		\begin{proof}
			By Proposition A.1. in \cite{cheridito2003} it follows that 
			$$
			\int_a^b  e^{\a_i u}dB_u^{H_i}= e^{\a_i b}B_b^{H_i}-e^{\a_i a}B_a^{H_i}-\a_i \int_a^b e^{\a_i u}B_u^{H_i}du \qquad{i=1,2}.
			$$
			
			We first consider the case $b=c=0$ and compute:
			\begin{align*}
				&\E\Big[\int_a^0  e^{\a_1 u}dB_u^{H_1}             \int_0^d  e^{\a_2 v}dB_v^{H_2} \Big] \\
				&=\E\Big[\Big( -e^{\a_1 a}B_a^{H_1}-\a_1\int_a^0 e^{\a_1u}B_u^{H_1}du\Big)\Big( e^{\a_2 d}B_d^{H_2}-\a_2\int_0^d e^{\a_2 v}B_v^{H_2}dv\Big)\Big]\\
				&=\underbrace{-e^{\a_1 a +\a_2 d} \E[ B_a^{H_1} B_d^{H_2}]}_{A}+\underbrace{\a_2e^{\a_1 a} \int_0^d  e^{\a_2 v}\E[B_a^{H_1} B_v^{H_2}] dv}_{B}\\
				&\underbrace{-\a_1e^{\a_2 d}\int_a^0  e^{\a_1 u}\E[B_u^{H_1} B_d^{H_2}] dv}_{C}+\underbrace{\a_1\a_2  \int_0^d e^{\a_2 v}  \int_a^0 e^{\a_1u} \E[B_u^{H_1}B_v^{H_2}] dudv}_{D}.
			\end{align*}
			We recall that the form of the covariance between $B^{H_1}$ and $B^{H_2}$ is different when $H_1+H_2\neq 1$ and $H_1+H_2=1$. We start with $H_1+H_2 \neq 1$: we have that
			\begin{align*}
				&A=-\left(\frac{\rho-\eta_{12}}{2}\right)e^{\a_1 a+\a_2 d}\left((-a)^{H}+d^{H}-(d-a)^{H}\right),\\
				&B= \a_2 e^{\a_1 a}\left(\frac{\rho-\eta_{12}}{2}\right)  \int_0^d e^{\a_2 v}\left( (-a)^H+v^{H}-(v-a)^{H} \right)dv, \\
				&C=-\a_1 e^{\a_2 d}\left(\frac{\rho-\eta_{12}}{2}\right) \int_a^0 e^{\a_1u} \left((-u)^{H}+d^{H}-(d-u)^{H}\right) du, \\
				&D=\a_1 \a_2\left(\frac{\rho-\eta_{12}}{2}\right) \int_0^d e^{\a_2 v}   \int_a^0 e^{\a_1u} \left( (-u)^{H}+v^{H}-(v-u)^{H}\right)du dv.
			\end{align*}
			
			In $C$ and $D$ we integrate by parts with respect to $u$ obtaining
			\begin{align*}
				&C=-\a_1 e^{\a_2 d}\left(\frac{\rho-\eta_{12}}{2}\right) \int_a^0 e^{\a_1u} [(-u)^{H}+d^{H}-(d-u)^{H}] du\\
				&=-e^{\a_2 d}\left(\frac{\rho-\eta_{12}}{2}\right)   \Big(-e^{\a_1 a}\left( (-a)^{H}+d^{H}-(d-a)^{H}\right)  +\\
				& +H  \int_a^0 e^{\a_1 u}  ((-u)^{H-1} -(d-u)^{H-1})du \Big)\\
				&=-A-H e^{\a_2 d}  \left(\frac{\rho-\eta_{12}}{2} \right) \int_a^0 e^{\a_1 u}  ((-u)^{H-1} -(d-u)^{H-1})du
			\end{align*}
			and
			\begin{align*}
				&D=\a_1 \a_2\left(\frac{ \rho-\eta_{12}}{2}\right)   \int_0^d e^{\a_2 v}   \int_a^0 e^{\a_1u} \left( (-u)^{H}+v^{H}-(v-u)^{H}\right)du dv\\
				&= \a_2\left(\frac{\rho-\eta_{12}}{2} \right)   \int_0^d e^{\a_2 v} \Big(-e^{\a_1a}\left((-a)^{H}+v^{H}-(v-a)^{H}\right)\\
				&-\int_a^0 e^{\a_1u}\left(-H (-u)^{H-1} +H(v-u)^{H-1} \right)du \Big)dv\\
				&=-B+\a_2 H\left(\frac{\rho-\eta_{12}}{2}\right)  \int_0^d e^{\a_2 v}\int_a^0 e^{\a_1u}( (-u)^{H-1} -(v-u)^{H-1} )dudv.
			\end{align*}
			It follows that
			\begin{align*}
				&A+B+C+D=-H e^{\a_2 d} \left(\frac{\rho-\eta_{12}}{2} \right) \int_a^0 e^{\a_1 u}  ((-u)^{H-1} -(d-u)^{H-1})du\\
				&+\a_2 H\left(\frac{\rho-\eta_{12}}{2}\right) \int_0^d e^{\a_2 v}\int_a^0 e^{\a_1u}( (-u)^{H-1} -(v-u)^{H-1} )dudv.
			\end{align*}
			Now, we integrate by parts with respect to $v$:
			\begin{align*}
				&A+B+C+D=-H e^{\a_2 d}\left(\frac{\rho-\eta_{12}}{2}\right)   \int_a^0 e^{\a_1 u}  ((-u)^{H-1} -(d-u)^{H-1})du\\
				&+\a_2 H\left(\frac{\rho-\eta_{12}}{2}\right)\int_0^d e^{\a_2 v}\int_a^0 e^{\a_1u}( (-u)^{H-1} -(v-u)^{H-1} )dudv\\
				&=-H e^{\a_2 d}\left(\frac{\rho-\eta_{12}}{2}\right)   \int_a^0 e^{\a_1 u}  ((-u)^{H-1} -(d-u)^{H-1})du\\
				&+He^{\a_2 d}\left(\frac{\rho-\eta_{12}}{2}\right)\int_a^0 e^{\a_1 u}  ((-u)^{H-1} -(d-u)^{H-1})du\\
				&+H(H-1)\left(\frac{\rho-\eta_{12}}{2}\right) \int_0^d e^{\a_2 v}\int_a^0 e^{\a_1u}(v-u)^{H-2}dudv\\
				&=H(H-1)\left(\frac{\rho-\eta_{12}}{2}\right) \int_0^d e^{\a_2 v}\int_a^0 e^{\a_1u}(v-u)^{H-2}dudv
			\end{align*}

			Then, for $b=c=0$, \eqref{2} holds. Now we prove the same equality for $H_1+H_2=1$. The proof is similar. The values of $A, B, C$ and $D$ change according to the change of 
			covariance structure. 
			\begin{align*}
				&A=-e^{\a_1 a+\a_2 d}\frac{\eta_{12}}{2}\left(d \log d -a \log(- a) -(d-a) \log (d-a)\right), \\
				&B= e^{\a_1 a}\a_2\frac{\eta_{12}}{2}  \int_0^d e^{\a_2 v}\left(v \log v - a\log (-a)- (v-a)\log (v-a)\right)dv, \\
				&C=-\a_1 e^{\a_2 d}\frac{\eta_{12}}{2} \int_a^0 e^{\a_1u} \left(d\log d-u\log(-u)-(d-u)\log(d-u)\right) du, \\
				&D=\a_1 \a_2 \frac{ \eta_{12}}{2}   \int_0^d e^{\a_2 v}   \int_a^0 e^{\a_1u} \left(v\log v -u\log(-u)-(v-u)\log(v-u)\right)du dv.
			\end{align*}
			Using the same technique as above, it follows that
			$$
			\E\Big[\int_a^0  e^{\a_1 u}dB_u^{H_1}             \int_0^d  e^{\a_2 v}dB_v^{H_2} \Big]=\frac{\eta_{12}}2\int_0^d e^{\a_2 v}\int_a^0 e^{\a_1 u}(v-u)^{-1} du dv.
			$$
			
			Now, we take $b=0<c$. We have that
			\begin{align*}
				&\E\Big[\int_a^0  e^{\a_1 u}dB_u^{H_1}   \int_c^d  e^{\a_1 v}dB_v^{H_2} \Big]\\
				&=\E\Big[\int_a^0  e^{\a_1 u}dB_u^{H_1}\int_0^d  e^{\a_1 v}dB_v^{H_2}-\int_a^0  e^{\a_1 u}dB_u^{H_1}\int_0^c  e^{\a_1 v}dB_v^{H_2}\Big]\\
				&=d_H \int_0^d e^{\a_2 v}  \int_a^0  e^{\a_1 u} (v-u)^{H-2} du dv-d_H \int_0^c e^{\a_2 v}  \int_a^0  e^{\a_1u}(v-u)^{H-2}du dv\\
				&=d_H \int_c^d e^{\a_2 v}  \int_a^0  e^{\a_1u} (v-u)^{H-2} du dv,
			\end{align*}
			where $d_H=H(H-1)\left(\frac{\rho-\eta_{12}}2 \right)$ where $H\neq 1$ and $d_H=\frac{\eta_{12}}{2}$ when $H=1$. We finally consider the case $b\neq 0$. 
			We notice that the process $\overline{B}^{H_i}_t=B^{H_i}_{t+b}-B^{H_i}_{b}$ is a fBm with Hurst index $H_i$, for $i=1,2$ and for all $b\in \R$. Moreover we have that:
			\begin{align*}
				&\E[\overline{B}^{H_1}_u\overline{B}^{H_2}_v]=\E\big[(B^{H_1}_{u+b}-B^{H_1}_{b})(B^{H_2}_{v+b}-B^{H_2}_{b})\big] =\E[B_u^{H_1}B_v^{H_2}]
			\end{align*}
			using the stationarity of the increments of $(B^{H_1}, B^{H_2})$.
			So $(\overline{B}^{H_1},\overline{B}^{H_2})$ is a $2$-dimensional fBm equal in law to $(B^{H_1}, B^{H_2})$, in the sense of the Definition \ref{fBm2}. So, we can work with  $\overline{B}^{H_1}$ and $\overline{B}^{H_2}$ instead of $B^{H_1}$ and $B^{H_2}$.
			If $-\infty\leq a <b \leq c < d<+\infty$, we have
			\begin{align*}
				&\E\Big[\int_a^b  e^{\a_1 u}dB_u^{H_1}   \int_c^d  e^{\a_1 v}dB_v^{H_2}\Big ]=e^{(\a_1+\a_2)b }\E\Big[\int_{a-b}^0  e^{\a_1 u}d\overline{B}_u^{H_1}   \int_{c-b}^{d-b}  e^{\a_1 v}d\overline{B}_v^{H_2} \Big]\\
				&=e^{(\a_1+\a_2)b }d_H \int_{c-b}^{d-b} e^{\a_2 v}  \int_{a-b}^0  e^{\a_1u}(v-u)^{H-2} du dv\\
				&=d_H\int_c^d e^{\a_2 v}  \int_a^b  e^{\a_1 u}(v-u)^{H-2} du dv.
			\end{align*}
			
		\end{proof}
		
		Using Lemma \ref{1}, we compute the cross-covariance functions of $(Y^1,Y^2)$. 
		\begin{lemma}\label{CrossYY}
			Let $(Y^1, Y^2)$ be the 2fOU as in Definition \ref{fou}. For $t,s \in \R$ and $i,j\in \{1,2\}, \,\,\, i\neq j$ and $H=H_1+H_2 \neq 1$ we have that
			\begin{align*}
				&\Cov(Y^i_t, Y^j_{t+s})=\\
				&= e^{-\a_j s} \Cov(Y^1_0,Y^2_0) +\nu_1\nu_2 e^{-\a_j s} H(H-1)\frac{\rho-\eta_{ij}}{2} \int_0^s e^{\a_j v} \Big( \int_{-\infty}^0 e^{\a_i u}(v-u)^{H-2} du\Big) dv,
			\end{align*}
			while for $H=H_1+H_2=1$, we have that
			\begin{align*}
				&\Cov(Y^i_t, Y^j_{t+s})=e^{-\a_j s}\Cov(Y^1_0,Y^2_0)+\nu_1\nu_2 e^{-\a_j s} \frac{\eta_{ij}}{2} \int_0^s e^{\a_j v} \Big( \int_{-\infty}^0 \frac{e^{\a_i u}}{v-u} du\Big) dv.
			\end{align*}
			
		\end{lemma}
		\noindent\begin{proof} From Lemma \ref{stat} and Lemma \ref{1}, it follows that
			\begin{align*}
				&\Cov(Y^i_t, Y^j_{t+s})=\Cov(Y^i_0, Y^j_s)=\nu_1\nu_2 \E\Big[\int_{-\infty}^0 e^{\a_i u} dB^{H_i}_u \int_{-\infty}^s e^{-\a_j(s-v)} dB_v^{H_j}\Big]\\
				&=\nu_1\nu_2 e^{-\a_j s}\Big(\E\Big[\int_{-\infty}^0  e^{\a_i u} dB^{H_i}_u \int_{-\infty}^0 e^{\a_j v} dB_v^{H_j}\Big]+\E\Big[\int_{-\infty}^0 e^{\a_i u} dB^{H_i}_u \int_{0}^s e^{\a_j v} dB_v^{H_j}\Big]\Big)\\
				&= e^{-\a_j s} \Cov(Y^1_0,Y^2_0) +\nu_1\nu_2 e^{-\a_j s} H(H-1)\frac{\rho-\eta_{ij}}{2} \int_0^s e^{\a_j v} \Big( \int_{-\infty}^0 e^{\a_i u}(v-u)^{H-2} du\Big) dv.
			\end{align*}
			When $H=1$, analogously 
			\begin{align*}
				&\Cov(Y^i_t, Y^j_{t+s})=\Cov(Y^i_0, Y^j_s)=\nu_1\nu_2 \E\Big[\int_{-\infty}^0 e^{\a_i u} dB^{H_i}_u \int_{-\infty}^s e^{-\a_j(s-v)} dB_v^{H_j}\Big]\\
				&=\nu_1\nu_2 e^{-\a_j s}\Cov(Y_0^1, Y_0^2)+\nu_1\nu_2 e^{-\a_j s}\frac{\eta_{ij}}{2}\int_c^d e^{\a_j v} \Big( \int_a^b  e^{\a_i u}(v-u)^{-1} du\Big) dv.
			\end{align*}
		\end{proof}
		
		For the cross-covariance at lag $0$, i.e. $\Cov(Y^1_t, Y^2_t)$, we have the following explicit formula.
		\begin{lemma}\label{crss0}
			For $t\in \R$
			\begin{enumerate}
				\item[1)] for $H_1+H_2\neq 1$,   
				$$
				\Cov(Y_t^1,Y^2_t)=\frac{\Gamma(H+1)\nu_1\nu_2}{2(\a_1+\a_2)}\Big( (\a_1^{1-H}+\a_2^{1-H})\rho+(\a_2^{1-H}-\a_1^{1-H})\eta_{12}     \Big);
				$$ 
				\item[2)] for $H_1+H_2=1$,
				$$
				\Cov(Y_t^1,Y^2_t)=\frac{\nu_1 \nu_2 }{\a_1+\a_2}\Big( \rho + \frac{\eta_{12}}{2}(\log \a_2-\log \a_1)\Big).
				$$
			\end{enumerate}
			The correlation $\Corr(Y_t^1, Y_t^2)$ is given by
			\begin{enumerate}
				\item[1)] for $H_1+H_2\neq 1$,   
				\begin{align*}
					\Corr(Y_t^1 &,Y^2_t)=\\&\frac{\Gamma(H+1)}{\sqrt{\Gamma(2H_1+1)\Gamma(2H_2+1)}}\Big(\frac{\a_1^{H_1}\a_2^{H_2}}{\a_1+\a_2} \Big) \Big( (\a_1^{1-H}+\a_2^{1-H})\rho+(\a_2^{1-H}-\a_1^{1-H})\eta_{12}     \Big);
				\end{align*} 
				\item[2)] for $H_1+H_2=1$,
				\begin{align*}
					\Corr(Y_t^1 &,Y^2_t)=\\&\frac{1}{\sqrt{\Gamma(2H_1+1)\Gamma(2H_2+1)}}\Big(\frac{2\a_1^{H_1}\a_2^{H_2}}{\a_1+\a_2} \Big) \Big( \rho + \frac{\eta_{12}}{2}(\log \a_2-\log \a_1)\Big).
				\end{align*} 
			\end{enumerate}
		\end{lemma}
		
		\noindent\begin{proof}
			
			Lemma \ref{stat} guarantees that $\Cov(Y_t^1, Y_t^2)=\Cov(Y_0^1,Y_0^2)$. Then 
			\begin{align}\label{comp1}
				&\Cov(Y^1_0,Y^2_0)=\nu_1\nu_2 \E\Big[\int_{-\infty}^0 e^{\a_1 u }dB^{H_1}_v \int_{-\infty}^0 e^{\a_2 u }dB^{H_2}_u  \Big]\notag\\
				&=\nu_1\nu_2\a_1\a_2 \E\Big[\int_{-\infty}^0 e^{\a_1 u }B^{H_1}_u du \int_{-\infty}^0 e^{\a_2 u }B^{H_2}_v  dv\Big]\qquad{\text{by }\eqref{repY}}\notag \\
				&=\nu_1\nu_2\a_1\a_2\int_{-\infty}^0 \int_{-\infty}^0 e^{\a_1 u+\a_2 v}\E[B^{H_1}_u B^{H_2}_v] dudv\notag \\
				&=\frac{\nu_1\nu_2\a_1\a_2}{2} \int_{-\infty}^0 \int_{-\infty}^0  \Big( (\rho-\eta_{12})(-u)^H+\notag\\&+(\rho+\eta_{12})(-v)^H -(\rho-\eta_{12}\text{sign}(v-u))|v-u|^H\Big)e^{\a_1u+\a_2v}dudv\notag \\
				&=\frac{\nu_1\nu_2\a_1\a_2}{2}(A+B-C),
			\end{align}
			where 
			\begin{align*}
				A&=(\rho-\eta_{12}) \int_{-\infty}^0 \int_{-\infty}^0   (-u)^H e^{\a_1u+\a_2v} dv du =\frac{\rho-\eta_{12}}{\a_2} \int_{-\infty}^0 e^{\a_1 u} (-u)^H du  \\
				&= \frac{\Gamma(H+1)(\rho-\eta_{12})}{\a_2 \a_1^{H+1} },
			\end{align*}
			\begin{align*}
				B&=(\rho+\eta_{12}) \int_{-\infty}^0 \int_{-\infty}^0   (-v)^H e^{\a_1u+\a_2v} du dv =\frac{\rho+\eta_{12}}{\a_1} \int_{-\infty}^0 e^{\a_2 v} (-v)^H dv  \\
				&= \frac{\Gamma(H+1)(\rho+\eta_{12})}{\a_1 \a_2^{H+1} },
			\end{align*}
			and
			\begin{align*}
				C&=\int_{-\infty}^0 \int_{-\infty}^0 (\rho-\eta_{12}\text{sign}(v-u))|v-u|^H e^{\a_1u+\a_2v} dv du.
			\end{align*}
			Developing $C$, we have 
			\begin{align*}
				C=(\rho+\eta_{12})& \underset{C_1}{\underbrace{  \int_{-\infty}^0 \int_{-\infty}^u (u-v)^{H} e^{\a_1 u+\a_2 v} dv du}}\\
				&+(\rho-\eta_{12})\underset{C_2}{\underbrace{\int_{-\infty}^0 \int_{u}^0 (v-u)^{H} e^{\a_1 u+\a_2 v} dv du }}.
			\end{align*}
			To compute $C_1$, we use the following change of variables:
			$$
			\begin{cases}
				x=\a_1 u+\a_2 v\\
				y=u-v
			\end{cases}
			\\
			\begin{cases}
				u=\frac{x+\a_2 y}{\a_1+\a_2}\\
				v=\frac{x-\a_1 y}{\a_1+\a_2}
			\end{cases}
			$$
			The determinant of the Jacobian matrix is $\frac{1}{\a_1+\a_2}$. We consider now the domain of integration. Being $v\leq u\leq 0$, then $y=u-v\geq 0$. Moreover
			$x-\a_1 y \leq  x+\a_2 y \leq 0$. It follows that $x\leq -\a_2 y$. Then
			\begin{align*}
				C_1&=\frac{1}{\a_1+\a_2}\int_{0}^{+\infty} y^H \int_{-\infty}^{-\a_2 y} e^x  dx \,dy=\frac{1}{\a_1+\a_2} \int_{0}^{+\infty}  e^{-\a_2 y} y^H  dy=\frac{\Gamma(H+1)}{(\a_1+\a_2) \a_2^{H+1}}.
			\end{align*}
			Using the same change of variables for $C_2$, we obtain
			\begin{align*}
				C_2&=\frac{1}{\a_1+\a_2}\int_{-\infty}^0  (-y)^H \int_{-\infty}^{\a_1 y} e^x dx\,dy=\frac{1}{\a_1+\a_2} \int_{-\infty}^0 e^{\a_1 y} (-y)^H dy=\frac{\Gamma(H+1)}{(\a_1+\a_2)\a_1^{H+1} }.
			\end{align*}
			Then, substituting $C_1$ and $C_2$ in $C$ and finally in \eqref{comp1} we obtain 
			\begin{align*}
				\Cov(Y^1_0,Y^2_0)&=\frac{\Gamma(H+1)\nu_1\nu_2}{2(\a_1+\a_2)}\Big( \frac{\rho+\eta_{12}}{\a_2^{H-1}} + \frac{\rho-\eta_{12}}{\a_1^{H-1}}     \Big).
			\end{align*}
			
			Now we compute $\Cov(Y_0^1, Y_0^2)$ when $H_1+H_2=1$. From Definition \ref{fBm2}, we have
			\begin{align*}
				&\Cov(Y_0^1, Y_0^2)\\&=\frac{\nu_1 \nu_2 \a_1 \a_2}{2}\int_{-\infty}^0  du \int_{-\infty}^0 e^{\a_1 u+\a_2 v}\times \\
				&\times\Big(\rho \big(-u-v-|u-v|\big)+\eta_{12}\big(v\log(-v)-u\log(-u)-(v-u)\log|v-u| \big)\Big)dv\\
				&=\frac{\nu_1 \nu_2 \a_1 \a_2}{2} \Big( 2\rho \frac{1}{\a_1\a_2(\a_1+\a_2)}+\eta_{12}\frac{\log \a_2 -\log \a_1}{\a_1\a_2(\a_1+\a_2)} \Big)\\
				&=\frac{\nu_1 \nu_2 }{\a_1+\a_2}\Big( \rho + \frac{\eta_{12}}{2}(\log \a_2-\log \a_1)\Big).
			\end{align*}

		\end{proof}
		
		We summarize in the next theorem all the results about the 2fOU process.  
		\begin{theorem}\label{StatC}
			The process $Y=(Y^1, Y^2)$ is strongly stationary with 
			$$
			\E[Y_t]=(0,0)
			$$
			and covariance function $\Gamma_{ij}(t,s)=\Cov(Y_t^i, Y_s^j)$, for $i,j\in\{1,2\}$ and $t>s$, given as follows:
			\begin{itemize}
				\item when $i=j$
				\begin{align*}
					\Gamma_{ii}(t,s)=\nu_i^2 \frac{\Gamma(2H_i+1)\sin \pi H_i}{2\pi} \int_{-\infty}^{\infty} e^{i(t-s)x} \frac{|x|^{1-2H_i}}{\a_i^2+x^2}dx;
				\end{align*} 
				\item when $i\neq j$ and $H=H_1+H_2\neq 1$ 
				\begin{align*}
					\Gamma_{ij}(t,s)&=e^{-\a_i(t-s)}\Cov(Y_0^1, Y_0^2)+\nu_1 \nu_2 e^{-\a_i (t-s)}H(H-1)\frac{\rho+\eta_{ij}}{2} I_{ij}(t-s)  
				\end{align*}
				with $I_{ij}(t-s)=\int_0^{t-s} e^{\a_i u} \Big(\int_{-\infty}^{0} e^{\a_j v} (u-v)^{H-2} dv\Big) du$
				\item when $i\neq j$ and $H=H_1+H_2=1$, 
				\begin{align*}
					\Gamma_{ij}(t,s)=e^{-\a_i (t-s)}\Cov(Y^1_0,Y^2_0)-\nu_1\nu_2 e^{-\a_i(t-s)} \frac{\eta_{ij}}{2}I_{ij}(t-s)
				\end{align*}
				with $I_{ij}(t-s)=\int_0^{t-s} e^{\a_i u} \Big(\int_{-\infty}^{0} \frac{e^{\a_j v}}{u-v} dv\Big) du$. 
			\end{itemize}
			
			The analogous correlation functions are given by
			
			\begin{itemize}
				\item when $i=j$
				\begin{align*}
					\Corr(Y_t^i, Y_s^i)= \frac{\a_i^{2H_i}\sin \pi H_i}{2\pi} \int_{-\infty}^{\infty} e^{i(t-s)x} \frac{|x|^{1-2H_i}}{\a_i^2+x^2}dx;
				\end{align*} 
				\item when $i\neq j$ and $H=H_1+H_2\neq 1$, 
				\begin{align*}
					\Corr(Y_t^i,& Y_s^j)\\&=e^{-\a_i(t-s)}\Corr(Y_0^1, Y_0^2)+\frac{H(H-1)\a_1^{H_1} \a_2^{H_2} e^{-\a_i (t-s)}}{\sqrt{\Gamma(2H_1+1)\Gamma(2H_2+1)}}(\rho+\eta_{ij}) I_{ij}(t-s)  
				\end{align*}
				with $I_{ij}(t-s)=\int_0^{t-s} e^{\a_i u} \Big(\int_{-\infty}^{0} e^{\a_j v} (u-v)^{H-2} dv\Big) du$; 
				\item when $i\neq j$ and $H=H_1+H_2=1$, 
				\begin{align*}
					\Corr(Y_t^i,& Y_s^j)=e^{-\a_i (t-s)}\Corr(Y^1_0,Y^2_0)- \frac{\a_1^{H_1}\a_2^{H_2} e^{-\a_i(t-s)}}{\sqrt{\Gamma(2H_1+1)\Gamma(2H_2+1)}} \eta_{ij} I_{ij}(t-s)
				\end{align*}
				with $I_{ij}(t-s)=\int_0^{t-s} e^{\a_i u} \Big(\int_{-\infty}^{0} \frac{e^{\a_j v}}{u-v} dv\Big) du$.
			\end{itemize}
		\end{theorem}
		
		\begin{proof}
			The statement easily follows from Lemma \eqref{PTformula} and Lemma \ref{CrossYY}. We notice that the expectation of the process is constant in time, and the covariance function $\Gamma(t,s)$ depends only on $t-s$. These properties guarantee the weak stationarity of the process. Being the process Gaussian, then it is strongly stationary. 
		\end{proof}

		\subsection{Asymptotic behaviour of the cross-covariance}
		In this subsection we analyze both the decay of the cross-covariance $\Cov(Y^1_t, Y^2_{t+s})$ when $s\to +\infty$ and the behaviour when $s\to 0$.
		
		We recall that the decay of the autocovariance function of the fOU is proved and analyzed in Theorem 2.3 in \cite{cheridito2003} and reported in Theorem \ref{Decay1}. We extend this result to the cross-covariances. This result will be very useful in Section \ref{Correlation_section}.
		
		\begin{theorem}\label{cov}
			Let $H_1, H_2 \in (0,\frac{1}{2})\cup (\frac{1}{2},1]$, $H=H_1+H_1\neq 1$ and $N\in \N$. Then for fixed $t\in \R$, as $s\to \infty$ 
			\begin{align}\label{decCr}
				&\Cov(Y_t^i,Y_{t+s}^j)=\notag \\
				&=\frac{\nu_1\nu_2(\rho+\eta_{ji})}{2(\a_1+\a_2)}\sum_{n=0}^N \Big(\frac{(-1)^n}{\a_j^{n+1}}+\frac{1}{\a_i^{n+1}}\Big)  \Big ( \prod_{k=0}^{n+1} (H-k) \Big)s^{H-2-n}+O(s^{H-N-3}).
			\end{align}
			
			When $H=1$ we have
			\begin{align}\label{decCr1}
				&\Cov(Y_t^i,Y_{t+s}^j)=\notag \\
				&=\frac{\nu_1\nu_2\eta_{ji}}{2\a_1\a_2}\frac{1}{s}+\frac{\nu_1\nu_2\eta_{ji}}{2(\a_1+\a_2)}\sum_{n=1}^N \Big(\frac{(-1)^n}{\a_j^{n+1}}+\frac{1}{\a_i^{n+1}}\Big)  \Big( \prod_{k=0}^{n-1} (-k-1) \Big)s^{-1-n}+O(s^{-N-2}).
			\end{align}
		\end{theorem}
		\begin{remark}
			
			We recall that, if we use the one-dimensional model in \cite{cheridito2003}, that is
			$$
			Y_t^H= \nu \int_{-\infty}^t e^{-\a(t-u)}dB_u^H+m,
			$$
			we have 
			\begin{equation}\label{1dim_asym}
				\Cov(Y_t^H,Y_{t+s}^H)=\frac{1}{2}\nu^2 \sum_{n=1}^N \a^{-2n}\Big(\prod_{k=0}^{2n-1}(2H-k)   \Big)s^{2H-2n} +O(s^{2H-2N-2})
			\end{equation}
			(see \eqref{cov1dim}). We notice that, if $\a_1=\a_2=\a$, $H=H_1+H_2$ and $H_1=H_2$, $\nu_1=\nu_2$, $\rho=1$ and $\eta=0$, the form of $\text{Cov}(Y_t^{H_1},Y_{t+s}^{H_2})$ in Theorem \eqref{cov1dim} can be recovered from Theorem \ref{cov}.
			
		\end{remark}
		\noindent\begin{proof}[Proof of Theorem \ref{cov}]
			By Lemma \ref{1}
			\begin{align*}
				&\Cov(Y_t^{1},Y_{t+s}^{2})=\Cov(Y_0^{1},Y_{s}^{2})=\E\Big[\nu_1\int_{-\infty}^0 e^{\a_1 u}dB_u^{H_1}\nu_2\int_{-\infty}^s e^{-\a_2(s-v)}dB_v^{H_2}\Big]\\
				&=\nu_1 \nu_2 e^{-\a_2 s}\E\Big[ \int_{-\infty}^0 e^{\a_1 u}dB_u^{H_1}\int_{-\infty}^{\frac{1}{\a_2}} e^{\a_2 v}dB_v^{H_2}+\int_{-\infty}^0 e^{\a_1 u}dB_u^{H_1}\int_{\frac{1}{\a_2}}^s e^{\a_2 v}dB_v^{H_2}\Big]\\
				&=\nu_1\nu_2 e^{-\a_2 s}\Big(\E\Big[ \int_{-\infty}^0 e^{\a_1 u}dB_u^{H_1}\int_{-\infty}^{\frac{1}{\a_2}} e^{\a_2 v}dB_v^{H_2}\Big]+d_H\int_{\frac{1}{\a_2}}^s e^{\a_2 v} \int_{-\infty}^0  e^{\a_1 u}(v-u)^{H-2} du dv\Big).
			\end{align*}
			The constant $d_H$ is equal to $H(H-1)\frac{\rho-\eta_{12}}{2}$ when $H\neq 1$ and $d_H=\frac{\eta_{12}}{2}$ when $H=1$.
			We notice that $\E[\nu_1 \int_{-\infty}^0 e^{\a_1 u}dB_u^{H_1}\nu_2\int_{-\infty}^{\frac{1}{\a_2}} e^{\a_2 v}dB_v^{H_2}]$ does not depend on $s$. We have 
			$$
			\Cov(Y_t^{1},Y_{t+s}^{2})=O(e^{-\a_2 s})+e^{-\a_2 s}\nu_1 \nu_2 d_H \int_{-\infty}^0  e^{\a_1 u}\int_{\frac{1}{\a_2}}^s e^{\a_2 v} (v-u)^{H-2} dv du.
			$$
			
			By employing the change of variables $y = v - u$ and $z = v + u$, we derive
			\begin{align*}
				&\Cov(Y_t^{1},Y_{t+s}^{2})=\frac{\nu_1 \nu_2 d_H}{2} e^{-\a_2 s}\Big( \int_{\frac{1}{\a_2}}^s y^{H-2} e^{\frac{\a_2-\a_1}{2}y}  \int_{\frac{2}{\a_2}-y}^{y} e^{\frac{\a_2+\a_1}{2}z} dzdy+\\
				&+\int_s^{+\infty} y^{H-2} e^{\frac{\a_2-\a_1}{2}y} \int_{\frac{2}{\a_2}-y}^{2s-y} e^{\frac{\a_2+\a_1}{2}z} dzdy \Big)+O(e^{-\a_2 s})\\
				&=\frac{\nu_1 \nu_2 d_H}{(\a_2+\a_1)}e^{-\a_2 s}\Big( \int_{\frac{1}{\a_2}}^s y^{H-2}e^{\a_2 y}dy +e^{(\a_2+\a_1)s}\int_s^{+\infty} y^{H-2}e^{-\a_1 y}dy\\
				& -e^{\frac{\a_2+\a_1}{\a_2}}\int_{\frac{1}{\a_2}}^{+\infty} y^{H-2} e^{-\a_1 y}dy \Big)+O(e^{-\a_2 s}).
			\end{align*}
			We notice that $\frac{\nu_1 \nu_2 d_H}{(\a_2+\a_1)}e^{-\a_2 s}e^{\frac{\a_2+\a_1}{\a_2}}\int_{\frac{1}{\a_2}}^{+\infty} y^{H-2} e^{-\a_1 y}dy\in O(e^{-\a_2 s})$. Then
			\begin{align*}
				&\Cov(Y_t^{1},Y_{t+s}^{2})\\
				&=\frac{\nu_1 \nu_2 d_H}{(\a_2+\a_1)}\Big( \frac{1}{\a_2^{H-1}}e^{-\a_2 s}\int_1^{\a_2 s} y^{H-2}e^y dy+\frac{1}{\a_1^{H-1}} e^{\a_1 s}\int_{\a_1 s}^{+\infty} y^{H-2}e^{-y}dy\Big)+O(e^{-\a_2 s})\\
			\end{align*}
			and by Lemma 2.2 in \cite{cheridito2003}, we have
			\begin{align*}
				&\Cov(Y_t^{1},Y_{t+s}^{2})=\frac{\nu_1 \nu_2 d_H}{(\a_2+\a_1)} \Big( \frac{1}{\a_2}s^{H-2}+\sum_{n=1}^N \frac{(-1)^n}{\a_2^{n+1}}  \big[ \prod_{k=0}^{n-1} (H-2-k) \big]s^{H-2-n}\\
				&+\frac{1}{\a_1}s^{H-2}+\sum_{n=1}^N \frac{1}{\a_1^{n+1}}  \big( \prod_{k=0}^{n-1} (H-2-k) \big)s^{H-2-n}\Big)+O(s^{H-N-3})\\
				&=\frac{\nu_1 \nu_2 d_H}{(\a_2+\a_1)} \Big( \frac{\a_1+\a_2}{\a_2\a_1}s^{H-2}+\sum_{n=1}^N \Big(\frac{(-1)^n}{\a_2^{n+1}}+\frac{1}{\a_1^{n+1}}\Big)  \big(\prod_{k=0}^{n-1} (H-2-k) \big)s^{H-2-n}\Big)+O(s^{H-N-3})\\
				&=\frac{\nu_1\nu_2(\rho-\eta_{12})}{(\a_1+\a_2)}\sum_{n=0}^N \Big(\frac{(-1)^n}{\a_2^{n+1}}+\frac{1}{\a_1^{n+1}}\Big)  \big( \prod_{k=0}^{n+1} (H-k) \big)s^{H-2-n}+O(s^{H-N-3}).
			\end{align*}
			Changing the role of $Y^1$ and $Y^2$ we obtain the general form of the cross-covariance in \eqref{decCr}. For $H=1$, analogous computations lead to \eqref{decCr1}.
		\end{proof}
	
 Theorem \ref{cov} allows to prove the cross-ergodicity of the 2fOUs process. 
		
		\begin{lemma}\label{cross-ergod}
		The bivariate process $Y$ is cross-covariance ergodic, i.e. for all $\tau >0$, $i\neq j\in \{1,2\}$ then
		$$
		\hat r_{ij}(\tau):=\frac{1}{2T}\int_{-T}^T (Y_{t+\tau}^i-\E[Y_{t+\tau}^i])(Y_t^j-\E[Y_t^j])dt\to \E[Y_{\tau}^iY_0^j]
		$$
		in probability.
		\end{lemma}
       \begin{proof}
       Let us consider $\tau>0$ and 
       $$
       \hat r_{ij}(\tau)=\frac 1{2T} \int_{-T}^T (Y_{t+\tau}^i-\E[Y_{t+\tau}^i])(Y_{t}^j-\E[Y_t^j]) dt
       $$ 
       and we prove that $\hat r_{ij}^T(\tau)\underset{T\to+\infty}\to \E[Y_{\tau}^i Y_0^j]$ in probability. We have that $\E[\hat r_{ij}^T(\tau)]=\E[Y^i_{\tau}Y^j_0]$. We study the variance.
       \begin{align*}
       	\Var(\hat r_{ij}^T)&=\frac 1{4T^2}\int_{-T}^T \int_{-T}^T\Big( \Cov(Y_t^i, Y_s^i)\Cov(Y_t^j, Y_s^j)+\Cov(Y_{t+u}^i, Y_s^j)\Cov(Y_t^j, Y_{s+u}^i)	\Big)dt ds\\
       	&=\frac 1T \int_{-2T}^{2T} \Big(1-\frac{|t|}{2T}\Big)\Big(\Cov(Y_t^i, Y_0^i)\Cov(Y_t^j, Y_0^j)+\Cov(Y_{t+\tau}^i,Y_0^j)\Cov(Y_{t-\tau}^i, Y_0^j)\Big)dt. 
       \end{align*}
       By Theorem \ref{Decay1} and Theorem \ref{cov}, when $H<\frac 32$ the integrand is in $L^1(\R)$, then $\Var(\hat r_{ij}^T)\underset{T\to +\infty}{\to} 0$. If $H=\frac 32$, the integrand is $O(\frac 1t)$ when $t\to \infty$, then the integral is $O(\log t)$, whereas for $H>\frac 32$ the integral is $O(T^{2H-3})$, then $\Var(\hat r_{ij}^T)=O(T^{2H-4})$. In each case $\Var(\hat r_{ij}^T)\underset{T\to +\infty}{\to} 0$. 
       \end{proof}
   
		Next, we investigate the short-time asymptotic behavior of the cross-covariance functions. We begin by introducing the following technical lemma.
		
		\begin{lemma}\label{Short_time_integral}
			Let $i,j=1,2$, $i\neq j$, $H\in(0,2)$. Then, when $s\to 0$ and $H\neq 1$
			
			\begin{align*}
				&\int_0^s dv\int_{-\infty}^0 e^{\a_i u +\a_j v}(v-u)^{H-2} du\\
				&=\frac{s^H}{H(1-H)} - \frac{\a_i^{1-H}\Gamma(H)}{1-H}s+\frac{(H\a_j+\a_i)}{H(1-H)(H+1)}s^{H+1}-\frac{(\a_i+\a_j)\a_i^{1-H}}{2(1-H)}\Gamma(H) s^2\\
				&+\Big(\frac{\a_j^2}{6(1-H)} -\frac{(H-3)\a_i\a_j}{6H(1-H)}+ \frac{(H^2-2H+3)\a_i^2}{6H(1-H)(1+H)}\Big)s^{H-2}+\frac{\a_j^2-\a_i\a_j+\a_i^2}{6(H+2)}s^{H+2}+o(s^{H+2}).
			\end{align*}
			while when $H=1$
			$$
			\int_0^s  dv\int_{-\infty}^0 e^{\a_1 u+\a_j v} (v-u)^{-1} du=-s\log s-\frac{\a_2-\a_1}{2}s^2\log s +o(s^2\log s).
			$$
		\end{lemma}
		\noindent\begin{proof}
			Without losing in generality, we consider 
			$$
			\int_0^s e^{\a_2 v}  \int_{-\infty}^0 e^{\a_1 u}(v-u)^{H-2} du dv= \int_0^s dv\int_{-\infty}^0 e^{\a_1 u +\a_2 v}(v-u)^{H-2} du.
			$$
			We use the following change of variables
			\[
			\begin{cases}
				\a_1 u+\a_2 v=x\\
				v-u =y
				
			\end{cases}
			\]
			The domain of $y$ is $[0,+\infty)$ while the domain of $x$ depends on $y$: if $y\in [0,s)$, we have that $x\in (-\a_1 y, \a_2 y)$ and if $y>s$, $x\in (-\a_1 y, -\a_1 y+ (\a_1+\a_2)s )$. The Jacobian determinant is $\frac{1}{\a_1+\a_2}$.
			Then
			\begin{align*}
				&\int_0^s dv\int_{-\infty}^0 e^{\a_1 u +\a_2 v}(v-u)^{H-2} du\\
				&=\frac{1}{\a_1+\a_2}\int_0^s dy \int_{-\a_1 y}^{\a_2 y} e^x y^{H-2} dx+\frac{1}{\a_1+\a_2}\int_s^{+\infty}dy \int_{-\a_1 y}^{-\a_1 y+(\a_1+\a_2)s} e^x y^{H-2} dx\\
				&=\frac{1}{\a_1+\a_2} \int_0^s y^{H-2}(e^{\a_2 y}-e^{-\a_1 y})  dy +\frac{e^{(\a_1+\a_2 )s}-1}{\a_1+\a_2} \int_s^{+\infty} y^{H-2} e^{-\a_1 y} dy.
			\end{align*}
			Now we easily obtain that 
			\begin{align*}
				&\frac{1}{\a_1+\a_2} \int_0^s y^{H-2}(e^{\a_2 y}-e^{-\a_1 y})  dy=\frac{1}{\a_1+\a_2} \sum_{k=1}^{+\infty} \frac{\a_2^k-(-1)^k \a_1^k}{k! (H-1+k)} s^{H-1+k}.
			\end{align*}
			Now we compute the second integral when $H\neq 1$:
			
			\begin{align*}
				&\frac{e^{(\a_1+\a_2 )s}-1}{\a_1+\a_2} \int_s^{+\infty} y^{H-2} e^{-\a_1 y} dy\\
				&=\frac{e^{(\a_1+\a_2)s}-1}{\a_1+\a_2}\Big(-\frac{s^{H-1} e^{-\a_1 s}}{H-1}+\frac{\a_1}{H-1}\int_s^{+\infty} y^{H-1}e^{-\a_1 y} dy\Big)\\
				&=\frac{e^{(\a_1+\a_2)s}-1}{\a_1+\a_2}\Big(\frac{s^{H-1} e^{-\a_1 s}}{1-H}+\frac{\a_1^{1-H}}{H-1}\Gamma(H)-\frac{\a_1}{H-1}\sum_{k=0}^{+\infty} \frac{(-1)^k \a_1^k}{k!(H+k)}s^{H+k}\Big)\\
				&=\frac{e^{\a_2 s}-e^{-\a_1 s}}{\a_1+\a_2} \frac{s^{H-1}}{1-H}+\frac{e^{(\a_1+\a_2)s}-1}{\a_1+\a_2} \frac{\a_1^{1-H}}{H-1}\Gamma(H)-\frac{e^{(\a_1+\a_2)s}-1}{\a_1+\a_2} \frac{\a_1}{H-1}\sum_{k=0}^{+\infty} \frac{(-1)^k \a_1^k}{k!(H+k)}s^{H+k}\\
				&=\sum_{k=1}^{+\infty} \frac{(\a_1+\a_2)^{k-1}}{k!} s^{k} \frac{\a_1^{1-H}}{H-1}\Gamma(H)+\frac{s^H}{1-H}+s^{H+1}\Big(\frac{\a_2}{2(1-H)}+\frac{(2-H)\a_1}{2H(1-H)}\Big)+\\
				&+s^{H+2}\Big(\frac{\a_2^2}{6(1-H)} -\frac{\a_1\a_2(H-3)}{6H(1-H)}+\frac{\a_1^2(H^2-2H+3)}{6H(1-H)(1+H)}\Big)+o(s^{H+2})\\
			\end{align*}
			Then, combining the above computations, we have, for $H\neq 1$
			\begin{align*}
				&\int_0^s dv\int_{-\infty}^0 e^{\a_1 u +\a_2 v}(v-u)^{H-2} du\\
				&=\frac{s^H}{H(1-H)} - \frac{\a_1^{1-H}\Gamma(H)}{1-H}s+\frac{(H\a_2+\a_1)}{H(1-H)(H+1)}s^{H+1}-\frac{(\a_1+\a_2)\a_1^{1-H}}{2(1-H)}\Gamma(H) s^2\\
				&+\Big(\frac{\a_2^2}{6(1-H)} -\frac{(H-3)\a_1\a_2}{6H(1-H)}+ \frac{(H^2-2H+3)\a_1^2}{6H(1-H)(1+H)}\Big)s^{H-2}+\frac{\a_2^2-\a_1\a_2+\a_1^2}{6(H+2)}s^{H+2}+o(s^{H+2}).
			\end{align*}
			When $H\neq 1$ we approach the computation of the second integral in a different manner. We have
			\begin{align*}
				&\frac{e^{(\a_1+\a_2 )s}-1}{\a_1+\a_2} \int_s^{+\infty} y^{H-2} e^{-\a_1 y} dy\\
				&=\sum_{k=1}^{+\infty} \frac{(\a_1+\a_2)^{k-1}s^k}{k!}  \Big[-e^{-\a_1 s}\log s+\a_1\int_{s}^{+\infty} e^{-\a_1 y}\log y dy \Big]\\
				&=\sum_{k=1}^{+\infty} \frac{(\a_1+\a_2)^{k-1}s^k}{k!}  \Big[-e^{-\a_1 s}\log s+\a_1\int_0^\infty e^{-\a_1 y} \log y dy -\a_1 \sum_{k=1}^{\infty}\frac{(-1)^k\a_1^k}{(k+1)!}s^{k+1}(\log s-1) \Big]\\
				&=-\sum_{k=1}^{+\infty}\sum_{h=0}^{\infty} \frac{(-1)^h \a_1^h (\a_1+\a_2)^{k-1}s^{k+h}}{k!h!} \log s +\a_1\sum_{k=1}^{+\infty} \frac{(\a_1+\a_2)^{k-1}s^k}{k!} \int_0^\infty e^{-\a_1 y} \log y dy \\
				&- \sum_{k,h=1}^{+\infty} \frac{(-1)^h\a_1^{h+1}(\a_1+\a_2)^{k-1}s^{k+h+1}}{k!(h+1)!}(\log s-1) \\
				&=-s\log s-\frac{\a_2-\a_1}{2} s^2\log s +o(s^2\log s). 
			\end{align*}
			The first integral is $o(s^2\log s )$, then 
			$$
			\int_0^s  dv\int_{-\infty}^0 e^{\a_1 u+\a_j v} (v-u)^{-1} du=-s\log s-\frac{\a_2-\a_1}{2}s^2\log s +o(s^2\log s).
			$$
		\end{proof}
		\begin{lemma}\label{shorttime}
			For all $t\in \R$ and $s\to 0$, when $H\neq 1$ and $i\neq j$, we have that
			
			\begin{align*}
				&\Cov(Y_t^i, Y_{t+s}^j)\\
				&=\Cov(Y_0^1, Y_0^2)-\nu_1 \nu_2 \frac{\rho-\eta_{ij}}{2}s^H +\Big(-\a_j \Cov(Y_0^1, Y_0^2)+\a_i^{1-H} \Gamma(H+1) \nu_1\nu_2\frac{\rho-\eta_{ij}}{2}\Big) s+ \\
				&+\frac{(\a_j-\a_i)\nu_1 \nu_2 }{H+1}\frac{\rho-\eta_{ij}}{2} s^{1+H}+\Big(\frac{\a_j^2}{2} \Cov(Y_0^1, Y_0^2)-\frac 12\nu_1\nu_2 \frac{\rho-\eta_{ij}}{2} \Gamma(H+1)(\a_j\a_i^{1-H} -\a_i^{2-H})\Big)s^2\\
				&+o(s^{\max\{2,1+H\}}). 
			\end{align*}
			When $H=1$, we have
			\begin{align*}
				&\Cov(Y_t^i, Y_{t+s}^j)=\Cov(Y_0^1, Y_0^2)-\nu_1\nu_2 \frac{\eta_{ij}}{2} s\log s +o(s^2\log s).
			\end{align*}
		
		\end{lemma}
		\noindent\begin{proof}
			We have that
			$$
			\Cov(Y^1_t, Y^2_{t+s})=e^{-\a_2 s} \Cov(Y^1_0,Y^2_0) -\nu_1\nu_2 e^{-\a_2 s} H(1-H)\Big(\frac{\rho-\eta_{12}}{2}\Big) \int_0^s e^{\a_2 v}  \int_{-\infty}^0 e^{\a_1 u}(v-u)^{H-2} du dv
			$$
			for $t,s \in \R, s>0$.

			Then, developing with Taylor's formula and denoting with $K_i= \frac{\rho+\eta_{ij} }{2}$, by Lemma \ref{Short_time_integral}, we obtain
			\begin{align*}
				&\Cov(Y_t^1, Y_{t+s}^2)=\Cov(Y_0^1, Y_0^2)-K_2\nu_1 \nu_2 s^H +\Big(-\a_2 \Cov(Y_0^1, Y_0^2)+\a_1^{1-H} \Gamma(H+1) \nu_1\nu_2 K_2\Big) s+ \\
				&+\frac{(\a_2-\a_1)\nu_1 \nu_2 }{H+1} K_2 s^{1+H}+\Big(\frac{\a_2^2}{2} \Cov(Y_0^1, Y_0^2)-\frac 12\nu_1\nu_2 K_2 \Gamma(H+1)(\a_2\a_1^{1-H} -\a_1^{2-H})\Big)s^2 -\\
				&-\nu_1 \nu_2 K_2 \frac{\a_1^2-\a_1\a_2 +\a_2^2}{(1+H)(2+H)}s^{H+2}+o(s^{H+2}) .
			\end{align*}
			
			When $H=H_1+H_2=1$, by Lemma \ref{Short_time_integral} we have
			\begin{align*}
				&\Cov(Y_t^1, Y_{t+s}^2)-\Cov(Y_0^1, Y_0^2)\\
				&=-\a_2 s \Cov(Y_0^1 , Y_0^2)+o(s)+\nu_1\nu_2 e^{-\a_2 s} \frac{\eta_{12}}{2}\int_0^s e^{\a_2 v} \int_{-\infty}^0 e^{\a_1 u} \frac 1{v-u} du dv +o(s^3)\\
				&=-\a_2 s \Cov(Y_0^1 , Y_0^2)+o(s)-\nu_1\nu_2 \frac{\eta_{12}}{2} s\log s +o(s^2\log s)\\
				&=-\nu_1\nu_2 \frac{\eta_{12}}{2} s\log s +o(s^2\log s).
			\end{align*}
			
		\end{proof}
		
		As a consequence of Lemma \ref{shorttime}, we derive formulas for $\rho$ and $\eta_{12}$ depending on the cross-covariances at lag $0, s, -s$ for all $s\in\R$. We observe that in Lemma \ref{shorttime} we consider the expansion of the covariance up to terms of order $1+H$ and $2$, because in the next expression \eqref{rho_s} we have a cancellation in the term of order $1$. 
		\begin{lemma}\label{short-time-inverse}
			For $s\to 0$, $\rho$ and $\eta_{12}$ have the following formulas:
			for all $H\in(0,2)$, $H\neq 1$, 
			\begin{equation}\label{rho_s}
				\rho=\frac{2\Cov(Y_0^1, Y_0^2)-\Cov(Y_s^1,Y_0^2)-\Cov(Y_0^1,Y_s^2)}{\nu_1\nu_2s^H}+O(s^{\min(1,2-H)})
			\end{equation}
			whereas only for $H\in(0,1)$, 
			\begin{equation}\label{eta_s}
				\eta_{12}=
				\frac{\Cov(Y_0^1,Y_s^2)-\Cov(Y_s^1,Y_0^2)}{\nu_1\nu_2s^H}	+O(s^{1-H}).
			\end{equation}
		\end{lemma}
		\noindent\begin{proof}
			From Lemma \ref{shorttime}, we have, for $H\neq 1$, 
			\begin{align*}
			&	\Cov(Y_0^1, Y_s^2)+\Cov(Y_s^1, Y_0^2)-2\Cov(Y_0^1, Y_0^2)=2\Cov(Y_0^1, Y_0^2)-2\Cov(
			Y_0^1, Y_0^2) -\nu_1\nu_2 \rho s^H+\\
			&+\Big(-(\a_1+\a_2)\Cov(Y_0^1, Y_0^2)+(\a_1^{1-H}\frac{\rho-\eta_{12}}2+\a_2^{1-H}\frac{\rho+\eta_{12}}2)\Gamma(H+1)\nu_1\nu_2 \Big)s+O(s^{\min\{1,2-H\}}).
			\end{align*} 
		Since by Lemma \ref{crss0}, 
		$$
		\Cov(Y_0^1, Y_0^2)=\frac{\Gamma(H+1)\nu_1\nu_2}{2(\a_1+\a_2)}\Big((\a_1^{1-H}+\a_2^{1-H})\rho+(\a_2^{1-H}-\a_1^{1-H})\eta_{12}\Big)
		$$ 
		the coefficient of $s$ is $0$. 
		For $\eta_{12}$ we have
					\begin{align*}
				&\Cov(Y_t^1, Y_{t+s}^2)-\Cov(Y_{t+s}^1, Y_t^2)=\\ &=\eta_{12}\nu_1\nu_2 s^H+\frac{\Gamma(H+1)\nu_1\nu_2}{\a_1+\a_2}\Big(\rho(\a_1^{2-H}-\a_2^{2-H})-\eta_{12}(\a_1^{2-H}+\a_2^{2-H})\Big)s+o(s).
			\end{align*}
			Inverting the above formulas we obtain the statement.
			
		\end{proof}
		\begin{remark}
		We point out that we will define in \eqref{hat_rho_n} and \eqref{hat_eta_n} estimators for $\rho$ and $\eta_{12}$ based on expressions \eqref{rho_s} and \eqref{eta_s}. 
		It is worth noticing that for $\eta_{12}$, when $H>1$, we could consider the linear term in $s$ and invert it, resulting in a different formulation.  However, if we pursued this approach, the formulation would depend on $\a_1$ and $\a_2$. Instead, we will use Lemma \ref{short-time-inverse} to derive estimators for $\rho$ and $\eta_{12}$ that notably will not depend on $\a_1, \a_2$. Consequently, for the estimator for $\eta_{12}$, we will confine our discussion to  $H<1$.  It is worth emphasizing that this should not be a significant limitation if the goal we have in mind is to use the estimators on log-volatility time series, since in this case we expect $H=H_1+H_2<1$, see for example \cite{GJR18}. 
		For example, in \cite{GJR18}, the authors model the log-volatility of various stocks as a fOU process. They estimate the values of the Hurst indexes according to their model for different stocks, obtaining very small indexes, both $H_1<\frac 12 $ and $H_2<\frac 12$. We also notice that for $H=1$ we do not provide a similar formula. 
		\end{remark}

		\chapter{Estimators for the correlation parameters}\label{Correlation_section}
		
		Our 2fOU process depends on parameters $H_1, H_2$, $  \nu_1, \nu_2, \a_1, \a_2$, $\rho, \eta_{12}$.	In our bivariate scenario, we must consider the presence of a cross-correlation structure. Therefore, we explore two distinct approaches to estimating $\rho$ and $\eta_{12}$. In both cases, we make the assumption that parameters related to the marginal distributions of the 2fOU process (i.e.  $H_1, H_2, \nu_1, \nu_2, \a_1, \a_2$) are known. Even if not ideal, this seems a reasonable starting point since the problem of estimating a one-dimensional fOU has already been widely considered in the literature both in theory and practice, see for example \cite{HN10}, \cite{WXY23}, \cite{Bolko_et_al}.

		Before introducing the estimators, we recall the definition of discrete convolution and the related Young's inequality, and we prove a simple but useful technical lemma.
		\begin{definition}\label{Discrete_convolution}
			Let $(f(k))_{k\in\Z}, (g(k))_{k\in\Z}$ be two sequences. The discrete convolution between $f$ and $g$ is 
			$$
			f*g(k)=\sum_{m\in \Z} f(m) g(k-m).
			$$
			
		\end{definition}
		\begin{proposition}[Young's inequality]\label{Young_ineq}
			Let $(f(k))_{k\in\Z}, (g(k))_{k\in\Z}$ be two sequences. For $p,q,s\geq 1$ such that $\frac 1p+\frac 1q=1+\frac 1s$, we have 
			$$
			\|f*g\|_{\ell^s(\Z)}\leq \|f\|_{\ell^p(\Z)}\|g\|_{\ell^q(\Z)}.
			$$
		\end{proposition}
		\begin{lemma}\label{estimate_sum}
			Let $\gamma >0$, $\beta \in(0,1)$. Then, as $n\to+\infty$,
			$$
			\sum_{k=n^{\beta}}^n \frac{1}{k^{\gamma}} =\begin{cases}
				O(n^{(1-\gamma)\beta})\, \quad{\text{if }\,\gamma >1}	\\
				O(\log n)\qquad{\text{if }\,\gamma=1}\\
				O(n^{1-\gamma})\qquad{\text{if }\,\gamma <1}
			\end{cases}
			$$	
			and 
			$$
			\sum_{k=1}^n \frac{1}{k^{\gamma}} =\begin{cases}
				O(1)\,\,\,\,\,\,\,\,\,\, \qquad{\text{if }\,\gamma >1}	\\
				O(\log n)\qquad{\text{if }\,\gamma=1}\\
				O(n^{1-\gamma})\qquad{\text{if }\,\gamma <1}.
			\end{cases}
			$$	
		\end{lemma}
		\begin{proof} 
			Let $\gamma\neq 1$. Then
			\begin{align*}
				&\sum_{k=n^\b}^n \frac 1{k^\gamma}\leq \int_{n^\b-1}^n \frac1{x^{\gamma}} dx = \frac{n^{1-\gamma}}{1-\gamma}-\frac{n^{\b(1-\gamma)}}{1-\gamma}.
			\end{align*}
			When $\gamma<1$, the above sequence tends to $+\infty$ as $O(n^{1-\gamma})$, while if $\gamma>1$ it tends to $0$ as $O(n^{(1-\gamma)\b})$. When $\gamma=1$
			\begin{align*}
				&\sum_{k=n^\b}^n \frac 1{k}\leq \int_{n^\b-1}^n \frac1{x} dx = \log n-\log(n^{\b}-1)\approx =O(\log n).
			\end{align*}
			The second part of the statement follows directly from the previous calculation.
		\end{proof}

		\section{First kind}\label{Firstkind}
		
		The first approach is based on the ergodic properties of the 2fOU process. A univariate fOU is ergodic (see \cite{cheridito2003}) and then the components of $2$fOUs are ergodic. The $2$fOUs is also cross-covariance ergodic (see Lemma \ref{cross-ergod}). Let us start from the cross-covariances to construct estimators for $\rho$ and $\eta_{12}$. 
		
		Let us recall the equations for cross covariance in Lemma \ref{crss0} and Lemma \ref{CrossYY}. For fixed $t,s \in \R$, inverting the equations for $\rho\pm\eta_{12}$, we obtain
		\begin{align}\label{r+}
			&\rho+\eta_{12}=2\frac{\Cov(Y_{t+s}^1, Y_t^2)-e^{-\a_1 s}\Cov(Y_t^1, Y_t^2)}{\nu_1\nu_2 H(H-1)e^{-\a_1 s}I_{12}(s)}\\
			&\label{r-}
			\rho-\eta_{12}=2\frac{\Cov(Y_t^1, Y_{t+s}^2)-e^{-\a_2 s}\Cov(Y_t^1, Y_t^2)}{\nu_1\nu_2 H(H-1)e^{-\a_2 s}I_{21}(s)}.
		\end{align}
		Combining \eqref{r+} and \eqref{r-}, it follows that
		\begin{equation}\label{111}
			\rho=a_{1}(s) \,\Cov(Y_t^1,Y_t^2)+a_{2}(s)\,\Cov(Y_{t+s}^1,Y^2_t)+a_3(s)\,\Cov(Y^1_{t}, Y_{t+s}^2)
		\end{equation}
		and
		\begin{equation}\label{222}
			\eta_{12}=b_1(s)\,\Cov(Y_t^1,Y_t^2)+b_2(s)\,\Cov(Y_{t+s}^1,Y^2_t)+b_3(s)\,\Cov(Y_t^1, Y_{t+s}^2).
		\end{equation}
		where 
		\begin{align}\label{coefficients}
			&a_1(s)=-\frac{I_{12}(s)+I_{21}(s)}{\nu_1\nu_2H(H-1) I_{12}(s)I_{21}(s)},\\
			&a_2(s)=\frac{1}{\nu_1\nu_2  H(H-1)e^{-\a_1 s}I_{12}(s)},\notag\\
			&a_3(s)=\frac{1}{\nu_1\nu_2  H(H-1)e^{-\a_2 s}I_{21}(s)},\notag
		\end{align}
	and
		\begin{align}\label{coefficients2}
			&b_1(s)=\frac{I_{12}(s)-I_{21}(s)}{\nu_1\nu_2 H(H-1) I_{12}(s)I_{21}(s)},\\
			&b_2(s)=\frac{1}{\nu_1\nu_2 H(H-1)e^{-\a_1 s}I_{12}(s)},\notag\\
			&b_3(s)=-\frac{1}{\nu_1\nu_2 H(H-1)e^{-\a_2 s}I_{21}(s)}\notag.
		\end{align}
		We can also consider the representation via correlations. In this case we have
		\begin{align*}
			\rho&=-\frac{\sqrt{\Gamma(2H_1+1)\Gamma(2H_2+1)}}{2\a_1^{H_1}\a_2^{H_2}H(H-1)}\Big(\frac{1}{I_{12}(s)}+\frac{1}{I_{21}(s)}\Big)\Corr(Y_t^1, Y_t^2)\\
			&+\frac{\sqrt{\Gamma(2H_1+1)\Gamma(2H_2+1)}}{2\a_1^{H_1}\a_2^{H_2}H(H-1)}\frac{e^{\a_1 s}}{I_{12}(s)}\Corr(Y^1_{t+s},Y^2_t)\\
			&+\frac{\sqrt{\Gamma(2H_1+1)\Gamma(2H_2+1)}}{2\a_1^{H_1}\a_2^{H_2}H(H-1)}\frac{e^{\a_2 s}}{I_{21}(s)}\Corr(Y^1_t, Y^2_{t+s})\\
		\end{align*}
		and
		\begin{align*}
			\eta_{12}&=-\frac{\sqrt{\Gamma(2H_1+1)\Gamma(2H_2+1)}}{2\a_1^{H_1}\a_2^{H_2}H(H-1)}\Big(\frac{1}{I_{12}(s)}-\frac{1}{I_{21}(s)}\Big)\Corr(Y_t^1, Y_t^2)\\
			&+\frac{\sqrt{\Gamma(2H_1+1)\Gamma(2H_2+1)}}{2\a_1^{H_1}\a_2^{H_2}H(H-1)}\frac{e^{\a_1 s}}{I_{12}(s)}\Corr(Y^1_{t+s},Y^2_t)\\
			&-\frac{\sqrt{\Gamma(2H_1+1)\Gamma(2H_2+1)}}{2\a_1^{H_1}\a_2^{H_2}H(H-1)}\frac{e^{\a_2 s}}{I_{21}(s)}\Corr(Y^1_t, Y^2_{t+s}).\\
		\end{align*}
		
		\noindent In case that we suppose a priori that we know that $\eta=0$, have the simpler representation of $\rho$
		\begin{align*}
			\rho=\frac{\sqrt{\Gamma(2H_1+1)\Gamma(2H_2+1)}}{\Gamma(H+1)}\frac{\a_1+\a_2}{\a_1^{H_1}\a_2^{H_2}(\a_1^{1-H}+\a_2^{1-H})} \Corr(Y_t^1, Y_t^2).
		\end{align*}

		\noindent Motivated by \eqref{111} and \eqref{222}, we construct the following estimators based on discrete observations for $\rho$ and $\eta_{12}$
		\begin{definition}\label{first_kind_estimator}
	Let $s\in \N$. Let us consider $Y_k=(Y_k^1, Y_k^2)$ for $k=0,\ldots, n$. We define 
		\begin{equation}\label{estRHO}
			\hat \rho_n=a_1(s)\,\frac{1}{n}\sum_{j=1}^n Y_j^1 Y_j^2+a_2(s)\,\frac{1}{n}\sum_{j=1}^{n-s }Y_{j+s}^1 Y_j^2+ a_3(s)\,\frac{1}{n}\sum_{j=1}^{n-s} Y_{j}^1 Y_{j+s}^2
		\end{equation}
		and
		\begin{equation}\label{estETA}
			\hat \eta_{12,n}=b_1(s) \,\frac{1}{n}\sum_{j=1}^n Y_j^1 Y_j^2+b_2(s)\,\frac{1}{n}\sum_{j=1}^{n-s} Y_{j+s}^1 Y_j^2+b_3(s)\,\frac{1}{n}\sum_{j=1}^{n-s} Y_{j}^1 Y_{j+s}^2.
		\end{equation}
	where $a_i(s), b_i(s), i=1,2,3$ are given in \eqref{coefficients} and \eqref{coefficients2}.
\end{definition}
\begin{remark}
	We establish our estimator based on \eqref{111} and \eqref{222}. Here, we write $\rho$ and $\eta_{12}$ as functions of $\Cov(Y^1_t, Y^2_t)$, $\Cov(Y^1_{t+s}, Y_t^2)$ and $\Cov(Y_t^1, Y_{t+s}^2)$. In practical applications it could be better to consider the representation via correlations, but the proofs of what follows do not change. In fact, by stationarity of $Y$, we have
	\begin{align*}
		&\rho=\sqrt{\Var(Y_0^1)\Var(Y_0^2)}\Big(a_1(s)\Corr(Y^1_t, Y^2_t)+a_2(s)\Corr(Y_{t+s}^1, Y_t^2)+a_3(s)\Corr(Y_t^1, Y_{t+s}^2)\Big)\\
		&\eta_{12}=\sqrt{\Var(Y_0^1)\Var(Y_0^2)}\Big(b_1(s)\Corr(Y^1_t, Y^2_t)+b_2(s)\Corr(Y_{t+s}^1, Y_t^2)+b_3(s)\Corr(Y_t^1, Y_{t+s}^2)\Big),
	\end{align*} 
then, in practice, one may normalize \eqref{estRHO} and \eqref{estETA} by the empirical standard deviations of $Y^1_t$ and $Y^2_t$. However, assuming knowledge of the process component parameters, including $\Var(Y_0^1)$ and $\Var(Y_0^2)$, the subsequent results follow analogously.
\end{remark}		
		
		\begin{lemma}\label{bias_lemma}
			Let $n \in \N$ and $\hat \rho_n, \hat \eta_{12, n}$ be as in \eqref{estRHO} and \eqref{estETA}. Then they are asymptotically unbiased estimators for $\rho$ and $\eta_{12}$ respectively. 
		\end{lemma}

		\noindent\begin{proof}
			Computing the expectation of $\hat \rho_{n}$, we have
			\begin{align*}
				\E[\hat \rho_n]&= \frac{a_1(s)}{n}\sum_{j=1}^n \E[Y_j^1 Y_j^2]+\frac{a_2(s)}{n}\sum_{j=1}^{n-s} \E[Y_{j+s}^1 Y_j^2]+\frac{a_3(s)}{n}\sum_{j=1}^{n-s} \E[Y_{j}^1 Y_{j+s}^2]\\
				&=\frac{a_1(s)}{n}\sum_{j=1}^n \Cov(Y_0^1, Y_0^2)+\frac{a_2(s)}{n}\sum_{j=1}^{n-s} \Cov(Y_s^1,Y_0^2)+\frac{a_3(s)}{n}\sum_{j=1}^{n-s} \Cov(Y_0^1,Y_s^2)\\
				&=a_1(s)\, \Cov(Y_0^1, Y_0^2)+a_2(s)\frac{n-s}{n}\, \Cov(Y_s^1,Y_0^2)+a_3(s)\frac{n-s}{n}\, \Cov(Y_0^1,Y_s^2)\\
				&\to \rho.
						\end{align*}
			Similarly, it holds that $\E[\hat \eta_{12,n}]\to\eta_{12}$.
			
		\end{proof}
		
		\noindent Let us notice that the expectation of the error is 
		$$
		\E[\hat \rho_n -\rho]=\frac{\a_2(s)s \,\Cov(Y_s^1, Y_0^2)+\a_3(s)s\,\Cov(Y_0^1, Y_s^2)}{n}=O\Big(\frac 1n\Big)$$
		
		\noindent Let us introduce the random variables $\overline \rho_n$ and $\overline \eta_{12, n}$ defined as
		
		\begin{equation}\label{estRHO2}
			\overline \rho_n=a_1(s)\,\frac{1}{n}\sum_{j=1}^n Y_j^1 Y_j^2+a_2(s)\,\frac{1}{n}\sum_{j=1}^{n }Y_{j+s}^1 Y_j^2+ a_3(s)\,\frac{1}{n}\sum_{j=1}^{n} Y_{j}^1 Y_{j+s}^2
		\end{equation}
		and
		\begin{equation}\label{estETA2}
			\overline \eta_{12,n}=b_1(s) \,\frac{1}{n}\sum_{j=1}^n Y_j^1 Y_j^2+b_2(s)\,\frac{1}{n}\sum_{j=1}^{n} Y_{j+s}^1 Y_j^2+b_3(s)\,\frac{1}{n}\sum_{j=1}^{n} Y_{j}^1 Y_{j+s}^2.
		\end{equation}
		We notice that $\E[\overline \rho_n]=\rho$ and $\E[\overline \eta_{12,n}]=\eta_{12}$. Then $\overline \rho_n$ and $\overline\eta_{12, n}$ are unbiased estimators (based on $n+s$ observations of $Y$). 
		Moving forward, whenever we need to extensively expand the expectation of the product of $4$ Gaussian random variables, we use Corollary \ref{corDiagFor} without explicitly recalling it each time.
		\begin{lemma}\label{errorVar}
			For $n\to\infty$, we have
			\begin{align*}
				&\Var(\hat\rho_n-\overline\rho_n)=O\Big(\frac 1{n^{4-2H}}\Big)\\
				&\Var(\hat\eta_{12,n}-\overline\eta_{12,n})=O\Big(\frac 1{n^{4-2H}}\Big).
			\end{align*}
			Therefore, $\hat\rho_n-\overline\rho_n$ and $\hat\eta_{12,n}-\overline\eta_{12,n}$ converge to $0$ in $L^2(\P)$ and then in probability. Additionally, when $H<\frac 32$, $\sqrt{n}(\hat\rho_n-\overline\rho_n)$ and $\sqrt{n}(\hat\eta_{12,n}-\overline\eta_{12,n})$ converge to $0$ in $L^2(\P)$ and then in probability.
		\end{lemma}
		\noindent\begin{proof} We have
			\begin{align*}
				&\hat \rho_n -\overline \rho_n=\frac{1}{n} \sum_{j=n-s+1}^n \Big(a_2(s)Y_{j+s}^1 Y_j^2+a_3(s)Y_j^1 Y_{j+s}^2\Big).
			\end{align*}
			Its expectation is given by
			\begin{align*}
				&\E[\hat \rho_n -\overline \rho_n]=\frac{s}{n} \Big(a_2(s)\Cov(Y_s^1, Y_0^2)+a_3(s)\Cov(Y_0^1, Y_s^2)\Big)=O\Big(\frac 1n\Big).
			\end{align*}
			We study the variance.
			\begin{align*}
				&\Var(\hat \rho_n -\overline \rho_n)\\
				&=\frac{1}{n^2} \sum_{i,j=n-s+1}^n \E[(a_2(s)Y_{j+s}^1 Y_j^2+a_3(s)Y_j^1 Y_{j+s}^2)(a_2(s)Y_{i+s}^1 Y_i^2+a_3(s)Y_i^1 Y_{i+s}^2)]-\E[\hat \rho_n -\overline \rho_n]^2\\
				&=\frac{a_2(s)^2}{n^2} \sum_{i,j=n-s+1}^n \E[Y_{j+s}^1Y_j^2Y_{i+s}^1Y_i^2]+\frac{a_3(s)^2}{n^2} \sum_{i,j=n-s+1}^n\E[Y_{j+s}^2Y_j^1Y_{i+s}^2Y_i^1]\\
				&+\frac{a_2(s)a_3(s)}{n^2} \sum_{i,j=n-s+1}^n\big(\E[Y_{j+s}^1Y_j^2 Y_i^1 Y_{i+s}^2]+\E[Y_j^1Y_{j+s}^2 Y_{i+s}^1 Y_i^2]\big) -\E[\hat \rho_n -\overline \rho_n]^2\\
				&=\frac{a_2(s)^2+a_3(s)^2}{n^2}\sum_{i,j=n-s+1}^n \Big(\Cov(Y_j^1, Y_i^1)\Cov(Y_i^2, Y_j^2)+\Cov(Y_{j+s}^1,Y_i^2)\Cov(Y_j^1,Y_{i+s}^2)\Big)\\
				&+\frac{a_2(s)a_3(s)}{n^2}\Big(\sum_{i,j=n-s+1}^n \Cov(Y^1_{j+s},Y_i^1)\Cov(Y_j^2 Y_{i+s}^2)+2\sum_{i,j=n-s+1}^n\Cov(Y^1_j, Y^2_i)\Cov(Y_i^1, Y_j^2)\\
				&+\sum_{i,j=n-s+1}^n\Cov(Y^1_j, Y_{i+s}^1)\Cov(Y_{j+s}^2, Y_i^2)\Big)\\
				&=\frac{a_2(s)^2+a_3(s)^2}{n^2}\sum_{i,j=n-s+1}^n\Big( \Cov(Y_0^1, Y_{i-j}^1)\Cov(Y_{i-j}^2, Y_0^2)+\Cov(Y_{i-j-s}^1,Y_0^2)\Cov(Y_0^1,Y_{i+s-j}^2)\Big)\\
				&+\frac{a_2(s)a_3(s)}{n^2}\Big(\sum_{i,j=n-s+1}^n \Cov(Y^1_{i-j-s},Y_0^1)\Cov(Y_0^2 Y_{i-j+s}^2)
											\end{align*}
			\begin{align*}
				&+2\sum_{i,j=n-s+1}^n\Cov(Y^1_0, Y^2_{i-j})\Cov(Y_{i-j}^1, Y_0^2)+\sum_{i,j=n-s+1}^n\Cov(Y^1_0, Y_{i-j+s}^1)\Cov(Y_{0}^2, Y_{i-j-s}^2)\Big)\\
				&=\frac{a_2(s)^2+a_3(s)^2}{n}\sum_{k=n-s+1}^n \frac{n-|k|}{n}\Big(\Cov(Y_0^1, Y_{k}^1)\Cov(Y_{k}^2, Y_0^2)+\Cov(Y_{k-s}^1,Y_0^2)\Cov(Y_0^1,Y_{k+s}^2)\Big)\\
				&+\frac{a_2(s)a_3(s)}{n}\sum_{k=n-s+1}^n \frac{n-|k|}{n}\Big( \Cov(Y^1_{k+s},Y_0^1)\Cov(Y_0^2 Y_{k+s}^2)+\\
				&+2\Cov(Y^1_0, Y^2_{k})\Cov(Y_{k}^1, Y_0^2)+\Cov(Y^1_0, Y_{k+s}^1)\Cov(Y_{0}^2, Y_{k-s}^2)\Big).
			\end{align*}
			Now, we observe that according to Theorems \ref{Decay1} and \ref{cov}, as $k\to+\infty$ all the sequences in the above sums are $O(k^{2H-4})$. Thus, there exists a constant $C>0$ such that, for sufficiently large $n$ 
			\begin{align*}
				&\Var(\hat \rho_n -\overline \rho_n)\leq   \frac C{n}\sum_{k=n-s+1}^n k^{2H-4}= O\Big(\frac 1{n^{4-2H}}\Big)
			\end{align*}
			where the last inequality follows by Lemma \ref{estimate_sum}. 
			Moreover, 
			\begin{align*}
				&n\E[(\hat \rho_n-\overline \rho_n)^2]\leq n\Var(\hat \rho_n-\overline \rho_n)-n\E[(\hat \rho_n-\overline \rho_n)]^2=O\Big(\frac 1n\Big).
			\end{align*}
			The computations for $\hat\eta_{12, n}-\overline{\eta}_{12,n}$ are the same. 
			
		\end{proof}

		\subsection{Asymptotic theory for $\hat \rho_n$ and $\hat \eta_{12,n}$}\label{Asymptotic_distribution1}
		
		We have observed that $\hat{\rho}_n$ and $\hat{\eta}_{12, n}$ serve as asymptotically unbiased estimators for the correlation parameters $\rho$ and $\eta_{12}$. In this section, we establish the consistency of these estimators and their asymptotic normality within a specific range of $H$, provided appropriate normalization is applied. We introduce an approach based on Malliavin calculus and Stein's method to derive the convergence to a Gaussian law of our normalized estimators in various probability distances. For additional details and background information, we refer to \cite{NP12}, and in Chapter \ref{background} we recall most of the information relevant for our computations. 
		
		We use $r_{12}$, $r_{11}$, and $r_{22}$ to denote the cross-correlation, auto-correlation of the first component, and auto-correlation of the second component, respectively. For $k\in \Z$, this is expressed as follows:
		\begin{align*}
			&r_{12}(k) = r_{21}(-k)=\frac{\E[Y_k^1 Y_0^2]}{\sqrt{\Var(Y_0^1)\Var(Y_0^2) }}\\
			&r_{11}(k) = \frac{\E[Y_k^1 Y_0^1]}{\Var(Y_0^1)}\\
			&r_{22}(k) =\frac{ \E[Y_k^2 Y_0^2]}{\Var(Y_0^2)}.
		\end{align*}
		
		\noindent The autocorrelation functions are symmetric, i.e., $r_{ii}(k) = r_{ii}(-k)$ for $i=1,2$, while the cross-correlation is not symmetric. Thus, we write $r_{12}(k) = \E[Y_k^1 Y_0^2]$ for $k\geq 0$ and use $r_{21}(k)$ to denote $\E[Y_0^1 Y_k^2] = r_{12}(-k)$ for $k\geq 0$. The asymptotic properties of the function $r_{21}$ mirror those of the function $r_{12}$, as we show in Theorem \ref{cov}. 
		
		\noindent Our processes can be represented as Wiener It\^o integrals with respect to a bidimensional Gaussian noise $W$. Recalling \eqref{movAverY}, we can write 
		\begin{equation}\label{Wiener-Ito_fOU}
			Y_k^i=\sum_{j=1}^2\int_{\R} K^j_i(k,s) W_j(ds)=\int_{\R} \<\tilde f_k^i(s), W(ds)\>_{\R^2},
		\end{equation}
		where $\tilde f_k^i(s)=(K_1^i(k,s), K_2^i(k,s))\in L^2(\R;\R^2)$. Clearly, from It\^o isometry, we can deduce that
		\begin{equation}\label{cross_isometry}
			r_{12}(h)=\frac{\E[Y_{k+h}^1 Y_k^2]}{\sqrt{\Var(Y_0^1)\Var(Y_0^2)}}=\int_{\R} \frac{\<\tilde f_{k+h}^1(x), \tilde f_k^2 (x)\>_{\R^2}}{\sqrt{\Var(Y_0^1)\Var(Y_0^2)}} dx=\<f_{k+h}^1, f_k^2\>_{L^2(\R; \R^2)},
		\end{equation}
		and 
		\begin{equation}\label{auto-isometry}
			r_{ii}(h)=\frac{\E[Y_{k+h}^i Y_k^i]}{\sqrt{\Var(Y_0^i)\Var(Y_0^i)}}=\int_{\R} \frac{\<\tilde f_{k+h}^i(x), \tilde f_k^i (x)\>_{\R^2}}{\sqrt{\Var(Y_0^i)\Var(Y_0^i)}} dx=\<f_{k+h}^i, f_k^i\>_{L^2(\R;\R^2)}.
		\end{equation}
		for all $h,k \in \N$, where 
		\begin{equation}\label{k}
		f^i_{k}=\frac{\tilde f^i_k}{\sqrt{\Var(Y_0^i)}}, \qquad{\|f_k^i\|_{L^2(\R)}=1}.
		\end{equation}
		Let us interpret $(Y_k^i)_{k\in\N}$, $i=1,2$, as a sub-field of $X$, where $X$ is the isonormal Gaussian field given by
		\begin{equation}\label{Gaus_fieldX}
			X(f)=\int_{\R} \<f(s),W(ds)\>_{\R^2} \qquad{f\in L^2(\R; \R^2)}. 
		\end{equation}
		Then, we consider 
		$$
		Y_k^i=I_1(\tilde f_k^i)
		$$
		using the notation in Definition \ref{multipleintegral}.  
		The goal of this section is to give a quantitative central limit theorem for the estimators $\hat \rho_n$ and $\hat \eta_{12, n}$, when $H=H_1+H_2$ is in a suitable sub-interval of $(0,2)$. Let us introduce
		\begin{align}\label{HATS_n}
			\overline S_n=\frac 1{n} \sum_{k=1}^n\Big(  a_1 Y_k^1Y_k^2+a_2 Y_{k+s}^1 Y_k^2+a_3 Y_k^1 Y_{k+s}^2 -\E[ a_1 Y_k^1Y_k^2+a_2 Y_{k+s}^1 Y_k^2+a_3 Y_k^1 Y_{k+s}^2]  \Big)
		\end{align}
		noticing that, for $a_1=a_1(s)$, $a_2=a_2(s)$ and $a_3=a_3(s)$, with $a_1(s), a_2(s), a_3(s)$ in \eqref{coefficients}, $\overline S_n=\hat\rho_n-\rho$, while  for $a_1=b_1(s)$, $a_2=b_2(s)$ and $a_3=b_3(s)$, with $b_1(s), b_2(s), b_3(s)$ in \eqref{coefficients2}, $\hat S_n=\hat\eta_{12,n}-\eta_{12,n}$. From the Wiener-It\^o representation  in \eqref{Wiener-Ito_fOU}, we can represent $\overline S_n$ as
		\begin{align*}
			\overline S_n&=\frac 1{n} \sum_{k=1}^{n} \Big(a_1I_1(\tilde f_k^1) I_1(\tilde f_k^2)+a_2I_1(\tilde f_{k+s}^1) I_1(\tilde f_k^2)+a_3I_1(\tilde f_k^1) I_1(\tilde f_{k+s}^2)-\\
			&-a_1\E[ I_1(Y_k^1) I_1(Y_k^2)]-a_2\E[ I_1(Y_{k+s}^1) I_1(Y_k^2)]-a_3\E[ I_1(Y_k^1) I_1(Y_{k+s}^2)]\Big).
		\end{align*}
		By applying the product formula in Theorem \ref{product_formula} for $p=q=1$, $S_n$ becomes
		\begin{align*}
			\overline S_n=\frac 1{n} \sum_{k=1}^{n} I_2\Big(a_1 \tilde f_k^1\tilde\otimes \tilde f_k^2+a_2\tilde  f_{k+s}^1\tilde\otimes \tilde f_k^2+a_3\tilde f_k^1\tilde\otimes \tilde f_{k+s}^2\Big),
		\end{align*}
		where $I_2$ denotes the double Wiener-It\^o integral and $\tilde\otimes$ denotes the symmetric tensor product (see Definition \ref{simple_symmetric}).
		
		\noindent We work with 
		\begin{align}\label{Structure_est}
			S_n&=\frac {\overline S_n}{\sqrt{\Var(Y_0^1)\Var(Y_0^2)}}\notag\\
			&=\frac 1{n\sqrt{\Var(Y_0^1)\Var(Y_0^2)}} \sum_{k=1}^{n} I_2\Big(a_1  \tilde f_k^1\otimes \tilde f_k^2+a_2 \tilde f_{k+s}^1\otimes \tilde f_k^2+a_3\tilde f_k^1\otimes \tilde  f_{k+s}^2\Big)\notag\\
			&=\frac 1{n} \sum_{k=1}^{n} I_2\Big(a_1   f_k^1\otimes  f_k^2+a_2  f_{k+s}^1\otimes  f_k^2+a_3 f_k^1\otimes   f_{k+s}^2\Big),
		\end{align}
	where $f_k^i$ are given in \eqref{k}. 
		\noindent Let us denote $z_k^s=a_1 f_k^1\tilde\otimes f_k^2+a_2f_{k+s}^1\tilde\otimes f_k^2+a_3f_k^1\tilde\otimes f_{k+s}^2$ to simplify the notation. 
		The variance of $S_n$ is given by
		\begin{align*}
			&\Var(S_n)=\E[S_n^2]=\frac{1}{n^2}\sum_{k,h=1}^n \E\Big[I_2(z_k^s) I_2(z_h^s)\Big],
		\end{align*}
		and, applying again the product formula when $p=q=2$ and recalling that $\E[I_q(f)]=0$ when $q\neq 0$, we have 
		\begin{align}\label{variance_dev}
			&\Var(S_n)=\frac{1}{n^2}\sum_{k,h=1}^n\E\big[z_k^s\,\tilde \otimes_2\, z_h^s\big]=\frac{1}{n^2}\sum_{k,h=1}^n\<z_k^s, z_h^s\>_{L^2(\R;\R^2)^{\otimes 2}}.
		\end{align}    
		Let us prove the following theorem.
		
		\begin{theorem}\label{VAR_SN}
			Let $S_n$ be as in \eqref{Structure_est}. Then, for $H<\frac 32$, 
	\begin{align*}
	&\lim_{n\to +\infty}	\Var(\sqrt{n} S_n)=\frac{1}{\Var(Y_0^1 )\Var(Y_0^2)}\Var(a_1Y_0^1Y_0^2+a_2Y_{s}^1Y_0^2+ a_3 Y_0^1 Y_{s}^2)\\
	&+\frac{2}{\Var(Y_0^1 )\Var(Y_0^2)}\sum_{k=1}^{+\infty} \Cov(a_1Y_0^1Y_0^2+a_2Y_{s}^1Y_0^2+ a_3 Y_0^1 Y_{s}^2,a_1Y_k^1Y_k^2+a_2Y_{k+s}^1Y_k^2+ a_3 Y_k^1 Y_{k+s}^2).\\
\end{align*}
			For $H=\frac 32$ we have
			$$
			\Var(S_n)=O\Big(\frac{\log n}n\Big)
			$$
			and for $H>\frac 32$
			$$
			\Var(S_n)=O\Big(\frac1{n^{4-2H}}\Big).
			$$
		\end{theorem}
		\noindent\begin{proof}
			We recall that $f\tilde \otimes g=\frac 12(f\otimes g+g\otimes f)$ and, for a given Hilbert space $\H$, and $v_1, v_2, v_3 ,v_4 \in \H$,  
			$$
			\< v_1\otimes v_2,v_3\otimes v_4\>_{\H^{\otimes 2}}=\<v_1, v_3\>_\H\<v_2, v_4\>_\H.
			$$

			Then, developing \eqref{variance_dev}, the variance can be written as \begin{align*}
				&\Var(\sqrt n S_n)\\
				&=\frac { a_1^2}{2n}\sum_{k,h=1}^n \Big( \<f_k^1\otimes f_k^2,f_h^1\otimes f_h^2\>+\<f_k^1\otimes f_k^2,f_h^2\otimes f_h^1\>+\\
				&+\<f_k^2\otimes f_k^1,f_h^1\otimes f_h^2\>+\<f_k^2\otimes f_k^1,f_h^2\otimes f_h^1\> \Big) +\\
				&+\frac {a_2^2}{2n}\sum_{k,h=1}^n  \Big( \<f_{k+s}^1\otimes f_k^2,f_{h+s}^1\otimes f_h^2\>+\<f_{k+s}^1\otimes f_k^2,f_h^2\otimes f_{h+s}^1\>+\\
				&+\<f_k^2\otimes f_{k+s}^1,f_{h+s}^1\otimes f_h^2\>+\<f_k^2\otimes f_{k+s}^1,f_h^2\otimes f_{h+s}^1\> \Big) +\\&\cdots\\
				&=\frac { a_1^2}{n}\sum_{k,h=1}^n \Big( r_{11}(|k-h|) r_{22}(|k-h|)
				+r_{12}(|k-h|)r_{21}(|k-h|)\Big)+\\
				&+\frac { a_2^2}{n}\sum_{k,h=1}^n \Big( r_{11}(|k-h|) r_{22}(|k-h|)+
				r_{12}(|k+s-h|)r_{21}(|k-s-h|)\Big)\\
				&\cdots
			\end{align*}
			We omit to write all the sums. Then, the variance of $S_n$ is a sum of sequences of the form 
			$$
			\frac 1{n}\sum_{k,h=1}^n r_{11}(|k-h|)r_{22}(|k-h|)
			$$
			or
			$$
			\frac 1{n}\sum_{k,h=1}^n r_{12}(|k-h|)r_{21}(|k-h|).
			$$
			We can have that the variables in functions $r_{ij}$ are shifted with the constant factor $s$, but the asymptotic behaviour does not change, so we can reduce the analysis to the above sequences. Then 
			\begin{align*}
				&\frac{1}{n} \sum_{k,h=1}^n r_{11}(|k-h|)r_{22}(|k-h|)= \sum_{|\tau|\leq n} \Big(1-\frac{|\tau|}n\Big)r_{11}(\tau)r_{22}(\tau)+\frac{1}{n}\\
				&=\sum_{\tau=1}^{\infty} \Big(1-\frac{|\tau|}n\Big)r_{11}(\tau)r_{22}(\tau)\1_{|\tau|\leq n}+\frac1n=\frac 2n \sum_{\tau=1}^{\infty} \Big(1-\frac{\tau}n\Big)r_{11}(\tau)r_{22}(\tau)\1_{1\leq\tau\leq n}+\frac1n 
			\end{align*}
			We recall that, when $\tau\to \infty$, $r_{ii}(\tau)=O(\tau^{2H_i-2})$, then $r_{11}(\tau)r_{22}(\tau)=O(\tau^{2H-4})$, and it is summable when $H<\frac 32$. Then,  using dominated convergence, 
			$$
			\lim_{n \to 
				+\infty}2\sum_{\tau=1}^{\infty} \Big(1-\frac{\tau}n\Big)r_{11}(\tau)r_{22}(\tau)\1_{\tau<n}+ r_{11}(0)r_{22}(0)=2\sum_{\tau=1}^{\infty} r_{11}(\tau)r_{22}(\tau)+1. 
			$$
			The same argument can be used for the second sequence, recalling that $r_{12}(\tau)=O(\tau^{H-2})$, $r_{21}(\tau)=O(\tau^{H-2})$, then $r_{12}(\tau) r_{21}(\tau)=O(\tau^{2H-4})$. 
			It follows that, when $H<\frac 32$, $\Var(S_n)=O\Big(\frac 1n\Big)$. 
			Moreover we easily see that 
			\begin{align*}
				&\lim_{n\to +\infty}	\Var(\sqrt{n} S_n)=\frac{1}{\Var(Y_0^1 )\Var(Y_0^2)}\Var(a_1Y_0^1Y_0^2+a_2Y_{s}^1Y_0^2+ a_3 Y_0^1 Y_{s}^2)\\
				&+\frac{2}{\Var(Y_0^1 )\Var(Y_0^2)}\sum_{k=1}^{+\infty} \Cov(a_1Y_0^1Y_0^2+a_2Y_{s}^1Y_0^2+ a_3 Y_0^1 Y_{s}^2,a_1Y_k^1Y_k^2+a_2Y_{k+s}^1Y_k^2+ a_3 Y_k^1 Y_{k+s}^2).\\
			\end{align*}
			
			\noindent When $H=\frac 32$, there exists a constant $C>0$ and an integer $N>0$ such that
			\begin{align*}
				&\frac{1}{n} \sum_{|\tau|=0}^n r_{11}(\tau)r_{22}(\tau)=O\Big(\frac1{n}\Big)+\frac{C}{n} \sum_{|\tau|\geq N}^n \frac{1}{\tau}+ \frac{C_2}{n}\sum_{|\tau|\geq N}^n O\Big(\frac1{\tau^{3}}\Big)=O\Big(\frac{\log n}n \Big).
			\end{align*}
			The same holds for $\frac{1}{n} \sum_{|\tau|=0}^n r_{12}(\tau)r_{21}(\tau)$.
			When $H>\frac 32$, we have
			\begin{align*}
				&\frac{1}{n} \sum_{|\tau|=0}^n r_{11}(\tau)r_{22}(\tau)=O\Big(\frac1{n}\Big)+\frac{C}{n} \sum_{|\tau|\geq N}^n \frac{1}{\tau^{4-2H}}+ \frac{C_2}{n}\sum_{|\tau|\geq N}^n O\Big(\frac1{\tau^{6-2H}}\Big)=O\Big(\frac1{n^{4-2H}} \Big).
			\end{align*}
			The same holds for $\frac{1}{n} \sum_{|\tau|=0}^n r_{12}(\tau)r_{21}(\tau)$.
			
		\end{proof}

	\noindent	Theorem \ref{VAR_SN} allows us to prove the consistency of $\hat \rho_n$ and $\hat \eta_{12,n}$.
		
		\begin{theorem}\label{consistency}
			Let $n\in \N$ and  $\hat \rho_n$, $\hat \eta_{12,n}$ given in \eqref{estRHO} and \eqref{estETA}. Then, for all $H\in (0,2)$, $\hat \rho_n $ and $\hat \eta_{12,n}$ converge to $\rho$ and $\eta_{12}$ in $L^2(\P)$ and so in probability.
		\end{theorem}
		\noindent\begin{proof}
			We have that, for $a_1=a_1(s)$, $a_2=a_2(s)$ and $a_3=a_3(s)$, where $a_1(s), a_2(s), a_3(s)$ are given in \eqref{coefficients}, then $\overline \rho_n=\overline S_n=\sqrt{\Var(Y_0^1)\Var(Y_0^2)} S_n$, then, by Theorem \ref{VAR_SN}, 
			$$
			\Var(\sqrt{n}\overline \rho_n)=\Var(\sqrt{n}\overline S_n)=\Var(Y_0^1)\Var(Y_0^2)\Var(\sqrt{n} S_n)\to 0.
			$$
 Moreover, by Theorem \ref{errorVar}, we have $\|\hat \rho_n -\overline \rho_n\|_{L^2(\P)}\to 0$. Then
			\begin{align*}
				&\E[(\hat \rho_n-\rho)^2]=\E[(\hat \rho_n- \overline \rho_n)^2]+\Var(\overline \rho_n)\to 0. 
			\end{align*}
			Then the statement follows for $\overline \rho_n$. The same computation can be applied for $\overline \eta_{12, n}$.
			
		\end{proof}
		
		\noindent In view of Theorem \ref{VAR_SN}, we here restrict our analysis to the case $H<\frac 32$. Let us consider
		\begin{align*}
			\sqrt{n} S_n&=  I_2\Big(\frac 1{\sqrt{n}} \sum_{k=1}^{n} \Big(a_1f_k^1\tilde\otimes f_k^2+a_2f_{k+s}^1\tilde\otimes f_k^2+a_3f_k^1\tilde\otimes f_{k+s}^2\Big)\Big).
		\end{align*}
		We denote with
		\begin{equation}\label{kernelSN}
			\theta_n=\frac 1{\sqrt{n}} \sum_{k=1}^{n} \Big(a_1f_k^1\tilde\otimes f_k^2+a_2f_{k+s}^1\tilde\otimes f_k^2+a_3f_k^1\tilde\otimes f_{k+s}^2\Big)
		\end{equation}
		the function kernel of $\sqrt{n}S_n$. Let us consider an orthonormal basis $\{e_j\}_j$ of $L^2(\R;\R^2)$. The symmetric tensor product is 
		\begin{align*}
			f_k^i\tilde\otimes f_h^j (x,y)&=\frac 12\Big( f_k^i\otimes f_h^j+f_h^j\otimes f_k^i \Big),
		\end{align*}
		then $\theta_n$ can be written as
		\begin{align*}
			\theta_n=\frac1{2\sqrt{n}}\sum_{k=1}^n \Big(a_1( f_k^1\otimes f_k^2+f_k^2\otimes f_k^1)+a_2(f_{k+s}^1\otimes &f_k^2+f_k^2\otimes f_{k+s}^1)\\
			&+a_3(f_k^1\otimes f_{k+s}^2+f_{k+s}^2\otimes f_k^1)\Big).
		\end{align*}
		
		\noindent The approach is quite similar to the approach in Chapter 5 in \cite{NP12} and in Chapter 7 in \cite{N13}. We proceed to prove the main theorem of this section. First of all, we need to estimate $\|\theta_n \otimes_1 \theta_n\|_{L^2(\R;\R^2)^{\otimes 2}}$, i.e. the $L^2(\R;\R^2)$ norm of the first order contraction $\theta_n \otimes_1\theta_n$. Let us start proving the following technical lemma.
		\begin{lemma}\label{gammaAn}
			Let us suppose $H=H_1+H_2<\frac 32$. Let us denote with $\gamma_{ij}$, with $i,j\in\{1,2\}$ four real functions such that
			\begin{itemize}
				\item[1)] $|\gamma_{ij}(k)|\leq \gamma_{ij}(0)$;
				\item[2)]  there exists $\ell_{ij}>0$, $i,j\in\{1,2\}$ such that
				$$\lim_{k\to\infty} \frac{|\gamma_{ij}(k)|}{k^p}=
				\begin{cases} 
					\ell_{ii}  \quad{\text{if }i=j,p=2H_i-2}\\
					\ell_{ij}  \quad{\text{if }i\neq j, p=H-2}.
				\end{cases}
				$$
				
			\end{itemize}
			Let $F(k_1,k_2,k_3,k_4)$ be one of the following functions: 
			\begin{align*}
				&\gamma_{11}(k_1-k_2)\gamma_{22}(k_2-k_3)\gamma_{11}(k_3-k_4)\gamma_{22}(k_4-k_1),\\
				&\gamma_{12}(k_1-k_2)\gamma_{12}(k_2-k_3)\gamma_{12}(k_3-k_4)\gamma_{12}(k_4-k_1),\\
				&\gamma_{12}(k_1-k_2)\gamma_{12}(k_2-k_3)\gamma_{11}(k_3-k_4)\gamma_{22}(k_4-k_1),\\
				&\gamma_{12}(k_1-k_2)\gamma_{11}(k_2-k_3)\gamma_{21}(k_3-k_4)\gamma_{22}(k_4-k_1).
			\end{align*}
			Then 
			$$
			A_n=\frac 1{n^2}\sum_{k_1,\ldots, k_4=0}^{n-1} F(k_1,k_2,k_3, k_4) \overset{n\to \infty}{\to} 0. 
			$$

		\end{lemma}
		\noindent \begin{proof} We split our proof according to the function $F$. Let us begin our analysis for $F(k_1, k_2, k_3, k_4)=\gamma_{11}(k_1-k_2)\gamma_{22}(k_2-k_3)\gamma_{11}(k_3-k_4)\gamma_{22}(k_4-k_1)$. Then
			$$
			A_n=\frac 1{n^2}\sum_{k_1,k_2, k_3, k_4=0}^{n-1} \gamma_{11}(k_1-k_2)\gamma_{22}(k_2-k_3)\gamma_{11}(k_3-k_4)\gamma_{22}(k_4-k_1).
			$$
			or equivalently,
			$$
			A_n=\frac 1{n^2}\sum_{k_1,k_2, k_3, k_4=0}^{n-1} \gamma_{22}(k_1-k_2)\gamma_{11}(k_2-k_3)\gamma_{22}(k_3-k_4)\gamma_{11}(k_4-k_1).
			$$
			(Note that permuting the indices returns us to the first form).

			\noindent We recall that $|\gamma_{ii}(k)|\leq \ell_{ii}$ and $\gamma_{ii}(|k|)\sim k^{2H_i-2}$ when $k\to \infty$, then there exist $C_i>0$, $i=1,2$ such that $|\gamma_{ii}(k)|\leq C_i |k|^{2H_i-2}$ (see \eqref{1dim_asym}). Let us study $|A_n|$. We have
			\begin{align*}
				|A_n|&\leq \frac 1{n^2}\sum_{k_1,k_2, k_3, k_4=1}^{n} |\gamma_{11}(k_1-k_2)\gamma_{22}(k_2-k_3)\gamma_{11}(k_3-k_4)\gamma_{22}(k_4-k_1)| \\
				&\leq \frac 1{n^2}\sum_{k_1, k_3=1 }^{n} \sum_{k_2, k_4 \in \Z}  |\gamma_{11}(k_1-k_2)\gamma_{22}(k_2-k_3)\gamma_{11}(k_3-k_4)\gamma_{22}(k_4-k_1)|.
			\end{align*}
	Let us write the above expression using the discrete convolution of $\gamma_{11}^n (k):=|\gamma_{11}(k)|\1_{|k|< n}$ and $\gamma_{22}^n (k):=|\gamma_{22}(k)|\1_{|k|< n}$ (see Definition \ref{Discrete_convolution}). We have
			\begin{align*}
				|A_n|\leq \frac{1}{n^2} \sum_{k_1, k_3=1}^n \Big(|\gamma_{11}^n|*|\gamma_{22}^n|(k_1-k_3)\Big)^2\leq \frac{1}{n} \sum_{k=-n}^n \Big(|\gamma_{11}^n|*|\gamma_{22}^n|(k)\Big)^2
			\end{align*}
			We recall Young's inequality for convolution in Proposition \ref{Young_ineq}: for $p,q, s\geq 1$ such that $\frac 1p+\frac 1q=1+\frac 1s$, we have 
			$$
			\|f*g\|_{\ell^s(\Z)}\leq \|f\|_{\ell^p(\Z)}\|g\|_{\ell^q(\Z)}.
			$$
			We will use convolution and Young's inequality at several points in the proof.
			Let us distinguish different cases, according to the values of $H_1, H_2$ and their sum $H$. In each case we see that $|A_n|\to 0$ when $n\to +\infty$. 
			\begin{itemize}
			
			\item[1)] When $H_1, H_2 <\frac 34$ then both $\gamma_{11}^n$ and $\gamma_{22}^n$ are in $\ell^2(\Z)$, because $\gamma_{ii}(k)\approx k^{2H_i-2}$ and $2H_i-2<-1$. Then
			\begin{itemize}
				\item[a)] $\max\{H_1,H_2\}\in\big(0,\frac 34\big)$, $\min\{H_1, H_2\}<\frac 12$. Without losing in generality, assume that $H_1=\max\{H_1,H_2\}$. Then $\gamma_{11}\in \ell^2(\Z)$, $\gamma_{22}\in \ell^1(\Z)$.  By selecting $s=2$, $p=2$, $q=1$, we can apply Young's inequality. 
				\begin{align*}
					|A_n|\leq \frac 1n \|\gamma_{11}\|_{\ell^2(\Z)}^2\|\gamma_{22}\|_{\ell^1(\Z)}^2\leq \frac C n ;
				\end{align*}

				\item[b)] $H_1, H_2\in \big(\frac 12, \frac 34\big)$. In this case we notice that
				\begin{align*}
					|A_n|&\leq \frac{1}{n^2}\sum_{k_1,\ldots, k_4=0}^{n-1} |\gamma_{11}(k_1-k_2)\gamma_{22}(k_2-k_3)\gamma_{11}(k_3-k_4)\gamma_{22}(k_4-k_1)| \\
					&\leq \frac{C}{n^2} \int_{[0,n]^4} |x_1-x_2|^{2H_1-2}|x_2-x_3|^{2H_2-2}|x_3-x_4|^{2H_1-2}|x_4-x_1|^{2H_2-2}dx\\
					&=\frac C{n^{6-4H}} \int_{[0,1]^4} |y_1-y_2|^{2H_1-2}|y_2-y_3|^{2H_2-2}|y_3-y_4|^{2H_1-2}|y_4-y_1|^{2H_2-2}dy
				\end{align*}
				and being $2-2H_i<1$ when $H_i>\frac 12$, then the integral is finite, and then 
				$$
				|A_n|\leq \frac C{n^{6-4H}}\to 0.
				$$
			\end{itemize}
			
		\item[2)]	When $H_1=\max\{H_1,H_2\}\geq \frac34$, then $\gamma^n_{11}\notin \ell^2(\Z)$. We distinguish again two cases:
			\begin{itemize}
				\item [a)] when $H_2=\min\{H_1,H_2\}\in \big(\frac 12, \frac 34)$, then 
				\begin{align*}
					|A_n|\leq Cn^{4H-6} \int_{[0,1]^4} |y_1-y_2|^{2H_1-2}|y_2-y_3|^{2H_2-2}|y_3-y_4|^{2H_1-2}|y_4-y_1|^{2H_2-2}dy,
				\end{align*}
				the integral is finite and $$|A_n|\leq \frac C{n^{6-4H}};$$
				\item[b)] $H_2<\frac 12$. Then $\gamma^n_{22}\in \ell^1$, then we choose $s=2$, $p=2$, $q=1$ and we have
				\begin{align*}
					&|A_n|\leq \frac C{n}\Big(\sum_{k=0}^{n-1} k^{4H_1-4}\Big)\|\gamma^n_{22}\|_{\ell^1(\Z)}^2\leq \frac{C_2}{n^{1-4H_1+3}}, 
				\end{align*}
				then 
				$$
				|A_n|\leq \frac C{n^{4-4\max\{H_1, H_2\}}}\to 0.
				$$
			\end{itemize}

			\end{itemize}
		When $F(k_1, k_2, k_3, k_4)=\gamma_{12}(k_1-k_2)\gamma_{12}(k_2-k_3)\gamma_{12}(k_3-k_4)\gamma_{12}(k_4-k_1)$, we have
			\begin{align*}
				|A_n|&\leq \frac{1}{n^2} \sum_{k=10}^{n-1} |\gamma_{12}(k_1-k_2)\gamma_{12}(k_2-k_3)\gamma_{12}(k_3-k_4)\gamma_{12}(k_4-k_1)|\\
				&\leq \frac C n\sum_{k=0}^{n-1} (|\gamma_{12}^n|*|\gamma_{12}^n|)^2(k)
			\end{align*}
		 The approach is equal to the first case. We consider the discrete convolution and apply Young's inequality with a suitable choice of $p,q,s$, according to the values of $H_1, H_2$ and $H$. 
		 We  recall that $\gamma^n_{12}\in \ell^2(\Z)$ when $H<\frac 32$, then we apply Young's inequality with $s=2$, $p=2$ and $q=1$ and recalling Lemma \ref{estimate_sum}, we have
			$$
			|A_n|\leq 
			\begin{cases}
				\frac{C}{n} \,\,\,\,\,\,\qquad{\text{if }H<1}\\
				\frac{C\log^2n}{n}\quad{\text{if }H=1}\\
				\frac{C}{n^{3-2H} }\,\,\quad{\text{if }H>1}	.
			\end{cases}
			$$

		\noindent	When $F(k_1, k_2, k_3, k_4)=\gamma_{12}(k_1-k_2)\gamma_{12}(k_2-k_3)\gamma_{11}(k_3-k_4)\gamma_{22}(k_4-k_1)$, then
			\begin{align*}
				 A_n=&\frac 1{n^2}\sum_{k_1,k_2, k_3, k_4=0}^{n-1} \gamma_{12}(k_1-k_2)\gamma_{12}(k_2-k_3)\gamma_{11}(k_3-k_4)\gamma_{22}(k_4-k_1)\\
				&\leq \frac 1{n^2}\sum_{k_1, k_4=0}^{n-1} \gamma_{12}^n*\gamma_{12}^n(k_1-k_4)\gamma_{11}^n*\gamma_{22}^n (k_1-k_4)\\ 
				&\leq \frac 1{n^2}\Big(\sum_{k_1, k_4=0}^{n-1} \big(\gamma_{12}^n*\gamma_{12}^n(k_1-k_4)\big)^2 \Big)^{\frac 12} \Big(\sum_{k_1, k_4=0}^{n-1}\big(\gamma_{11}^n*\gamma_{22}^n (k_1-k_4)\big)^2\Big)^{\frac 12}\\
				&\leq \frac 1{n}\Big(\sum_{k=0}^{n-1} \big(\gamma_{12}^n*\gamma_{12}^n(k)\big)^2 \Big)^{\frac 12} \Big(\sum_{k=0}^{n-1}\big(\gamma_{11}^n*\gamma_{22}^n (k)\big)^2\Big)^{\frac 12}
			\end{align*}
			
			We apply Young's inequality in according to the values of $H_1, H_2$ and $H$, distinguishing some cases:
			\begin{itemize}
				
				\item[1)] when $H<1$, we choose $p=2$, $q=1$, and we have
				$$
				\| \gamma_{12}^n *\gamma_{12}^n\|_{\ell^2}\leq \|\gamma_{12}^n\|_{\ell^1}\|\gamma_{12}^n\|_{\ell^{2}}
				$$
				that is uniformly bounded in $n$.
				Then, we have to estimate $\|\gamma_{11}^n*\gamma_{22}^n\|_{\ell^2}$. Therefore we need to investigate the behavior of the auto-correlation terms. 
				\begin{itemize}
				\item[a)] If $\min\{H_1, H_2\}<\frac{1}{2}$ and $\max\{H_1, H_2\}<\frac{3}{4}$, we observe that the auto-correlation related to the minimum Hurst index is in $\ell^1(\mathbb{Z})$, while the auto-correlation related to the maximum Hurst index is in $\ell^2(\mathbb{Z})$. Consequently, selecting $s=2$, $p=2$, and $q=1$ (or vice versa, depending on which auto-correlation is related to the minimum), we can apply Young's inequality. Supposing that $H_1=\max\{H_1, H_2\}$, we have
				\begin{align*}
				|A_n|&\leq \frac{1}{n}\| \gamma_{12}^n *\gamma_{12}^n\|_{\ell^2}\| \gamma_{11}^n *\gamma_{22}^n\|_{\ell^2} \\
				&\leq \frac 1n\|\gamma_{12}^n\|_{\ell^{2}}\|\gamma_{12}^n\|_{\ell^1}\|\gamma_{11}^n\|_{\ell^{2}}\|\gamma_{22}^n\|_{\ell^1}\leq \frac Cn;
				\end{align*}
				\item[b)] when $H<1$ and $H_1=\max\{H_1, H_2\}=\frac 34$, then
			\begin{align*}
				|A_n|&\leq \frac{1}{n}\| \gamma_{12}^n *\gamma_{12}^n\|_{\ell^2}\| \gamma_{11}^n *\gamma_{22}^n\|_{\ell^2} \\
				&\leq \frac 1n\|\gamma_{12}^n\|_{\ell^{2}}\|\gamma_{12}^n\|_{\ell^1}\|\gamma_{11}^n\|_{\ell^{2}}\|\gamma_{22}^n\|_{\ell^1}.
			\end{align*}
		 Since 
		 $$
		 \|\gamma_{11}^n\|_{\ell^2} \leq C\Big(\sum_{k=1}^{n-1} k^{4H_1-4}\Big)^{\frac 12}= C\Big(\sum_{k=1}^{n-1} \frac1k\Big)^{\frac 12}\approx \sqrt{\log n}
		 		 $$
		 		 then 
		 		 $$
		 		 |A_n|\leq \frac{C\sqrt{\log n}}{n}.
		 		 $$
				\item[c)] when $H<1$ and $H_1=\max\{H_1, H_2\}>\frac 34$, then we select the exponent of the norm related to the maximum to be $2$,  the exponent of the norm related to the minimum to be $1$.  We observe that if $H_1>\frac 34$, then $H_2<\frac 14$ and $\gamma_{22}\in \ell^1(\Z)$). Moreover, by Lemma \ref{estimate_sum} we have
				\begin{align*}
				 \|\gamma_{11}^n\|_{\ell^2}\leq C\Big(\sum_{k=1}^{n-1} k^{4H_1-4}\Big)^{\frac 12}\leq C n^{2H_1-\frac 32}
				\end{align*}
			and
					\begin{align*}
					|A_n|&\leq \frac{1}{n}\| \gamma_{12}^n *\gamma_{12}^n\|_{\ell^2}\| \gamma_{11}^n *\gamma_{22}^n\|_{\ell^2} \\
					&\leq \frac 1n\|\gamma_{12}^n\|_{\ell^{2}}\|\gamma_{12}^n\|_{\ell^1}\|\gamma_{11}^n\|_{\ell^{2}}\|\gamma_{22}^n\|_{\ell^1}\\
					&\leq \frac{C}{n^{\frac 52-2H_1}}=\frac{C}{n^{\frac 52-2\max\{H_1\}}}.
				\end{align*}
		    	\end{itemize}

				\item[2)] when $H=1$, we have
				$\|\gamma_{12}^n\|_{\ell^1(\Z)}\sim \log{n}$,
				then 
				\begin{itemize}
					\item[a)] if $\max\{H_1, H_2\}<\frac 34$ 
					\begin{align*}
						|A_n|&\leq \frac{1}{n}\| \gamma_{12}^n *\gamma_{12}^n\|_{\ell^2}\| \gamma_{11}^n *\gamma_{22}^n\|_{\ell^2}\leq C\frac{\log n}{n}.
					\end{align*}
				\item[b)] if $H_1=\max\{H_1, H_2\}<\frac 34$ we have 
				\begin{align*}
					|A_n|&\leq \frac{1}{n}\| \gamma_{12}^n *\gamma_{12}^n\|_{\ell^2}\| \gamma_{11}^n *\gamma_{22}^n\|_{\ell^2}\leq C\frac{(\log n)^2}{n}.
				\end{align*}
			\item[c)] if $H_1=\max\{H_1, H_2\}>\frac 34$ we have
				\begin{align*}
				|A_n|&\leq \frac{1}{n}\| \gamma_{12}^n *\gamma_{12}^n\|_{\ell^2}\| \gamma_{11}^n *\gamma_{22}^n\|_{\ell^2}\leq C\frac{\log n}{n^{4-4H_1}}=C\frac{\log n}{n^{4-4\max\{H_1\}}}.
			\end{align*}
				\end{itemize}
				\item[3)] when $1<H<\frac 32$ we have $\|\gamma_{12}^n\|_{\ell^1}\approx n^{H-1}$. Then
				\begin{itemize}
					\item[a)] if $H_1=\max\{H_1, H_2\}<\frac 34$ and $H_2=\min\{H_1, H_2\}<\frac 12$,  then
					$\gamma_{11}^n \in\ell^2$ and $\gamma_{22}^n\in\ell^1$, then
					\begin{align*}
						|A_n|&\leq \frac{1}{n}\| \gamma_{12}^n *\gamma_{12}^n\|_{\ell^2}\| \gamma_{11}^n *\gamma_{22}^n\|_{\ell^2}\leq C\frac{1}{n^{2-H}},
					\end{align*}
				while when $H_2=\min\{H_1, H_2\}> \frac 12$, we recall that 
				\begin{align*}
				|A_n|\leq Cn^{4H-6} \int_{[0,1]^4} |y_1-y_2|^{H-2}|y_2-y_3|^{H-2}|y_3-y_4|^{2H_1-2}|y_4-y_1|^{2H_2-2}dy,
			\end{align*}
		and then 
	$$
|A_n|\leq \frac{C}{n^{4H-6}}.
$$
				\item[b)] if $H_1=\max\{H_1, H_2\}=\frac 34$,  then by Lemma \ref{estimate_sum}
				$$
				\|\gamma_{11}^n\|_{\ell^2} \approx \log n, 
				$$
				and $\gamma_{22}^n\in\ell^1$, then
				\begin{align*}
					|A_n|&\leq \frac{1}{n}\| \gamma_{12}^n *\gamma_{12}^n\|_{\ell^2}\| \gamma_{11}^n *\gamma_{22}^n\|_{\ell^2}\leq C\frac{\log n}{n^{2-H}}=C\frac{1}{n^{\frac 72-H-2\max\{H_1, H_2\}}}.
				\end{align*}
			
				\item[c)] if $H_1=\max\{H_1, H_2\}>\frac 34$, then by Lemma \ref{estimate_sum}
				$$
				\|\gamma_{11}^n\|_{\ell^2} \approx n^{2H_1-\frac 32}, 
				$$
				and $\gamma_{22}^n\in\ell^1$, then
				\begin{align*}
					|A_n|&\leq \frac{1}{n}\| \gamma_{12}^n *\gamma_{12}^n\|_{\ell^2}\| \gamma_{11}^n *\gamma_{22}^n\|_{\ell^2}\leq C\frac{1}{n^{\frac 72-H-2H_1}}=C\frac{1}{n^{\frac 72-H-2\max\{H_1, H_2\}}}.
				\end{align*}
				\end{itemize}

			\end{itemize}

			\noindent We conclude with $F(k_1,k_2,k_3,k_4)=\gamma_{12}^n(k_1-k_2)\gamma_{11}^N(k_2-k_3)\gamma_{12}^n(k_3-k_4)\gamma_{22}^n(k_4-k_1)$. We have
			
			\begin{align*}
				|A_n|&=\frac1{n^2}\sum_{k_1, \ldots, k_4=1}^n |\gamma_{12}(k_1-k_2)\gamma_{11}(k_2-k_3)\gamma_{12}(k_3-k_4)\gamma_{22}(k_4-k_1)|\\
				&\leq \frac 1n \sum_{k=1}^n |\gamma_{12}^n * \gamma_{11}^n (k) \gamma_{12}^n * \gamma_{22}^n (k)|\\
				&\leq \frac 1n \|\gamma_{12}^n*\gamma_{11}^n\|_{\ell^2(\Z)}\|\gamma_{12}^n*\gamma_{22}^n\|_{\ell^2(\Z)}. 
			\end{align*}
			
			Then 

			\begin{itemize}
				\item [1)] when $H<1$ we have that $\gamma_{12}^n \in \ell^1$ and $\gamma_{12}^n \in \ell^2$, then 
				
				\begin{itemize}
					\item[a)] if $H_1=\max\{H_1, H_2\}<\frac 34$, then $\gamma_{11}^n\in\ell^1$ and $\gamma^n_{22}\in\ell^2$, and so, by applying Young's inequality for both norms, we have
					\begin{align*}
					|A_n|\leq \frac{1}{n} \|\gamma_{12}^n\|_{\ell^2}\|\gamma_{11}^n\|_{\ell^2} \|\gamma_{12}^n\|_{\ell^2}\|\gamma_{22}^n\|_{\ell^1}\leq \frac Cn.
					\end{align*}
					 \item[b)] if $H_1=\max\{H_1, H_2\}=\frac 34$, then $\|\gamma_{11}^n\|_{\ell^1}\approx \log n$ and $\gamma_{22}^n\in\ell^2$
					 by applying Young's inequality for both norms, we have
					 \begin{align*}
					 	|A_n|\leq \frac{1}{n} \|\gamma_{12}^n\|_{\ell^2}\|\gamma_{11}^n \|\gamma_{12}^n\|_{\ell^2}\|\gamma_{22}^n\|_{\ell^1}\leq C\frac {\log n}n.
					 \end{align*}
				 \item[c)] if $H_1=\max\{H_1, H_2\}>\frac 34$, then $\|\gamma_{11}^n\|_{\ell^1}\approx \log n$ and $\gamma_{22}^n\in\ell^2$
				 by applying Young's inequality for both norms, we have
				 \begin{align*}
				 	|A_n|\leq \frac{1}{n} \|\gamma_{12}^n\|_{\ell^2}\|\gamma_{11}^n\|_{\ell^2} \|\gamma_{12}^n\|_{\ell^2}\|\gamma_{22}^n\|_{\ell^1}\leq \frac {C}{n^{2-2H_1}}.
				 \end{align*}
				\end{itemize}
			\item[2)] when $H=1$, then $\gamma_{12}^n\in \ell^2$ but $\|\gamma_{12}^n\|_{\ell^1}\approx \log n$, then
			\begin{itemize}
			 \item[a)] if $H_1=\max\{H_1, H_2\}<\frac 34$, then
			 \begin{align*}
			 	|A_n|\leq \frac{1}{n} \|\gamma_{12}^n\|_{\ell^2}\|\gamma_{11}^n \|\gamma_{12}^n\|_{\ell^2}\|\gamma_{22}^n\|_{\ell^1}\leq \frac {C\log n}{n}.
			 \end{align*} 
		 \item[b)] if $H_1=\max\{H_1, H_2\}=\frac 34$, then
		 \begin{align*}
		 	|A_n|\leq \frac{1}{n} \|\gamma_{12}^n\|_{\ell^2}\|\gamma_{11}^n\|_{\ell^2} \|\gamma_{12}^n\|_{\ell^2}\|\gamma_{22}^n\|_{\ell^1}\leq \frac {C(\log n)^2}{n}.
		 \end{align*} 
	 \item[c)] if $H_1=\max\{H_1, H_2\}>\frac 34$, then
	 \begin{align*}
	 	|A_n|\leq \frac{1}{n} \|\gamma_{12}^n\|_{\ell^2}\|\gamma_{11}^n\|_{\ell^2} \|\gamma_{12}^n\|_{\ell^2}\|\gamma_{22}^n\|_{\ell^1}\leq \frac {C\log n}{n^{2-2H_1}}.
	 \end{align*} 
 \end{itemize}
 \item[3)] if $1<H<\frac 32$, we have $\|\gamma_{12}^n \|_{\ell^1} \approx n^{H-1}$, then
 			\begin{itemize}
 	\item[a)] if $H_1=\max\{H_1, H_2\}<\frac 34$, then
 	\begin{align*}
 		|A_n|\leq \frac{1}{n} \|\gamma_{12}^n\|_{\ell^2}\|\gamma_{11}^n\|_{\ell^2} \|\gamma_{12}^n\|_{\ell^2}\|\gamma_{22}^n\|_{\ell^1}\leq \frac {C}{n^{2-H}}.
 	\end{align*} 
 	\item[b)] if $H_1=\max\{H_1, H_2\}=\frac 34$, and $H_2=\min\{H_1, H_2\}\leq \frac 12$ then
 	\begin{align*}
 		|A_n|\leq \frac{1}{n} \|\gamma_{12}^n\|_{\ell^2}\|\gamma_{11}^n \|_{\ell^2}\|\gamma_{12}^n\|_{\ell^2}\|\gamma_{22}^n\|_{\ell^1}\leq \frac {C(\log n)^2}{n^{2-H}},
 	\end{align*} 
 while if $H_2=\min{H_1,H_2}>\frac  12$ we recall 
 \begin{align*}
 	|A_n|\leq Cn^{4H-6} \int_{[0,1]^4} |y_1-y_2|^{H-2}|y_2-y_3|^{2H_1-2}|y_3-y_4|^{H-2}|y_4-y_1|^{2H_2-2}dy,
 \end{align*}
 	\item[c)] if $H_1=\max\{H_1, H_2\}>\frac 34$ and $H_2=\min\{H_1, H_2\}\leq \frac 12$ then
 	\begin{align*}
 		|A_n|\leq \frac{1}{n} \|\gamma_{12}^n\|_{\ell^2}\|\gamma_{11}^n \|_{\ell^2}\|\gamma_{12}^n\|_{\ell^2}\|\gamma_{22}^n\|_{\ell^1}\leq \frac {C}{n^{\frac 72 -H_1-H}}=\frac {C\log n}{n^{\frac 72 -\max\{H_1, H_2\}-H}},
 	\end{align*} 
 while if $H_2=\min\{H_1, H_2\}>\frac  12$,  we recall that 
 \begin{align*}
 	|A_n|\leq Cn^{4H-6} \int_{[0,1]^4} |y_1-y_2|^{H-2}|y_2-y_3|^{H-2}|y_3-y_4|^{2H_1-2}|y_4-y_1|^{2H_2-2}dy,
 \end{align*}
 and then 
 $$
 |A_n|\leq \frac{C}{n^{4H-6}}.
 $$
 
			\end{itemize}
  \end{itemize}
			In each case $|A_n|\to 0$. 
			
		\end{proof}

		\begin{proposition}\label{contr0}
			Let $H<\frac 32$. Then 
			$$
			\|\theta_n \otimes_1\theta_n\|_{L^2(\R;\R^2)}\to 0
			$$ 
			and 
			$$
			\kappa_4(\sqrt{n}S_n)\to 0,
			$$
			where $\kappa_4(\sqrt{n}S_n)$ denotes the fourth cumulant of $\sqrt{n}S_n$. 
		\end{proposition}
		\noindent\begin{proof}
			Let us compute the contraction. We have 
			\begin{align*}
				&  \theta_n \otimes_1 \theta_n =\frac{1}{n}\sum_{k_1, k_2 =1}^n z_{k_1}^s\otimes_1 z_{k_2}^s.
			\end{align*}
			Let us recall that
			\begin{align*}
				& (f_{k_1}^1\tilde\otimes f_{k_1}^2)  \otimes_1 (f_{k_2}^1\tilde\otimes f_{k_2}^2) =\\
				&=\sum_{r_1, r_2, r_3=1}^{\infty} 
				\<f_{k_1}^1\tilde\otimes f_{k_1}^2, e_{r_1} \otimes e_{r_2}\> \<f_{k_2}^1\tilde\otimes f_{k_2}^2, e_{r_1} \otimes e_{r_3}\> e_{r_2}\otimes e_{r_3}\\
				&=\sum_{r_2, r_3=1}^{\infty} 
				\Big(\sum_{r_1=1}^{\infty}\<f_{k_1}^1\tilde\otimes f_{k_1}^2, e_{r_1} \otimes e_{r_2}\> \<f_{k_2}^1\tilde\otimes f_{k_2}^2, e_{r_1} \otimes e_{r_3}\> \Big)e_{r_2} \otimes e_{r_3}\\
				&=\sum_{r_2, r_3=1}^{\infty} 
				q(k_1, k_2, r_2, r_3)e_{r_2} \otimes e_{r_3}.
			\end{align*}
			Here, $\{e_{r}\}_{r\in\N}$ denotes an orthonormal bases of $L^2(\R; \R^2)$ and $q(k, h, r, s)$ denotes $\sum_{\ell=1}^{\infty}\<f_{k}^1\tilde\otimes f_{k}^2, e_{\ell} \otimes e_{r}\> \<f_{h}^1\tilde\otimes f_{h}^2, e_{\ell} \otimes e_{s}\>$. Then
			\begin{align*}
				q(k_1, k_2, r_2, r_3)&=\<f_{k_1}^2, e_{r_2}\>\<f_{k_2}^2, e_{r_3}\>\<f_{k_1}^1, f_{k_{2}}^1\>+\<f_{k_1}^2, e_{r_2}\>\<f_{k_2}^1, e_{r_3}\>\<f_{k_1}^1, f_{k_{2}}^2\>\\
				&+\<f_{k_1}^1, e_{r_2}\>\<f_{k_2}^2, e_{r_3}\>\<f_{k_1}^2, f_{k_2}^1\>+\<f_{k_1}^1, e_{r_2}\>\<f_{k_2}^1, e_{r_3}\>\<f_{k_1}^2, f_{k_{2}}^2\>
			\end{align*}

			Then
			\begin{align*}
				&\|\theta_n\otimes_1\theta_n\|_{L^2(\R;\R^2)}^2=\<\theta_n\otimes_1\theta_n,\theta_n\otimes \theta_n\>=\frac 1{n^2}\sum_{k_1,\ldots,k_4=1}^n 
				\big\< z_{k_1}^s\otimes_1\,z_{k_2}^s, z_{k_3}^s\otimes_1\,z_{k_4}^s\big\>.
			\end{align*}
			We compute completely the first contraction by linearity:
			\begin{align}\label{pattern}
				&z_{k_1}^s\otimes_1 z_{k_2}^s\\\
				&=\big(a_1f_{k_1}^1\tilde\otimes f_{k_1}^2+a_2f_{k_1+s}^1\tilde\otimes f_{k_1}^2+a_3f_{k_1}^1\tilde\otimes f_{k_1+s}^2\big)\otimes_1\big(a_1f_{k_2}^1\tilde\otimes f_{k_2}^2+a_2f_{k_2+s}^1\tilde\otimes f_{k_2}^2+a_3f_{k_2}^1\tilde\otimes f_{k_2+s}^2 \big)\notag\\
				&=a_1^2\Big( \<f_{k_1}^1, f_{k_2}^1\> f_{k_1}^2\otimes f_{k_2}^2 +\<f_{k_1}^1, f_{k_2}^2\> f_{k_1}^2\otimes f_{k_2}^1+\<f_{k_1}^2, f_{k_2}^1\>f_{k_1}^1 \otimes f_{k_2}^2+\<f_{k_1}^2, f_{k_2}^2\>f_{k_1}^1\otimes f_{k_2}^1\Big)\notag\\
				&+a_2^2\Big( \<f_{k_1+s}^1, f_{k_2+s}^1\> f_{k_1}^2\otimes f_{k_2}^2 +\<f_{k_1+s}^1, f_{k_2}^2\> f_{k_1}^2\otimes f_{k_2+s}^1+\notag\\
				&+\<f_{k_1}^2, f_{k_2+s}^1\>f_{k_1+s}^1 \otimes f_{k_2}^2+\<f_{k_1}^2, f_{k_2}^2\>f_{k_1+s}^1\otimes f_{k_2+s}^1\Big)\notag\\
				&+a_3^2\Big( \<f_{k_1}^1, f_{k_2}^1\> f_{k_1+s}^2\otimes f_{k_2+s}^2 +\<f_{k_1}^1, f_{k_2+s}^2\> f_{k_1+s}^2\otimes f_{k_2}^1+\notag\\
				&+\<f_{k_1+s}^2, f_{k_2}^1\>f_{k_1}^1 \otimes f_{k_2+s}^2+\<f_{k_1+s}^2, f_{k_2+s}^2\>f_{k_1}^1\otimes f_{k_2}^1\Big)\notag\\
				&+a_1a_2\Big(\<f_{k_1}^1, f_{k_2+s}^1\> f_{k_1}^2\otimes f_{k_2}^2 +\<f_{k_1}^1, f_{k_2}^2\> f_{k_1}^2\otimes f_{k_2+s}^1+\notag\\
				&+\<f_{k_1}^2, f_{k_2+s}^1\>f_{k_1}^1 \otimes f_{k_2}^2+\<f_{k_1}^2, f_{k_2}^2\>f_{k_1}^1\otimes f_{k_2+s}^1\Big)\notag\\
				&+a_1 a_3\Big(\<f_{k_1}^1, f_{k_2}^1\> f_{k_1}^2\otimes f_{k_2+s}^2 +\<f_{k_1}^1, f_{k_2+s}^2\> f_{k_1}^2\otimes f_{k_2}^1+\notag\\
				&+\<f_{k_1}^2, f_{k_2}^1\>f_{k_1}^1 \otimes f_{k_2+s}^2+\<f_{k_1}^2, f_{k_2+s}^2\>f_{k_1}^1\otimes f_{k_2}^1\Big)\notag\\
				&+a_2a_1\Big(\<f_{k_1+s}^1, f_{k_2}^1\> f_{k_1}^2\otimes f_{k_2}^2 +\<f_{k_1+s}^1, f_{k_2}^2\> f_{k_1}^2\otimes f_{k_2}^1+\notag\\
				&+\<f_{k_1}^2, f_{k_2}^1\>f_{k_1+s}^1 \otimes f_{k_2}^2+\<f_{k_1}^2, f_{k_2}^2\>f_{k_1+s}^1\otimes f_{k_2}^1\Big)\notag\\
				&+a_2a_3\Big(\<f_{k_1+s}^1, f_{k_2}^1\> f_{k_1}^2\otimes f_{k_2+s}^2 +\<f_{k_1+s}^1, f_{k_2+s}^2\> f_{k_1}^2\otimes f_{k_2}^1+\notag\\
				&+\<f_{k_1}^2, f_{k_2}^1\>f_{k_1+s}^1 \otimes f_{k_2+s}^2+\<f_{k_1}^2, f_{k_2+s}^2\>f_{k_1+s}^1\otimes f_{k_2}^1\Big)\notag\\
				&+a_3a_1\Big(\<f_{k_1}^1, f_{k_2}^1\> f_{k_1+s}^2\otimes f_{k_2}^2 +\<f_{k_1}^1, f_{k_2}^2\> f_{k_1+s}^2\otimes f_{k_2}^1+\notag\\
				&+\<f_{k_1+s}^2, f_{k_2}^1\>f_{k_1}^1 \otimes f_{k_2}^2+\<f_{k_1+s}^2, f_{k_2}^2\>f_{k_1}^1\otimes f_{k_2}^1\Big)\notag\\
				&+a_3a_2\Big(\<f_{k_1}^1, f_{k_2+s}^1\> f_{k_1+s}^2\otimes f_{k_2}^2 +\<f_{k_1}^1, f_{k_2}^2\> f_{k_1+s}^2\otimes f_{k_2+s}^1+\notag\\
				&+\<f_{k_1+s}^2, f_{k_2+s}^1\>f_{k_1}^1 \otimes f_{k_2}^2+\<f_{k_1+s}^2, f_{k_2}^2\>f_{k_1}^1\otimes f_{k_2+s}^1\Big).\notag
			\end{align}
			
			\noindent The contraction reveals terms comprised of one inner product and one tensorial product. Each term involves eight kernels: four associated with the first component of the process and four associated with the second component. Recalling the inner product $\<f_1\otimes g_1, f_2\otimes g_2\>=\<f_1,f_2\>\<g_1,g_2\>$
			along with \eqref{cross_isometry} and \eqref{auto-isometry},
			we observe that the first contraction of $\theta_n$ with itself yields terms of the following forms:
			\begin{align*}
				&r_{11}(k_1-k_2)r_{22}(k_2-k_3)r_{11}(k_3-k_4)r_{22}(k_4-k_1),\\
				&r_{12}(k_1-k_2)r_{12}(k_2-k_3)r_{12}(k_3-k_4)r_{12}(k_4-k_1),\\
				&r_{12}(k_1-k_2)r_{12}(k_2-k_3)r_{11}(k_3-k_4)r_{22}(k_4-k_1),\\
				&r_{12}(k_1-k_2)r_{11}(k_2-k_3)r_{21}(k_3-k_4)r_{22}(k_4-k_1).
			\end{align*}
			Additionally, the same functions can be evaluated with a suitable shift of $k-h$ determined by the constant $s$. Given that $s$ is fixed, the asymptotic behavior of the sum and the convergence analysis remain unaffected for the second type of terms.
			
			Furthermore, although there can be variations in the order in which the functions are evaluated with respect to the variables $k_1, \ldots, k_4$, this does not alter the analysis (up to permutation of the variables). Thus, we can confine the asymptotic analysis to the sum involving the four above terms. 
			By Theorem \ref{Decay1} and Theorem \ref{cov}, the assumptions of Lemma \ref{gammaAn} are satisfied. It follows that each sum involved sequences of the above forms tends to zero, then $\|\theta_n\otimes_1\theta_n\|\to 0$.
			
			Finally, by \eqref{v_4_est} and \eqref{c_4_est}
			$$
			\kappa_4(\sqrt{n}S_n)\leq C(2)\|\theta_n \otimes_1 \theta_n\|^2_{L^2(\R;\R^2)^{\otimes 2}} \to 0,
			$$
			where $C(2)$ is a positive constant depending only on the order of multiple integral. 
			
		\end{proof}

		Now we are ready to prove the convergence in various probability metrics of $\sqrt{n}S_n$ to a Gaussian random variable. 
		
		\begin{theorem}\label{S_n_conv}
			Let $H<	\frac 32$ and $N\sim\mathcal N(0,\sigma^2)$, where $\sigma^2=\underset{n\to+\infty}{\lim}\Var(\sqrt{n}S_n)>0$. Then 
			$$
			\sqrt{n}S_n\overset{d}{\to} N. 
			$$
			
		\end{theorem}
		
	\noindent	\begin{proof} 
			Let us denote with $N_n\sim \mathcal N(0,\sigma_n^2)$, where $\sigma_n^2=\Var(\sqrt{n}S_n)$. Then by Theorem \ref{principal_0}, \eqref{v_4_est} and \eqref{c_4_est} we have
			\begin{align*}
				&\dW\,\,\Big(\sqrt{n}S_n, N_n\Big)\leq C_{W}(2)\sqrt{\frac{\kappa_4(\sqrt{n}S_n)}{\sigma_n^2}},\\
				&\dTV\Big( \sqrt{n}S_n, N_n\Big)\leq C_{TV}(2)\sqrt{\frac{\kappa_4(\sqrt{n}S_n)}{\sigma_n^2}},\\
				&d_{K}\,\,\Big( \sqrt{n}S_n, N_n\Big) \leq C_{K}(2)\sqrt{\frac{\kappa_4(\sqrt{n}S_n)}{\sigma_n^2}}.
			\end{align*}
			By Proposition \ref{dWTVKgaussian}, we have
			\begin{align*}
				&\dW(\sqrt{n}S_n, N)\leq \dW(S_n,N_n)+\dW(N_n, N)\leq  C_{W}(2)\sqrt{\frac{\kappa_4(\sqrt{n}S_n)}{\sigma_n^2}}+\frac{\tilde C_{W}|\sigma_n^2-\sigma^2|}{\sigma_n\vee \sigma}  \\
				&\dTV(\sqrt{n}S_n, N)\leq \dTV(S_n,N_n)+\dTV(N_n, N)\leq  C_{TV}(2)\sqrt{\frac{\kappa_4(\sqrt{n}S_n)}{\sigma_n^2}}+\frac{\tilde C_{TV}|\sigma_n^2-\sigma^2|}{\sigma_n^2\vee \sigma^2}, \\
				&d_K(\sqrt{n}S_n, N)\leq d_K(S_n,N_n)+d_K(N_n, N)\leq  C_{K}(2)\sqrt{\frac{\kappa_4(\sqrt{n}S_n)}{\sigma_n^2}}+\frac{\tilde C_{K}|\sigma_n^2-\sigma^2|}{\sigma_n^2\vee \sigma^2}. 
			\end{align*}
			By Theorem \ref{contr0} we have that $\kappa_4(\sqrt{n}S_n)\to 0$, then the right-hand sides of above equations tends to $0$.  It implies the convergence in law of $\sqrt{n}S_n$ to $N$. 
			
		\end{proof}
		
		\noindent As a consequence of Theorem \ref{S_n_conv}, we prove the main result of this section.
		\begin{theorem}\label{CLT_hat_rho_eta}
			Let $\hat \rho_n$ and $\hat \eta_n$ in \eqref{estRHO} and \eqref{estETA}. Let $\sigma_\rho^2=\underset{n\to+\infty}{\lim}\Var(\sqrt{n}(\hat \rho_n-\rho))>0$ and $\sigma_\eta^2=\underset{n\to+\infty}{\lim}\Var(\sqrt{n}(\hat \eta_{12, n}-\eta_{12}))>0$. Then, for $H<\frac 32$,
			$$
			\sqrt{n}(\hat \rho_n-\rho)\overset{d}{\to} N_\rho 
			$$
			and 
			$$
			\sqrt{n}(\hat \eta_{12,n}-\eta_{12})\overset{d}{\to} N_{\eta}
			$$
			where $N_\rho\sim \mathcal N(0,\sigma_\rho^2)$ and $N_\eta\sim\mathcal N(0,\sigma_\eta^2)$. Additionally 
			\begin{equation*}
				\sqrt{n}(\hat \rho_n- \rho, \hat \eta_{12, n}-\eta_{12}) \overset{d}{\to} (N_\rho, N_\eta)
			\end{equation*} 
			where $\Cov(N_\rho, N_\eta)=\lim_{n\to \infty} n\E[(\hat \rho_n-\rho)(\hat \eta_{12, n}- \eta_{12})]$.
		\end{theorem}
		\noindent\begin{proof}
			In Theorem \ref{S_n_conv} we prove that
			\begin{align*}
				&\dW\,\,\Big( \sqrt{n}S_n, N\Big)\to 0,\\
				&\dTV\,\,\Big( \sqrt{n}S_n, N\Big)\to 0,\\
				&d_{K}\,\,\Big( \sqrt{n}S_n, N\Big)\to 0.
			\end{align*}
			where $N\sim N(0,\sigma^2)$, with $\sigma^2=\lim_{n\to \infty} \Var(\sqrt{n} S_n)$. 
			In Equation \eqref{Structure_est}, we define $S_n=\frac{\hat S_n}{\sqrt{\Var(Y_0^1)\Var(Y_0^2)}}$, with $\hat S_n$ in \eqref{HATS_n}. We observe that, taking $a_1, a_2$ and $a_3$ equal to $a_1(s), a_2(s)$ and $a_3(s)$ in \eqref{coefficients}, we recover $\hat S_n=\hat \rho_n-\rho$, while  taking $a_1, a_2$ and $a_3$ equal to $b_1(s), b_2(s)$ and $b_3(s)$ in \eqref{coefficients2}, we recover $\hat S_n=\hat \eta_{12,n}-\eta_{12}$. Then, let us consider $\hat N\sim  \mathcal N\Big(0,\Var(Y_0^1)\Var(Y_0^2)\sigma^2\Big)$, we have 
			\begin{align*}
				\dW(\sqrt{n}\hat S_n, \hat N)&= \dW\big(\sqrt{n}  \sqrt{\Var(Y_0^1)\Var(Y_0^2)}S_n, \hat N\big)\\
				&=\sqrt{\Var(Y_0^1)\Var(Y_0^2)}\dW\big(\sqrt{n}S_n,  \frac{\hat N}{\sqrt{\Var(Y_0^1)\Var(Y_0^2)}}\big) \\
				&=\sqrt{\Var(Y_0^1)\Var(Y_0^2)}\dW(\sqrt{n}S_n, N)\to 0.
			\end{align*}
			It holds also for Total Variation distance and Kolmogorov distance. Let us call $N_\rho$ and $N_{\eta}$ the Gaussian random variables such that $\sqrt{n}(\hat \rho_n- \rho)\to N_\rho$ and $\sqrt{n}(\hat \eta_{12,n}- \rho)\to N_\eta$ in Wasserstein, Total Variation and Kolmogorov. The second part of the statement follows from Theorem \ref{jointconvergence}. If the components of the random vector are Wiener-It\^o integral, then the componentwise convergence to Gaussian always implies joint convergence, where the covariance of the components of the Gaussian limit is given by
			\begin{align*}
			&\Cov(N_\rho, N_{\eta})=\lim_{n\to \infty} n \E[(\hat \rho_n- \rho)(\hat \eta_{12, n}-\eta_{12})].
			\end{align*} 
			(see Theorem \ref{jointconvergence}).
			
		\end{proof}

		When $H=\frac 32$, we need a different normalization to prove a CLT for $\rho_n$. Indeed, by Theorem \ref{VAR_SN} we have proved that $\lim_{n\to \infty} \Var(S_n)=O\Big(\frac{\log n}n\Big)$. In this case we denote with $\hat S_n$ the normalized sequence given by
		$$
		\hat S_n=\sqrt{\frac{n}{\log{n}}} \Big(\frac 1{n} \sum_{k=1}^{n} I_2\Big(a_1f_k^1\tilde\otimes f_k^2+a_2f_{k+s}^1\tilde\otimes f_k^2+a_3f_k^1\tilde\otimes f_{k+s}^2\Big)\Big).
		$$
		Then, we denote with 
		$$
		g_n=\frac{1}{\sqrt{n\log{n}}}\sum_{k=1}^{n} \Big(a_1f_k^1\tilde\otimes f_k^2+a_2f_{k+s}^1\tilde\otimes f_k^2+a_3f_k^1\tilde\otimes f_{k+s}^2\Big).
		$$
		The analysis is similar to the case $H<\frac 32$. We start proving the convergence of the variances sequence.
		\begin{proposition}\label{Var32}
			Let us suppose $H=\frac 32$. Then there exists $\sigma^2>0$ such that 
			$$
			\lim_{n\to+\infty}
			\Var(\hat S_n)=\sigma^2.	
			$$
		\end{proposition}
		\noindent\begin{proof}
			We have to prove that 
			$$
			\E[I_2(g_n)^2]\to \sigma^2 
			$$
			for $\sigma^2>0$. We have
			\begin{align*}
				&\E[I_2(g_n)^2]=\\&=\frac{1}{\log(n) n}\sum_{k_1, k_2=1}^n \Big\< a_1 f_{k_1}^1\tilde \otimes f_{k_1}^2+a_2f_{k_1+s}^1\tilde \otimes f_{k_1}^2+a_3f_{k_1}^1\tilde     \otimes f_{k_1+s}^2,a_1 f_{k_2}^1\tilde \otimes f_{k_2}^2+a_2f_{k_2+s}^1\tilde \otimes f_{k_2}^2+a_3f_{k_2}^1\tilde     \otimes f_{k_2+s}^2 \Big\>. 
			\end{align*}
			We study $\frac{ a_1^2}{\log(n)n} \sum_{k_1, k_2=1}^n \<f_{k_1}^1\tilde \otimes f_{k_1}^2\>\<f_{k_2}^1\tilde \otimes f_{k_2}^2\>$. We have
			\begin{align*}
				&\frac{ a_1^2}{\log(n)n} \sum_{k_1, k_2=1}^n \<f_{k_1}^1\tilde \otimes f_{k_1}^2\>\<f_{k_2}^1\tilde \otimes f_{k_2}^2\>=\frac{ a_1^2}{2\log(n)n} \sum_{k_1, k_2=1}^n \Big(\<f_{k_1}^1, f_{k_2}^1\>\<f_{k_1}^2, f_{k_2}^2\>+\<f_{k_1}^1, f_{k_2}^2\>\<f_{k_1}^2, f_{k_2}^1\>\Big)\\
				&=\frac{ a_1^2}{2\log(n)n} \sum_{k_1, k_2=1}^n \Big(r_{11}(k_1-k_2)r_{22}(k_2-k_1)+r_{12}(k_1-k_2)r_{21}(k_1-k_2)\Big)\\
				&=\frac{ a_1^2}{2\log(n)} \sum_{k=-n}^n \Big(1-\frac{|k|}{n}\Big)\Big(r_{11}(k)r_{22}(k)+r_{12}(k)r_{21}(k)\Big)\\
				&=\frac{ 2a_1^2}{2\log(n)} \sum_{k=1}^n \Big(1-\frac{k}{n}\Big)\Big(r_{11}(k)r_{22}(k)+r_{12}(k)r_{21}(k)\Big).
			\end{align*}
	By Lemma \ref{Decay1} and Lemma \ref{cov}, for $H=\frac 32$, we have that 
	\begin{align*}
		&\lim_{k\to +\infty} \frac{r_{11}(k)r_{22}(k)+r_{12}(k)r_{21}(k)}{k^{-1}}=\frac{\nu_1^2\nu_2^2}{4\a_1^2\a_2^2}\big((\rho^2-\eta_{12}^2)\frac {9}{16}+4(2H_1-1)(2H_2-1)H_1H_2\big)=\ell.
	\end{align*}
We prove that $\ell\neq 0$. For $H=\frac 32$, by condition in \eqref{domain}, we have that 
$$
\rho^2+\eta_{12}^2\leq \frac{\Gamma(\frac52)^2}{2\sin(\pi H_1)\sin(\pi H_2)\Gamma(2H_1+1)\Gamma(2H_2+1)}.
$$
Let us suppose that $\ell=0$, then 
$$
\eta_{12}^2=\rho^2+\frac{64(2H_1+1)(2H_2+1)H_1H_2}{9}.
$$
Then we obtain the following condition for $\rho^2$, recalling that $H_2=\frac 32-H_1$ and $H_1,H_2>\frac 12$:
\begin{align*}
\rho^2\leq  -\frac{\Gamma(\frac52)^2}{4\sin(\pi H_1)\cos(\pi H_1)\Gamma(2H_1+1)\Gamma(4+2H_1)}-\frac{32(2H_1+1)(4+2H_1)H_1(\frac 32-H_1)}{9}
\end{align*}
and it can be verify that the right-hand side is negative for all $H_1>\frac 12$, then it is not possible that $\ell=0$.

and moreover
$$
r_{11}(k)r_{22}(k)+r_{12}(k)r_{21}(k)- \frac \ell k=O\Big(\frac 1{k^2}\Big)
$$
 Then, if the limit is different to $0$, we have
 \begin{align*}
 	\lim_{n\to \infty} \sum_{k=1}^{n}\Big(1-\frac kn\Big)\Big( r_{11}(k)r_{22}(k)+r_{12}(k)r_{21}(k)- \frac \ell k\Big) =\lim_{n\to \infty} \sum_{k=1}^{n} O\Big(\frac{1}{k^2}\Big)  
 \end{align*}
that exists and it is finite. Then
	 \begin{align*}
		\lim_{n\to \infty} \frac{1}{\log n}\sum_{k=1}^{n}\Big(1-\frac kn\Big)\Big( r_{11}(k)r_{22}(k)+r_{12}(k)r_{21}(k)- \frac \ell k\Big)=0.
	\end{align*}
But $\lim_{n\to+\infty} \frac 1{\log n}\sum_{k=1}^n \Big(1-\frac kn\Big)\frac 1k=1$, then

	 \begin{align*}
	\lim_{n\to \infty} \frac{1}{\log n}\sum_{k=1}^{n}\Big(1-\frac kn\Big)\Big( r_{11}(k)r_{22}(k)+r_{12}(k)r_{21}(k)\Big)=\ell<\infty. 
\end{align*}

			We recall that $r_{11}(k)r_{22}(k)= O\Big(\frac 1k\Big) $ and $r_{12}(|k|)r_{21}(|k|)\approx \frac 1k$ when $k\to \infty$, then there exists 
		\begin{align*}
			&\lim_{n\to \infty} \frac{ a_1^2}{\log(n)n} \sum_{k_1, k_2=1}^n \<f_{k_1}^1\tilde \otimes f_{k_1}^2\>\<f_{k_2}^1\tilde \otimes f_{k_2}^2\>\\
			&=\lim_{n\to +\infty}\frac{ a_1^2}{2\log(n)n} \sum_{k_1, k_2=1}^n \Big(\<f_{k_1}^1, f_{k_2}^1\>\<f_{k_1}^2, f_{k_2}^2\>+\<f_{k_1}^1, f_{k_2}^2\>\<f_{k_1}^2, f_{k_2}^1\>\Big)
			\end{align*}
			and it is finite. We use the same argument to the entire sum and conclude the proof.   
			
		\end{proof}
		
		We can prove the following theorem.
		\begin{theorem}\label{Convergence_H_32}
			Let us suppose that $H=	\frac 32$ and $N\sim\mathcal N(0,\sigma^2)$, where $\sigma^2=\underset{n\to +\infty}{\lim}\Var(\hat S_n)$. Then 
			$$
			\hat S_n \overset{d}{\to} N
			$$
		\end{theorem}
		
		\noindent\begin{proof}
			We need to prove that the norm of the  contraction $$\|g_n\otimes_1g_n\|\to 0.$$
			We will omit the computation since it is identical to the one in proof of Proposition \ref{contr0}. The only difference lies in the normalization coefficient. 
			Let us start with 
			$$
			A_n=\frac 1{(\log{n})^2 n^2}\sum_{k_1,k_2, k_3, k_4=0}^{n-1} r_{11}(|k_1-k_2|)r_{22}(|k_2-k_3|)r_{11}(|k_3-k_4|)r_{22}(|k_4-k_1|)
			$$
			We notice that $H_1,H_2>\frac 12$, then 
			$$
			|A_n|\leq C\frac{1}{(\log{n})^2}\int_{[0,1]^4} |y_1-y_2|^{2H_1-2}|y_2-y_3|^{2H_2-2}|y_3-y_4|^{2H_1-2}|y_4-y_1|^{2H_2-2}dy_1dy_2dy_3dy_4,  
			$$
			the integral is finite and then $A_n\to 0$.
			When
			$$
			A_n=\frac 1{(\log{n})^2 n^2}\sum_{k_1,k_2, k_3, k_4=0}^{n-1} r_{12}(|k_1-k_2|)r_{12}(|k_2-k_3|)r_{12}(|k_3-k_4|)r_{22}(|k_4-k_1|)
			$$
			then
			$$
			|A_n|\leq C\frac{1}{(\log{n})^2}\int_{[0,1]^4} |y_1-y_2|^{H-2}|y_2-y_3|^{H-2}|y_3-y_4|^{H-2}|y_4-y_1|^{H-2}dy_1dy_2dy_3dy_4  
			$$
			and $A_n\to 0$.
			When 
			$$
			A_n=\frac 1{(\log{n})^2 n^2}\sum_{k_1,k_2, k_3, k_4=0}^{n-1} r_{12}(|k_1-k_2|)r_{11}(|k_2-k_3|)r_{12}(|k_3-k_4|)r_{22}(|k_4-k_1|)
			$$
			then
			$$
			|A_n|\leq C\frac{1}{(\log{n})^2}\int_{[0,1]^4} |y_1-y_2|^{H-2}|y_2-y_3|^{2H_1-2}|y_3-y_4|^{H-2}|y_4-y_1|^{2H_2-2}dy_1dy_2dy_3dy_4  
			$$
			and $A_n\to 0$.
			When 
			$$
			A_n=\frac 1{(\log{n})^2 n^2}\sum_{k_1,k_2, k_3, k_4=0}^{n-1} r_{12}(|k_1-k_2|)r_{12}(|k_2-k_3|)r_{11}(|k_3-k_4|)r_{22}(|k_4-k_1|)
			$$
			then
			$$
			|A_n|\leq C\frac{1}{(\log{n})^2}\int_{[0,1]^4} |y_1-y_2|^{H-2}|y_2-y_3|^{H-2}|y_3-y_4|^{2H_1-2}|y_4-y_1|^{2H_2-2}dy_1dy_2dy_3dy_4  
			$$
			and $A_n\to 0$. It follows that $k_4(\hat S_n)\to 0$. By Proposition \ref{Var32}	there exists $\lim_{n\to+\infty} \Var(\hat S_n)=\sigma^2>0$, then we can apply the Fourth moment Theorem and we can conclude that the statement follows.
		\end{proof}
		
		Then we have the following result.
			
	\begin{theorem}
	Let $H=\frac 32$ and $\hat \rho_n$ and $\hat \eta_n$ in \eqref{estRHO} and \eqref{estETA}. Let $\sigma_\rho^2=\underset{n\to+\infty}{\lim}\Var(\sqrt{n/\log n}(\hat \rho_n-\rho))$ and $\sigma_\eta^2=\underset{n\to+\infty}{\lim}\Var(\sqrt{n/\log n}(\hat \eta_{12, n}-\eta_{12}))$. Then 
	$$
	\sqrt{\frac{n}{\log n}}(\hat \rho_n-\rho)\overset{d}{\to} N_\rho
	$$
	and 
	$$
	\sqrt{\frac{n}{\log n}}(\hat \eta_{12,n}-\eta_{12})\overset{d}{\to} N_\eta
	$$
	where $N_\rho\sim \mathcal N(0,\sigma_\rho^2)$ and $N_\eta\sim\mathcal N(0,\sigma_\eta^2)$. Additionally 
	\begin{equation*}
		\sqrt{\frac{n}{\log n}}(\hat \rho_n- \rho, \hat \eta_{12, n}-\eta_{12}) \overset{d}{\to} (N_\rho, N_\eta)
	\end{equation*} 
	where $\Cov(N_\rho, N_\eta)=\lim_{n\to \infty} n\E[(\hat \rho_n-\rho)(\hat \eta_{12, n}- \eta_{12})]$.
\end{theorem}
	\begin{proof}
		The proof follows from the same argument in proof of Theorem \ref{CLT_hat_rho_eta}, applying Theorem \ref{Convergence_H_32} instead of Theorem \ref{S_n_conv}.  
	\end{proof}	
		\subsection{Non-central case}\label{NON_central}
		
		In Section \ref{Asymptotic_distribution1} we prove that when $H<\frac 32$, the normalized statistic $\sqrt{n}S_n$ converges to a standard Gaussian random variable in Wasserstein, Total variation and Kolmogorov distances, and so in distribution as well. However, when $H\geq  \frac 32$, at least one component of the process exhibits an autocorrelation function that does not belong to $\ell^2(\Z)$, and the cross-correlation is also not in $\ell^2(\Z)$. In particular when $H>\frac 32$ we have
		$$
		\Var\Big(\frac 1n\sum_{k=1}^n\big(  a_1 Y_k^1Y_k^2+a_2Y_{k+s}^1 Y_k^2+a_3Y_k^1 Y_{k+s}^2 -\rho  \big)\Big)=O(n^{2H-4})
		$$
		(see Theorem \ref{VAR_SN}). Then we approach the study of the asymptotic distribution of the normalized statistic in a different way. First of all we normalize the sequence with a different rate than $\sqrt n$. Let us consider the sequence
		\begin{align}\label{newSn}
			\tilde S_n&=\frac{n^{2-H}}{\sqrt{\Var(Y_0^1)\Var(Y_0^2)}}\Big(\frac 1n\sum_{k=1}^n\big(  a_1Y_k^1Y_k^2+a_2Y_{k+s}^1 Y_k^2+a_3Y_k^1 Y_{k+s}^2 -\rho  \big)\Big)\notag\\
			&=n^{1-H}\sqrt{\Var(Y_0^1)\Var(Y_0^2)}\sum_{k=1}^{n}\big(  a_1 Y_k^1Y_k^2+a_2Y_{k+s}^1 Y_k^2+a_3Y_k^1 Y_{k+s}^2 -\rho  \big).
		\end{align}
		
		The aim of this section is to prove that the variance of $\tilde S_n$ converges to $\sigma^2>0$ and to compute the limits of the cumulant of all orders of $\tilde S_n$. We adopt the approach outlined in \cite[\S 7.3]{N13}, modifying the proof to suit our setting. In that context, the author examines a functional of a one-dimensional process. In our case, we need to adjust some details, encountering more difficulties, particularly in how the correlation functions dictate the asymptotic behavior of the cumulants of our statistic. In particular in our case we were unable to write down the limiting random variable. In this case, in view of \cite{PN12}, the limiting law is a random variable $S_{\infty}$ of the form
		$$
		S_{\infty}=N+G
		$$
		where $G$ is in the second Wiener chaos and $N$ is a centered Gaussian law independent of the Gaussian noise $W$ (see \eqref{Gaus_fieldX}).

		By \eqref{Structure_est}, we deduce that our statistic can be written as
		$$
		\tilde S_n=I_2\Big(\frac 1{n^{H-1}} \sum_{k=1}^{n} \Big(a_1f_k^1\tilde\otimes f_k^2+a_2f_{k+s}^1\tilde\otimes f_k^2+a_3f_k^1\tilde\otimes f_{k+s}^2\Big)\Big
		)$$
		where $\{f_h^i\}_{i=1,2, h\in \N}$ are suitable functions in $L^2(\R; \R^2)$ (see \eqref{Structure_est}). Let us write 
		$$
		g_n=\frac 1{n^{H-1}} \sum_{k=1}^{n} \Big(a_1f_k^1\tilde\otimes f_k^2+a_2f_{k+s}^1\tilde\otimes f_k^2+a_3f_k^1\tilde\otimes f_{k+s}^2\Big).
		$$	
		The main result of this section is the following.
		\begin{theorem}\label{cumulantNC}
			Let $\tilde S_n$ be defined in \eqref{newSn} and $\kappa_p(\tilde S_n)$ be the cumulant of order $p$ of $\tilde S_n$. Let $a_1, a_2, a_3$ be such that $a_1+a_2+a_3\neq0$. Then, for all $p\geq 2$ 
			\begin{align*}
				&\lim_{n\to+\infty} \kappa_p(\tilde S_n)\\
				&= 2^p(p-1)!\lim_{n\to\infty} \frac{1}{n^{p(H-1)}} \sum_{k_1,\ldots, k_p=1}^n \sum_{h_1,\ldots, h_p=1}^{\infty}\prod_{i=1}^p \<g_n, e_{h_i}\otimes e_{h_{i+1}}\>\\
				&=2^p(p-1)!\hat C_p(a_1, a_2, a_3)\sum_{i_1,\ldots, i_p=1}^2 \sum_{\underset{i_2\neq j_1,\ldots, i_p\neq j_{p-1}, i_1\neq j_p}{j_1,\ldots, j_p=1}}^2\int_{[0,1]^p}  z_{i_1 j_1}(x_1, x_2) z_{i_2 j_2}(x_2, x_3)\ldots z_{i_p j_p}(x_p, x_1) dx
			\end{align*}
			where 
			$$
			z_{11}(x,y)= \frac{2}{\a_1^{2-2H_1}} \frac{H_1(2H_1-1)}{\Gamma(2H_1+1)} |x-y|^{2H_1-2}
			$$
			$$
			z_{22}(x,y)=\frac{2}{\a_2^{2-2H_2}} \frac{H_2(2H_2-1)}{\Gamma(2H_2+1)} |x-y|^{2H_2-2},
			$$
			$$
			z_{12}(x, y)=\frac{1}{\a_1^{1-H_1}\a_2^{1-H_2}}  \frac{H(H-1) }{\sqrt{\Gamma(2H_1+1)\Gamma(2H_2+1)}} \times \begin{cases}
				(\rho-\eta_{12})(x-y)^{H-2} \quad{x>y}\\
				(\rho+\eta_{12})  (y-x)^{H-2} \quad{x\leq y}
			\end{cases}
			$$
			$$
			z_{21}(x, y)=\frac{1}{\a_1^{1-H_1}\a_2^{1-H_2}}  \frac{H(H-1) }{\sqrt{\Gamma(2H_1+1)\Gamma(2H_2+1)}} \times \begin{cases}
				(\rho+\eta_{12})(x-y)^{H-2} \quad{x>y}\\
				(\rho-\eta_{12})  (y-x)^{H-2} \quad{x\leq y}
			\end{cases}
			$$
			$\hat C_p(a_1, a_2, a_3)=(a_1+a_2+a_3)^p$. 
			Additionally
			\begin{align*}
				&\lim_{n\to +\infty} \frac{\Var(\tilde S_n)}{(a_1+a_2+a_3)^2}=\frac{16\Big((\rho^2-\eta_{12}^2)H^2(H-1)^2+4H_1H_2(2H_1-1)(2H_2-1)\Big)}{\a_1^{2-2H_1}\a_2^{2-2H_2}\Gamma(2H_1+1)\Gamma(2H_2+1)(2H-3)(2H-2)}>0.
			\end{align*}
		\end{theorem}
	
	The existence of the limit in law for $\tilde S_n$ is a consequence of Theorem \ref{cumulantNC}. Indeed, for $\varepsilon>0$, we have that
	\begin{align*}
		\P\Big(|\tilde S_n|>\sqrt{\frac{\sup_{n\in\N}\kappa_2(\tilde S_n)}{\varepsilon}}\Big)\leq \frac{\Var(\tilde S_n)\varepsilon}{\sup_{n\in\N}\kappa_2(\tilde S_n)}\leq \varepsilon. 
	\end{align*}
	Then, for every $n\in\N$ and $\varepsilon>0$ there exists $M=\sqrt{\sup_{n\in\N}\kappa_2(\tilde S_n)/\varepsilon}$ (that does not depend on $n$) such that $\P(|\tilde S_n|>M)\leq \varepsilon$. It follows that the sequence $\tilde S_n$ is tight. The Hally-Bray theorem and the tightness of $\tilde S_n$ allows us to conclude the existence of the limiting law. 
	\begin{remark}
		As we can see in the result of Theorem \ref{cumulantNC}, it is difficult to establish that the fourth cumulant does not converge to $0$. The functions $z_{12}$ and $z_{21}$ are not continuous and are not positive for all $x$ in $(0,1)$. Then in general we can have cancellation in the limit of the fourth cumulant. But we can conclude that, when both $\rho-\eta_{12}$ and $\rho+\eta_{12}$ are positive, then the fourth cumulant does not converge to $0$. Moreover, in the particular case $\rho=\eta_{12}=0$, we have that the components of the $2$-fOU are independent and the fourth cumulant does not converge to $0$.  
	
	\end{remark}	
		The proof of Theorem \ref{cumulantNC} needs to be constructed by steps. 
		Following the approach of \cite{N13}, we start by defining the functional 
		$A_{g_n}:L^2(\R; \R^2)\to L^2(\R; \R^2)$ as the contraction of order $1$ against $g_n$. For $f\in L^2(\R; \R^2)$, this is given by
		$$
		A_{g_n}(f)=g_n\otimes_1 f.
		$$
		We can represent $A_{g_n}$ more explicitly using a fixed orthonormal basis of $L^2(\R; \R)$. Let $\{e_k\}_{k\in\N}$ be a generic orthonormal basis of $L^2(\R; \R^2)$. Then $\{e_k\otimes e_h\}_{k,h \in\N}$ forms an orthonormal basis of $L^2(\R; \R^2)^{\otimes 2}$ and
		\begin{align*}
			&A_{g_n}(f):=A_{g_n} f=g_n\otimes_1 f=\sum_{k=1}^{\infty} \Big(\sum_{h=1}^{\infty} \< g_n, e_k\otimes e_h\>_{(L^2(\R; \R^2))^{\otimes 2}}\<f, e_h\>_{L^2(\R;\R^2)}\Big)e_k . 
		\end{align*}
		From now on, we omit specifying the Hilbert space of the inner products when the context is clear.
		
		It is worth noting that $A_{g_n}$ is a Hilbert-Schmidt operator. Recall that for an operator $T: H\to K$, where $H, K$ are separable Hilbert spaces, $T$ is Hilbert-Schmidt if the Hilbert-Schmidt norm of $T$ is finite, i.e.,
		\begin{align*}
			\text{Tr}|T|^2=\text{Tr}(TT^*)<\infty,
		\end{align*}
		where $T^*$ is the transpose of $T$ and $\text{Tr}$ is the trace of an operator (see \S \ref{functAnalysis})
		The trace of $T:H\to H$ can be defined by fixing an orthonormal basis of $H$. Denoting $\{g_k\}_{k\in \N}$ as the orthonormal basis of $H$, we have
		$$
		\text{Tr} (T) =\sum_{k=1}^{\infty} \<Tg_k, g_k\>. 
		$$
		
		The above definition is indeed independent of the basis we choose. If we consider the basis of the eigenvectors of $T$, denoted as $e_{T,j}$, with corresponding eigenvalues $\lambda_{T, j}$, we have
		\begin{align*}
			\text{Tr}(T)=\sum_{j=1}^{\infty}\<Te_{T,j}, e_{T,j}\>=\sum_{j=1}^{\infty} \lambda_{T,j}<\infty.
		\end{align*}
		Thus, we can readily prove that $A_{g_n}$ is a Hilbert-Schmidt. Let ${e_k}$ be an orthonormal basis of $L^2(\mathbb{R}; \mathbb{R}^2)$, then
		
		\begin{align*}
			&\text{Tr}(A_{g_n}^* A_{g_n})=\sum_{k=1}^{\infty} \<A_{g_n} e_k, A_{g_n} e_k\>=\sum_{k=1}^{\infty} \sum_{h_2=1}^{\infty}\Big(\sum_{h_1=1}^{\infty}\<g_n, e_{h_2 }\otimes e_{h_1}\>\<e_k, e_{h_1}\>\Big)\<e_{h_2},e_k\>\\
			&=\sum_{k=1}^{\infty} \<g_n, e_k\otimes e_k\>=\frac{1}{n^{1-H}}\sum_{h=1}^{n}\sum_{k=1}^{\infty}\Big(a_1\<f_h^1\tilde \otimes f_{h}^2, e_k\otimes e_k\>+a_2\<f_{h+s}^1\tilde \otimes f_{h}^2, e_k\otimes e_k\>+\\&+a_3\<f_h^1\tilde \otimes f_{h+s}^2, e_k\otimes e_k\>\Big)=\frac{1}{n^{1-H}}\sum_{h=1}^{n}\Big(a_1\sum_{k=1}^{\infty}\<f_h^1,e_k\>\<f_{h}^2, e_k\>+\\
			&+a_2\sum_{k=1}^{\infty}\<f_{h+s}^1,e_k\>\<f_{h}^2, e_k\>+a_3\sum_{k=1}^{\infty}\<f_h^1,e_k\>\< f_{h+s}^2, e_k\>\Big)\\
			&=\frac{1}{n^{1-H}}\sum_{h=1}^{n}\Big(a_1\<f_{h}^1, f_{h}^2\>+a_2\<f_{h+s}^1, f_{h}^2\>+a_3\<f_{h}^1, f_{h+s}^2\>\Big)<\infty.
		\end{align*}
		From now on we denote with $\lambda_{g_n, j}$ and $e_{g_n, j}$, $j\in\N$, the eigenvalues with their respective eigenvectors of $A_{g_n}$ (which form an orthonormal basis of $L^2(\R; \R^2)$). Let us give an operative representation of the trace of $A_{g_n}^p$, for all $p\geq 2$. 
		\begin{lemma}\label{p-power}
			Let $p\in \N$, $p\geq 2$. Then the series 
			$$
			\sum_{j=1}^\infty \lambda_{g_n, j}^p
			$$
			converges and 
			$g_n$ can be written as
			$$
			g_n=\sum_{j=1}^{\infty} \lambda_{g_n, j} e_{g_n, j}\otimes e_{g_n, j}.
			$$
			Moreover, taking a generic orthonormal bases $e_k$ of $L^2(\R;\R)$, then 
			\begin{equation}\label{trace_A}
				\text{Tr}(A_{g_n}^p)=\sum_{h_1,\ldots, h_p=1}^{\infty}\prod_{i=1}^p \<g_n, e_{h_i}\otimes e_{h_{i+1}}\>=\sum_{j=1}^{\infty} \lambda_{g_n, j}^p,
			\end{equation}
			where $h_{p+1}=h_1$. 
		\end{lemma}
		\begin{proof}
			The first statement follows from the properties of Hilbert-Schmidt operator. Being $A_{g_n}$ an Hilbert-Schmidt operator, then
			$$
			\sum_{j=1}^{\infty } \<A_{g_n} e_j, A_{g_n}e_j\><\infty
			$$
			for every orthonormal bases $\{e_j\}$ of $L^2(\R;\R)$. Let us take $\{e_{g_n, j}\}_j$. Then
			$$
			\sum_{j=1}^{\infty} \lambda_{A_{g_n}, j}^2 =\sum_{j=1}^{\infty} \<A_{g_n} e_{g_n,j}, A_{g_n}e_{g_n,j}\><\infty. 
			$$
			We also recall that each Hilbert-Schmidt operator is compact, then the sequence of the eigenvalues decreases to $0$ (reordering the sequence). Then, there exists $m\in \N$ such that for all $k>m$, $\lambda_{g_n, k}<1$, and for all $p\geq 2$
			\begin{align*}
				&\sum_{j=1}^\infty \lambda_{g_n, j}^p=\sum_{j=1}^m \lambda_{g_n, j}^p+\sum_{j=1+m}^\infty \lambda_{g_n, j}^p   \leq \sum_{j=1}^m \lambda_{g_n, j}^p+\sum_{j=1+m}^\infty \lambda_{g_n, j}^2 <\infty.
			\end{align*}
			Moreover, being $\{e_{g_n, i}\otimes e_{g_n, j}\}_{i,j \in \N}$ an orthonormal basis of $(L^2(\R; \R^2))^{\otimes 2}$, then
			\begin{align*}
				&g_n=\sum_{i,j=1}^{n} \<g_n, e_{g_n, i}\otimes e_{g_n, j}\> e_{g_n, i}\otimes e_{g_n, j}
			\end{align*}
			but we can easily see that
			\begin{align*}
				&\<g_n, e_{g_n, i}\otimes e_{g_n, j}\>=\<A_{g_n} e_{g_n, j},e_{g_n, i}\>=0 \qquad{i\neq j}
			\end{align*}
			whereas the above quantity is equal to $\lambda_{g_n, j}$ when $i=j$. 
			
			The second equality in the last statement follows from the definition of trace of $A_{g_n}$ computed choosing the eigenvectors of $A_{g_n}$:
			\begin{align*}
				\text{Tr}(A_{g_n}^p)=\sum_{j=1}^{\infty}\<A_{g_n}e_{g_n, j}, e_{g_n, j}\>=\sum_{j=1}^\infty \lambda_{g_n, j}^p.
			\end{align*}
			Now let us consider a generic orthonormal basis of $L^2(\R; \R^2)$, denoted with $\{e_k\}_k$. Then 
			\begin{align*}
				&\text{Tr}(A_{g_n}^p)=\sum_{k=1}^{+\infty} \<A_{g_n}^p e_k, e_k\>=\sum_{k=1}^{+\infty} \<A_{g_n}( A_{g_n}^{p-1} e_k), e_k\>\\
				&=\sum_{k=1}^{\infty} \Big\<\sum_{h_1=1}^{\infty}\Big( \sum_{j_1=1}^{\infty} \<g_n, e_{h_1}\otimes e_{j_1}\>\<A_{g_n}^{p-1} e_k, e_{j_1}\>\Big)e_{h_1}, e_k\Big\>\\
				&=\sum_{h_1=1}^{\infty} \sum_{j_1=1}^{\infty} \<g_n, e_{h_1}\otimes e_{j_1}\>\<A_{g_n}^{p-1} e_{h_1}, e_{j_1}\>\\
				&=\sum_{h_1=1}^{\infty} \sum_{j_1=1}^{\infty} \<g_n, e_{h_1}\otimes e_{j_1}\>\Big\<\sum_{h_2=1}^{\infty}\Big(\sum_{j_2=1}^{\infty} \<g_n, e_{h_2}\otimes e_{j_2}\>\<A_{g_n}^{p-2} e_{h_1}, e_{j_2}\> \Big) e_{h_2}, e_{j_1}\Big\>\\
				&=\sum_{h_1, h_2=1}^{\infty} \<g_n, e_{h_1}\otimes e_{h_2} \>\sum_{j_2=1}^{\infty}\<g_n, e_{h_2}\otimes e_{j_2}\>\<A_{g_n}^{p-2} e_{h_1}, e_{j_2}\> \\
				&=\sum_{h_1, h_2=1}^{\infty} \<g_n, e_{h_1}\otimes e_{h_2} \>\sum_{j_2=1}^{\infty}\<g_n, e_{h_2}\otimes e_{j_2}\>\Big\<\sum_{h_3=1}^{\infty}\Big(\sum_{j_3=1}^{\infty} \<g_n, e_{h_3}\otimes e_{j_3}\>\<A_{g_n}^{p-3} e_{h_1}, e_{j_3}\>\Big) e_{h_3} , e_{j_2}\Big\>\\
				&=\sum_{h_1, h_2, h_3=1}^{\infty} \<g_n, e_{h_1}\otimes e_{h_2} \>\<g_n, e_{h_2}\otimes e_{h_3}\>\sum_{j_3=1}^{\infty} \<g_n, e_{h_3}\otimes e_{j_3}\>\<A_{g_n}^{p-3} e_{h_1}, e_{j_3}\>\\
				&=\ldots\\
				&=\sum_{h_1, \ldots, h_p=1}^{\infty} \prod_{i=1}^{p}\<g_n, e_{h_i}\otimes e_{h_{i+1}}\>,
			\end{align*}
			where $h_{p+1}=h_1$.
			
		\end{proof}

		Next, we need to establish the connection between the trace of $A_{g_n}^p$ and the $p$-th order cumulant of $I_2(g_n)=\tilde S_n$. To do so, we compute the $p$-th order cumulant of $\tilde S_n$, where $p\geq 2$. Let us recall that for a given random variable $F$ with moments of all orders, the $p$-th order cumulant is defined as
		$$
		\kappa_p(F)=(-\mathbf{i})^p\frac{\partial^p}{\partial t^p}|_{t=0}\log\Big(\E\big[e^{\mathbf{i}tF}\big]\Big).
		$$
		Now, we invoke the following proposition, which demonstrates that each random variable in $\mathcal{C}_2$ is equal in law to a specific series of random variables (that depends on the kernel $g_n$). This proposition corresponds to Proposition 7.1 in \cite{N13}, tailored to our particular scenario.
		\begin{proposition}\label{CumulantDev}
			Let $\tilde S_n$ defined in \eqref{newSn}. Then 
			$$
			\tilde S_n=\sum_{j=1}^{\infty}\lambda_{g_n, j} (N_j^2-1)
			$$
			when $\{N_j\}_j$ are a sequence of independent standard normal random variables, and the series converges in $L^2(\Omega)$. Moreover 
			\begin{align*}
				\kappa_p(\tilde S_n)=2^p (p-1)!\sum_{h_1, \ldots, h_p=1}^{\infty} \prod_{i=1}^{p}\<g_n, e_{h_i}\otimes e_{h_{i+1}}\>.
			\end{align*}
		\end{proposition}
		\begin{proof}
			Recalling the representation of $g_n$ in Lemma \ref{p-power}, and the product formula in \eqref{product_formula}, we have
			\begin{align*}
				\tilde S_n&=I_2(g_n)=\sum_{j=1}^{\infty}\lambda_{g_n, j} I_2(e_{g_n, j} \otimes e_{g_n, j} )\\
				&=\sum_{j=1}^{\infty}\lambda_{g_n, j} (I_1(e_{g_n, j})I_1( e_{g_n, j} )-1)\\
				&=\sum_{j=1}^{\infty}\lambda_{g_n, j} (I_1(e_{g_n, j})^2 -1)
			\end{align*}
			where the sequence $N_j=I_1(e_{g_n, j})$, $j\in\N$, is a sequence of independent standard normal random variables. The above representation implies
			\begin{align*}
				&\E[e^{\mathbf{i} t \tilde S_n}]=\prod_{j=1}^{\infty} \frac{e^{-\mathbf{i}t \lambda_{g_n, j}}}{\sqrt{1-2\mathbf i t\lambda_{g_n, j}}}.
			\end{align*}
			Then
			\begin{align*}
				&\kappa_p(\tilde S_n)=(-\mathbf{i})^p\frac{\partial^p}{\partial t^p}|_{t=0}\log\Big(\prod_{j=1}^{\infty} \frac{e^{-\mathbf{i}t \lambda_{g_n, j}}}{\sqrt{1-2\mathbf i t\lambda_{g_n, j}}}\Big)\\
				&=(-\mathbf{i})^p\frac{\partial^p}{\partial t^p}|_{t=0}\Big(\sum_{j=1}^\infty \log\Big(\frac{e^{-\mathbf{i}t \lambda_{g_n, j}}}{\sqrt{1-2\mathbf i t\lambda_{g_n, j}}}\Big)\Big)=\cdots =2^p(p-1)!\sum_{j=1}^{\infty} \lambda_{g_n, j}^p,
			\end{align*}
			and recalling \eqref{trace_A}, we have the statement. 
			
		\end{proof}

		Applying Proposition \ref{CumulantDev}, we obtain
		\begin{align*}
			&\kappa_p(\tilde S_n)=2^p (p-1)!\sum_{h_1, \ldots, h_p=1}^{\infty} \prod_{i=1}^{p}\<g_n, e_{h_i}\otimes e_{h_{i+1}}\>\\
			&=2^p (p-1)!\sum_{h_1, \ldots, h_p=1}^{\infty} \prod_{i=1}^{p}\Big\<\frac{1}{n^{H-1}}\sum_{k=1}^{n} \Big(a_1f_k^1\tilde \otimes f_{k}^2+a_2f_{k+s}^1\tilde \otimes f_{k}^2+a_3f_k^1\tilde \otimes f_{k+s}^2 \Big), e_{h_i}\otimes e_{h_{i+1}}  \Big\>  \\
			&=\frac{2^p (p-1)!}{n^{p(H-1)}}\sum_{h_1, \ldots, h_p=1}^{\infty}\sum_{k_1,\ldots, k_p=1}^n  \prod_{i=1}^{p} \Big\<a_1f_{k_i}^1\tilde \otimes f_{k_i}^2+a_2f_{k_i+s}^1\tilde \otimes f_{k_i}^2+a_3f_{k_i}^1\tilde \otimes f_{k_i+s}^2, e_{h_i}\otimes e_{h_{i+1}}  \Big\>\\
			&=\frac{2^p (p-1)!}{n^{p(H-1)}}\sum_{h_1, \ldots, h_p=1}^{\infty}\sum_{k_1,\ldots, k_p=1}^n\\
			&  \prod_{i=1}^{p} \Big(a_1\<f_{k_i}^1\tilde \otimes f_{k_i}^2, e_{h_i}\otimes e_{h_{i+1}}  \>+a_2\<f_{k_i+s}^1\tilde \otimes f_{k_i}^2, e_{h_i}\otimes e_{h_{i+1}}  \>+a_3\<f_{k_i}^1\tilde \otimes f_{k_i+s}^2, e_{h_i}\otimes e_{h_{i+1}}  \>\Big)
		\end{align*}

		We can expand $\kappa_p(\tilde S_n)$ such as finite sum of 
		\begin{align}\label{genC}
			\frac{2^p (p-1)!C(a_1,a_2,a_3)}{n^{p(H-1)} }\sum_{h_1, \ldots, h_p=1}^{\infty}  \sum_{k_1,\ldots, k_p=1}^n  \prod_{i=1}^p \<f_{ k_i^1}^{r_i^1}\otimes f_{ k_i^2}^{r_i^2}, e_{h_i}\otimes e_{h_i+1}\>\\.
		\end{align}
		Here $C(a_1,a_2,a_3)=a_1^{p_1}a_2^{p_2}a_3^{p_3}$ with $p_1+p_2+p_3=p$ while for all $i=1,\ldots, p $ we have 
		\begin{align*}
			&(r_i^1, r_i^2)\in \{(1,2), (2,1)\} \\
			&(k_i^1, k_i^2)\in \{(k_i, k_i), (k_i+s, k_i), (k_i, k_i+s)\}.
		\end{align*}

		Then 
		
		\begin{align*}
			&\frac{2^p (p-1)!C(a_1,a_2,a_3)}{n^{p(H-1)}}\sum_{h_1, \ldots, h_p=1}^{\infty}  \sum_{k_1,\ldots, k_p=1}^n  \prod_{i=1}^p \<f_{ k_i^1}^{r_i^1}\otimes f_{ k_i^2}^{r_i^2}, e_{h_i}\otimes e_{h_i+1}\>\\
			&=\frac{2^p (p-1)!C(a_1,a_2,a_3)}{n^{p(H-1)} } \sum_{k_1,\ldots, k_p=1}^n \<f_{k_1^2}^{r_1^2}, f_{k_2^1}^{r_2^1}\>\<f_{k_2^2}^{r_2^2},f_{k_3^1}^{r_3^1}\>\cdots \<f_{k_{p-1}^2}^{r_{p-1}^2}, f_{k_p^1}^{r_{p}^1}\>\<f_{k_p^2}^{r_p^2}, f_{k_1^1}^{r_1^1}\>.
		\end{align*}
		The conditions on the indexes $r$ enable us to restrict the range of different products in the sum. In particular we can notice that:
		\begin{itemize}
			\item if $(r_i^2, r_{i+1}^1)=(1,1)$, then $(r_{i+1}^2, r_{i+2}^1)\in\{(2,2), (2,1)\}$; 
			\item if $(r_i^2, r_{i+1}^1)=(2,2)$, then $(r_{i+1}^2, r_{i+2}^1)\in\{(1,1), (1,2)\}$;
			\item if $(r_i^2, r_{i+1}^1)=(1,2)$, then $(r_{i+1}^2, r_{i+2}^1)\in\{(1,2), (1,1)\}$;
			\item if $(r_i^2, r_{i+1}^1)=(2,1)$, then $(r_{i+1}^2, r_{i+2}^1)\in\{(2,1), (2,2)\}$;
			\item if $p$ is odd, there exists at least one couple $(r_i^2, r_{i+1}^2)\in\{(1,2), (2,1)\}$ and the number of couple in $\{(1,2), (2,1)\}$ are odd;
			\item if there exists a couple equal to $(1,1)$, then there exists a couple equal to $(2,2)$, and if the number of couple $(1,1)$ is $m$, then the number of couple $(2,2)$ is $m$.
		\end{itemize}
		Our aim is to prove that $\kappa_2(\tilde S_n)=\Var(\tilde S_n)\to \sigma^2>0$, where $\sigma^2$ is a strictly positive real value, while $\kappa_4(\tilde S_n)\not \to 0$. Then we can conclude that $\tilde S_n$ does not converge to a Gaussian random variable, as a consequence of Fourth Moment theorem (Theorem \ref{FOURTH}). Let us consider
		\begin{align}\label{Piece}
			&\frac1{n^{p(H-1)}}\sum_{h_1, \ldots, h_p=1}^{\infty}  \sum_{k_1,\ldots, k_p=1}^n  \prod_{i=1}^p \<f_{ k_i^1}^{r_i^1}\otimes f_{ k_i^2}^{r_i^2}, e_{h_i}\otimes e_{h_i+1}\>\notag\\
			&=\frac{1}{n^{p(H-1)} } \sum_{k_1,\ldots, k_p=1}^n \<f_{k_1^2}^{r_1^2}, f_{k_2^1}^{r_2^1}\>\<f_{k_2^2}^{r_2^2},f_{k_3^1}^{r_3^1}\>\cdots \<f_{k_{p-1}^2}^{r_{p-1}^2}, f_{k_p^1}^{r_{p}^1}\>\<f_{k_p^2}^{r_p^2}, f_{k_1^1}^{r_1^1}\>.
		\end{align}
		We aim to prove that it converges to certain values. We proceed as follows. 
		
		\begin{remark}
			
			Here, we make observations regarding the inner product in equation \eqref{Piece}.
			\begin{itemize}
				\item when $r_i^2=r_{i+1}^1=j \in \{1,2\}$, being $r_{jj}$ symmetric, we do not have notation problem, then $\<f_{k_i^2}^{r_1^2}, f_{k_{i+1}^1}^{r_2^1}\>=r_{jj}(k_{i+1}^1-k_{i}^2)$;
				\item when $r_i^2=1$ and $r_{i+1}^1=2$ we have 
				$$
				\<f_{k_i^2}^{1}, f_{k_{i+1}^1}^{2}\>=r_{21}(k_{i+1}^2-k_i^1)\1_{k_{i+1}^2>k_i^1}+r_{12}(k_i^1-k_{i+1}^2)\1_{k_{i+1}^2\leq k_i^1}.
				$$
			\end{itemize}
			
		\end{remark}
		
		To prove Theorem \ref{cumulantNC}, we split the sum in \eqref{Piece} into two parts. The first part consists of the sum over the indices $k_1,\ldots, k_p$ where there exist two indices $k_i, k_{i+1}$ such that $|k_{i+1}-k_i|\leq s+2$.

		\begin{lemma}\label{restrict_sum}
			For $n\to+\infty$
			$$
			\frac1{n^{p(H-1)}}\sum_{h_1, \ldots, h_p=1}^{\infty}  \sum_{\underset{\exists i:|k_{i+1}-k_i|\leq s+2}{k_1,\ldots, k_p=1}}^n   \prod_{i=1}^p \<f_{ k_i^1}^{r_i^1}\otimes f_{ k_i^2}^{r_i^2}, e_{h_i}\otimes e_{h_i+1}\>\to 0
			$$
			
		\end{lemma}
		Let's provide an example to clarify the concept further.
		
		\begin{example}\label{Example1}
			Let us propose an example of summand. Recall that $r_{12}(k)=\frac{\E[Y_k^1 Y_0^2]}{\sqrt{\Var(Y_0^1)\Var(Y_0^2)}}$ and $r_{21}(k)=\frac{\E[Y_0^1 Y_k^2]}{\sqrt{\Var(Y_0^1)\Var(Y_0^2)}}$ for $k>0$. We use $R(k)$ to indicate the function 
			$$
			R(k)=\begin{cases}
				r_{12}(k) \quad{k\geq 0}\\
				r_{21}(|k|) \quad{k<0}
			\end{cases}
			$$
			Let us consider 
			$$
			A_n=\frac{1}{n^{p(H-1)}}\sum_{k_1,\ldots, k_p=1, |k_p-k_1|\leq s+2}^{n} C_{k_1,\ldots, k_p}\prod_{i=1}^p R(k_{i+1}-k_i).
			$$
			Then
			\begin{align*}
				&|A_n|=\frac{1}{n^{p(H-1)}}\Big|\sum_{k_1,\ldots, k_p=1, |k_p-k_1|\leq s+2}^{n}\prod_{i=1}^p R(k_{i+1}-k_i)\Big|\\
				&=\frac{1}{n^{p(H-1))}} \Big|\sum_{k_1,\ldots, k_p=1, |k_p-k_1|\leq s+2}^{n}R(k_1-k_p) \times  \\
				&\times\sum_{h_1,\ldots, h_{p-1}=1 }^{\infty} \<f_{k_1}^2, e_{h_1}\>\<f_{k_2}^1\otimes f_{k_2}^2, e_{h_1}\otimes e_{h_2} \>\cdots \<f_{k_{p-1}}^1\otimes f_{k_{p-1}}^2, e_{h_{p-2}}\otimes e_{h_{p-1}} \> \<f_{k_{p}}^1, e_{h_{p-1}} \>\Big|. 
			\end{align*}
			We can prove that
			\begin{align*}
				&\frac{1}{n^{(p-2)(H-1)}}\sum_{k_2, \ldots, k_{p-1}=1}^n\sum_{h_1,\ldots, h_{p-1}=1 }^{\infty} \<f_{k_1}^2, e_{h_1}\>\<f_{k_2}^1\otimes f_{k_2}^2, e_{h_1}\otimes e_{h_2} \>\cdots  \<f_{k_{p}}^1, e_{h_{p-1}}\>\\
				&=\Big\<\Big( \cdots \Big( f_{k_1}^2 \otimes_1\Big(\frac{1}{n^{H-1}} \sum_{k_2=1}^n f_{k_2}^1\otimes f_{k_2}^2 \Big) \Big)\otimes_1\cdots \Big)\otimes_1 \Big(\frac{1}{n^{H-1}}\sum_{k_{p-1}=1}^n f_{k_{p-1}}^1\otimes f_{k_{p-1}}^2\Big),  f_{k_{p}}^1\Big>
			\end{align*}
			and, since we have also that
			\begin{align*}
				\|g\otimes_1 f\|= \Big(\sum_{h=1}^\infty &\Big(\sum_{k=1}^{\infty} \<g, e_k\otimes e_h\>\<f, e_k\> \Big)^2\Big)^{\frac 12}\\
				&\leq\Big( \sum_{h=1}^{\infty} \sum_{k=1}^{\infty} \<g, e_k\otimes e_h\>^2 \sum_{\ell=1}^{\infty} \<f, e_{\ell}\>^2\Big)^{\frac 12}=\|g\|\|f\|
			\end{align*}
			then
			\begin{align*}
				&\Big\<\Big( \cdots \Big( f_{k_1}^2 \otimes_1\Big(\frac{1}{n^{H-1}} \sum_{k_2=1}^n f_{k_2}^1\otimes f_{k_2}^2 \Big) \Big)\otimes_1\cdots \Big)\otimes_1 \Big(\frac{1}{n^{H-1}}\sum_{k_{p-1}=1}^n f_{k_{p-1}}^1\otimes f_{k_{p-1}}^2\Big),  f_{k_{p}}^1\Big>   \\
				&\leq \|f_{k_1}^2\| \|f_{k_p}^1\|\prod_{i=2}^{p-1} \Big\|\frac{1}{n^{H-1}}\sum_{k_i=1}^n f_{k_i}^1\otimes f_{k_i}^2\Big\|.
			\end{align*}
			Then, since $|R(k_1-k_p)|\leq 1$, we have
			\begin{align*}
				|A_n|&\leq \sum_{k_1=1}^n \sum_{k_p:|k_1-k_p|\leq s+2} \|f_{k_1}^2\| \|f_{k_p}^1\|\prod_{i=2}^{p-1} \Big\|\frac{1}{n^{H-1}}\sum_{k_i=1}^n f_{k_i}^1\otimes f_{k_i}^2\Big\|\\
				&=\sum_{k_1=1}^n \sum_{k_p:|k_1-k_p|\leq s+2}\sqrt{r_{11}(0)r_{22}(0) }\prod_{i=2}^{p-1} \Big\|\frac{1}{n^{H-1}}\sum_{k_i=1}^n f_{k_i}^1\otimes f_{k_i}^2\Big\|\\
				&=C_1 n \sqrt{r_{11}(0)r_{22}(0) }\prod_{i=2}^{p-1} \Big\|\frac{1}{n^{H-1}}\sum_{k_i=1}^n f_{k_i}^1\otimes f_{k_i}^2\Big\|\\
				&=C_1 n \prod_{i=2}^{p-1} \Big\|\frac{1}{n^{H-1}}\sum_{k_i=1}^n f_{k_i}^1\otimes f_{k_i}^2\Big\|
			\end{align*}
			where $C_1$ is a positive constant. Now we have 
			\begin{align*}
				&\sup_n\Big\|\frac 1{n^{H-1}} \sum_{k=1}^n f_k^1\otimes f_{k}^2\Big\|^2=\sup_n \frac{1}{n^{2H-2}}\sum_{k,h=1}^n R(k-h)^{2}\leq \sup_n\frac{C_2}{n^{2H-3-2H+3}} \leq C_2.
			\end{align*}
			Then 
			\begin{align*}
				|A_n|\leq \frac{C_1C_2}{n^{2H-3}}\to 0
			\end{align*}
			when $H>\frac 32$. 
		\end{example}
		\begin{proof}[Proof of Lemma \ref{restrict_sum}]

			In the general case we can write 
			\begin{align*}
				A_n&=	\frac1{n^{p(H-1)}}\sum_{h_1, \ldots, h_p=1}^{\infty}  \sum_{\underset{\exists i:|k_{i+1}-k_i|\leq s+2}{k_1,\ldots, k_p=1}}^n   \prod_{i=1}^p \<f_{ k_i^1}^{r_i^1}\otimes f_{ k_i^2}^{r_i^2}, e_{h_i}\otimes e_{h_i+1}\>\\
				&=\frac{1}{n^{p(H-1)} }  \sum_{\underset{\exists i:|k_{i+1}-k_i|\leq s+2}{k_1,\ldots, k_p=1}}^n  \<f_{k_1^2}^{r_1^2}, f_{k_2^1}^{r_2^1}\>\<f_{k_2^2}^{r_2^2},f_{k_3^1}^{r_3^1}\>\cdots \<f_{k_{p-1}^2}^{r_{p-1}^2}, f_{k_p^1}^{r_{p}^1}\>\<f_{k_p^2}^{r_p^2}, f_{k_1^1}^{r_1^1}\>\\
			\end{align*}

			We also recall that, when $p$ is even, there exists $m=0,\ldots, \frac{p}2$, such that the number of functions $r_{11}$ is equal to the number of functions $r_{22}$ and it is $m$. The number of functions $r_{12}$ and $r_{21}$ is $p-2m$. When $p$ is odd, there exists at least one function equals to a cross correlation function ($r_{12}$ or $r_{21}$). Then there exists $m=0,\ldots, \frac{p-1}2$ such that there are $m$ functions $r_{11}$, $m$ functions $r_{22}$ and $p-2m$ cross correlation functions.

			Without losing in generality, we consider the case $|k_1^1-k_p^2|\leq s+2$. 
			We can see that the procedure in Example \ref{Example1} is completely general, indeed
			\begin{align*}
				& |A_n|\leq \frac {1}{n^{(H-1)p}}\Big|\sum_{\underset{|k_1-k_p|\leq s+2}{k_1,\ldots, k_p=1}}^n \<f_{k_1^1}^{i_{1}^1}, f_{k_p^2}^{i_{p}^2}\>\<f_{k_1^2}^{i_{1}^2}, f_{k_2^1}^{i_2^1}\>\cdots \<f_{k_{p-1}^2}^{i_{p-1}^2}, f_{k_p^1}^{i_p^1}\>\Big|\\
				&=\frac {1}{n^{(H-1)p}}\Big|\sum_{\underset{|k_1-k_p|\leq s+2}{k_1, k_p=1}}^n \<f_{k_1^1}^{i_{1}^1}, f_{k_p^2}^{i_{p}^2}\>
				\sum_{h_1, \ldots, h_{p-1}=1}^{+\infty} \<f_{k_1^2}^{i_{1}^2}, e_{h_1}\> \<f_{k_2^1}^{i_2^1}, e_{h_1}\>\cdots \<f_{k_{p-1}^2}^{i_{p-1}^2}, e_{h_{p-1}}\>\<f_{k_p^1}^{i_p^1}, e_{h_{p-1}}\>\Big|\\
				&=\frac {1}{n^{(H-1)p}}\Big|\sum_{\underset{|k_1-k_p|\leq s+2}{k_1,\ldots, k_p=1}}^n \<f_{k_1^1}^{i_{1}^1}, f_{k_p^2}^{i_{p}^2}\> \sum_{h_1, \ldots, h_{p-1}=1}^{+\infty} \<f_{k_1^2}^{i_{1}^2}, e_{h_1}\> \<f_{k_2^1}^{i_2^1}\otimes f_{k_2^2}^{i_2^2}, e_{h_1}\otimes e_{h_2}\>\times\\
				&\times\<f_{k_{p-1}^1}^{i_{p-1}^1}\otimes f_{k_{p-1}^2}^{i_{p-1}^2}, e_{h_{p-2}}\otimes e_{h_{p-1}}\> \<f_{k_p^1}^{i_p^1}, e_{h_{p-1}}\>\Big|\\
				&=\frac {1}{n^{(H-1)2}}\Big|\sum_{\underset{|k_1-k_p|\leq s+2}{k_1, k_p=1}}^n \<f_{k_1^1}^{i_{1}^1}, f_{k_p^2}^{i_{p}^2}\> \Big\<\cdots\Big( f_{k_1^2}^{i_1^2}\otimes_1\Big(\frac{1}{n^{H-1}}\sum_{k_2=}^n f_{k_2^1}^{i_2^1}\otimes f_{k_2^2}^{i_2^2}\Big)\otimes_1\\
				&\otimes_1\Big(\frac1{n^{H-1}}\sum_{k_{p-1}}^n f_{k_{p-1}^1}^{i_{p-1}^1}\otimes f_{k_{p-1}^2}^{i_{p-1}^2}\Big)\Big), f_{k_p^1}^{i_p^1}\Big\> \\
				&\leq \frac {1}{n^{(H-1)2}}\sum_{\underset{|k_1-k_p|\leq s+2}{k_1, k_p=1}}^n   \|f_{k_1^1}^{i_{1}^2}\|\| f_{k_p^2}^{i_{p}^1}\| \Big\|\frac{1}{n^{H-1}}\sum_{k_2=1}^n f_{k_2^1}^{i_2^1}\otimes f_{k_2^2}^{i_2^2}\Big\|\cdots\Big\|\frac1{n^{H-1}}\sum_{k_{p-1}}^n f_{k_{p-1}^1}^{i_{p-1}^1}\otimes f_{k_{p-1}^2}^{i_{p-1}^2}\Big\|\\
				&\leq \frac{C}{n^{(H-1)2}L^p(n)} \sum_{\underset{|k_1-k_p|\leq s+2}{k_1, k_p=1}}^n  1\leq \frac{C_2}{n^{2H-3}}\to 0
			\end{align*}
			when $H>\frac 32$.

		\end{proof}

		Now we want to compute 
		$$
		\lim_{n\to+\infty} \frac{1}{n^{p(H-1)}}\sum_{k_1,\ldots, k_p=1, \forall i |k_{i+1}-k_i|\geq s+3}^n \prod_{i=1}^{p} \<f_{k_{i}^2}^{r_i^2}, f_{k_{i+1}^1}^{r_{i+1}^1}\>
		$$
		
		We need to prove the following technical lemmas.
		\begin{lemma}\label{limitLL}
			Let $L_{12}, L_{21}, L_{11}$ and $L_{22}$ be the following functions:
			$$
			L_{12}(|k|)=\frac{r_{12}(|k|)}{|k|^{H-2}}\quad{}
			L_{21}(|k|)=\frac{r_{21}(|k|)}{|k|^{H-2}}
			$$
			$$
			L_{11}(|k|)=\frac{r_{11}(k)}{|k|^{2H_1-2}}
			\quad{}
			L_{22}(|k|)=\frac{r_{22}(k)}{|k|^{2H_2-2}}.
			$$
			Then $|L_{12}|, |L_{21}|, L_{11}, L_{22}$ are slowly varying functions when $H_1, H_2>\frac 12$. 
			Additionally, 
			\begin{align*}
				&	\lim_{k\to \infty} L_{11}(k)= \frac{2 H_1(2H_1-1)}{\a_1^{2-2H_1}\Gamma(2H_1+1)}\\
				&	\lim_{k\to \infty} L_{22}(k)= \frac{2 H_2(2H_2-1)}{\a_2^{2-2H_2}\Gamma(2H_2+1)}\\
				&	\lim_{k\to \infty} L_{12}(k)= \frac{(\rho+\eta_{12})}{\a_1^{1-H_1}\a_2^{1-H_2}} \frac{H(H-1)}{\sqrt{\Gamma(2H_1+1)\Gamma(2H_2+1)} }\\
				&\lim_{k\to \infty} L_{21}(k)= \frac{(\rho-\eta_{12})}{\a_1^{1-H_1}\a_2^{1-H_2}} \frac{H(H-1)}{\sqrt{\Gamma(2H_1+1)\Gamma(2H_2+1)} }.
			\end{align*} 
			
		\end{lemma}
		\begin{proof}
			Let us recall that a measurable function $L:(0,+\infty)\to (0,+\infty)$ is slowly varying (at infinity) when for all $a>0$ 
			$$
			\lim_{t\to \infty} \frac{L(at)}{L(t)}=1.
			$$
			Clearly, if $\lim_{t\to\infty} L(t)=\ell >0$ then $L$ is slowly varying. Then, by Theorem \ref{Decay1} we have
			$$\lim_{k\to +\infty} L_{ii}(k)=\lim_{k\to+\infty}\frac{r_{ii}(k)}{k^{2H_i-2}}=\lim_{k\to +\infty} \frac{\E[Y_k^i Y_0^i] }{k^{2H_i-2}\Var(Y_0^i)}=\frac{2 H_i(2H_i-1)}{\a_i^{2-2H_i}\Gamma(2H_i+1)}>0$$
			when $H_i>\frac 12$. 
			For the cross correlation the same relation does not hold. In general, we can deduce by Theorem \ref{cov}
			$$
			L_{ij}(k)=\frac{r_{ij}(k)}{k^{H-2}}=r_{ij}(k)k^{2-H}
			$$
			but 
			\begin{align*}
			&\lim_{k\to +\infty} L_{ij}(k)=\lim_{k\to +\infty} \frac{r_{ij}(k)}{k^{H-2}}=\lim_{k\to+\infty} \frac{\E[Y_k^i Y_0^j]}{k^{H-2}\sqrt{\Var(Y_0^1 )\Var(Y_0^2)}}		\\
			&=\frac{(\rho+\eta_{ij})}{\a_1^{1-H_1}\a_2^{1-H_2}} \frac{H(H-1)}{\sqrt{\Gamma(2H_1+1)\Gamma(2H_2+1)} }
			\end{align*}
			and even if $H>1$, if $\rho-\eta_{ij}<0$ then the limit is negative. But 
			$$
			\lim_{k\to \infty} |L_{ij}(k)|=\frac{|\rho+\eta_{ij}|}{\a_1^{1-H_1}\a_2^{1-H_2}} \frac{H|H-1|}{\sqrt{\Gamma(2H_1+1)\Gamma(2H_2+1)} } 
			$$
			then the absolute value of $L_{ij}$ is slowly varying.
			
		\end{proof}
		%
		%
		%

		We observe that the functions $L_{ij}$ are continuous, and there exist $\lim_{x\to \infty} |L_{ij}|$. Additionally, we have $\lim_{h\to 0} |L_{ij}(h)|=0$. Hence, the absolute value of $L_{ij}$ is bounded on $\R_+$.

	\noindent	\begin{proof}[Proof of Theorem \ref{cumulantNC}]
			
			We need to compute 
			\begin{align*}
				&\lim_{n\to +\infty} \frac{1}{n^{p(H-1)} } \sum_{\underset{\forall i |k_{i+1}-k_i|\geq s+3}{k_1,\ldots, k_p=1}}^n\<f_{k_1^2}^{r_1^2}, f_{k_2^1}^{r_2^1}\>\<f_{k_2^2}^{r_2^2},f_{k_3^1}^{r_3^1}\>\cdots \<f_{k_{p-1}^2}^{r_{p-1}^2}, f_{k_p^1}^{r_{p}^1}\>\<f_{k_p^2}^{r_p^2}, f_{k_1^1}^{r_1^1}\>.
			\end{align*}
		 since we have proved in Lemma \ref{restrict_sum} that 
			\begin{align*}
				&\lim_{n\to +\infty} \frac{1}{n^{p(H-1)} } \sum_{\underset{\exists i: |k_{i+1}-k_i|<s+2 }{k_1,\ldots, k_p=1}}^n\<f_{k_1^2}^{r_1^2}, f_{k_2^1}^{r_2^1}\>\<f_{k_2^2}^{r_2^2},f_{k_3^1}^{r_3^1}\>\cdots \<f_{k_{p-1}^2}^{r_{p-1}^2}, f_{k_p^1}^{r_{p}^1}\>\<f_{k_p^2}^{r_p^2}, f_{k_1^1}^{r_1^1}\>\to 0.
			\end{align*}
			
			Let us recall that
			$$
			\<f_{k_{i}^2}^{r_i^2}, f_{k_{i+1}^1}^{r_{i+1}^1}\>=\begin{cases}
				|k_{i+1}^2-k_i^1|^{2H_j-2}L_{jj}(k_{i+1}^2-k_i^1) \quad{r_i^2=r_{i+1}^1=j}\\
				(k_{i+1}^2-k_i^1)^{H-2}L_{21}(k_{i+1}^2-k_i^1) \quad{r_{i}^2=1, r_{i+1}^1=2, k_{i+1}^2>k_i^1 }\\
				(k_i^1-k_{i+1}^2)^{H-2}L_{12}(k_i^1-k_{i+1}^2) \quad{r_{i}^2=1, r_{i+1}^1=2, k_{i+1}^2\leq k_i^1 }\\
				(k_{i+1}^2-k_i^1)^{H-2}L_{12}(k_{i+1}^2-k_i^1) \quad{r_{i}^2=2, r_{i+1}^1=1, k_{i+1}^2>k_i^1 }\\
				(k_i^1-k_{i+1}^2)^{H-2}L_{21}(k_i^1-k_{i+1}^2) \quad{r_{i}^2=2, r_{i+1}^1=1, k_{i+1}^2\leq k_i^1 }.
			\end{cases}
			$$
			
			Then we can write
			\begin{align*}
				&\frac{1}{n^{p(H-1)}}\sum_{k_1,\ldots, k_p=1, \forall i |k_{i+1}-k_i|\geq s+3}^n \prod_{i=1}^{p} \<f_{k_{i}^2}^{r_i^2}, f_{k_{i+1}^1}^{r_{i+1}^1}\>\\
				&=\frac{1}{n^p} \sum_{\underset{\forall i |k_{i+1}-k_i|\geq s+3}{k_1,\ldots, k_p=1}}^n \prod_{i=1}^{p}\Big(\frac{|k_{i+1}^2-k_i^1|}{n}\Big)^{-\b_i} L^i(k_{i+1}^2-k_i^1) 
			\end{align*}
			where $\b_i\in\{2-2H_1, 2-2H_2, 2-H\}$ for all $i$. $L^i$ follows the above notation. We notice that $\sum_{i=1}^n \b_i=p(2-H)$, since the number of $r_{11}$ is equal to the number of $r_{22}$.
			We can also write
			\begin{align*}
				&\frac{1}{n^p} \sum_{\underset{\forall i |k_{i+1}-k_i|\geq s+3}{k_1,\ldots, k_p=1}}^n \prod_{i=1}^{p}\Big(\frac{|k_{i+1}^2-k_i^1|}{n}\Big)^{-\b_i} L^i(k_{i+1}^2-k_i^1) \\
				&=\int_{\R_+^p}\sum_{\underset{\forall i |k_{i+1}-k_i|\geq s+3}{k_1,\ldots, k_p=1}}^n  \prod_{i=1}^{p} \Big(\frac{|k_{i+1}^2-k_i^1|}{n}\Big)^{-\b_i} L^i(k_{i+1}^2-k_i^1)\1_{[\frac{k_1-1}n, \frac{k_1}n)}(x_1) \cdots  \1_{[\frac{k_p-1}n, \frac{k_p}n)}(x_p) dx\\
				&=\int_{\R_+^p} \lambda_n(x_1,\ldots, x_p)dx_1\ldots dx_p,
			\end{align*}
			where $\lambda_n$ denotes the sequence of the integrands. 
			Now we notice that
			$$
			n x_{i+1}-nx_i -1\leq k_{i+1}-k_i\leq  n x_{i+1}-nx_i +1,
			$$
			then, for fixed $x_1,\ldots, x_p\in[0,1]$, $\underset{n\to\infty}{\lim} |k_{i+1}-k_i|=+\infty$ and $\lim_{n\to\infty}|k_{i+1}^2-k_i^1|=|x_{i+1}-x_i|$. 
			Then 
			$\underset{n\to\infty}{\lim}  L^i(k_{i+1}^2-k_i^1)= \ell_i$, where $\ell_i$ depends on the function $L^i$ as in Lemma \ref{limitLL}. In particular 
			$$
			\lim_{n\to \infty} L^i(k_{i+1}^2-k_i^1)=\begin{cases}
				\ell_{jj}=\frac{2H_j(2H_j-1)}{\a_j^{2-2H_j}\Gamma(2H_j+1)} \,\,\,\,\,\,\,\,\,\,\,\,\,\,\,\,\,\,\,\,\,\,\,\,\,\,\,\,\,\qquad{\text{if  }L^i=L_{jj}}\\
				\ell_{12}=\frac{ (\rho+\eta_{12})H(H-1)}{\a_1^{1-H_1}\a_2^{1-H_2}\sqrt{\Gamma(2H_1+1)\Gamma(2H_2+1)}}\quad{\text{if  }L^i=L_{12}}\\
				\ell_{21}=\frac{ (\rho-\eta_{12})H(H-1)}{\a_1^{1-H_1}\a_2^{1-H_2}\sqrt{\Gamma(2H_1+1)\Gamma(2H_2+1)}}\quad{\text{if  }L^i=L_{21}}
			\end{cases}
			$$

			Now we observe that, since $k_{i+1}-k_i\geq s+3$ (or $k_i-k_{i+1}\geq s+3$) for all $i$, and $x_i\in\Big[\frac{k_i-1}{n}, \frac{k_i}n\Big)$ then
			$$
			\frac {s+3}n \leq \frac{k_{i+1}-k_i}{n}\leq x_{i+1}+\frac 1n-x_i.
			$$ 
			Hence, $x_{i+1}-x_i \geq \frac {s+2}n$ and for $n$ sufficiently large
			\begin{align*}
				&\frac{k_{i+1}-k_{i}}n \geq x_{i+1}-x_i-\frac 1n\geq \frac{s+1}{s+2}(x_{i+1}-x_i)\geq \frac{x_{i+1}-x_i}{s+2}\\
				&\frac{k_{i+1}-k_{i}+s}n \geq x_{i+1}-x_i+\frac{s-1}n \geq \frac{s+1}{s+2}(x_{i+1}-x_i)\geq \frac{x_{i+1}-x_i}{s+2}\\
				&\frac{k_{i+1}-k_{i}-s}n \geq x_{i+1}-x_i-\frac{s+1}n\geq \frac{x_{i+1}-x_i}{s+2}. 
			\end{align*}

			The functions $L^i$ are continuous in $(0,+\infty)$, when $x\to 0$, $L^i(x)=0$, then there exists $M_i=\max_{(0,+ \infty)} L^i(x)$, then 
			$$
			|\lambda_n(x_1,\ldots, x_p)| \leq C(s+2)^{p(2-H)} \1_{[0,1]^p}(x_1,\ldots, x_p) \prod_{i=1}^p |x_{i+1}-x_i|^{-\b_i}
			$$
			that is a $L^1$ function, because $\b_i<1$ for all $i=1,\ldots, p$. Then we can apply Lebesgue's theorem, and we can write the limit. We observe that
			$$
			\lim_{n\to+\infty}|k_{i+1}^2-k_i^1|\geq \lim_{n\to+\infty}  n \, C\, |x_{i+1}-x_i|=+\infty
			$$
			then 
			\begin{align*}
				&\lim_{n\to +\infty} \lambda_n(x_1,\ldots, x_p)\\
				&= \sum_{\underset{\forall i |k_{i+1}-k_i|\geq s+3}{k_1,\ldots, k_p=1}}^n \prod_{i=1}^{p} \Big(\frac{|k_{i+1}^2-k_i^1|}{n}\Big)^{-\b_i} L^i(k_{i+1}^2-k_i^1)\1_{[\frac{k_1-1}n, \frac{k_1}n)}(x_1) \cdots  \1_{[\frac{k_p-1}n, \frac{k_p}n)}(x_p) \\
				&= \prod_{i=1}^p\Big ( \ell_{21}\1_{\underset{ x_{i+1}>x_i}{\b_i=2-H}}+\ell_{12}\1_{\underset{x_{i+1}\leq x_1}{\b_i=2-H} }+\ell_{11}\1_{\b_i=2-2H_1}+\ell_{22}\1_{\b_i=2-2H_2} \Big)|x_{i+1}-x_i|^{-\b_i}\1_{[0,1]}(x_i) .
			\end{align*}
			Finally we can conclude that
			\begin{align*}
				&\lim_{n\to\infty} \frac{1}{n^{p(H-1)}} \sum_{k_1,\ldots, k_p=1}^n \sum_{h_1,\ldots, h_p=1}^{\infty}\prod_{i=1}^p \<g_n, e_{h_i}\otimes e_{h_{i+1}}\>\\
				&=\hat C(a_1, a_2, a_3)\sum_{i_1,\ldots, i_p=1}^2 \sum_{\underset{i_2\neq j_1,\ldots, i_p\neq j_{p-1}, i_1\neq j_p}{j_1,\ldots, j_p=1}}^2 \int_{[0,1]^p}  z_{i_1 j_1}(x_1, x_2) z_{i_2 j_2}(x_2, x_3)\ldots z_{i_p j_p}(x_p, x_1) dx_1\cdots dx_p
			\end{align*}
			where 
			$$
			z_{11}(x,y)= \frac{2H_1(2H_1-1)}{\a_1^{2-2H_1}} |x-y|^{2H_1-2}
			$$
			$$
			z_{22}(x,y)=\frac{2H_2(2H_2-1)}{\a_2^{2-2H_2}}|x-y|^{2H_2-2},
			$$
			$$
			z_{12}(x, y)=\begin{cases}
			\frac{ (\rho-\eta_{12})H(H-1)}{\a_1^{1-H_1}\a_2^{1-H_2}\sqrt{\Gamma(2H_1+1)\Gamma(2H_2+1)}} (x-y)^{H-2} \quad{x>y}\\
				\frac{ (\rho+\eta_{12})H(H-1)}{\a_1^{1-H_1}\a_2^{1-H_2}\sqrt{\Gamma(2H_1+1)\Gamma(2H_2+1)}} (y-x)^{H-2} \quad{x\leq y}
			\end{cases}
			$$
			$$
			z_{21}(x, y)=\begin{cases}
			\frac{ (\rho+\eta_{12})H(H-1)}{\a_1^{1-H_1}\a_2^{1-H_2}\sqrt{\Gamma(2H_1+1)\Gamma(2H_2+1)}} (x-y)^{H-2} \quad{x>y}\\
			\frac{ (\rho-\eta_{12})H(H-1)}{\a_1^{1-H_1}\a_2^{1-H_2}\sqrt{\Gamma(2H_1+1)\Gamma(2H_2+1)}} (y-x)^{H-2} \quad{x\leq y}
			\end{cases}
			$$
			and $\hat C(a_1, a_2, a_3)=(a_1+a_2+a_3)^p$. This constant derives from the complete computation of the product $\prod_{i=1}^p\<g_n, e_{h_i}\otimes e_{h_{i+1}}\>$.
			Then we can observe that $z_{12}(x,y)=z_{21}(y,x)$ and 
			\begin{align*}
				&\lim_{n\to +\infty} \frac{\Var(\tilde S_n)}{(a_1+a_2+a_3)^2}=\lim_{n\to +\infty}\frac{\kappa_2(\tilde S_n)}{(a_1+a_2+a_3)^2}\\
				&=4\int_{[0,1]^2} z_{11}(x_1, x_2) z_{22}(x_2 x_1) dx_1 dx_2+4\int_{[0,1]^2} z_{12}(x_1, x_2) z_{12}(x_2, x_1) dx_1 dx_2\\
				&+4\int_{[0,1]^2} z_{21}(x_1, x_2) z_{21}(x_2, x_1) dx_1 dx_2+4\int_{[0,1]^2} z_{22}(x_1, x_2) z_{11}(x_2, x_1) dx_1 dx_2\\
				&=8\int_{[0,1]^2} z_{11}(x_1, x_2) z_{22}(x_2 x_1) dx_1 dx_2+8\int_{[0,1]^2} z_{12}(x_1, x_2) z_{12}(x_2, x_1) dx_1 dx_2.
			\end{align*}
			We notice that 
			\begin{align*}
			&\int_{[0,1]^2} z_{12}(x_1, x_2)z_{21}(x_1, x_2) dx_1 dx_2\\
			&=(\rho^2-\eta_{12}^2)\frac{H^2(H-1)^2}{\a_1^{2-2H_1}\a_2^{2-H_2}\Gamma(2H_1+1)\Gamma(2H_2+1)}\int_0^1 \int_0^{x_2} (x_2-x_1)^{2H-4} dx_1 dx_2+\\
			&+(\rho^2-\eta_{12}^2)\frac{H^2(H-1)^2}{\a_1^{2-2H_1}\a_2^{2-H_2}\Gamma(2H_1+1)\Gamma(2H_2+1)}\int_0^1 \int_{x_2}^1 (x_1-x_2)^{2H-4} dx_1 dx_2\\
			&=(\rho^2-\eta_{12}^2)\frac{H^2(H-1)^2}{\a_1^{2-2H_1}\a_2^{2-H_2}\Gamma(2H_1+1)\Gamma(2H_2+1)(2H-3)}\int_0^1 (x_2^{2H-3} +(1-x_2)^{2H-3} )dx_2 \\
			&=(\rho^2-\eta_{12}^2)\frac{2H^2(H-1)^2}{\a_1^{2-2H_1}\a_2^{2-H_2}\Gamma(2H_1+1)\Gamma(2H_2+1)(2H-3)(2H-2)}
			\end{align*}
			and 
			\begin{align*}
			&\int_{[0,1]^2} z_{11}(x_1, x_2)z_{22}(x_2, x_1) dx_1 dx_2\\
			&=-\frac{4H_1 H_2 (2H_1-1)(2H_2-1)}{\a_1^{2-2H_1}\a_2^{2-2H_2}\Gamma(2H_1+1)\Gamma(2H_2+1)} \int_0^1 \int_0^{x_2} (x_2-x_1)^{2H-4} dx_1 dx_2\\
			&\frac{4H_1 H_2 (2H_1-1)(2H_2-1)}{\a_1^{2-2H_1}\a_2^{2-2H_2}\Gamma(2H_1+1)\Gamma(2H_2+1)(2H-3)}\int_0^1 \int_{x_2}^{1} (x_1-x_2)^{2H-4} dx_1 dx_2\\
			&\frac{4H_1 H_2 (2H_1-1)(2H_2-1)}{\a_1^{2-2H_1}\a_2^{2-2H_2}\Gamma(2H_1+1)\Gamma(2H_2+1)(2H-3)} \int_0^1 ( x_2^{2H-3}+(1-x_2)^{2H-3}    ) dx_2\\
			&=\frac{8H_1 H_2 (2H_1-1)(2H_2-1)}{\a_1^{2-2H_1}\a_2^{2-2H_2}\Gamma(2H_1+1)\Gamma(2H_2+1)(2H-3)(2H-2)}. 
			\end{align*}
			Then 
			\begin{align*}
			&\lim_{n\to \infty}\frac{\Var(\tilde S_n)}{(a_1+a_2+a_3)^2}=\frac{16\Big((\rho^2-\eta_{12}^2)H^2(H-1)^2+4H_1H_2(2H_1-1)(2H_2-1)\Big)}{\a_1^{2-2H_1}\a_2^{2-2H_2}\Gamma(2H_1+1)\Gamma(2H_2+1)(2H-3)(2H-2)}.
			\end{align*}
		    The above limit is $0$ when 
		    \begin{align}\label{conditionIper}
		    \eta_{12}^2-\rho^2= \frac{4H_1H_2(2H_1-1)(2H_2-1)}{H^2(H-1)^2}
		    \end{align}
		    and the right-hand side is positive when $H_1+H_2>\frac 32$. The couple $(\rho, \eta_{12})$ that satisfies \eqref{conditionIper} is on an hyperbole. But we have that the couple $(\rho, \eta_{12})$ satisfies condition in \eqref{domain}. We recall that condition in \eqref{domain} is satisfied when $(\rho, \eta_{12})$ belongs to the interior of an ellipse $\frac{\rho^2}{a^2}+\frac{\eta_{12}^2 }{b^2}=1$, where 
		    $$
		    b=\frac{\cos(\frac{\pi }{2}H) \Gamma(H+1)}{\sqrt{\sin(\pi H_1)\sin(\pi H_2)\Gamma(2H_1+1)\Gamma(2H_2+1)}}.
		    $$ 	
		    \begin{figure}[t]\label{fig_ell}
		    	\centering
		    	\includegraphics[width=0.5\linewidth]{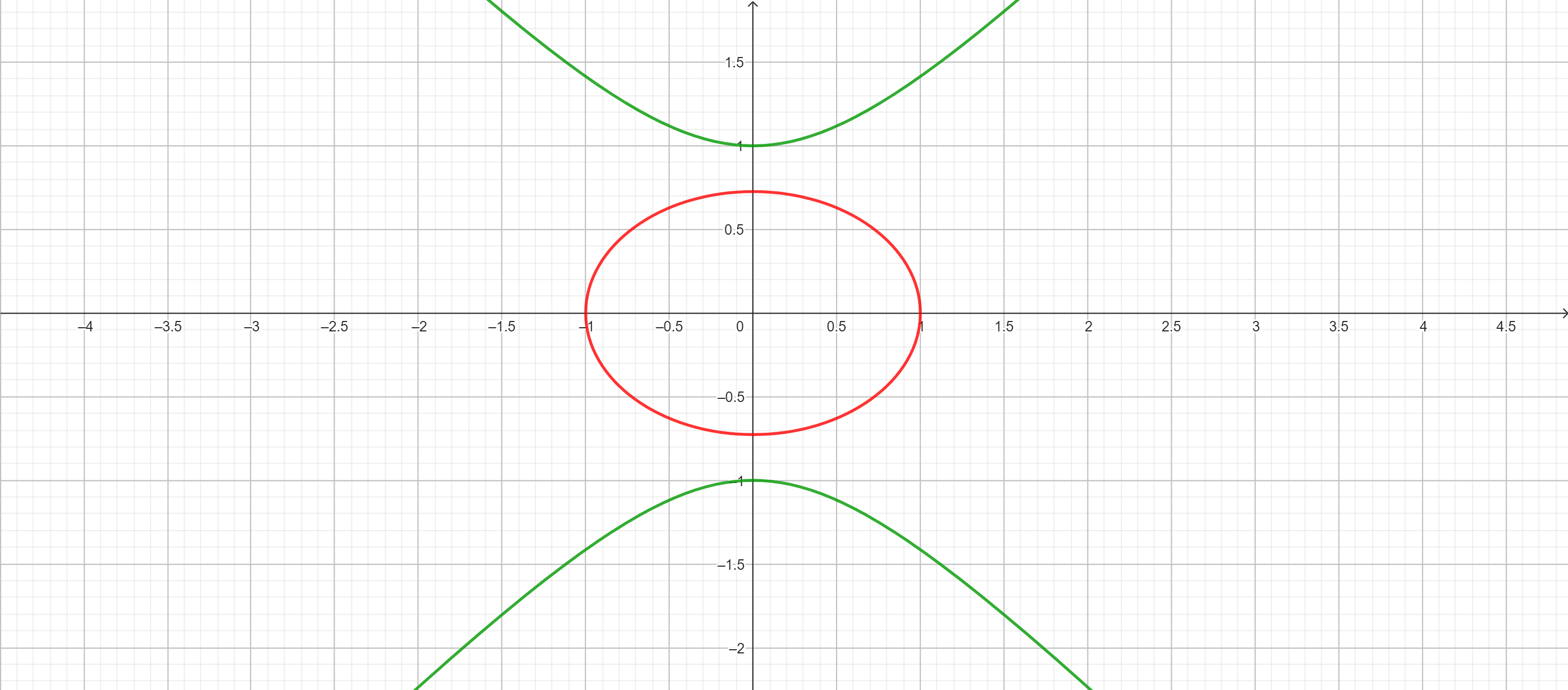}
		    	\small{\caption{Boundary of the condition \eqref{conditionIper} (in green) and boundary of the condition \eqref{domain} (in red) for $H_1=H_2=0.8$. $\rho$ is in $x$-axis and $\eta_{12}$ in the $y$-axis.    
		    	}}
		    \end{figure}
		    Then conditions in \eqref{conditionIper} and \ref{domain} are both satisfied when 
		    $$
		    \frac{\cos(\frac{\pi }{2}H) \Gamma(H+1)}{\sqrt{\sin(\pi H_1)\sin(\pi H_2)\Gamma(2H_1+1)\Gamma(2H_2+1)}}< \frac{2\sqrt{H_1 H_2 (2H_1-1)(2H_2-1)}}{H(H-1)}, 
		    $$
		    and it is verified when $H=H_1+H_2>\frac 32$ (see for example Figure \ref{fig_ell} ). Then we have that $
		    \underset{n\to+\infty }{\lim} \Var(\tilde S_n)\neq 0$.

		\end{proof}

\vspace{1cm}
		\section{Second kind}\label{SecondKind}
	In this section, we present an alternative approach to estimate $\rho$ and $\eta_{12}$. It is important to note that our estimators rely on knowing the parameters of the marginal distributions. In practice, this means we must first estimate the parameters of the marginals before using these estimates to define estimators for $\rho$ and $\eta_{12}$. Consequently, these estimators may be influenced by errors stemming from the estimation of the marginal parameters. This is the case of the mean reversion parameters $\a_1, \a_2$. It is well documented in the literature that log-volatility time series mean-revert slowly. In \cite{WXY23},  authors estimate the fOU process on the logarithm of realized volatility time series and find small mean reversion parameters $\a$. In \cite{GJR18} authors argue that at all time-scales of interest, due to the very slow mean-reversion, the logarithm of realized volatility behaves like a fBm. The authors formalize this concept with an asymptotic result as $\a\to 0$. The same argument is made in \cite{bayer2016pricing} in the context of derivatives pricing.
	For the first kind of estimators $\hat \rho_n$ and $\eta_{12, n}$, small $\a_1, \a_2$ might be problematic for two reasons: 
	\begin{itemize}
		\item[1)] the coefficients in \eqref{coefficients} and \eqref{coefficients2} heavily depend on $\a_1,\a_2$;
		\item[2)] the factor $\a_i$, $i=1,2$ always multiplies the time in covariance functions, affecting the theoretical rates of convergence. 
 	\end{itemize}
	In addition to the reasons related to the small values of $\a_1,\a_2$, the estimation procedure available for the the mean reversion parameter $\a$ in the setting of fOU is documented to be biased in small samples (see for example \cite{WXY23}), and it affects negatively the estimations of $\rho$ and $\eta_{12}$. For these reasons in this section we define and discuss estimators for $\rho$ and $\eta_{12}$ based on Lemma \ref{shorttime}. We will notice that the coefficients of these estimators do not depend on $\a_1, \a_2$, then in this setting it is not necessary to first estimate the mean reversion parameters, and these kind of estimators are not affected by the bias on $\a_1, \a_2$.

	 From Lemma \ref{short-time-inverse} we have the following relations, as $s\to 0$:
		\begin{equation}\label{rel_1}
			\rho=\frac{2\Cov(Y_0^1, Y_0^2)-\Cov(Y_s^1,Y_0^2)-\Cov(Y_0^1,Y_s^2)}{\nu_1\nu_2s^H}+O(s^{\min(1,2-H)})
		\end{equation}
		and 
		\begin{equation}\label{rel_2}
			\eta_{12}=\frac{\Cov(Y_0^1,Y_s^2)-\Cov(Y_s^1,Y_0^2)}{\nu_1\nu_2s^H}	+O(s^{1-H}).
		\end{equation}
		It is important to note that the above relations are valid for $\rho$ when $H>1 $ and they hold for both $\rho$ and $\eta_{12}$ when $H<1$, following by Lemma \ref{short-time-inverse}. We notice that, in Lemma \ref{shorttime}, we provide two different relations, one for $H\neq 1$ and one for $H=1$. In the latter case $H=1$ the asymptotic relation for small lag in Lemma \ref{shorttime} is different and we cannot deduce a asymptotic relation for $\rho$, then we do not study this case. In this chapter, all results are to be understood under the assumption that $H\neq 1$ in the estimation of $\rho$, that $H<1$ in the estimation of $\eta_{12}$. Now, let us consider $n\in\mathbb{N}$ and $T_n>0$, along with the discretization of the time interval $[0,T_n]$ given by $t_k^n=k\frac{T_n}{n}$ for $k=0,\ldots,n$. We denote the time step as $\Delta_n=\frac{T_n}{n}$. Let us define the estimators based on $n$ discrete observations:
		\begin{equation}\label{hat_rho_n}
			\tilde\rho_n=\frac{1}{\nu_1\nu_2n\Delta_n^H} \sum_{k=0}^{n-1}\Big(Y_{(k+1)\Delta_n}^1-Y_{k\Delta_n}^1\Big)\Big(Y_{(k+1)\Delta_n}^2-Y_{k\Delta_n}^2\Big)
		\end{equation}
		and 
		\begin{equation}\label{hat_eta_n}
			\tilde\eta_{12,n}=\frac{1}{\nu_1\nu_2n\Delta_n^H} \sum_{k=0}^{n-1}\Big(Y_{k\Delta_n}^1Y_{(k+1)\Delta_n}^2-Y_{(k+1)\Delta_n}^1Y_{k\Delta_n}^2\Big).
		\end{equation}
	We observe that $\tilde \rho_n$ and $\tilde \eta_{12, n}$ follow from \eqref{rel_1} and \eqref{rel_2} taking $s=\Delta_n$. For this reason, we consider the asymptotic framework of $\Delta_n\to 0$, as $n\to +\infty$, to take advantage of the small lag asymptotic relations \eqref{rel_1} and \eqref{rel_2}, while in estimators in \eqref{estRHO} and \eqref{estETA} the time-lag is fixed and equal to $1$. In the second kind we therefore use high frequency observations. We also notice that $\tilde \rho_n$ and $\eta_{12, n}$ do not depend on $\a_1$ and $\a_2$, then they are not affected by the bias on the estimation of the mean reversion. 
		The estimators $\tilde{\rho}_n$ and $\eta_{12,n}$ exhibit asymptotic unbiasedness.
		\begin{proposition}\label{unbiasEst}
			Let $\tilde\rho_n$ and $\tilde\eta_{12,n}$ be the random variables in \eqref{hat_rho_n} and \eqref{hat_eta_n}. If $T_n\to +\infty$ and $\Delta_n\to 0$ when $n\to \infty$, it follows that
			$$
			\E[\tilde\rho_n] \underset{n\to +\infty}{\to} \rho
			$$
			and 
			$$
			\E[\tilde\eta_{12,n}]\underset{n\to +\infty}{\to} \eta_{12}. 
			$$
			Then $\tilde\rho_n$ and $\eta_{12,n}$ are asymptotically unbiased estimators for $\rho$ and $\eta_{12}$.
		\end{proposition}
		\begin{proof}
			Let us start with $\tilde\rho_n$: from the stationarity of the process $Y=(Y^1,Y^2)$ we have
			\begin{align*}	
				\E[\tilde\rho_n]&=\frac{1}{\nu_1\nu_2n\Delta_n^H}\sum_{k=0}^{n-1} \E\Big[\Big(Y_{(k+1)\Delta_n}^1-Y_{k\Delta_n}^1\Big)\Big(Y_{(k+1)\Delta_n}^2-Y_{k\Delta_n}^2\Big)\Big]\\
				&=\frac{1}{\nu_1\nu_2n\Delta_n^H}\sum_{k=0}^{n-1} \Big(2\Cov(Y_0^1,Y_0^2)-\Cov(Y_{\Delta_n}^1, Y_0^2)-\Cov(Y_0^1,Y_{\Delta_n}^2)\Big)\\
				&=\frac{2\Cov(Y_0^1,Y_0^2)-\Cov(Y_{\Delta_n}^1, Y_0^2)-\Cov(Y_0^1,Y_{\Delta_n}^2)}{\nu_1\nu_2\Delta_n^H}\\
				&=\rho+O(\Delta_n^{\min(1,2-H)})\overset{n\to \infty}{\to} \rho.
			\end{align*} 
			The second part of the statement follows from analogous computation.
			\begin{align*}
				\E[\tilde \eta_{12,n}]&=\frac{1}{\nu_1\nu_2n\Delta_n^{H}}\sum_{k=0}^{n-1} \E[Y_{k\Delta_n}^1Y_{(k+1)\Delta_n}^2-Y_{(k+1)\Delta_n}^1Y_{k\Delta_n}^2]\\
				&=\frac{\Cov(Y_0^1,Y_{\Delta_n}^2)-\Cov(Y_{\Delta_n}^1,Y_0^2)}{\nu_1\nu_2\Delta_n^{H}}=\eta_{12}+O(\Delta_n^{1-H}).
			\end{align*}
			
		\end{proof}

		\begin{remark}
			It is useful to observe that, for $i=1,2$,  
			\begin{equation}\label{Y_k+1}
				Y^i_{(k+1)\Delta_n}=Y_{k\Delta_n}^i e^{-\a_i \Delta_n}+\xi_{(k+1)\Delta_n}^i
			\end{equation}
			where 
			\begin{equation}\label{xi_k}
				\xi_{k\Delta_n}^i= \nu_i \int_{(k-1)\Delta_n}^{k\Delta_n} e^{-k\Delta_n \a_i} e^{\a_i u} dB_u^{H_i}.
			\end{equation}
			Relation \eqref{Y_k+1} allows to represent $\hat\rho_n$ in \eqref{hat_rho_n} in the following way
			\begin{equation}\label{rep_hatrho}
				\hat \rho_n=\frac{1}{\nu_1\nu_2n\Delta_n^H}\sum_{k=0}^{n-1}\Big(Y_{k\Delta_n}^1(e^{-\a_1\Delta_n}-1)+\xi_{(k+1)\Delta_n}^1\Big)\Big(Y_{k\Delta_n}^2(e^{-\a_2\Delta_n}-1)+\xi_{(k+1)\Delta_n}^2\Big).
			\end{equation}
			
		\end{remark}
		
		Let us study the processes $(\xi_{k\Delta_n}^i)_{k,n}$, $i=1,2$. 
		
		\begin{lemma}\label{approx}
			Let $k= 0,\ldots, n$ and $\xi^i_{k\Delta_n}$ be the random variables in \eqref{xi_k}. Then, for $i=1,2$    
			\begin{equation}\label{rap_xi}
				\xi^i_{k\Delta_n}=\nu_i(B_{k\Delta_n}^{H_i}-B_{(k-1)\Delta_n}^{H_i})+R_{k\Delta_n}^i,	\end{equation}
			where $R_{k\Delta_n}^i$ is a random variable whose variance does not depend on $k$ and, when $\Delta_n\to 0$,   
			\begin{equation}\label{var_Rk}
				\Var(R_{k\Delta_n}^i)=C_R^i \Delta_n^{2H_i+2}+o(\Delta_n^{2H_i+2})
			\end{equation}
			where $C_R^i=\frac{2H_i^2-H_i+2}{2(1+H_i)(1+2H_i)} \a_i^2 \nu_i^2$.   
			
		\end{lemma}
		
		\begin{proof}
			Let us begin by computing the variance of $\xi_{k\Delta_n}^i$ using \eqref{Y_k+1}:
			\begin{align*}
				\Var(\xi_{k\Delta_n}^i)&=\Var(Y_{k\Delta_n}^i-Y_{(k-1)\Delta_n}^ie^{-\a_i\Delta_n})\\
				&=\Var(Y_{k\Delta_n}^i)+e^{-2\a_i\Delta_n}\Var(Y_{(k-1)\Delta_n}^i)-2e^{-\a_i\Delta_n}\Cov(Y_{k\Delta_n}^i,Y_{(k-1)\Delta_n}^i )\\
				&=(1+e^{-2\a_i\Delta_n})\Var(Y_0^i)-2e^{-\a_i\Delta_n}\Cov(Y_{\Delta_n}^i,Y_0^i). 
			\end{align*}
			
			Notably, the variance of $\xi_{k\Delta_n}^i$ does not depend on $k$. We utilize Lemma \ref{shorttime1dim} to analyze the asymptotic behavior as $\Delta_n\to 0$:
			\begin{align*}
				&\Var(\xi_{k\Delta_n}^i)=\Big(2-2\a_i\Delta_n+2\a_i^2\Delta_n^{2}-\frac 43 \a_i^3\Delta_n^3+o(\Delta_n^3)\Big) \Var(Y_0^i)-\\
				&-2\Big(1-\a_i\Delta_n+\frac{\a_i^2}2\Delta_n^2-\frac{\a_i^3}6\Delta_n^3+o(\Delta_n^3)\Big)\Big(\Var(Y_0^i)-\frac{\nu_i^2}2\Delta_n^{2H_i}+\frac{\a_i^2}{2}\Var(Y_0^i)\Delta_n^2-\\
				&-\nu_i^2\a_i^2 \frac{1}{4(1+2H_i)(H_i+1)}\Delta_n^{2H_i+2}+ o(\Delta_n^{2H_i+2})\Big)\\
				&=\nu_i^2\Delta_n^{2H_i}-\a_i\nu_i^2 \Delta_n^{1+2H_i}+\nu_i^2\a_i^2 \frac{2H_i^2+3H_i+2}{2(1+2H_i)(H_i+1)}\Delta_n^{2H_i+2}+o(\Delta_n^{2H_i+2}).
			\end{align*}

			Let $R_{k\Delta_n}^i=\xi_{k\Delta_n}^i-\nu_i(B_{k\Delta_n}^{H_i}-B_{(k-1)\Delta_n}^{H_i})$. The variance of $R_{k\Delta_n}^i$ is given by
			\begin{align}\label{R_kk}
				&\Var(R_{k\Delta_n}^i)=\Var(\xi_{k\Delta_n}^i- \nu_i(B_{k\Delta_n}^{H_i}-B_{(k-1)\Delta_n}^{H_i}))\notag\\
				&=\Var(\xi_{k\Delta_n}^i)+\nu_i^2 \Var(B_{k\Delta_n}^{H_i}-B_{(k-1)\Delta_n}^{H_i})-2\nu_i\Cov(\xi_{k\Delta_n}^i,B_{k\Delta_n}^{H_i}-B_{(k-1)\Delta_n}^{H_i}).
			\end{align}
			Here we notice that 
			\begin{align*}
				\xi_{k\Delta_n}^i=\nu_i(B_{k\Delta_n}^{H_i}-e^{-\a_i\Delta_n}B_{(k-1)\Delta_n}^{H_i})-\a_i\nu_ie^{-\a_ik\Delta_n}\int_{(k-1)\Delta_n}^{k\Delta_n} e^{\a_i u} B_u^{H_i}du,  
			\end{align*}
			then 
			\begin{align*}
				&\Cov(\xi_{k\Delta_n}^i,B_{k\Delta_n}^{H_i}-B_{(k-1)\Delta_n}^{H_i})\\
				&=\nu_i\E[(B_{k\Delta_n}^{H_i}-e^{-\a_i\Delta_n}B_{(k-1)\Delta_n}^{H_i})(B_{k\Delta_n}^{H_i}-B_{(k-1)\Delta_n}^{H_i})]\\
				&-\a_i\nu_ie^{-\a_ik\Delta_n}\int_{(k-1)\Delta_n}^{k\Delta_n} e^{\a_iu}\E[B_u^{H_i} (B_{k\Delta_n}^{H_i}-B_{(k-1)\Delta_n}^{H_i})]du\\
				&=\nu_i \Var(B^{H_i}_{k\Delta_n}-B^{H_i}_{(k-1)\Delta_n})+\nu_i(1-e^{-\a_i\Delta_n})\E[B^{H_i}_{(k-1)\Delta_n}(B^{H_i}_{k\Delta_n}-B^{H_i}_{(k-1)\Delta_n})]-\\
				&-\frac{\a_i\nu_i}{2}  e^{-\a_ik\Delta_n}\int_{(k-1)\Delta_n}^{k\Delta_n} e^{\a_iu}\Big((k\Delta_n)^{2H_i}-((k-1)\Delta_n)^{2H_i}-(k\Delta_n-u)^{2H_i}\\
				&+(u-(k-1)\Delta_n)^{2H_i}\Big)du\\
				&=\nu_i\Var(B^{H_i}_{k\Delta_n}-B^{H_i}_{(k-1)\Delta_n})-\frac{\nu_i}2 (1-e^{-\a_i\Delta_n})\Delta_n^{2H_i}+\frac{\nu_i\a_i}{2}  \int_0^{\Delta_n} e^{-\a_i v} v^{2H_i} dv\\
				&-\frac{\nu_i\a_i}{2} e^{-\a_i \Delta_n} \int_0^s e^{\a_iv} v^{2H_i} dv.\\
			\end{align*}
			Being the increment of the fBm $B^{H_i}$ stationary, $\Var(B_{k\Delta_n}^{H_i}-B_{(k-1)\Delta_n}^{H_i})$ does not depend on $k$. Then, returning to \eqref{R_kk}, we can notice that $\Var(R_{k\Delta_n}^i)$ does not depend on $k$. We have that
			\begin{align*}
				&	\Cov(\xi_{k\Delta_n}^i, B_{k\Delta_n}^{H_i}-B_{(k-1)\Delta_n}^{H_i})\\&=\nu_i\Var(B_{k\Delta_n}-B_{(k-1)\Delta_n})-\frac{\nu_i\a_i}2 \Delta_n^{2H_i+1}+\frac{\nu_i\a_i^2}{2}\Delta_n^{2H_i+2}+\frac{\a_i^2 \nu_i}{2H_i+1} \Delta_n^{2H_i+2}+o(\Delta_n^{2H_i+2})\\
				&=\nu_i\Delta_n^{2H_i}-\frac{\nu_i\a_i}2 \Delta_n^{2H_i+1}+\frac{\nu_i\a_i^2}{2}\frac{2H_i}{(2H_i+1)(1+H_i)}\Delta_n^{2H_i+2}+o(\Delta_n^{2H_i+2})
			\end{align*}  
			Finally we have
			\begin{align*}
				\Var(R_{k\Delta_n}^i)=\frac{2H_i^2-H_i+2	}{2(1+H_i)(1+2H_i)}\a_i^2 \nu_i^2\Delta_n^{2H_i+2}+o(\Delta_n^{2H_i+2})
			\end{align*}
			
			Then the statement holds.
			
		\end{proof}

		\begin{lemma}\label{CrossXI}
			For all $k,h \in\N$ there exists a positive constant $c_1^i$ not depending  on $k, h$ such that    
			\begin{align*}
				&|\Cov(\xi_{k\Delta_n}^i, \xi_{h\Delta_n}^i)|\\
				&\leq \nu_i^2\Delta_n^{2H_i} \Big|\big|k-h+1\big|^{2H_i}-2\big|k-h\big|^{2H_i}+\big|k-h-1\big|^{2H_i}\Big|+c^i \Delta_n^{2H_i+1}+o(\Delta_n^{2H_i+1}).
			\end{align*}
			The constants involved in $o(\Delta_n^{2H_i+1})$ don't depend on $k,h$.
		\end{lemma}
		
		\noindent\begin{proof}
			Let fix $h, k \in \N$, then 
			\begin{align*}
				&\Cov(\xi_{k\Delta_n}^i, \xi_{h\Delta_n}^i)=\nu_i^2\Delta_n^{2H_i}\E[(B_k^{H_i}-B_{k-1}^{H_i})(B_h^{H_i}-B_{h-1}^{H_i})]+\nu_i\E[(B_{k\Delta_n}^{H_i}-B_{(k-1)\Delta_n}^1)R_{h\Delta_n}^i]\\
				&+\nu_i\E[(B_{h\Delta_n}^{H_i}-B_{(h-1)\Delta_n}^{H_i})R_{k\Delta_n}^i]+\E[R_{k\Delta_n}^i R_{h\Delta_n}^i]
			\end{align*}
			then, using Cauchy-Schwarz inequality we have  
			\begin{align*}
				|\Cov(\xi_{k\Delta_n}^i, \xi_{h\Delta_n}^i)|&\leq \nu_i^2\Delta_n^{2H_i} \Big|\big|k-h+1\big|^{2H_i}-2\big|k-h\big|^{2H_i}+\big|k-h-1\big|^{2H_i}\Big|\\
				&+2\nu_i \sqrt{ |C_R^i|} \Delta_n^{2H_i+1}+ o(\Delta_n^{2H_i+1}).
			\end{align*}
			
		\end{proof}

		\begin{remark}\label{asymptoticresult}
			Let us recall that, for $H\neq 1$, we have
			\begin{equation}\label{inf_asym}
				\lim_{t\to +\infty}  \frac{r_{ij}(t)}{t^{H-2}}= \frac{\nu_i\nu_j}{\a_i\a_j}\frac{\rho+\eta_{ij}}{2}H(H-1),
			\end{equation}	
			where $r_{12}(t)=\E[Y_t^1 Y_0^2]$ and $r_{21}(t)=\E[Y_0^1 Y_t^2]$, with $t>0$. 
			So for all $\varepsilon >0$ there exists $M_\varepsilon>0$ such that for all $t>M_\varepsilon$, 
			\begin{equation}\label{limit}
				\Big(\frac{\nu_i\nu_j}{\a_i\a_j}\frac{\rho+\eta_{ij}}{2}H(H-1)-\varepsilon\Big)t^{H-2}\leq r_{ij}(t)\leq\Big (\frac{\nu_i\nu_j}{\a_i\a_j}\frac{\rho+\eta_{ij}}{2}H(H-1)+\varepsilon\Big)t^{H-2}
			\end{equation}
			From now on we write \begin{equation}\label{const}
				q_{ij}=\frac{\nu_i\nu_j}{\a_i\a_j}\frac{\rho+\eta_{ij}}{2}H(H-1).
			\end{equation}
		\end{remark}

		\subsection{Asymptotic theory for $\tilde \rho_n$}\label{theory_tilde_rho}
		This section is devoted to compute the asymptotic distribution of a suitable normalization of $(\tilde \rho_n-\rho)$. Let us consider 
		$$
		\tilde \rho_n-\rho=\frac{1}{\nu_1\nu_2n\Delta_n^H}\sum_{k=1}^n (Y_{(k+1)\Delta_n}^1-Y_{k\Delta_n}^1)(Y_{(k+1)\Delta_n}^2-Y_{k\Delta_n}^2)-\rho.
		$$
		Here we underline the assumptions we need to prove the main result of this section, recalling that $H\neq 1$ in all the results of this section.
		\begin{assumption}\label{AssumptionDelta_a}
			For $n\to +\infty$
			\begin{itemize}
				\item[1)] $\Delta_n\to 0$;
				\item[2)] $n\Delta_n\to +\infty$.
		   \end{itemize}
	   \end{assumption}
   \begin{assumption}\label{AssumptionDelta_b}
 For $n\to+\infty$
	        \begin{itemize}
				\item[3)] $n\Delta_n^2 \to 0$;
				\item[4)] $n \Delta_n^{4-2H}\to 0$. 
			\end{itemize}

			We notice that, when $H<1$, assumption $(3)$ implies assumption $(4)$, while when $H>1 $, assumption $(4)$ implies assumption $(3)$. 
		\end{assumption}
		
		The main result of this section is the following. 
		\begin{theorem}\label{CLT_tilde_rho}
			Let $\sigma^2=\underset{n\to +\infty}{\lim} \Var(\sqrt{n}(\tilde \rho_n-\rho))$ and $N\sim\mathcal N(0,\sigma^2)$. Then, for $H<\frac 32$, under Assumption \ref{AssumptionDelta_a} and \ref{AssumptionDelta_b}
			\begin{equation*}
				\sqrt{n}(\tilde \rho_n-\rho) \overset{d}{\to}N
			\end{equation*}
			
		\end{theorem}
		
		The normalization by $\sqrt{n}$ arises from the behavior of the variance of $\tilde{\rho}_n$ as $n$ increases. We will prove that when $H < \frac{3}{2}$, it is of order $O\left(\frac{1}{n}\right)$.
		Let us start denoting 
		\begin{align*}
			R_{ii}^n(k)&=\E[(Y_{(k+1)\Delta_n}^i-Y_{k\Delta_n}^i)(Y_{\Delta_n}^i-Y_{0}^i)]\\
			&=2r_{ii}(k\Delta_n)-r_{ii}((k-1)\Delta_n)-r_{ii}((k+1)\Delta_n),
		\end{align*}
		where $r_{ii}$ is the function in \eqref{auto-isometry}. When $k\to \infty$, then $n\to \infty$ and 
		\begin{align*}
			|R_{ii}^n(k)|&\leq C|2(k\Delta_n)^{2H_i-2}-((k-1)\Delta_n)^{2H_i-2}-((k+1)\Delta_n)^{2H_i-2}|\\
			&\leq C\Delta_n^{2H_i-2}|2k^{2H_i-2}-(k-1)^{2H_i-2}-(k+1)^{2H_i-2}|\leq \tilde C\Delta_n^{2H_i-2} k^{2H_i-4}.
		\end{align*}
		It also holds for the cross-correlation function 
		\begin{align*}
			R_{ij}^n(k)&=\E[(Y_{(k+1)\Delta_n}^i-Y_{k\Delta_n}^i)(Y_{\Delta_n}^j-Y_{0}^j)]\\
			&=2r_{ij}(k\Delta_n)-r_{ij}((k-1)\Delta_n)-r_{ij}((k+1)\Delta_n).
		\end{align*}
		There exists $M>0$ such that, for $k>\frac M{\Delta_n}$, 
		\begin{align*}
			|R_{ij}^n(k)|&\leq C|2(k\Delta_n)^{H-2}-((k-1)\Delta_n)^{H-2}-|(k+1)\Delta_n|^{H-2}|\\
			&\leq C\Delta_n^{H-2}|2k^{H-2}-|k-1|^{H-2}-(k+1)^{H-2}|\leq \tilde C\Delta_n^{H-2} k^{H-4}.
		\end{align*}

		We use \eqref{rep_hatrho} to study the variance of $\tilde \rho_n$. 
		\begin{align*}
			&\Var(\tilde \rho_n)\\
			&\leq  4\Var\Big(\frac{(e^{-\a_1\Delta_n}-1)(e^{-\a_2\Delta_n}-1)}{\nu_1\nu_2n\Delta_n^H}\sum_{k=1}^{n-1} Y_{k\Delta_n}^1  Y_{k\Delta_n}^2 \Big)+4\Var\Big(\frac{e^{-\a_1\Delta_n}-1}{\nu_1\nu_2n\Delta_n^H}\sum_{k=1}^{n-1} Y_{k\Delta_n}^1  \xi_{(k+1)\Delta_n}^2 \Big)\\
			&+4\Var\Big(\frac{e^{-\a_2\Delta_n}-1)}{\nu_1\nu_2n\Delta_n^H}\sum_{k=1}^{n-1} Y_{k\Delta_n}^2 \xi_{(k+1)\Delta_n}^1 \Big)+4\Var\Big(\frac{1}{\nu_1\nu_2n\Delta_n^H}\sum_{k=1}^{n-1} \xi_{(k+1)\Delta_n}^1\xi_{(k+1)\Delta_n}^2 \Big).
		\end{align*}
		
		Let us focus on each summand. 
		
		\begin{theorem}\label{seriesYY}
			Under Assumption \ref{AssumptionDelta_a} and \ref{AssumptionDelta_b}, when $n\to \infty$, then
			\begin{equation}\label{first_variance1}
				\Var\Big(\frac{(e^{-\a_1\Delta_n}-1)(e^{-\a_2\Delta_n}-1)}{n\Delta_n^H}\sum_{k=0}^{n-1}Y_{k\Delta_n}^1Y_{k\Delta_n}^2\Big)\notag= \begin{cases}O\Big(\frac{\Delta_n^{3-2H}}{n}\Big) \,\,\,\quad{\text{if }H<\frac 32}\\O\Big(\frac{\log n}{n}\Big)\,\,\,\,\,\,\,\quad{\text{if }H=\frac 32}\\O\Big(\frac{1}{n^{4-2H}}\Big)\,\,\,\quad{\text{if } H>\frac 32}.
				\end{cases}
			\end{equation}
			and when $H<\frac 32$, 
			\begin{align*}
				&\sqrt{n}\frac{(e^{-\a_1\Delta_n}-1)(e^{-\a_2\Delta_n}-1)}{n\Delta_n^H}\sum_{k=0}^{n-1}Y_{k\Delta_n}^1Y_{k\Delta_n}^2\to 0,\\
			\end{align*} in $L^2(\P)$ and so in probability.
		\end{theorem}
		\begin{proof}
			Let us study $\Var\Big(\frac{1}{n\Delta_n^H}\sum_{k=0}^{n-1}Y_{k\Delta_n}^1Y_{k\Delta_n}^2\Big)$. We have that
			\begin{align*}
				&\Var\Big(\frac{1}{n\Delta_n^H}\sum_{k=0}^{n-1}Y_{k\Delta_n}^1Y_{k\Delta_n}^2\Big)=\frac{1}{n\Delta_n^{2H}}\sum_{|\tau|=0}^{n-1}\Big(1-\frac {|\tau|} n\Big)\Big(r_{11}(\tau\Delta_n)r_{22}(\tau\Delta_n)+r_{12}(\tau\Delta_n)r_{21}(\tau\Delta_n)\Big)
			\end{align*}	
			where $r_{ij}(\tau \Delta_n)=\E[Y_{\tau \Delta_n}^i Y_0^j]$. 
			Remark \ref{asymptoticresult} allows us to compute the series in this way:
			\begin{align*}
				&\Var\Big(\frac{1}{n\Delta_n^H}\sum_{k=0}^{n-1}Y_{k\Delta_n}^1Y_{k\Delta_n}^2\Big)\\
				&\leq \frac{2}{n\Delta_n^{2H}}\sum_{\tau=0}^{\lfloor\frac{M_{\varepsilon}}{\Delta_n}\rfloor}	\Big(1-\frac \tau n\Big)\Big(|r_{11}(\tau\Delta_n)r_{22}(\tau\Delta_n)|+|r_{12}(\tau\Delta_n)r_{21}(\tau\Delta_n)|\Big)\\
				&+\frac{2}{n\Delta_n^{2H}}\sum_{\tau=\lfloor\frac{M_{\varepsilon}}{\Delta_n}\rfloor+1}^{n-1}	\Big(1-\frac \tau n\Big)\Big(|r_{11}(\tau\Delta_n)r_{22}(\tau\Delta_n)|+|r_{12}(\tau\Delta_n)r_{21}(\tau\Delta_n)|\Big)\\
				&\leq 2M_\varepsilon \frac{|r_{11}(0)r_{22}(0)|+|r_{12}(0)r_{21}(0)|}{n\Delta_n^{2H+1}}+\frac{2|(q_{11}+\varepsilon)(q_{22}+\varepsilon)+(q_{12}+\varepsilon)(q_{21}+\varepsilon))|}{n\Delta_n^{4}}\sum_{\tau=\lfloor\frac{M_{\varepsilon}}{\Delta_n}\rfloor+1}^{n-1} \tau^{2H-4}
			\end{align*}
			where $q_{ij}$ are defined in \eqref{const}.
			We apply Lemma \ref{estimate_sum} with $\gamma=4-2H$. Then
			\begin{align*}
				&\Var\Big(\frac{(e^{-\a_1\Delta_n}-1)(e^{-\a_2\Delta_n}-1)}{n\Delta_n^H}\sum_{k=0}^{n-1}Y_{k\Delta_n}^1Y_{k\Delta_n}^2\Big)\\
				&=O\Big(\frac{\Delta_n^{3-2H}}{n}\Big)\mathbbm 1_{H<\frac 32}+O\Big(\frac{\log n}{n}\Big)\mathbbm 1_{H=\frac 32}+O\Big(\frac{1}{n^{4-2H}}\Big)\mathbbm 1_{H>\frac 32}.
			\end{align*}
			Now, let us prove that, for $H<\frac 32$,  
			\begin{align*}
				\E\Big[\Big(\frac{(e^{-\a_1\Delta_n}-1)(e^{-\a_2\Delta_n}-1)}{\sqrt{n}\Delta_n^H}\sum_{k=0}^{n-1}Y_{k\Delta_n}^1Y_{k\Delta_n}^2\Big)^2\Big]\to 0
			\end{align*}
			It follows from 
			\begin{align*}
				&\E\Big[\Big(\frac{(e^{-\a_1\Delta_n}-1)(e^{-\a_2\Delta_n}-1)}{\sqrt{n}\Delta_n^H}\sum_{k=0}^{n-1}Y_{k\Delta_n}^1Y_{k\Delta_n}^2\Big)^2\Big]\\
				&\leq 2\Var\Big(\frac{(e^{-\a_1\Delta_n}-1)(e^{-\a_2\Delta_n}-1)}{\sqrt{n}\Delta_n^H}\sum_{k=0}^{n-1}Y_{k\Delta_n}^1Y_{k\Delta_n}^2\Big)\\
				&+\E\Big[\frac{(e^{-\a_1\Delta_n}-1)(e^{-\a_2\Delta_n}-1)}{\sqrt{n}\Delta_n^H}\sum_{k=0}^{n-1}Y_{k\Delta_n}^1Y_{k\Delta_n}^2\Big]^2\\
				&\leq O(n^{2H-3})+O(n \Delta_n^{4-2H})\to 0
			\end{align*}
			under Assumptions \ref{AssumptionDelta_a} and \ref{AssumptionDelta_b}. 
			
		\end{proof}
		
		\begin{theorem}\label{serieXiY}
			Under Assumptions \ref{AssumptionDelta_a} and \ref{AssumptionDelta_b}, when $n\to +\infty$ then 
			\begin{equation}\label{second_variance2}
				\Var\Big(\frac{(e^{-\a_j\Delta_n}-1)}{n\Delta_n^H}\sum_{k=0}^{n-1} \xi_{(k+1)\Delta_n}^iY_{k\Delta_n}^j \Big)\to 0.
			\end{equation}
			Moreover, when $H<\frac 32$ we have that 
			\begin{align*}
				\frac{\sqrt{n}(e^{-\a_j\Delta_n}-1)}{n\Delta_n^H}\sum_{k=0}^{n-1} \xi_{(k+1)\Delta_n}^iY_{k\Delta_n}^j{\to} 0
			\end{align*}  
			in $L^2(\P)$ and so in probability.

		\end{theorem}
		\begin{proof}
			We have
			\begin{align*}
				&	\E\Big[\Big(\frac{1}{n\Delta_n^H}\sum_{k=0}^{n-1}\Big( \xi_{(k+1)\Delta_n}^1Y_{k\Delta_n}^2-\frac{1}{\Delta_n^H}\E[\xi_{(k+1)\Delta_n}^1Y_{k\Delta_n}^2]\Big) \Big)^2\Big]\\
				&=	\underbrace{\frac{1}{n^2\Delta_n^{2H}}\sum_{k,h=0}^{n-1} \E[\xi_{(k+1)\Delta_n}^1\xi_{(h+1)\Delta_n}^1]\E[Y_{k\Delta_n}^2Y_{h\Delta_n}^2]}_A+\\
				&+\underbrace{\frac{1}{n^2\Delta_n^{2H}}\sum_{k,h=0}^{n-1} \E[\xi_{(k+1)\Delta_n}^1Y_{h\Delta_n}^2]\E[\xi_{(h+1)\Delta_n}^1Y_{k\Delta_n}^2]}_B.
			\end{align*}
			From Lemma \ref{CrossXI}, the summand $A$ can be estimated by
			\begin{align*}
				&|A|\leq \frac{2}{n^2\Delta_n^{2H}} \sum_{k>h=0}^{n-1}  \Big(\nu_1^2 \Delta_n^{2H_1} |(k-h+1)^{2H_1}-2(k-h)^{2H_1}+(k-h-1)^{2H_1} |+c_1^1\Delta_n^{2H_1+1}\\
				&+o(\Delta_n^{2H_1+1})\Big) |\E[Y_{k\Delta_n}^2Y_{h\Delta_n}^2]|+ \frac{\Var(Y_0^2)}{n^2\Delta_n^{2H}} \sum_{k=h=0}^{n-1}  \Big(\nu_1^2 \Delta_n^{2H_1} +o(\Delta_n^{2H_1})\Big)\\
				&=\frac{2}{n\Delta_n^{2H}} \sum_{\tau=1}^{n-1} \Big(1-\frac{\tau+1}{n}\Big) \Big|\nu_1^2 \Delta_n^{2H_1} ((\tau+1)^{2H_1}-2\tau^{2H_1}+(\tau-1)^{2H_1} )+c_1^1\Delta_n^{2H_1+1}\\
				&+o(\Delta_n^{2H_1+1})\Big| |\E[Y_{\tau\Delta_n}^2 Y_0^2]|+ \frac{\var(Y_0^2)}{n^2\Delta_n^{2H}} \sum_{k=h=0}^{n-1}  \Big(\nu_1^2 \Delta_n^{2H_1} +o(\Delta_n^{2H_1})\Big)\\
				&\leq \frac{C_1}{n\Delta_n^{2H_2}} \sum_{\tau=1}^{n-1} \Big(\tau^{2H_1-2}+O(\Delta_n)\Big)|\E[Y_{k\Delta_n}^2Y_{h\Delta_n}^2]| +O\Big(\frac{1}{n\Delta_n^{2H_2}}\Big)\\
				&\leq \frac{C_1 \Var(Y_0^2)}{n\Delta_n^{2H_2}} \sum_{\tau=1}^{[\frac M{\Delta_n}]}\Big(\tau^{2H_1-2}+O(\Delta_n)\Big)+\frac{C_1 }{n\Delta_n^{2}} \sum_{\tau=[\frac M{\Delta_n}]}^{n-1}\Big(\tau^{2H-4}+\tau^{2H_2-2}O(\Delta_n)\Big)+O\Big(\frac{1}{n\Delta_n^{2H_2}}\Big)\\
				&\leq O\Big(\frac 1{n\Delta_n^{2H_2}}\Big)\1_{H_1<\frac 12}+O\Big(\frac{\log n}{n \Delta_n^{2H_2}}\Big)\1_{H_1=\frac 12}+O\Big(\frac{1}{n\Delta_n^{2H_2-1}}\Big)\1_{H_1>\frac 12}+O\Big(\frac{1}{n\Delta_n^{2H-1}}\Big)\1_{H<\frac 32}		\\
				&+O\Big(\frac{\log n}{n}\Big)\1_{H=\frac 32}+O\Big(\frac{1}{n^{4-2H}\Delta_n^2}\Big)\1_{H>\frac 32}+O\Big(\frac 1{n\Delta_n^{2H_2}}\Big)\1_{H_2<\frac 12}+O\Big(\frac{\log n}{n \Delta_n}\Big)\1_{H_2=\frac 12}\\
				&+O\Big(\frac{1}{n^{2-2H_2}\Delta_n}\Big)\1_{H_2>\frac 12}+O\Big(\frac{1}{n\Delta_n^{2H_2}}\Big).
			\end{align*}
			
		It follows that
			\begin{align*}
				&|(e^{-\a_2\Delta_n}-1)^2 A|\to 0.
			\end{align*}

			The summand $B$ can be estimated with the following: for $k\geq h$ we have
			\begin{align*}
				&|\E[\xi_{(k+1)\Delta_n}^1 Y_{h\Delta_n}^2]|=\Big|\nu_1\nu_2\E\Big[\int_{k\Delta_n}^{(k+1)\Delta_n}e^{-\a_1(k+1)\Delta_n} e^{\a_1 u} dB_u^{H_1}\int_{-\infty}^{h\Delta_n} e^{-\a_2h\Delta_n} e^{\a_2 v} dB_v^{H_2}\Big]\Big|\\
				&=\Big|\nu_1\nu_2 H(H-1)\frac{\rho-\eta_{12}}{2}e^{-\a_1(k+1)\Delta_n-\a_2h\Delta_n}\int_{k\Delta_n}^{(k+1)\Delta_n} e^{\a_1u} \int_{-\infty }^{h\Delta_n} e^{\a_2 v} (u-v)^{H-2}dvdu\Big|\\
				&= \nu_1\nu_2 |H(H-1)\frac{\rho-\eta_{12}}{2}|e^{-\a_1\Delta_n}\int_{0}^{\Delta_n} e^{\a_1u} \int_{-\infty }^{0} e^{\a_2 v} (u-v+(k-h)\Delta_n)^{H-2}dvdu\\
				&\leq \nu_1\nu_2 |H(H-1)|\frac{\rho-\eta_{12}}{2}e^{-\a_1\Delta_n}\int_{0}^{\Delta_n} e^{\a_1u} \int_{-\infty }^{0} e^{\a_2 v} (u-v)^{H-2}dvdu.
			\end{align*}
			By Lemma \ref{Short_time_integral}, we have that 
			$$
			\int_{0}^{\Delta_n} e^{\a_1 u}\int_{-\infty}^0 e^{\a_2 v} (u-v)^{H-2}dvdu =O(\Delta_n^{1\wedge H}).
			$$
			Then 	
			
			\begin{align*}
				&|e^{-\a_2\Delta_n}-1||B|=\frac{2C_1}{n^2\Delta_n^{2H}} \sum_{k\geq h=0}^{n-1}\Delta_n^{2\wedge 2H}\leq \frac{2C_2}{\Delta_n^{2H-2H\wedge 2}}=O(\Delta_n^2)\1_{H\leq 1}+O(\Delta_n^{4-2H})\1_{H>1}.
			\end{align*}
			Then 
			\begin{align*}
				&\Var\Big(\frac{e^{-\a_j\Delta_n}-1}{n\Delta_n^H}\sum_{k=0}^{n-1} \xi_{(k+1)\Delta_n}^iY_{k\Delta_n}^j \Big)\to 0.
			\end{align*}
			Let us prove that, for $H<\frac 32$, 
			\begin{align*}
				\sqrt{n}\Big(\frac{e^{-\a_j\Delta_n}-1}{n\Delta_n^H}\sum_{k=0}^{n-1} \xi_{(k+1)\Delta_n}^iY_{k\Delta_n}^j \Big)\to 0
			\end{align*}
			in $L^2(\P)$ and in probability. 
			We have that 
			\begin{align*}
				&\E\Big[\Big(\frac{\sqrt{n}(e^{-\a_j\Delta_n}-1)}{n\Delta_n^H}\sum_{k=0}^{n-1} \xi_{(k+1)\Delta_n}^iY_{k\Delta_n}^j\Big)^2\Big]\\
				& \leq 2\Var\Big((\frac{\sqrt{n}(e^{-\a_j\Delta_n}-1)}{n\Delta_n^H}\sum_{k=0}^{n-1} \xi_{(k+1)\Delta_n}^iY_{k\Delta_n}^j\Big) +2\E\Big[\frac{\sqrt{n}(e^{-\a_j\Delta_n}-1)}{n\Delta_n^H}\sum_{k=0}^{n-1} \xi_{(k+1)\Delta_n}^iY_{k\Delta_n}^j\Big]^2\\
				&\leq O(n\Delta_n^2)\1_{H\leq 1}+O(n\Delta_n^{4-2H})\1_{H>1}.
			\end{align*}
			It converges to $0$ under Assumptions \ref{AssumptionDelta_b}. Then, when $H<\frac 32$,  under Assumption \ref{AssumptionDelta_a} and \ref{AssumptionDelta_b} the second part of the statement follows. 
			
		\end{proof}
		
		\noindent To study $\frac{1}{\nu_1\nu_2n\Delta_n^H}\sum_{k=0}^{n-1} \xi_{(k+1)\Delta_n}^1 \xi_{(k+1)\Delta_n}^2 - \rho$, we rewrite it by using Lemma \ref{approx}:
		\begin{align*}
			& \frac{1}{\nu_1\nu_2n\Delta_n^H}\sum_{k=0}^{n-1} \xi_{(k+1)\Delta_n}^1 \xi_{(k+1)\Delta_n}^2 - \rho\\
			&=\Big(  \frac{1}{\nu_1\nu_2n\Delta_n^H}\sum_{k=0}^{n-1} \Big(\nu_1(B_{(k+1)\Delta_n}^{H_1}-B_{k\Delta_n}^{H_1})+R_{(k+1)\Delta_n}^1\Big) \Big(\nu_2(B_{(k+1)\Delta_n}^{H_2}-B_{k\Delta_n}^{H_2})+R_{(k+1)\Delta_n}^2\Big)- \rho\Big)\\
			&=\frac{1}{n\Delta_n^H}\sum_{k=0}^{n-1}(B_{(k+1)\Delta_n}^{H_1}-B_{k\Delta_n}^{H_1}) (B_{(k+1)\Delta_n}^{H_2}-B_{k\Delta_n}^{H_2})-\rho+\frac{1}{\nu_1n\Delta_n^H}\sum_{k=0}^{n-1} (B_{(k+1)\Delta_n}^{H_2}-B_{k\Delta_n}^{H_2})R_{(k+1)\Delta_n}^1\\
			&+\frac{1}{\nu_2 n\Delta_n^H}\sum_{k=0}^{n-1} (B_{(k+1)\Delta_n}^{H_1}-B_{k\Delta_n}^{H_1})R_{(k+1)\Delta_n}^2+\frac{1}{\nu_1\nu_2n\Delta_n^H}\sum_{k=0}^{n-1}R_{(k+1)\Delta_n}^1R_{(k+1)\Delta_n}^2\\
		\end{align*}		
		\begin{lemma}\label{Remainder}
			For $H\in (0,2)$, under Assumption \ref{AssumptionDelta_a} and \ref{AssumptionDelta_b}, we have
			\begin{align*}
				&\Var\Big(\frac{1}{\nu_1 n\Delta_n^H}\sum_{k=0}^{n-1} (B_{(k+1)\Delta_n}^{H_i}-B_{k\Delta_n}^{H_i})R_{(k+1)\Delta_n}^j\Big) = o(\Delta_n^2)\\
				&\Var\Big(\frac{1}{\nu_1n\Delta_n^H}\sum_{k=0}^{n-1} R^i_{(k+1)\Delta_n} R_{(k+1)\Delta_n}^j\Big)=O(\Delta_n^2)
			\end{align*}
			and 
			\begin{align*}
				&\sqrt{n}\Big(\frac{1}{\nu_1n\Delta_n^H}\sum_{k=0}^{n-1} (B_{(k+1)\Delta_n}^{H_i}-B_{k\Delta_n}^{H_i})R_{(k+1)\Delta_n}^j\Big)\to0\\
				&\sqrt{n}\Big(\frac{1}{\nu_1n\Delta_n^H}\sum_{k=0}^{n-1} R^i_{(k+1)\Delta_n} R_{(k+1)\Delta_n}^j\Big)\to0.
			\end{align*}
			in $L^2(\P)$ and so in probability.
		\end{lemma}
		
		\noindent\begin{proof}
			We have
			\begin{align*}
				&\Var\Big(\frac{1}{\nu_1n\Delta_n^H}\sum_{k=0}^{n-1} (B_{(k+1)\Delta_n}^{H_2}-B_{k\Delta_n}^{H_2})R_{(k+1)\Delta_n}^1\Big)\\
				&=\frac{1}{\nu_1^2 n^2\Delta_n^{2H}}\sum_{k,h=0}^{n-1}\E[(B_{(k+1)\Delta_n}^{H_1}-B_{k\Delta_n}^{H_1})R_{(h+1)\Delta_n}^2]\E[(B_{(h+1)\Delta_n}^{H_1}-B_{h\Delta_n}^{H_1})R_{(k+1)\Delta_n}^2]\\
				&+	\frac{1}{\nu_1^2 n^2\Delta_n^{2H}}\sum_{k,h=0}^{n-1}\E[(B_{(k+1)\Delta_n}^{H_1}-B_{k\Delta_n}^{H_1})(B_{(h+1)\Delta_n}^{H_1}-B_{h\Delta_n}^{H_1})]\E[R_{(h+1)\Delta_n}^2)R_{(k+1)\Delta_n}^2]\\
				&\leq \frac{1}{\nu_1^2 n^2\Delta_n^{2H}}\sum_{k,h=0}^{n-1} \Delta_n^{2H_1}((C_R^2)^2\Delta_n^{2H_2+2}+o(\Delta_n^{2H_2+2}))+\\
				&+\frac{1}{\nu_1^2 n^2 \Delta_n^{2H}}\sum_{k,h=0}^{n-1} \Delta_n^{2H_1} \text{const}|k-h|^{2H_1-2} ((C_R^2)^2\Delta_n^{2H_2+2}+o(\Delta_n^{2H_2+2}))\\
				&= o(\Delta_n^2).
			\end{align*} 
			Moreover
			\begin{align*}
				&\Var\Big(\frac{1}{\nu_1\nu_2n\Delta_n^H}\sum_{k=0}^{n-1}R_{(k+1)\Delta_n}^1R_{(k+1)\Delta_n}^2\Big)=\frac{1}{\nu_1^2\nu_2^2n^2\Delta_n^{2H}}\sum_{k,h=0}^{n-1} \E[R_{(k+1)\Delta_n}^1R_{(h+1)\Delta_n}^2]\E[R_{(h+1)\Delta_n}^1R_{(k+1)\Delta_n}^2]	+\\
				&+\frac{1}{\nu_1^2\nu_2^2n^2\Delta_n^{2H}}\sum_{k,h=0}^{n-1}\E[R_{(k+1)\Delta_n}^1R_{(h+1)\Delta_n}^1]\E[R_{(k+1)\Delta_n}^2R_{(h+1)\Delta_n}^2]\\
				&\leq O(\Delta_n^2).
			\end{align*}
			Then, when $n\Delta_n^2\to 0$ 
			\begin{align*}
				&\Big	\|\sqrt{n}\Big(\frac{1}{\nu_1n\Delta_n^H}\sum_{k=0}^{n-1} (B_{(k+1)\Delta_n}^{H_2}-B_{k\Delta_n}^{H_2})R_{(k+1)\Delta_n}^1\Big)\Big\|^2_{L^2(\P)} \\
				&\leq \Var\Big(\Big(\frac{1}{\nu_1\sqrt{n}\Delta_n^H}\sum_{k=0}^{n-1} (B_{(k+1)\Delta_n}^{H_2}-B_{k\Delta_n}^{H_2})R_{(k+1)\Delta_n}^1\Big)\Big)\\
				&+\E\Big[\Big(\frac{1}{\nu_1\sqrt{n}\Delta_n^H}\sum_{k=0}^{n-1} (B_{(k+1)\Delta_n}^{H_2}-B_{k\Delta_n}^{H_2})R_{(k+1)\Delta_n}^1\Big)\Big]^2\\
				&=o(n\Delta_n^2)\to 0,
			\end{align*}
			and 
			\begin{align*}
				&\Big	\|\sqrt{n}\Big(\frac{1}{\nu_1\nu_2 n\Delta_n^H}\sum_{k=0}^{n-1} R_{(k+1)\Delta_n}^1R_{(k+1)\Delta_n}^2\Big)\Big\|^2_{L^2(\P)} \\
				&\leq \Var\Big(\Big(\frac{1}{\nu_1\sqrt{n}\Delta_n^H}\sum_{k=0}^{n-1} R_{(k+1)\Delta_n}^1R_{(k+1)\Delta_n}^2\Big)\Big)+\E\Big[\Big(\frac{1}{\nu_1\sqrt{n}\Delta_n^H}\sum_{k=0}^{n-1} R_{(k+1)\Delta_n}^1R_{(k+1)\Delta_n}^2\Big)\Big]^2\\
				&=O(n\Delta_n^2)\to 0,
			\end{align*}
			then the statement follows. 
		\end{proof}
		
		\begin{theorem}\label{serieXiXi}
			Under Assumption \ref{AssumptionDelta_a} and \ref{AssumptionDelta_b}, when $n\to\infty$, we have
			\begin{equation}\label{third_variance3}
				\Var\Big(\frac{1}{n\Delta_n^{H}} \sum_{k=0}^{n-1} \xi_{(k+1)\Delta_n}^1\xi_{(k+1)\Delta_n}^2\Big) =\begin{cases} O\Big(\frac 1n\Big) \,\,\,\,\qquad{\text{if }H<\frac 32}\\O\Big(\frac{\log n}{n}\Big)\,\,\,\,\quad{\text{if }H=\frac 32}\\O\Big(\frac{1}{n^{4-2H}}\Big)\quad{\text{if }H>\frac 32}.
				\end{cases}	
			\end{equation}
			Then
			\begin{align*}
				\Big(\frac{1}{n\Delta_n^{H}} \sum_{k=0}^{n-1} \xi_{(k+1)\Delta_n}^1\xi_{(k+1)\Delta_n}^2- \rho\Big)\to 0
			\end{align*}
			in $L^2(\P)$ and so in probability.
			
		\end{theorem}
		\noindent\begin{proof} We have
			\begin{align*}
				\Var\Big(\frac{1}{\nu_1\nu_2n\Delta_n^H}\sum_{k=0}^{n-1}& \xi_{(k+1)\Delta_n}^1 \xi_{(k+1)\Delta_n}^2 - \rho\Big)\\
				&\leq 4\Var\Big(\frac{1}{n\Delta_n^H}\sum_{k=0}^{n-1}(B_{(k+1)\Delta_n}^{H_1}-B_{k\Delta_n}^{H_1}) (B_{(k+1)\Delta_n}^{H_2}-B_{k\Delta_n}^{H_2})-\rho\Big)\\
				&+4\Var\Big(\frac{1}{\nu_1n\Delta_n^H}\sum_{k=0}^{n-1} (B_{(k+1)\Delta_n}^{H_2}-B_{k\Delta_n}^{H_2})R_{(k+1)\Delta_n}^1\Big)\\
			&+4\Var\Big(\frac{1}{\nu_2 n\Delta_n^H}\sum_{k=0}^{n-1} (B_{(k+1)\Delta_n}^{H_1}-B_{k\Delta_n}^{H_1})R_{(k+1)\Delta_n}^2\Big)\\&+4\Var\Big(\frac{1}{\nu_1\nu_2n\Delta_n^H}\sum_{k=0}^{n-1}R_{(k+1)\Delta_n}^1R_{(k+1)\Delta_n}^2\Big).
			\end{align*}
			By Lemma \ref{Remainder}, the last three variances tends to $0$ as $O(\Delta_n^2)$. Then we study the first. We will heavily rely on Lemma \ref{estimate_sum} in the following calculations.
			\begin{align*}
				&\Var\Big(\frac{1}{n\Delta_n^H}\sum_{k=0}^{n-1}(B_{(k+1)\Delta_n}^{H_1}-B_{k\Delta_n}^{H_1}) (B_{(k+1)\Delta_n}^{H_2}-B_{k\Delta_n}^{H_2})-\rho\Big)\\
				&\leq \frac{1}{n^2\Delta_n^{2H}} \sum_{k,h=0}^{n-1}\E[(B_{(k+1)\Delta_n}^{H_1}-B_{k\Delta_n}^{H_1}) (B_{(h+1)\Delta_n}^{H_2}-B_{h\Delta_n}^{H_2})]\E[(B_{(k+1)\Delta_n}^{H_2}-B_{k\Delta_n}^{H_2}) (B_{(h+1)\Delta_n}^{H_1}-B_{h\Delta_n}^{H_1})]\\
				&+\frac{1}{n^2\Delta_n^{2H}} \sum_{k,h=0}^{n-1}\E[(B_{(k+1)\Delta_n}^{H_1}-B_{k\Delta_n}^{H_1}) (B_{(h+1)\Delta_n}^{H_1}-B_{h\Delta_n}^{H_1})]\E[(B_{(k+1)\Delta_n}^{H_2}-B_{k\Delta_n}^{H_2}) (B_{(h+1)\Delta_n}^{H_2}-B_{h\Delta_n}^{H_2})]\\
				&= \frac{1}{n^2} \sum_{k,h=0}^{n-1}\E[(B_{k+1}^{H_1}-B_{k}^{H_1}) (B_{h+1}^{H_2}-B_{h}^{H_2})]\E[(B_{k+1}^{H_2}-B_{k}^{H_2}) (B_{h+1}^{H_1}-B_{h}^{H_1})]\\
				&+\frac{1}{n^2} \sum_{k,h=0}^{n-1}\E[(B_{k+1}^{H_1}-B_{k}^{H_1}) (B_{h+1}^{H_1}-B_{h}^{H_1})]\E[(B_{k+1}^{H_2}-B_{k}^{H_2}) (B_{h+1}^{H_2}-B_{h}^{H_2})]\\
				&=\frac{1}{4n}\sum_{|\tau|=0}^{n-1} \Big((\rho-\eta_{12}\text{sign}(\tau-1))|\tau-1|^H   -2(\rho-\eta_{12}\text{sign}(\tau))|\tau|^H+(\rho-\eta_{12}\text{sign}(\tau+1))|\tau+1|^H\Big)\times \\
				&\times \Big((\rho-\eta_{21}\text{sign}(\tau-1))|\tau-1|^H   -2(\rho-\eta_{21}\text{sign}(\tau))|\tau|^H+(\rho-\eta_{21}\text{sign}(\tau+1))|\tau+1|^H\Big)\\
				&+\frac{1}{n} \sum_{|\tau|=0}^{n-1} \Big(|\tau+1|^{2H_1}-2|\tau|^{2H_1}+|\tau+1|^{2H_1}\Big)\Big(|\tau+1|^{2H_2}-2|\tau|^{2H_2}+|\tau+1|^{2H_2}\Big)\\
				&\leq \frac{C_1}n \sum_{\tau=1}^{n-1} \tau^{2H-4}+O\Big(\frac 1n\Big)=O\Big(\frac 1n\Big)\1_{H<\frac 32}+ O\Big(\frac{\log n}n\Big)\1_{H=\frac 32}+O\Big(\frac{1}{n^{4-2H}}\Big)\1_{H>\frac 32}. 
			\end{align*}

			We notice that if $n\Delta_n^2\to 0$, then $\Delta_n^2=o(\frac 1n)$. Then
			
			\begin{align*}
				\Var\Big(\frac{1}{n\Delta_n^{H}} \sum_{k=0}^{n-1} \xi_{(k+1)\Delta_n}^1\xi_{(k+1)\Delta_n}^2\Big) = O\Big(\frac 1n\Big) \mathbbm 1_{H<\frac 32}+O\Big(\frac{\log n}{n}\Big)\mathbbm 1_{H=\frac 32}+O\Big(\frac{1}{n^{4-2H}}\Big)\mathbbm 1_{H>\frac 32}.
			\end{align*}
			
		\end{proof}
		
		\noindent The consistency of $\tilde \rho_n$ follows immediately.
		\begin{theorem}
			For all $H\in(0,2)$, under Assumptions \ref{AssumptionDelta_a}, when $n\to \infty$,  
			$$
			\tilde \rho_n\to \rho
			$$
			in $L^2(\P)$ and so in probability.
		\end{theorem}
		\begin{proof}
			We have
			\begin{align*}
				&	\E[(\tilde \rho_n- \rho)^2]=\E\Big[\Big(\frac{(e^{-\a_1\Delta_n}-1)(e^{-\a_2\Delta_n}-1)}{\nu_1\nu_2n\Delta_n^H}\sum_{k=1}^{n-1} Y_{k\Delta_n}^1  Y_{k\Delta_n}^2+\frac{e^{-\a_1\Delta_n}-1}{\nu_1\nu_2n\Delta_n^H}\sum_{k=1}^{n-1} Y_{k\Delta_n}^1  \xi_{(k+1)\Delta_n}^2 \\
				&+\frac{e^{-\a_2\Delta_n}-1}{\nu_1\nu_2n\Delta_n^H}\sum_{k=1}^{n-1} Y_{k\Delta_n}^2 \xi_{(k+1)\Delta_n}^1 +\frac{1}{\nu_1\nu_2n\Delta_n^H}\sum_{k=1}^{n-1} \xi_{(k+1)\Delta_n}^1\xi_{(k+1)\Delta_n}^2 -\rho\Big)^2\Big]\\
				&\leq 4\Var\Big(\frac{1}{n\Delta_n^{H}} \sum_{k=0}^{n-1} \xi_{(k+1)\Delta_n}^1\xi_{(k+1)\Delta_n}^2\Big)+4\E\Big[\Big(\frac{(e^{-\a_1\Delta_n}-1)(e^{-\a_2\Delta_n}-1)}{\nu_1\nu_2n\Delta_n^H}\sum_{k=1}^{n-1} Y_{k\Delta_n}^1  Y_{k\Delta_n}^2\Big)^2\Big]\\
				&+4\E\Big[\Big(\frac{e^{-\a_1\Delta_n}-1}{\nu_1\nu_2n\Delta_n^H}\sum_{k=1}^{n-1} Y_{k\Delta_n}^1  \xi_{(k+1)\Delta_n}^2\Big)^2\Big]+4\E\Big[\Big(\frac{e^{-\a_2\Delta_n}-1}{\nu_1\nu_2n\Delta_n^H}\sum_{k=1}^{n-1} Y_{k\Delta_n}^2 \xi_{(k+1)\Delta_n}^1\Big)^2\Big]\\
				&\leq 4\Var\Big(\frac{1}{n\Delta_n^{H}} \sum_{k=0}^{n-1} \xi_{(k+1)\Delta_n}^1\xi_{(k+1)\Delta_n}^2\Big)+8\Var\Big(\frac{(e^{-\a_1\Delta_n}-1)(e^{-\a_2\Delta_n}-1)}{\nu_1\nu_2n\Delta_n^H}\sum_{k=1}^{n-1} Y_{k\Delta_n}^1  Y_{k\Delta_n}^2\Big)\\
				&+8\E\Big[\frac{(e^{-\a_1\Delta_n}-1)(e^{-\a_2\Delta_n}-1)}{\nu_1\nu_2n\Delta_n^H}\sum_{k=1}^{n-1} Y_{k\Delta_n}^1  Y_{k\Delta_n}^2\Big]^2+8\Var\Big(\frac{e^{-\a_1\Delta_n}-1}{\nu_1\nu_2n\Delta_n^H}\sum_{k=1}^{n-1} Y_{k\Delta_n}^1  \xi_{(k+1)\Delta_n}^2\Big)\\
				&+8\E\Big[\frac{e^{-\a_1\Delta_n}-1}{\nu_1\nu_2n\Delta_n^H}\sum_{k=1}^{n-1} Y_{k\Delta_n}^1  \xi_{(k+1)\Delta_n}^2\Big]^2+8\Var\Big(\frac{e^{-\a_2\Delta_n}-1}{\nu_1\nu_2n\Delta_n^H}\sum_{k=1}^{n-1} Y_{k\Delta_n}^2 \xi_{(k+1)\Delta_n}^1\Big)\\
				&+8\E\Big[\frac{e^{-\a_2\Delta_n}-1}{\nu_1\nu_2n\Delta_n^H}\sum_{k=1}^{n-1} Y_{k\Delta_n}^2 \xi_{(k+1)\Delta_n}^1\Big]^2.
			\end{align*}
			By \eqref{first_variance1}, \eqref{second_variance2} and \eqref{third_variance3}, the variances converges to $0$, while the expectations are
			\begin{align*}
				&\E\Big[\frac{(e^{-\a_1(\Delta_n)}-1)(e^{-\a_2(\Delta_n)}-1)}{\nu_1\nu_2n\Delta_n^H}\sum_{k=1}^{n-1} Y_{k\Delta_n}^1  Y_{k\Delta_n}^2\Big]^2\leq O\Big(\Delta_n^{4-2H}\Big)\\
				&	\E\Big[\frac{e^{-\a_1(\Delta_n)}-1}{\nu_1\nu_2n\Delta_n^H}\sum_{k=1}^{n-1} Y_{k\Delta_n}^1  \xi_{(k+1)\Delta_n}^2\Big]^2=O(\Delta_n^{2+2\wedge 2H-2H})\\
				&\E\Big[\frac{e^{-\a_2(\Delta_n)}-1}{\nu_1\nu_2n\Delta_n^H}\sum_{k=1}^{n-1} Y_{k\Delta_n}^2 \xi_{(k+1)\Delta_n}^1\Big]^2=O(\Delta_n^{2+2\wedge 2H-2H}),
			\end{align*}
			then $\E[(\tilde \rho_n-\rho)^2]\to 0$. 
			
		\end{proof}
		
	\noindent 	When $H<\frac 32$ we can study the asymptotic distribution of $\sqrt{n}(\tilde \rho_n-\rho)$. First of all, we prove the following theorem, that will be crucial.  
		
		\begin{theorem}\label{CentralLT2fBm}
			Let $H<\frac 32$. Let us denote with 
			$$
			\sigma_n^2=\Var\Big(\frac{1}{n}\sum_{k=0}^{n-1} (B_{k+1}^{H_1}-B_k^{H_1})(B_{k+1}^{H_2}-B_k^{H_2})\Big).
			$$
			Then there exists 
			\begin{align*}
				\sigma^2&:=\lim_{n\to +\infty} \sigma_n^2 =\Var(B_1^{H_1} B_1^{H_2})+2\sum_{k=1}^{+\infty} \Cov\Big(B_1^{H_1} B_1^{H_2}, (B_{k+1}^{H_1}-B_k^{H_1})(B_{k+1}^{H_2}-B_k^{H_2})\Big).\\
			\end{align*}
			Let $N\sim \mathcal N(0,\sigma^2)$. Then, when $n\to +\infty$, 
			\begin{align*}
				&\dW\Big(\sqrt{n}\Big(\frac{1}{n}\sum_{k=0}^{n-1} (B_{k+1}^{H_1}-B_k^{H_1})(B_{k+1}^{H_2}-B_k^{H_2})-\rho\Big), N\Big)\to 0,\\
				&\dTV\Big(\sqrt{n}(\frac{1}{n}\sum_{k=0}^{n-1} (B_{k+1}^{H_1}-B_k^{H_1})(B_{k+1}^{H_2}-B_k^{H_2})-\rho\Big), N\Big)\to 0,\\
				&		d_K\Big(\sqrt{n}\Big(\frac{1}{n}\sum_{k=0}^{n-1} (B_{k+1}^{H_1}-B_k^{H_1})(B_{k+1}^{H_2}-B_k^{H_2})-\rho\Big), N\Big)\to 0.
			\end{align*}
		\end{theorem}
		\noindent\begin{proof} Let us study the asymptotic distribution of
			\begin{equation}\label{Product_increments}
				\sqrt{n}\Big(\frac{1}{n}\sum_{k=0}^{n-1} (B_{k+1}^{H_1}-B_k^{H_1})(B_{k+1}^{H_2}-B_k^{H_2})-\rho\Big)
			\end{equation}
			when $H<\frac 32$. The expectation of $\sqrt{n}\Big(\frac{1}{n}\sum_{k=0}^{n-1} (B_{k+1}^{H_1}-B_k^{H_1})(B_{k+1}^{H_2}-B_k^{H_2})-\rho\Big)$ is exactly $0$. Then
			\begin{align*}
				&\Var\Big(\sqrt{n}\Big(\frac{1}{n}\sum_{k=0}^{n-1} (B_{k+1}^{H_1}-B_k^{H_1})(B_{k+1}^{H_2}-B_k^{H_2})\Big)\\
				&=\frac1n\sum_{k,h=0}^{n-1} \E[(B_{k+1}^{H_1}-B_k^{H_1})(B_{h+1}^{H_2}-B_h^{H_2})]\E[(B_{h+1}^{H_1}-B_h^{H_1})(B_{k+1}^{H_2}-B_k^{H_2})]\\
				&+\frac1n\sum_{k,h=0}^{n-1} \E[(B_{k+1}^{H_1}-B_k^{H_1})(B_{h+1}^{H_1}-B_h^{H_1})]\E[(B_{h+1}^{H_2}-B_h^{H_2})(B_{k+1}^{H_2}-B_k^{H_2})]\\
				&\leq \frac 14\sum_{|\tau|=0}^{n-1}\Big(1-\frac{|\tau|}{n}\Big)\times \\
				&\times\Big((\rho-\eta_{12}\text{sign}(\tau-1))|\tau-1|^H   -2(\rho-\eta_{12}\text{sign}(\tau))|\tau|^H+(\rho-\eta_{12}\text{sign}(\tau+1))|\tau+1|^H\Big)\times \\
				&\times \Big((\rho-\eta_{21}\text{sign}(\tau-1))|\tau-1|^H   -2(\rho-\eta_{21}\text{sign}(\tau))|\tau|^H+(\rho-\eta_{21}\text{sign}(\tau+1))|\tau+1|^H\Big)\\
				&+ \frac 14\sum_{|\tau|=0}^{n-1}\Big(1-\frac{|\tau|}{n}\Big)\Big(|\tau-1|^{2H_1}   -2|\tau|^{2H_1}+|\tau+1|^{2H_1}\Big) \Big(|\tau-1|^{2H_2}   -2|\tau|^{2H_2}+|\tau+1|^{2H_2}\Big).
			\end{align*}
			
			\noindent The sequences in the above sums are in $\ell^1(\Z)$, because both of them as asymptotically equivalent to $\tau^{2H-4}$, and for $H<\frac 32$ then $4-2H>1$. 
			Then the limit of $\Var\Big(\sqrt{n}\Big(\frac{1}{n}\sum_{k=0}^{n-1} (B_{k+1}^{H_1}-B_k^{H_1})(B_{k+1}^{H_2}-B_k^{H_2})\Big)$ exists and it is equal to 
			\begin{align*}
				\lim_{n\to \infty}\Var\Big(\sqrt{n}&\Big(\frac{1}{n}\sum_{k=0}^{n-1} (B_{k+1}^{H_1}-B_k^{H_1})(B_{k+1}^{H_2}-B_k^{H_2})\Big)\\
				&=\Var(B_1^{H_1} B_1^{H_2})+2\sum_{k=1}^{+\infty} \Cov\Big(B_1^{H_1} B_1^{H_2}, (B_{k+1}^{H_1}-B_k^{H_1})(B_{k+1}^{H_2}-B_k^{H_2})\Big).
			\end{align*}
			Here we denote with 
			$Z_{k}^i=B_{k+1}^{H_i}-B_k^{H_i}$. We have that $(Z^1, Z^2)$ is a stationary process, such that $\E[(Z_k^i)^2]=1$, and we denote with 
			\begin{align*}
				&\omega_{ii}(k)=\E[(B_{k+1}^{H_i}-B_k^{H_i})B_1^{H_i}]\\
				&\omega_{ij}(k)=\E[(B_{k+1}^{H_i}-B_k^{H_i})B_1^{H_j}].
			\end{align*}
			
			\noindent Also in this case, our processes $Z^1, Z^2$ can be represented as Wiener It\^o integrals with respect to a bidimensional Gaussian noise $W$. There exist kernels $f_k^i\in L^2(\R; \R^2)$, $k\in \N$ and $i =1,2$, such that
			\begin{equation}\label{Wiener-Ito_fOU_1}
				Z_k^i=\int_{\R} \<f_k^i(x), W(dx)\>_{\R^2}.
			\end{equation}
			Clearly, from It\^o isometry, we can deduce that
			\begin{equation}\label{cross_isometry_1}
				\omega_{12}(h)=\E[Z_{k+h}^1 Z_k^2]=\int_{\R} \<f_{k+h}^1(x), f_k^2 (x)\>_{\R^2} dx=\<f_{k+h}^1, f_k^2\>_{L^2(\R; \R^2)},
			\end{equation}
			and 
			\begin{equation}\label{auto-isometry_2}
				\omega_{ii}(h)=\E[Z_{k+h}^i Z_k^i]=\int_{\R} \<f_{k+h}^i(x), f_k^i (x)\>_{\R^2} dx=\<f_{k+h}^i, f_k^i\>_{L^2(\R;\R^2)}.
			\end{equation}
			for all $h,k \in \N$.
			We use the same notation of Section \ref{Asymptotic_distribution1} because the proof requires the same approach. Being the Gaussian noise $W$ the same that is in Section \ref{Asymptotic_distribution1}, 
			we can see $(Z_k^i)_{k\in\N}$, $i=1,2$, as a sub-field of the Gaussian field $X$ in \eqref{Gaus_fieldX}.
			Then, we consider $Z_k^i=I_1(f_k^i)$. By applying the product formula in Theorem \ref{product_formula} for $p=q=1$, we have 
			\begin{align*}
				&\sqrt{n}\Big(\frac{1}{n}\sum_{k=0}^{n-1} (B_{k+1}^{H_1}-B_k^{H_1})(B_{k+1}^{H_2}-B_k^{H_2})-\rho\Big)\sqrt{n}	=\frac 1{\sqrt{n}} \sum_{k=1}^{n} I_2\Big( f_k^1\tilde\otimes f_k^2\Big).
			\end{align*}
			
			\noindent We denote with
			\begin{equation}\label{new_kernel}
				\theta_n=\frac 1{\sqrt{n}} \sum_{k=1}^{n} f_k^1\tilde\otimes f_k^2.
			\end{equation}
			It follows that
			\begin{align*}
				&  \theta_n \otimes_1 \theta_n=\frac{1}{n}\sum_{k_1, k_2 =1}^n (f_{k_1}^1\tilde\otimes f_{k_1}^2)\otimes_1(f_{k_2}^1\tilde\otimes f_{k_2}^2).
			\end{align*}
			and
			\begin{align*}
				&\|\theta_n\otimes_1\theta_n\|_{L^2(\R;\R^2)}^2=\<\theta_n\otimes_1\theta_n,\theta_n\otimes \theta_n\>\\
				&=\frac 1{n^2}\sum_{k_1,\ldots,k_4=1}^n 
				\big\< \big(f_{k_1}^1\tilde\otimes f_{k_1}^2\big)\otimes_1\big(f_{k_2}^1\tilde\otimes f_{k_2}^2\big),\big(f_{k_3}^1\tilde\otimes f_{k_3}^2\big)\otimes_1\big(f_{k_4}^1\tilde\otimes f_{k_4}^2\big)\big\>.
			\end{align*}
			We expand the first contraction $\xi_n\otimes_1\xi_n$:
			\begin{align*}
				&\big(f_{k_1}^1\tilde\otimes f_{k_1}^2\big)\otimes_1\big(f_{k_2}^1\tilde\otimes f_{k_2}^2\big)\\
				&=\<f_{k_1}^1, f_{k_2}^1\> f_{k_1}^2\otimes f_{k_2}^2 +\<f_{k_1}^1, f_{k_2}^2\> f_{k_1}^2\otimes f_{k_2}^1+\<f_{k_1}^2, f_{k_2}^1\>f_{k_1}^1 \otimes f_{k_2}^2+\<f_{k_1}^2, f_{k_2}^2\>f_{k_1}^1\otimes f_{k_2}^1.\\
			\end{align*}
			Then 
			\begin{align*}
				& 	\|\theta_n \otimes \theta_n\|_{L^2(\R; \R^2)^{\otimes 2}}^2\\
				&\frac1{n^2} \sum_{k_1,\ldots, k_4=0}^{n-1}\Big( \<f_{k_1}^1, f_{k_2}^1\> \<f_{k_1}^2, f_{k_3}^2\> \<f_{k_2}^2, f_{k_4}^2\>\<f_{k_3}^1, f_{k_4}^1\>+\<f_{k_1}^1, f_{k_2}^1\> \<f_{k_1}^2, f_{k_3}^2\> \<f_{k_2}^2, f_{k_4}^1\>\<f_{k_3}^1, f_{k_4}^2\>\\
				&+\<f_{k_1}^1, f_{k_2}^1\> \<f_{k_1}^2, f_{k_3}^1\> \<f_{k_2}^2, f_{k_4}^1\>\<f_{k_3}^2, f_{k_4}^2\>+\<f_{k_1}^1, f_{k_2}^1\> \<f_{k_1}^2, f_{k_3}^1\> \<f_{k_2}^2, f_{k_4}^2\>\<f_{k_3}^2, f_{k_4}^1\>\\
				&+\<f_{k_1}^1, f_{k_2}^2\> \<f_{k_1}^2, f_{k_3}^2\> \<f_{k_2}^1, f_{k_4}^2\>\<f_{k_3}^1, f_{k_4}^1\>+\<f_{k_1}^1, f_{k_2}^2\> \<f_{k_1}^2, f_{k_3}^2\> \<f_{k_2}^1, f_{k_4}^1\>\<f_{k_3}^1, f_{k_4}^2\>\\
				&+\<f_{k_1}^1, f_{k_2}^2\> \<f_{k_1}^2, f_{k_3}^1\> \<f_{k_2}^1, f_{k_4}^1\>\<f_{k_3}^2, f_{k_4}^2\>+\<f_{k_1}^1, f_{k_2}^2\> \<f_{k_1}^2, f_{k_3}^1\> \<f_{k_2}^1, f_{k_4}^2\>\<f_{k_3}^2, f_{k_4}^1\>\\
				&+\<f_{k_1}^2, f_{k_2}^2\> \<f_{k_1}^1, f_{k_3}^2\> \<f_{k_2}^1, f_{k_4}^2\>\<f_{k_3}^1, f_{k_4}^1\>+\<f_{k_1}^2, f_{k_2}^2\> \<f_{k_1}^1, f_{k_3}^2\> \<f_{k_2}^1, f_{k_4}^1\>\<f_{k_3}^1, f_{k_4}^2\>\\
				&+\<f_{k_1}^2, f_{k_2}^2\> \<f_{k_1}^1, f_{k_3}^1\> \<f_{k_2}^1, f_{k_4}^1\>\<f_{k_3}^2, f_{k_4}^2\>+\<f_{k_1}^2, f_{k_2}^2\> \<f_{k_1}^1, f_{k_3}^1\> \<f_{k_2}^1, f_{k_4}^2\>\<f_{k_3}^2, f_{k_4}^1\>\\
				&+\<f_{k_1}^2, f_{k_2}^1\> \<f_{k_1}^1, f_{k_3}^2\> \<f_{k_2}^2, f_{k_4}^2\>\<f_{k_3}^1, f_{k_4}^1\>+\<f_{k_1}^2, f_{k_2}^1\> \<f_{k_1}^1, f_{k_3}^2\> \<f_{k_2}^2, f_{k_4}^1\>\<f_{k_3}^1, f_{k_4}^2\>\\
				&+\<f_{k_1}^2, f_{k_2}^1\> \<f_{k_1}^1, f_{k_3}^1\> \<f_{k_2}^2, f_{k_4}^1\>\<f_{k_3}^2, f_{k_4}^2\>+\<f_{k_1}^2, f_{k_2}^1\> \<f_{k_1}^1, f_{k_3}^1\> \<f_{k_2}^2, f_{k_4}^2\>\<f_{k_3}^2, f_{k_4}^1\>\Big).
			\end{align*}
			It is the same pattern as \eqref{pattern}, then we only have terms of the following form
			\begin{align*}
				&\omega_{11}(k_1-k_2)\omega_{22}(k_2-k_3)\omega_{11}(k_3-k_4)\omega_{22}(k_4-k_1),\\
				&\omega_{12}(k_1-k_2)\omega_{12}(k_2-k_3)\omega_{12}(k_3-k_4)\omega_{12}(k_4-k_1),\\
				&\omega_{12}(k_1-k_2)\omega_{12}(k_2-k_3)\omega_{11}(k_3-k_4)\omega_{22}(k_4-k_1),\\
				&\omega_{12}(k_1-k_2)\omega_{11}(k_2-k_3)\omega_{21}(k_3-k_4)\omega_{22}(k_4-k_1),
			\end{align*}
			up to permutations of the variables. Thus, we restrict our analysis to the sum involving the four above terms.
			But we notice that the asymptotic behaviour of $\omega_{ii}$ is the same as $r_{ii}$ and $\omega_{ij}$ is the same as $r_{ij}$. Being the analysis based only on the asymptotic behaviour of the auto-correlation and cross-correlation functions, we can conclude that, when $H<\frac 32$, 
			$$
			\|\theta_n\otimes_1\theta_n\|^2_{L^2(\R;\R^2)^{\otimes 2}}\to 0
			$$
			that is equivalent to
			\begin{equation}\label{contrBB}
				\kappa_4\Big(\sqrt{n}\Big(\frac{1}{n}\sum_{k=0}^{n-1} (B_{k+1}^{H_1}-B_k^{H_1})(B_{k+1}^{H_2}-B_k^{H_2})-\rho\Big)\Big)\to 0.
			\end{equation}
			The detailed proof is exactly the same as Proposition \ref{contr0}. 
			Let us denote with
			$$
			\sigma_n^2=\Var\Big(\sqrt{n}\Big(\frac{1}{n}\sum_{k=0}^{n-1} (B_{k+1}^{H_1}-B_k^{H_1})(B_{k+1}^{H_2}-B_k^{H_2})-\rho\Big)\Big)
			$$
			and $\sigma^2=\lim_{n\to\infty} \sigma_n^2$. Let us also denote with
			$N\sim\mathcal N(0,\sigma^2)$ and $N_n\sim \mathcal N(0,\sigma_n^2)$. Then, by Proposition \ref{dWTVKgaussian}, 
			\begin{align*}
				&d_{D}\Big(\sqrt{n}\Big(\frac{1}{n}\sum_{k=0}^{n-1} (B_{k+1}^{H_1}-B_k^{H_1})(B_{k+1}^{H_2}-B_k^{H_2})-\rho\Big),N\Big)\\
				&\leq d_{D}\Big(\sqrt{n}\Big(\frac{1}{n}\sum_{k=0}^{n-1} (B_{k+1}^{H_1}-B_k^{H_1})(B_{k+1}^{H_2}-B_k^{H_2})-\rho\Big), N_n\Big)+d_{D}(N_n, N)\\
				&\leq \frac{8}{3\sigma_n^2}\kappa_4\Big(\Var\Big(\sqrt{n}\Big(\frac{1}{n}\sum_{k=0}^{n-1} (B_{k+1}^{H_1}-B_k^{H_1})(B_{k+1}^{H_2}-B_k^{H_2})-\rho\Big)\Big)+ \frac{c_D}{\sigma_n^2 \vee \sigma^2}|\sigma_n^2-\sigma^2|\to 0
			\end{align*}
			where $D\in\{\text{TV},\text{W}, \text{K}\}$. 
			
		\end{proof}

		\noindent \begin{proof}[Proof of Theorem \ref{CLT_tilde_rho}]
			By \eqref{rep_hatrho} we can write
			\begin{align*}
				&\sqrt{n}(\tilde \rho_n-\rho)	=\sqrt{n}\Big(\frac{(e^{-\a_1\Delta_n}-1)(e^{-\a_2\Delta_n}-1)}{\nu_1\nu_2n\Delta_n^H}\sum_{k=0}^{n-1} Y_{k\Delta_n}^1Y_{k\Delta_n}^2\Big)\\
				&+\sqrt{n}\Big(\frac{e^{-\a_1\Delta_n}-1}{\nu_1\nu_2n\Delta_n^H}\sum_{k=0}^{n-1} Y_{k\Delta_n}^1\xi_{(k+1)\Delta_n}^2\Big)+\sqrt{n}\Big(\frac{e^{-\a_2\Delta_n}-1}{\nu_1\nu_2n\Delta_n^H}\sum_{k=0}^{n-1} Y_{k\Delta_n}^2\xi_{(k+1)\Delta_n}^1\Big)\\
				&+\sqrt{n}\Big(\frac{1}{\nu_1\nu_2n\Delta_n^H}\sum_{k=0}^{n-1} \xi_{(k+1)\Delta_n}^1\xi_{(k+1)\Delta_n}^2-\rho\Big)
			\end{align*}
			Under Assumptions \ref{AssumptionDelta_a} and \ref{AssumptionDelta_b}, by Theorem \ref{seriesYY} we have 
			\begin{align*}
				&\sqrt{n}\Big(\frac{(e^{-\a_1\Delta_n}-1)(e^{-\a_2\Delta_n}-1)}{\nu_1\nu_2n\Delta_n^H}\sum_{k=0}^{n-1} Y_{k\Delta_n}^1Y_{k\Delta_n}^2\Big)\overset{p}{\to} 0.
			\end{align*}
			By Theorem \ref{serieXiY} we have
			\begin{align*}
				\sqrt{n}\Big(\frac{e^{-\a_1\Delta_n}-1}{\nu_1\nu_2n\Delta_n^H}\sum_{k=0}^{n-1} Y_{k\Delta_n}^1\xi_{(k+1)\Delta_n}^2\Big)\overset{p}{\to} 0
			\end{align*}
			and 
			\begin{align*}
				\sqrt{n}\Big(\frac{e^{-\a_2\Delta_n}-1}{\nu_1\nu_2n\Delta_n^H}\sum_{k=0}^{n-1} Y_{k\Delta_n}^2\xi_{(k+1)\Delta_n}^1\Big)\overset{p}{\to} 0.
			\end{align*}
			Moreover, by \eqref{rap_xi}, we have
			\begin{align*}
				&\sqrt{n}\Big(\frac{1}{\nu_1\nu_2n\Delta_n^H}\sum_{k=0}^{n-1} \xi_{(k+1)\Delta_n}^1\xi_{(k+1)\Delta_n}^2-\rho\Big)=\frac{1}{\sqrt{n}\Delta_n^H}\sum_{k=0}^{n-1}(B_{(k+1)\Delta_n}^{H_1}-B_{k\Delta_n}^{H_1})(B_{(k+1)\Delta_n}^{H_2}-B_{k\Delta_n}^{H_2})\\
				&+\frac{1}{\nu_2\sqrt{n}\Delta_n^H}\sum_{k=0}^{n-1}(B_{(k+1)\Delta_n}^{H_1}-B_{k\Delta_n}^{H_1})R_{k\Delta_n}^2+\frac{1}{\nu_1\sqrt{n}\Delta_n^H}\sum_{k=0}^{n-1}(B_{(k+1)\Delta_n}^{H_2}-B_{k\Delta_n}^{H_2})R_{k\Delta_n}^1\\
				&+\frac{1}{\nu_1\nu_2n\Delta_n^H}\sum_{k=0}^{n-1} R_{k\Delta_n}^1R_{k\Delta_n}^2.
			\end{align*}
			By proof of Theorem \ref{Remainder} we have that 
			\begin{align*}
				&\frac{1}{\nu_j\sqrt{n}\Delta_n^H}\sum_{k=0}^{n-1}(B_{(k+1)\Delta_n}^i-B_{k\Delta_n}^i)R_{k\Delta_n}^j\overset{p}{\to} 0 \\
				&\frac{1}{\nu_1\nu_2n\Delta_n^H}\sum_{k=0}^{n-1}R_{k\Delta_n}^1R_{k\Delta_n}^2\overset{p}{\to} 0 
			\end{align*}
			and by Theorem \ref{CentralLT2fBm} we have 
			\begin{align*}
				\frac{1}{\sqrt{n}\Delta_n^H}\sum_{k=0}^{n-1}(B_{(k+1)\Delta_n}^{H_1}-B_{k\Delta_n}^{H_1})(B_{(k+1)\Delta_n}^{H_2}-B_{k\Delta_n}^{H_2})\overset{d}{=} \frac{1}{\sqrt{n}}\sum_{k=0}^{n-1}(B_{k+1}^{H_1}-B_{k}^{H_1})(B_{k+1}^{H_2}-B_{k}^{H_2}) \overset{d}{\to} N 
			\end{align*}
			where $N\sim\mathcal N(0,\sigma^2)$.
			This concludes the proof because 
			\begin{align*}
				\lim_{n\to \infty} \Var(\sqrt{n}(\tilde \rho_n-\rho))=\lim_{n\to \infty} \Var\Big(\frac{1}{\sqrt{n}}\sum_{k=0}^{n-1}(B_{k+1}^{H_1}-B_k^{H_1})(B_{k+1}^{H_2}-B_k^{H_2})\Big)=\sigma^2. 	
			\end{align*}
		\end{proof}
		\subsection{Consistency of $\tilde \eta_{12,n}$}
		In this section we prove the consistence of $\tilde \eta_{12, n}$ when $H<1$. The main result is the following.
				\begin{theorem}\label{Cons_eta}
			Let $H<1$. Under Assumption \ref{A} we have that
			$$
			\tilde \eta_{12, n}\to \eta_{12, n}
			$$
			in $L^2(\P)$ and in probability.
		\end{theorem}
		Using \eqref{Y_k+1} we can write
		\begin{align*}
			\tilde \eta_{12,n}=\frac{e^{-\a_2 \Delta_n}-e^{-\a_1\Delta_n }}{\nu_1\nu_2n\Delta_n^H}\sum_{k=0}^{n-1} Y_{k\Delta_n}^1Y_{k\Delta_n}^2+ \frac{1}{\nu_1\nu_2n\Delta_n^H}\sum_{k=0}^{n-1}\Big(Y_{k\Delta_n}^1\xi_{(k+1)\Delta_n}^2-Y_{k\Delta_n}^2\xi_{(k+1)\Delta_n}^1\Big).
		\end{align*}
		\begin{lemma}\label{cETA_1}
			For $H<1$
			\begin{align*}
				&\frac{e^{-\a_2 \Delta_n}-e^{-\a_1\Delta_n }}{\nu_1\nu_2n\Delta_n^H}\sum_{k=0}^{n-1} Y_{k\Delta_n}^1Y_{k\Delta_n}^2\to 0 
			\end{align*}
			in $L^2(\P)$ and in probability.
			
		\end{lemma}
		
		\begin{proof}
			By Lemma \ref{seriesYY} we can deduce that 
			\begin{align*}
				&\Var\Big(\frac{e^{-\a_2 \Delta_n}-e^{-\a_1\Delta_n }}{\nu_1\nu_2n\Delta_n^H}\sum_{k=0}^{n-1} Y_{k\Delta_n}^1Y_{k\Delta_n}^2\Big)=O\Big(\frac{\Delta_n^{1-2H}}{n}\Big)\to 0
			\end{align*}
			when $H<1$. Moreover
			\begin{align*}
				\E\Big[\frac{e^{-\a_2 \Delta_n}-e^{-\a_1\Delta_n }}{\nu_1\nu_2n\Delta_n^H}\sum_{k=0}^{n-1} Y_{k\Delta_n}^1Y_{k\Delta_n}^2\Big]=O(\Delta_n^{1-H})\to 0
			\end{align*}
			then the convergence in $L^2(\P)$ and in probability follows. 
		\end{proof}
		
		We prove the following technical lemma.
		\begin{lemma}\label{Lim_0}
			For $H<1$ we have
			$$
			\lim_{t\to 0}	\frac{\int_0^{t} e^{\a_i u} \int_{-\infty}^0 e^{\a_j v} (u-v)^{H-2}dv du}{t^H}=\frac1{H(1-H)}.
			$$
		\end{lemma}
		\begin{proof}
			We have that
			\begin{align*}
				&\lim_{t\to 0}	\frac{\int_0^{t} e^{\a_i u} \int_{-\infty}^0 e^{\a_j v} (u-v)^{H-2}dv du}{t^H}=\lim_{t\to 0}	 \frac{e^{\a_i t} \int_{-\infty}^0 e^{\a_j v} (t-v)^{H-2}dv }{Ht^{H-1}}\\
				&=\lim_{t\to 0}	 \frac{ \int_{t}^{+\infty} e^{-\a_j v} v^{H-2}dv }{Ht^{H-1}}=\lim_{t\to 0}	 \frac{ \frac{e^{-\a_j t}t^{H-1}}{1-H}+\frac{\a_j}{H-1}\int_{t}^{+\infty}e^{-\a_j v} v^{H-1} dv }{Ht^{H-1}}=\frac{1}{H(1-H)}.
			\end{align*}
		\end{proof}
		\begin{lemma}\label{cETA_2}
			When $H<1$,
			$$
			\frac{1}{\nu_1\nu_2n\Delta_n^H}\sum_{k=0}^{n-1}Y_{k\Delta_n}^1\xi_{(k+1)\Delta_n}^2-Y_{k\Delta_n}^2\xi_{(k+1)\Delta_n}^1\to \eta_{12}
			$$
			in $L^2(\P)$ and in probability.
		\end{lemma}
		\begin{proof}
			We have that
			
			\begin{align*}
				&\E[\xi_{(k+1)\Delta_n}^1 Y_{k\Delta_n}^2]=\nu_1\nu_2\E\Big[\int_{k\Delta_n}^{(k+1)\Delta_n}e^{-\a_1(k+1)\Delta_n} e^{\a_1 u} dB_u^{H_1}\int_{-\infty}^{k\Delta_n} e^{-\a_2h\Delta_n} e^{\a_2 v} dB_v^{H_2}\Big]\\
				&=\nu_1\nu_2 H(1-H)\frac{\rho-\eta_{12}}{2}e^{-\a_1(k+1)\Delta_n-\a_2k\Delta_n}\int_{k\Delta_n}^{(k+1)\Delta_n} e^{\a_1u} \int_{-\infty }^{h\Delta_n} e^{\a_2 v} (u-v)^{H-2}dvdu\\
				&= \nu_1\nu_2 H(1-H)\frac{\rho-\eta_{12}}{2}e^{-\a_1\Delta_n}\int_{0}^{\Delta_n} e^{\a_1u} \int_{-\infty }^{0} e^{\a_2 v} (u-v)^{H-2}dvdu\\
			\end{align*}
			then, by Lemma \ref{Lim_0}
			\begin{align*}	&\lim_{n\to +\infty}\frac{1}{\nu_1\nu_2n\Delta_n^H}\sum_{k=0}^{n-1}\E\Big[Y_{k\Delta_n}^1\xi_{(k+1)\Delta_n}^2-Y_{k\Delta_n}^2\xi_{(k+1)\Delta_n}^1\Big]\\
				&=\lim_{n\to \infty}\frac{H(1-H)}{n\Delta_n^H} \sum_{k=0}^{n-1}\Big( \frac{\rho+\eta_{12}}{2}e^{-\a_2\Delta_n}\int_0^{\Delta_n} e^{\a_2 u} \int_{-\infty}^0 e^{\a_1 v} (u-v)^{H-2} dvdu\\
				&-\frac{\rho-\eta_{12}}{2}e^{-\a_1\Delta_n}\int_0^{\Delta_n} e^{\a_1 u} \int_{-\infty}^0 e^{\a_2 v} (u-v)^{H-2} dvdu\Big)\\
				&\lim_{n\to \infty}\frac{H(1-H)}{\Delta_n^H} \Big( \frac{\rho+\eta_{12}}{2}e^{-\a_2\Delta_n}\int_0^{\Delta_n} e^{\a_2 u} \int_{-\infty}^0 e^{\a_1 v} (u-v)^{H-2} dvdu\\
				&-\frac{\rho-\eta_{12}}{2}e^{-\a_1\Delta_n}\int_0^{\Delta_n} e^{\a_1 u} \int_{-\infty}^0 e^{\a_2 v} (u-v)^{H-2} dvdu\Big)=\eta_{12}.
			\end{align*}
			Moreover 
			\begin{align*}
				&\Var\Big(\frac{1}{\nu_1\nu_2n\Delta_n^H}\sum_{k=0}^{n-1}Y_{k\Delta_n}^1\xi_{(k+1)\Delta_n}^2-Y_{k\Delta_n}^2\xi_{(k+1)\Delta_n}^1\Big)\\
				&=\frac{1}{\nu_1^2\nu_2^2n^2\Delta_n^{2H}}\sum_{k,h=0}^{n-1}\Big( \E[Y_{k\Delta_n}^1\xi_{(h+1)\Delta_n}^2]\E[Y_{h\Delta_n}^1\xi_{(k+1)\Delta_n}^2]-\E[Y_{k\Delta_n}^2\xi_{(h+1)\Delta_n}^1]\E[Y_{h\Delta_n}^2\xi_{(k+1)\Delta_n}^1]\Big)\\
				&+\frac{1}{\nu_1^2\nu_2^2n^2\Delta_n^{2H}}\sum_{k,h=0}^{n-1}\Big(\E[Y_{k\Delta_n}^1 Y^1_{h\Delta_n}]\E[\xi_{(k+1)\Delta_n}^2\xi_{(h+1)\Delta_n}^2]-\E[Y_{k\Delta_n}^2 Y^2_{h\Delta_n}]\E[\xi_{(k+1)\Delta_n}^1\xi_{(h+1)\Delta_n}^1]\Big).\\
			\end{align*}
			We notice that, when $\Delta_n\to 0$ we have
			
			\begin{align*}
				&\E[Y_{k\Delta_n}^1 Y^1_{h\Delta_n}]\E[\xi_{(k+1)\Delta_n}^2\xi_{(h+1)\Delta_n}^2]-\E[Y_{k\Delta_n}^2 Y^2_{h\Delta_n}]\E[\xi_{(k+1)\Delta_n}^1\xi_{(h+1)\Delta_n}^1]\\
				&=\E[Y_{k\Delta_n}^1 Y^1_{h\Delta_n}]\E[(Y_{(k+1)\Delta_n}^2-e^{-\a_2\Delta_n}Y_{k\Delta_n}^2)(Y_{(h+1)\Delta_n}^2-e^{-\a_2\Delta_n}Y_{h\Delta_n}^2)]-\\
				&-\E[Y_{k\Delta_n}^2 Y^2_{h\Delta_n}]\E[(Y_{(k+1)\Delta_n}^1-e^{-\a_1\Delta_n}Y_{k\Delta_n}^1)(Y_{(h+1)\Delta_n}^1-e^{-\a_1\Delta_n}Y_{h\Delta_n}^1)]\\
				&=\E[Y_{k\Delta_n}^1 Y^1_{h\Delta_n}]\E[Y_{k\Delta_n}^2Y_{h\Delta_n}^2](e^{-2\a_2\Delta_n}-e^{-2\a_1\Delta_n})-e^{-\a_2\Delta_n}\E[Y_{k\Delta_n}^1 Y^1_{h\Delta_n}]\E[Y_{k\Delta_n}^2 Y_{(h+1)\Delta_n}^2]\\
				&-e^{-\a_2\Delta_n}\E[Y_{k\Delta_n}^1 Y^1_{h\Delta_n}]\E[Y_{(k+1)\Delta_n}^2 Y_{h\Delta_n}^2]+e^{-\a_1\Delta_n}\E[Y_{k\Delta_n}^2Y^2_{h\Delta_n}]\E[Y_{k\Delta_n}^1 Y_{(h+1)\Delta_n}^1]\\
				&+e^{-\a_1\Delta_n}\E[Y_{k\Delta_n}^2 Y^2_{h\Delta_n}]\E[Y_{(k+1)\Delta_n}^1Y_{h\Delta_n}^1]=O( \Delta_n^2 r_{11}((k-h)\Delta_n)r_{22}((k-h)\Delta_n))
			\end{align*}
			and 
			\begin{align*}
				&\E[Y_{k\Delta_n}^1\xi_{(h+1)\Delta_n}^2]\E[Y_{h\Delta_n}^1\xi_{(k+1)\Delta_n}^2]-\E[Y_{k\Delta_n}^2\xi_{(h+1)\Delta_n}^1]\E[Y_{h\Delta_n}^2\xi_{(k+1)\Delta_n}^1]\\
				&=\E[Y_{k\Delta_n}^1(Y_{(h+1)\Delta_n}^2-e^{-\a_2\Delta_n}Y_{h\Delta_n}^2)]\E[Y_{h\Delta_n}^1(Y_{(k+1)\Delta_n}^2-e^{-\a_2\Delta_n}Y_{k\Delta_n}^2)]\\
				&-\E[Y_{k\Delta_n}^2(Y_{(h+1)\Delta_n}^1-e^{-\a_1\Delta_n}Y_{h\Delta_n}^1)]\E[Y_{h\Delta_n}^2(Y_{(k+1)\Delta_n}^1-e^{-\a_1\Delta_n}Y_{k\Delta_n}^1)]\\
				&=\Big(\E[Y_{k\Delta_n}^1Y_{(h+1)\Delta_n}^2]-e^{-\a_2\Delta_n}\E[Y_{k\Delta_n}^1Y_{h\Delta_n}^2]\Big)\Big(\E[Y_{h\Delta_n}^1Y_{(k+1)\Delta_n}^2]-e^{-\a_2\Delta_n}\E[Y_{h\Delta_n}^1Y_{k\Delta_n}^2]\Big)\\
				&-\Big(\E[Y_{k\Delta_n}^2Y_{(h+1)\Delta_n}^1]-e^{-\a_1\Delta_n}\E[Y_{k\Delta_n}^2Y_{h\Delta_n}^1]\Big)\Big(\E[Y_{h\Delta_n}^2Y_{(k+1)\Delta_n}^1]-e^{-\a_1\Delta_n}\E[Y_{h\Delta_n}^2Y_{k\Delta_n}^1]\Big)\\
				&=\E[Y_{k\Delta_n}^1Y_{(h+1)\Delta_n}^2]\E[Y_{h\Delta_n}^1Y_{(k+1)\Delta_n}^2]-e^{-\a_2\Delta_n}\E[Y_{k\Delta_n}^1Y_{(h+1)\Delta_n}^2]\E[Y_{h\Delta_n}^1Y_{k\Delta_n}^2]\\
				&-e^{-\a_2\Delta_n}\E[Y_{k\Delta_n}^1Y_{h\Delta_n}^2]\E[Y_{h\Delta_n}^1Y_{(k+1)\Delta_n}^2]+e^{-2\a_2\Delta_n}\E[Y_{k\Delta_n}^1Y_{h\Delta_n}^2]\E[Y_{h\Delta_n}^1Y_{k\Delta_n}^2]\\
				&-\E[Y_{k\Delta_n}^2Y_{(h+1)\Delta_n}^1]\E[Y_{h\Delta_n}^2Y_{(k+1)\Delta_n}^1]+e^{-\a_1\Delta_n}\E[Y_{k\Delta_n}^2Y_{(h+1)\Delta_n}^1]\E[Y_{h\Delta_n}^2Y_{k\Delta_n}^1]\\
				&+e^{-\a_1\Delta_n}\E[Y_{k\Delta_n}^2Y_{h\Delta_n}^1]\E[Y_{h\Delta_n}^2Y_{(k+1)\Delta_n}^1]-e^{-2\a_1\Delta_n}\E[Y_{k\Delta_n}^2Y_{h\Delta_n}^1]\E[Y_{h\Delta_n}^2Y_{k\Delta_n}^1]\\
				&=O(\Delta_n^2 r_{12}(|k-h|
				\Delta_n) r_{21}(|k-h|\Delta_n))
			\end{align*}
			then
			\begin{align*}
				&\Var\Big(\frac{1}{\nu_1\nu_2n\Delta_n^H}\sum_{k=0}^{n-1}Y_{k\Delta_n}^1\xi_{(k+1)\Delta_n}^2-Y_{k\Delta_n}^2\xi_{(k+1)\Delta_n}^1\Big)\\
				&\leq \frac{C_1}{n\Delta_n^{2H-2}}\sum_{\tau=0}^{n-1} |r_{12}(\tau\Delta_n)r_{21}(\tau\Delta_n)|+|r_{11}(\tau\Delta_n)r_{22}(\tau\Delta_n)|=O\Big(\frac{1}{n\Delta_n^{2H-1}}\Big)
			\end{align*}
			and for $H<1$, when $n\Delta_n\to \infty$, it converges to $0$.  Then 
			$$
			\frac{1}{\nu_1\nu_2n\Delta_n^H}\sum_{k=0}^{n-1}Y_{k\Delta_n}^1\xi_{(k+1)\Delta_n}^2-Y_{k\Delta_n}^2\xi_{(k+1)\Delta_n}^1 \to \eta_{12}
			$$
			in $L^2(\P)$ and in probability. 
			
		\end{proof}
		
		\noindent Then the consistency of $\tilde \eta_{12, n}$ follows.

		\noindent\begin{proof}{Proof of Theorem \ref{Cons_eta}}
			The statement follows immediately by Lemma \eqref{cETA_1} and Lemma \eqref{cETA_2}.
		\end{proof}
		\begin{remark}We provide the proof of the consistency both for $\tilde \rho_n$ and $\tilde \eta_{12, n}$, but for $\tilde \eta_{12, n}$ we do not study the asymptotic distribution. It is due to the different expression of $\tilde \eta_{12, n}$, that makes the approach more complicated. While $\tilde \rho_n$ is a sum of product of two stationary process $Y^1_{(k+1)\Delta_n}-Y^1_{k\Delta_n}$ and $Y^2_{(k+1)\Delta_n}-Y^2_{k\Delta_n}$, the same is not true for $\tilde \eta_{12,n}$. We think that a similar but more difficult approach can be pursued for $\tilde \eta_{12, n}$. 
		
		\end{remark}
		\chapter{Short-time asymptotics for non self-similar stochastic volatility models}\label{Chapter:Short-time}
In this Chapter, based on \cite{GPP23}, we provide a short-time large deviation principle (LDP) for Volterra-driven stochastic volatilities. In in Section \ref{sec:ldp:volterra} we obtain a short-time LDP that does not rely on the self-similarity assumption. In Section \ref{sec:applications} we propose as examples the log-fractional Brownian motion and the fOU process, which satisfies the assumption of the result in Section \ref{sec:ldp:volterra}. In Section \ref{sec:pricing} we apply the result to option pricing and in Section \ref{sec:numerics} we test the accuracy of the short-time pricing formulas for fOU with numerical experiments. 
			\section{The large deviations principle}\label{app:ldp}
			Large deviations give an asymptotic computation of small
			probabilities on an exponential scale (see e.g.
			\cite{DemZei} as a reference on this topic). We recall some
			basic definitions (see e.g. Section 1.2 in \cite{DemZei}).
			Throughout this chapter a speed function is a sequence $\{v_n:n\geq
			1\}$ such that $\lim_{n\to\infty}v_n=\infty$. A sequence of random
			variables $\{Z_n:n\geq 1\}$, taking values on a topological space
			$\mathcal{X}$, satisfies the LDP with rate function $I$ and speed function $v_n$ if
			$I:\mathcal{X}\to[0,\infty]$ is a lower semicontinuous function,
			$$\liminf_{n\to\infty}\frac{1}{v_n}\log \P(Z_n\in O)\geq-\inf_{x\in O}I(x)$$
			for all open sets $O$, and
			$$\limsup_{n\to\infty}\frac{1}{v_n}\log \P(Z_n\in C)\leq-\inf_{x\in C}I(x)$$
			for all closed sets $C$. A rate function is said to be \emph{good} if all
			its level sets $\{\{z\in\mathcal{Z}:I(z)\leq\eta\}:\eta\geq 0\}$
			are compact.
			Therefore,  if an LDP holds, and $\Gamma$ is a Borel set such that  $\inf_{x\in \Gamma^o}I(x)=\inf_{x\in \bar\Gamma}I(x)$ ($\Gamma^o$ and $\bar\Gamma$ are the interior and the  closure of $\Gamma$ respectively), then
			$$\lim_{n\to\infty}\frac{1}{v_n}\log \P(Z_n\in \Gamma)=-I(\Gamma)$$
			where $ I(\Gamma)= \displaystyle\inf_{x\in \Gamma^o}I(x)=\inf_{x\in \bar\Gamma}I(x).$
			In this case we write
			$$
			\P(Z_n\in \Gamma)\approx e^{- {I(\Gamma)}{v_n}}.
			$$ 
			
			Moreover $\{Z_n:n\geq 1\}$ is exponentially tight
			with respect to the speed function $v_n$ if, for all $b>0$, there
			exists a compact $K_b\subset\mathcal{X}$ such that
			$$\limsup_{n\to\infty}\frac{1}{v_n}\log \P(Z_n\notin K_b)\leq-b.$$
			
			The concept of exponential tightness plays a crucial role in large
			deviations; in fact this condition is often required to establish
			that the LDP holds for a sequence of random
			variables taking values on an infinite dimensional topological
			space. In this chapter we refer to condition (8) and (9) in Section 2 in   \cite{MacPac})
			which yield the exponential tightness when
			the topological space $\mathcal{X}$ of the continuous function is equipped with the
			topology of the uniform convergence.

		\section{Large deviations for Volterra stochastic volatility models}
		\label{sec:ldp:volterra}
		
		\subsection{Small-noise large deviations for the log-price}
		
		We are interested in stochastic volatility models with asset price dynamics described by
		\begin{equation}\label{eq:price-SDE}
			dS_t =  S_t\,\sigma(V_t)d(\bar{\rho}\bar B_t+\rho B_t), \qquad 0\leq t \leq T,\\
		\end{equation}
		where  we set, without loss of generality, $S_0=1$ the initial price. The time horizon is $ T > 0 $,
		$ \bar B $ and $ B $ are two independent standard Brownian motions, $ \rho \in (-1,1) $ is a correlation coefficient and $ \bar{\rho}=\sqrt{1-\rho^2}$, so that $\widetilde{B}= \bar{\rho}\bar B+\rho B $ is a standard Brownian motion $ \rho $-correlated with $ B $.
		We assume that the process $ V $ is a non-degenerate, continuous Volterra type Gaussian process of the form
		\begin{equation}\label{eq:Volterra} V_t=
			\int_0^t K(t,s)\, dB_s, \quad 0\leq t \leq T. \end{equation}
		Here, the kernel  $ K $ is a  measurable and square integrable function on $ [0,T]^2 $, such that
		$K(0,0)=0$, $ K(t,s)=0 $ for all $ 0\leq t < s \leq T $ and
		$$\sup_{t\in[0,T]} \int_0^T K(t,s)^2 \,ds < \infty. $$
		One can verify that the covariance function of the process $ V $ defined as above is given by
		$$k(t,s)= \E[V_{t}V_{s}]=\int_0^{t\wedge s} K(t,u)K(s,u)\,du, \quad t,s \in [0,T].$$
		We introduce now the modulus of continuity of the kernel $ K $, defined as
		$$
		M(\delta)=  \sup_{\{t_1,t_2 \in [0,T]: |t_1-t_2|\leq \delta\}}  \int_0^T |K(t_1,s)-K(t_2,s)|^2\,ds, \quad 0 \leq \delta \leq T.
		$$
		In order to ensure the continuity of the paths of $V$, we assume that $K$ satisfies the following condition.
		\begin{enumerate}
			\item[\bf(A1)\rm]
			\label{ass:Volterra-process}
			There exist constants $ c > 0 $ and $ \vartheta > 0 $ such that $ M(\delta)\leq c\,\delta^\vartheta $ for all $ \delta \in [0,T] $.
		\end{enumerate}		
		Let us recall that the unique solution to equation \eqref{eq:price-SDE} is  $(e^{X_t})_{t\in[0,T]}$, where the log-price process is defined by
		\begin{equation}\label{eq:log-price}
			X_t=-\frac 12\int_0^t\sigma^2(V_s) ds +  \bar{\rho}\int_0^t\sigma(V_s) d\bar B_s + \rho\int_0^t \sigma(V_s) dB_s.
		\end{equation}
		\begin{definition}
			A modulus of continuity is an increasing function  $\omega:[0,+\infty)\to[0,+\infty)$
			such that $\omega(0) = 0 $ and $\lim_{x\to 0 }\omega (x)=0$.
			A function $f$  defined on $\R$  is
			called locally $\omega$-continuous, if for every $\delta>0$ there exists a constant $R(\delta)>0$  such that for all
			$x,y\in[-\delta, \delta]$, inequality  $|f(x)-f(y)|\leq R(\delta)\omega(|x-y|)$ holds.
		\end{definition}
		\begin{remark}\rm
			For instance, if $\omega(x)=x^{\vartheta}$, $\vartheta\in(0,1)$,   the function $f$ is locally $\vartheta $-H\"{o}lder continuous.
			If $\omega(x)=x$,    the function $f$ is locally Lipschitz continuous.
		\end{remark}
		
		We consider the following assumptions on the volatility function $\sigma$.
		
		\begin{enumerate}
			\item[($\bf\Sigma 1$)\rm] \label{ass:hp-sigma-II}\rm
			$ \sigma: \R \longrightarrow (0,+\infty) $ is a locally $\omega$-continuous function for some modulus of continuity $\omega$.
			\item[($\bf\Sigma 2$)\rm]
			\label{ass:hp-sigma-III}\rm
			There exist constants $\vartheta, M_1,M_2>0, $ such that
			$ \sigma(x)\leq M_1+ M_2 \,|x|^\vartheta, \quad  x\in \R.$
		\end{enumerate}
		
		From now on, we denote by  $C ([0, T])$ (respectively $ C_0([0,T])$) the set of
		continuous functions on $[0, T]$ (respectively the set of
		continuous functions on $[0, T]$ starting at $0$),  endowed with the topology induced by the sup-norm.
		
		Let $ \gamma_.: \mathbb{N} \to \mathbb{R}_+ $ be an infinitesimal, decreasing function,  i.e.    $ \gamma_n \downarrow 0 $, as $ n \to +\infty$ .  For every $ n \in \mathbb{N} $, we  consider the following scaled version of  equation \eqref{eq:price-SDE}
		$$\begin{array}{c}
			\begin{cases}
				dS_t^{n} =\gamma_n  S_t^{n}\sigma(V_t^n)d(\bar{\rho}\bar B_t+\rho B_t), \qquad 0\leq t \leq T,\\
				S_0^{n}=1,
			\end{cases}
		\end{array}$$
		The  log-price process $ X_t^{n}=\log S_t^{n}, $ $ 0\leq t \leq T,$
		in the scaled model is
		\begin{equation}\label{eq:log-price-scaled}
			X_t^{n}=-\frac 12\gamma_n^{2}\int_0^t\sigma^2(V_s^n) ds +  \gamma_n\,\bar{\rho}\int_0^t\sigma(V_s^n) d\bar B_s +  \gamma_n\,\rho\int_0^t\sigma(V_s^n) dB_s.
		\end{equation}
		Here  the Brownian motion $ \bar{\rho}\bar B+\rho B $ is multiplied by a \it  small-noise \rm parameter $ \gamma_n$
		and  the Volterra process $V^n$  is  of the form
		$$V^n_t=\displaystyle\int_0^t K^n(t,s)dB_s \quad 0\leq t \leq T,
		$$
		where $K^n$ is a suitable kernel. It can be verified that the covariance function of the process $ V^n$, for every $ n \in \mathbb{N}, $ is given by
		$$ k^n(t,s)=\displaystyle\int_0^{t\wedge s}K^n (t,u)K^n(s,u) \, du \quad \mbox{for } t,s \in [0,T].
		$$
		In the setting above, we are interested in an LDP for the family $( (\gamma_n B,V^n))_{n \in \mathbb{N}} $ (we recall basic facts and notations on LDP in Section \ref{app:ldp}).
		Such an LDP holds under the following conditions on  the covariance functions, as seen in Theorem 7.4 in \cite{CelPac}.
		\begin{enumerate}
			\rm
			\item[{\bf(K1)\rm}] \label{ass:cov-lim}	
			There exist an infinitesimal function $ \gamma_n $ and a kernel {$ \widehat K $ }(regular enough to be the kernel of a continuous Volterra process) such that
			
			\begin{equation} \label{eq:ker-limit}\displaystyle\lim_{n \to +\infty} \frac{K^n(t,s)}{{\gamma_n}}=\widehat K(t,s) \end{equation} and
			
			$$ 	 \displaystyle\lim_{n \to +\infty} \frac{\int_0^{T} K^n(t,u)K^n(s,u) du} {\gamma^2_n}
			=
			\int_0^{T}\widehat K(t,u)\widehat K(s,u) du
			$$
			uniformly for $ t,s \in [0,T].$
			
			\item[{\bf(K2)\rm}]\label{eq:ker-exp-tight} There exist constants $ \beta,M>0 $, such that, for every $ n\geq n_0$
			$$\displaystyle\sup_{s,t \in [0,T], s\neq t} \frac{\int_0^{T}(K^n(t,u)-K^n(s,u))^2du}{\gamma^2_n|t-s|^{2\beta}}\leq M.
			$$
		\end{enumerate}

		\begin{theorem}\label{th:LDP-B-hatB}
			Let $ \gamma: \mathbb{N} \to \mathbb{R}_+ $ be an infinitesimal function. Suppose Assumptions {\bf(K1)\rm} and {\bf(K2)\rm}  
			are fulfilled.
			Then $ ((\gamma_nB, V^n))_{n \in \mathbb{N}} $ satisfies an LDP on $C_0([0,T])^2$ with speed $ \gamma_n^{-2} $ and good rate function
			$$
			I_{(B,V)} (f,g)=
			\begin{cases}
				\displaystyle \frac12  \int_0^T \dot{f}(s)^2 \, ds & (f,g) \in \cl H_{(B,V)}\\
				\displaystyle +\infty & (f,g)\notin \cl H_{(B,V)}
			\end{cases}
			$$
			where
			$$\cl H_{(B,V)}=\{(f,g) \in C_0([0,T])^2: f \in H_0^1[0,T], \, g(t)=\int_0^t \widehat K(t,u)\dot{f}(u)\,du, \quad 0\leq t \leq T\},$$
			where $\widehat K$ is defined in equation \eqref{eq:ker-limit} and $H_0^1=H_0^1[0,T]$ is the Cameron-Martin space.
		\end{theorem}

		If  Assumptions  {\bf($\bf\Sigma 1$)} and ($\bf\Sigma 2$) 
		hold for the volatility function $\sigma(\cdot)$, 
		we also have a sample path LDP  for the family  of  processes $ ((X_t^{n})_{t \in [0,T]})_{n \in \mathbb{N}}$ and for the family of random  variables
		$ (X_T^{n})_{n \in \mathbb{N}}$ (see Section 7 in \cite{CelPac} for details).
		Let us denote $\hat f(t)=\int_0^t \widehat K(t,u)\dot{f}(u)\,du$ for  $f \in H_0^1$.
		\begin{theorem}\label{th:LDP-log-price}
			Under  Assumptions {\bf($\bf\Sigma 1$)}, {\bf($\bf\Sigma 2$)}, {\bf(K1)\rm} and {\bf(K2)\rm},  we have:
			i) the family  of  processes $ ((X_t^{n})_{t \in [0,T]})_{n \in \mathbb{N}}$
			satisfies an LDP
			with speed $\gamma_n^{-2}$ and good rate function
			\begin{equation}\label{eq:rate-fun-infinite} I_X(x)=\begin{cases} \inf_{f\in H_0^1[0,T]}\left[\frac12 \lVert f\rVert_{H_0^1[0,T]}^2+\frac12\int_0^T \big(\frac{\dot{x}(t)-\rho \sigma(\hat{f}(t))\dot{f}(t)}{\bar{\rho}\sigma(\hat{f}(t))}\big)^2\,dt\right]& x \in H_0^1[0,T] \\
					\displaystyle +\infty & x \notin H_0^1[0,T];
				\end{cases}
			\end{equation}
			ii) the family  of  random variables $ (X_T^{n})_{n \in \mathbb{N}}$
			satisfies an LDP
			with speed $\gamma_n^{-2}$ and good rate function
			\begin{equation}\label{eq:rate-fun-finite}I_{X_T}(y)=\inf_{f\in H_0^1[0,T]} \left[\frac12 \lVert f\rVert_{H_0^1[0,T]}^2+\frac12 \frac{\big(y-\int_0^T \rho \sigma(\hat{f}(t))\dot{f}(t)\,dt\big)^2}{\int_0^T \bar{\rho}^2\sigma^2(\hat{f}(t)) dt}\right], \quad y\in \R.\end{equation}
		\end{theorem}
		\begin{remark}\label{rem:I=J}\rm
			From Theorem 4.8 in \cite{FZ17}, it follows that $ I_{X_T}(0)=0,$
			$$
			\begin{cases}
				\inf_{y\geq x } I_{X_T}(y)=\inf_{y> x } I_{X_T}(y)=I_{X_T}(x)\quad \mbox{for}\,x>0,\\
				\inf_{y\leq x } I_{X_T}(y)=\inf_{y< x } I_{X_T}(y)=I_{X_T}(x)\quad \mbox{for}\,x<0.
			\end{cases}
			$$
		\end{remark}
		
		\subsection{Short-time large deviations for the log-price}\label{sec:shorttime}
		It is well known that if the Volterra process is self-similar we can pass  from small-noise to short-time regime (see the discussion at the end of Section 3 in \cite{Gu2}). However, in general this is not possible if the process is not self-similar.
		In this section, we  obtain  a short-time LDP that does not rely on the self-similarity assumption,   by using the results of the previous section. 
		
		Let $ \ep_.: \mathbb{N} \to \mathbb{R}_+ $ be a sequence decreasing to zero, i.e.   $\ep_n\downarrow 0$ as $n\to +\infty$.
		For every $n\in\mathbb{N}$ and $t\in[0,T]$, if $V$ is a Volterra process as in (\ref{eq:Volterra}) we have
		\begin{equation}\label{eq:short-Volterra}
			V_{\ep_nt}=\int_0^{\ep_n t} K(\ep_nt,s) dB_s
			\eqlaw
			\int_0^{t} \sqrt \ep_n  K(\ep_n t,\ep_n s) dB_s=\int_0^{t}  K^n(t,s) dB_s=V^n_t,\end{equation}
		with $K^n(t,s)=\sqrt \ep_n  K(\ep_nt,\ep_n s)$.
		Therefore for every $n\in\mathbb{N}$ and $t\in[0,T]$, if $X$ is as  in (\ref{eq:log-price}), we have
		\begin{eqnarray*}
			X_{\ep_nt}&=&-\frac 12\int_0^{\ep_nt}\sigma^2(V_s) ds +  \bar{\rho}\int_0^{\ep_nt}\sigma(V_s) d\bar B_s + \rho\int_0^{\ep_nt} \sigma(V_s) dB_s\\
			&\eqlaw&-\ep_n\frac 12\int_0^{t}\sigma^2(V^n_s) ds +\sqrt\ep_n\,  \bar{\rho}\int_0^{t}\sigma(V^n_s) d\bar B_s +\sqrt\ep_n\, \rho \int_0^{t} \sigma(V^n_s) dB_s.
		\end{eqnarray*}
		Define $V^n_t=V_{\ep_nt}$ and suppose the family of processes $ ( V^n)_{n \in \mathbb{N}} $ satisfies an LDP  with  speed $ \gamma_n^{-2} $ (depending on $\ep_n$). Suppose furthermore that the family 
		$ ((\gamma_nB, V^n))_{n \in \mathbb{N}} $ satisfies an LDP  with speed $ \gamma_n^{-2} $ (for details on this topic see Section 7 and in particular Theorem 7.4 in \cite{CelPac})
		and let $X^n$ be the process defined in (\ref{eq:log-price-scaled}).
		If we consider the processes,  defined on the same space, we have
		\begin{equation*}
			X^n_t-\gamma_n \ep_n^{-1/2}X_{\ep_nt}=\frac12(\gamma_n\ep_n^{1/2}-\gamma_n^2 )\int_0^{t}\sigma^2(V^n_s) ds.
		\end{equation*}
		
		Let us recall that  two families  $ (Z^n)_{n\in\N} $ and $ (\tilde{Z}^n)_{n\in\N} $ of  random variables are \emph{exponentially equivalent} (at the speed $v_n$, with $v_{n}\to \infty$ as $n\to \infty$) if  for any $  \delta > 0 $,
		$$ \limsup_{n\to +\infty }\frac{1}{v_n} \log \P(|\tilde{Z}^n-Z^n|>\delta) = -\infty. $$
		As far as the LDP is concerned, exponentially equivalent families  are indistinguishable. See  Theorem 4.2.13 in \cite{DemZei}.

		\begin{theorem}\label{th:small-time-LDP}
			Under  Assumptions {\bf($\bf\Sigma 1$)}, {\bf($\bf\Sigma 2$)}, {\bf(K1)\rm} and {\bf(K2)\rm}, the two families $((X^n_t)_{t\in[0,T]})_{n\in\mathbb{N}}$ and $((\gamma_n \ep_n^{-1/2}X_{\ep_nt})_{t\in[0,T]})_{n\in\mathbb{N}}$   are exponentially equivalent  and therefore satisfy the same LDP. In particular,
			
			(i)  the family $((\gamma_n \ep_n^{-1/2}X_{\ep_nt})_{t\in[0,T]})_{n\in\mathbb{N}}$ satisfies an LDP  with  speed $\gamma_n^{-2}$ and good rate function given by (\ref{eq:rate-fun-infinite});
			
			(ii)  the family of random variables $ (\gamma_n \ep_n^{-1/2}X_{\ep_n T})_{n \in \mathbb{N}}$
			satisfies an LDP  with  speed $\gamma_n^{-2}$ and good rate function given by (\ref{eq:rate-fun-finite}).
		\end{theorem}
		\noindent \begin{proof}
			We have
			\begin{equation*}
				|X^n_t-\gamma_n \ep_n^{-1/2}X_{\ep_nt}|=\frac12|\gamma_n\ep_n^{1/2}-\gamma_n^2 |\int_0^{t}\sigma^2(V^n_s)ds=\delta_n \int_0^{t}\sigma^2(V^n_s) ds,
			\end{equation*}
			where $\delta_n=\frac12|\gamma_n\ep_n^{1/2}-\gamma_n^2 |$ and $\delta_n\to 0$.
			The family $(V^n)_{n\in \N}$ satisfies an LDP with a good rate function. Then, it is exponentially tight (at the inverse speed $\gamma_n^2$). Therefore for every $R>0$, there exists a compact set $C_R$ (of equi-bounded functions) such that
			$\limsup_{n\to +\infty} \gamma_n^2 \log \P\Big( V^n\in C^{c}_R\Big)\leq -R$, with $(\cdot)^{c}$ indicating the complementary set.
			Thus, for every $\eta>0$,
			$$\begin{array}{c}
				\limsup_{n\to+\infty}\gamma_n^2\log\mathbb{P}\Big( \sup_{t\in[0,T]}| X^n_t-\sqrt{\gamma_n}X_{\gamma_nt} |>\eta \Big) \phantom{ghdfasasasasfgdfhgfhdsaddsdsdsdh}\\
				\leq\limsup_{n\to+\infty}\gamma_n^2\log\mathbb{P}\Big( \sup_{t\in[0,T]} \int_0^{t}\sigma^2(V^n_s) ds >\eta/\delta_n ,V^n\in C_R\Big)\\
				\phantom{ghdfasasasasfgdfhgfdsfdsfdsffdhds}+\limsup_{n\to+\infty}\gamma_n^2\log\mathbb{P}\Big(  V^n\in C_R^{c}\Big)=-\infty,
			\end{array}$$
			since the set $\Big\{ \sup_{t\in[0,T]} \int_0^{t}\sigma^2(V^n_s) ds >\delta/\delta_n , V^n\in C_R\Big\}$ is eventually empty.
			\cvd
		\end{proof}


		\section{Applications}\label{sec:applications}
		In this section, we consider some (non self-similar) Volterra processes that satisfy 
		assumption \bf (A1) \rm and such that the corresponding family $(V^n)_{n\in\N}$ defined by equation (\ref{eq:short-Volterra}) satisfies conditions {\bf(K1)\rm}  and {\bf(K2)\rm}.  
		We also suppose that assumptions ($\bf\Sigma 1$) \rm and ($\bf\Sigma 2$) \rm on the volatility function are satisfied and $T=1$.
		From Theorem \ref{th:small-time-LDP} we obtain a short-time LDP for the corresponding  log-price processes.

		\subsection{Log-fractional Brownian motion and modulated models}\label{sec:modulated}
		Let us consider the kernel, for $0\leq s\leq t\leq 1$,
		\begin{equation}\label{eq:kernel:logmod}
			{K}
			(t,s)=C (t-s)^{H-1/2} (-\log(t-s))^{-p},
		\end{equation}
		where $0\leq H \leq 1/2$,  $p>1$ and $C>0$ is a constant.
		The corresponding Volterra process $V$ essentially amounts to
		the log-fBM introduced in \cite{BHP20}. There, an additional
		cutoff of the logarithm function was introduced in order to normalize the variance of the volatility at time one, but we can avoid here this complication as it does not affect our analysis in any way, since we only consider short-time asymptotics.
		
		Condition \bf (A1) \rm for this kernel was proved in 
		\cite{BHP20} with $\vartheta=2H$.
		Note that $\kappa(t,s)=C (t-s)^{H-1/2}$ is the well known kernel of the RLp, which also satisfies Assumption \bf(A1) \rm with $\vartheta=2H$.

		For $n$ large enough, we set
		$$ K^n(t,s)=C  \ep_n^{H}(t-s)^{H-1/2} (-\log(\ep_n(t-s)))^{-p}.$$
		
		Let us verify that  conditions {\bf(K1)\rm}  and {\bf(K2)\rm} are satisfied for $H>0$. No small time LDP can be verified in the case $H=0$, as shown in Section 5.4 in \cite {BaPa}. 
		
		{\bf(K1)\rm}
		For $s,t\in[0,1]$, $s<t$, since we can suppose $\ep_n<1$, we have
		\begin{equation}\label{eq:ineq}\frac {\log \ep_n(t-s)}{\log \ep_n }\geq 1,\end{equation} and therefore
		$$\kappa(t,u)\Big(\frac{\log \ep_n(t-u) }{\log \ep_n}\Big)^{-p}\leq \kappa(t,u).$$ 
		Then, thanks to Lebesgue's dominated convergence Theorem, for  $s,t\in[0,1]$,
		$$\lim_{n\to \infty}\int_0^{s\wedge t} \kappa(t,u)\kappa(s,u)\Big(\frac{\log \ep_n(t-u) }{\log \ep_n}\Big)^{-p}\Big(\frac{\log \ep_n(s-u) }{\log \ep_n}\Big)^{-p}du=\int_0^{s\wedge t}\kappa(t,u)\kappa(s,u)du,$$ 
		so that
		$\lim_{n\to \infty} \frac{k^n(t,s)}{\ep_n^{2H}(-\log \ep_n)^{-2p}}=k(t,s)$.
		
		This convergence is actually uniform,  since
		$$\frac{k^n(t,s)}{\ep_n^{2H}(-\log \ep_n)^{-2p}}= C^2\int_0^{s\wedge t} (t-u)^{H-1/2}(s-u)^{H-1/2}\Big(1+ \frac {\log (t-u)}{\log \ep_n }\Big)^{-p}
		\Big(1+ \frac {\log (s-u)}{\log \ep_n }\Big)^{-p} du,$$
		and therefore the sequence $(k^n(\cdot,\cdot)/{\ep_n^{2H}(-\log \ep_n)^{-2p}})_n$ is a monotone sequence of continuous functions converging pointwise to a continuous function. Then  {\bf(K1)\rm}
		is proved (with $\widehat K=\kappa$ and $\gamma_n=\ep_n^{H}(-\log \ep_n )^{-p}$). 
		\medskip
		
		{\bf(K2)\rm} For every $ n \in \mathbb{N}$, $s<t$,
		we have
		$$ \displaylines{\int_0^t \Big(\kappa(t,u) \frac{\log \ep_n(t-u) }{\log \ep_n}\Big)^{-p}-\kappa(s,u) \frac{\log \ep_n(s-u) }{\log \ep_n}\Big)^{-p}\Big)^2 \, du\cr
			=\int_s^t \kappa(t,u)^2\Big(\frac{\log \ep_n(t-u) }{\log \ep_n}\Big)^{-2p}\, du+\int_0^s \Big(\kappa(t,u) \frac{\log \ep_n(t-u) }{\log \ep_n}\Big)^{-p}-\kappa(s,u)\frac{\log \ep_n(s-u) }{\log \ep_n}\Big)^{-p}\Big)^2\, du. }$$
		Now, thanks to (\ref{eq:ineq})
		$$
		\int_s^t \kappa(t,u)^2\Big(\frac{\log \ep_n(t-u) }{\log \ep_n}\Big)^{-2p}\, du\\
		\leq \int_s^t \kappa(t,u)^2\, du.$$
		
		Let us prove that
		$$\int_0^s \Big(\kappa(t,u) \Big(\frac{\log \ep_n(t-u) }{\log \ep_n}\Big)^{-p}-\kappa(s,u) \Big(\frac{\log \ep_n(s-u) }{\log \ep_n}\Big)^{-p}\Big)^2\, du\leq
		\int_0^s (\kappa(t,u) -\kappa(s,u) )^2\, du.$$
		
		The map $x\to x^{H-1/2}(-\log x)^{-p}$ defines a decreasing function in a neighbourhood of $x=0$  and $x\to (-\log x)^{-p}$  an increasing function for $x\in (0,1)$. Then,
		for $n$ large enough, for $u\in(0,s)$, we have
		\begin{eqnarray*}
			0&\leq& (\ep_n u)^{H-1/2}(-{\log(\ep_nu)})^{-p}- (\ep_n(t-s+u))^{H-1/2} (-{\log(\ep_n(t-s+u))} )^{-p}\\
			&\leq&   (\ep_n u)^{H-1/2}(-{\log(\ep_n(t-s+u))})^{-p}- (\ep_n(t-s+u))^{H-1/2} (-{\log(\ep_n(t-s+u))} )^{-p}\\ &=& 
			( (\ep_nu)^{H-1/2}- (\ep_n(t-s+u))^{H-1/2}) (-{\log(\ep_n(t-s+u)\ep_n)} )^{-p}.
		\end{eqnarray*}
		Therefore,
		\begin{eqnarray*}
			&&\int_0^s \Big(\kappa(t,u) \Big(\frac{\log \ep_n(t-u) }{\log \ep_n}\Big)^{-p}
			-\kappa(s,u) \Big(\frac{\log \ep_n(s-u) }{\log \ep_n}\Big)^{-p}\Big)^2\, du\\ 
			&=&\!\!\!\frac {C^2 }{\ep_n^{2H}(-\log \ep_n)^{-2p}}\!\!\! \int_0^s \!\!\!\Big((\ep_n(t-s+u))^{H-\frac 12} (-{\log(\ep_n(t-s+u)} )^{-p}- (\ep_n\,u)^{H-\frac 12}  (-{\log(\ep_n u)})^{-p}\Big)^2\!\! du\\&\leq& \frac {C^2 }{\ep_n^{2H}(-\log \ep_n)^{-2p}}\int_0^s ( (\ep_nu)^{H-1/2}- (\ep_n(t-s+u))^{H-1/2})^2 (-{\log(\ep_n(t-s+u))} )^{-2p} du
			\\& =& C^2 \int_0^s ( u^{H-1/2}- (t-s+u)^{H-1/2})^2 \Big(\frac{\log(\ep_n(t-s+u))}{\log \ep_n} \Big)^{-2p}\, du\\&\leq&
			C^2 \int_0^s ( u^{H-1/2}- (t-s+u)^{H-1/2})^2\, du=\int_0^s (\kappa(t,u) -\kappa(s,u) )^2\, du,
		\end{eqnarray*}
		Therefore conditions  {\bf(K1)\rm} and {\bf(K2)\rm}  are verified with infinitesimal function $\gamma_n=\ep_n^H(-\log \ep_n)^{-p}$, limit kernel $\widehat K=\kappa$, and
		$\beta=H$.
		A short-time LDP  holds with inverse speed $\ep_n^{2H}(-\log \ep_n )^{-2p}$  and limit kernel $\kappa(t,s)=C (t-s)^{H-1/2}.$ 
		
		\medskip
		
		The results proved for the log-fBM can be extended to a
		class of processes, that we refer to as
		\emph{modulated Volterra processes}, defined,  for $t\in[0,1]$, as
		\begin{equation}\label{eq:Volterra-modulated} V_t= \int_0^t \kappa(t,s) L(t-s) dB_s.\end{equation}  
		Here, 
		$\kappa$ is the kernel of a self-similar  Volterra process of index $H>0$,
		i.e.
		\begin{equation}\label{eq:self-similar}
			\kappa( c t,c s)= c^{H-1/2} \kappa(t,s), \quad \mbox{ for }  c>0,
		\end{equation}
		that satisfies Assumption \bf(A1)\rm  \it, modulated  \rm by a slowly varying function $L$, i.e. a function such that
		$$\lim_{x\to 0^+}\frac { L(x\lambda)}{L(x)}=1,$$
		for every $\lambda>0$. 
		Thanks to (\ref{eq:short-Volterra})  and (\ref{eq:self-similar}),   we have
		$$V^n_t=\int_0^{t}  \ep_n^H  \kappa(t, s) L(\ep_n (t-s))dB_s.$$
		First we note that here $K^n(t,s)=\ep_n^H  \kappa(t, s) L(\ep_n (t-s)) $ for $s,t\in[0,1]$, $s<t$ and
		$$\lim_{n\to \infty} \frac{K^n(t,s)}{L(\ep_n)\ep_n^H}=\lim_{n\to \infty}\kappa(t,s) \frac{L(\ep_n (t-s))}{L(\ep_n)}=\kappa(t,s).$$
		
		Note that the limit kernel is independent of $L$. For these processes, if assumptions  {\bf(K1)\rm}  and {\bf(K2)\rm} are satisfied,  we have a short-time LDP with the same rate function as the self-similar process and inverse speed
		$L(\ep_n)^{2}\ep_n^{2H}$. Therefore, the rate function does not depend on the modulating function $L$, but the speed of the LDP does.

		\subsection{Fractional Ornstein-Uhlenbeck process}\label{sec:fOU}
	
		We refer to Section \ref{Univariate_section}. Let us recall that the Mandelbrot-Van Ness fBM $B^{H}$ is the centered continuous Gaussian process with covariance function (\eqref{cfBM})
		\[
		\frac{1}{2}\left(t^{2H}+s^{2H}-|t-s|^{2H}\right).
		\]
		It is self-similar with exponent
		$H$. 
		Let us also recall that it admits a Volterra representation with kernel $K_H$ given in \ref{eqn:kernelfbm}.
		For $H\in(0,1)$ and $a > 0$, we consider the fOU process, solution to
		\[
		d V_{t} = -a V_{t } dt + dB^{H}_{t}, 
		\]
		which is given explicitly, with initial condition $V_{0}=0$, 
		by\footnote{
			Let us notice that, in Section \ref{Univariate_section}, we consider the stationary solution to the fractional SDE above, explicitly given by 
			$\int_{-\infty}^t e^{-a(t-u)} dB^H_u.$
			However, we are interested in this chapter in option valuation, so we take as volatility driver the process $V_t$ above, with $V_{0}=0$, so that $\sigma_{0}=\sigma(V_{0})=\sigma({0})$ is spot volatility in \eqref{eq:stoch:vol:fOU}. 
		}
		$$
		V_t= \int_0^t e^{-a(t-u)} dB^H_u, \quad t\geq 0.
		$$
		For more details see Section \ref{Univariate_section}. We see in \eqref{eq:FOU} that we can represent $V_t$ as
		\begin{equation*}
			V_t=B^H_t - a\int_0^t e^{-a(t-u)}B^H_u du.
		\end{equation*}
		
		We note from this equation that self-similarity for fOU is approximately inherited from the fBM, for small time scales.
		From \eqref{eq:FOU} we obtain, for $V$, the Volterra representation 
		$$V_t=\int_0^t  K(t,s) dB_s,$$
		with
		$$
		K(t,s)= K_H(t,s)- a \int_s^t e^{-a(t-u)} K_H(u,s) du,$$
		and $K_{H}$ as in \eqref{eqn:kernelfbm}
		(see, e.g., Section 2 in \cite{CelPac}). Condition \bf (A1)  \rm for this process, with $\vartheta=\min\{2H,1\}$, was established in Lemma 10 in \cite{Gu1}.
		Here we have
		$$  K^n(t,s)=  \ep_n^{H}K_H(t,s)- a \ep_n^{H+1} \int_s^t e^{-a\ep_n(t-u)} K_H(u,s) du.$$
	Let us verify that  conditions {\bf(K1)\rm}  and {\bf(K2)\rm} are satisfied.
		
		{\bf(K1)\rm}  
		It is enough to observe that 
		$$\left|\frac{ K^n(t,s)}{{\ep_n^H}}-K_H(t,s)\right|=a \ep_n\int_s^t e^{-a\ep_n(t-u)} K_H(u,s) du\leq C\ep_n, $$
		where $C>0$ is a constant independent of $s,t\in[0,1]$. Therefore
		$$\displaystyle\lim_{n \to +\infty} \frac{ K^n(t,s)}{{\ep_n^H}}=K_H(t,s),$$
		uniformly for $s,t\in[0,1]$.
		Therefore also  
		$$ \int_0^{t\wedge s}\widehat K(t,u)\widehat K(s,u) du= \displaystyle\lim_{n \to +\infty} \frac{\int_0^{t\wedge s} K^n(t,u)K^n(s,u) du} {\gamma^2_n}$$
		uniformly for $ t,s \in [0,T]$ and {\bf(K1)\rm} is proved (with $\widehat K=K_H$ and $\gamma_n=\ep_n^{H}$).
		
		{\bf(K2)\rm}  For $s<t$ we have 
		\begin{eqnarray*}&&\left| K^n(t,u)- K^n(s,u)\right|\\
			&\leq &\ep_n^{H}|K_H(t,u)- K_H(s,u)|+ a\ep_n^{H+1}\left| 
			\int_u^t e^{-a\ep_n(t-v)} K_H(v,u)dv -\int_u^s e^{-a\ep_n(s-v)} K_H(v,u)dv\right|\\
			&\leq& \ep_n^{H}|K_H(t,u)- K_H(s,u)|+ a\ep_n^{H+1}\int_s^t e^{-a\ep_n(t-v)} K_H(v,u)dv \cr & +&a\ep_n^{H+1}\left| 
			\int_u^s \Big(e^{-a\ep_n(t-v)}  - e^{-a\ep_n(s-v)} \Big)k_H(v,u)dv\right|\\ 
			&=&\ep_n^{H}|K_H(t,u)- K_H(s,u)|+ a\ep_n^{H+1}e^{-a\ep_n(t-s)} \int_s^t e^{-a\ep_n(s-v)} K_H(v,u)dv \cr &+&a\ep_n^{H+1}|e^{-a\ep_n(t-s)}-1| 
			\int_u^s e^{-a\ep_n(s-v)} k_H(v,u)dv.
		\end{eqnarray*}
		
		Therefore, denoting by $C>0$ a  constant (not depending on $s,t\in[0,1]$), we have
		\begin{eqnarray*}
			&&\int_0^{1}(K^n(t,u)-K^n(s,u))^2du\cr &\leq&
			\ep_n^{2H}{\int_0^{1}(K_H(t,u)-K_H(s,u))^2du}+
			\ep_n^{2H+2}\int_0^1 du \Big(\int_s^t e^{-a\ep_n(t-v)} K_H(v,u)dv\Big)^2\cr &+& C (t-s)^2 \ep_n^{2H+4}
			\int_0^1\Big(\int_u^s \Big(e^{-a\ep_n(t-v)} K_H(t,v)\Big)^2\cr&\leq&
			\ep_n^{2H}{\int_0^{1}(K_H(t,u)-K_H(s,u))^2du}+
			\ep_n^{2H+2}(t-s)\int_0^1  \int_0^1  K_H(v,u)^2 du\,dv\cr &+& C (t-s)^2 \ep_n^{2H+4}
			\int_0^1  \int_0^1  K_H(v,u)^2 du\,dv\leq C(t-s)^{2H\wedge 1}\ep_n^{2H},
		\end{eqnarray*}
		since (see for example Lemma 8 in \cite{Gu1})
		$$\displaystyle\sup_{s,t \in [0,T], s\neq t} \frac{\int_0^{1}(K_H(t,u)-K_H(s,u))^2du}{|t-s|^{2H}}\leq M.
		$$
		Condition  {\bf(K2)\rm} is  verified with infinitesimal function $\gamma_n=\ep_n^H$,  $\beta=H$ and limit kernel $\widehat K=K_H$. 
		Therefore, the short-time asymptotic behaviour of the model with volatility given as a function of the fOU process is exactly the same as the one of the model with volatility  given as a function of the fBM, meaning that they both satisfy LDPs where the speed and rate function are the same. Indeed, the rate function in \eqref{eq:rate-fun-finite} is the same that was found in \cite{FZ17}. This can be computed numerically as we do in Section \ref{sec:numerics}.

		\section{Short-time asymptotic pricing and implied volatility}\label{sec:pricing}
		In this section we discuss applications to option pricing and behaviour at short maturities of implied volatility for certain stochastic volatility models, using the LDP previously discussed.
		We denote 
		\begin{equation}\label{eq:putcall}
			p(t,k):=\E[(e^{k}-S_t)^+], \quad c(t,k):=\E[(S_t-e^{k})^+],
		\end{equation}
		the European put and call prices with maturity $t$ and log-moneyness $k$ (i.e., strike $e^{k}$, since $S_{0}=1$).
		
		\subsection{Large deviations pricing for log-modulated models}
		
		Let us consider the stochastic volatility model given by \eqref{eq:price-SDE} and \eqref{eq:Volterra-modulated}, i.e.
		\[
		\begin{cases}
			dS_t &=S_t\sigma(V_t)d(\bar{\rho}\bar B_t+\rho B_t),\\
			V_t&= \int_0^t \kappa(t,s) L(t-s) dB_s,
		\end{cases}
		\]
		with $\kappa$ kernel of a self-similar process, of exponent $H<1/2$, that satisfies {\bf(A1)\rm}, and $L$ slowly varying, such that $K(t,s)=\kappa(t,s) L(t-s)$ satisfies {\bf(A1)\rm}, {\bf(K1)\rm}, {\bf (K2)\rm}. In particular, this holds true for the kernel in \eqref{eq:kernel:logmod}, that essentially is the kernel of the log-fBM in \cite{BHP20}, for $H\in[0,1/2)$. Let 
		\[
		\Lambda(x)=I_{X_1}(x),
		\]  
		where $I_{X_1}$ is the rate function in \eqref{eq:rate-fun-finite}. 
		
		Let us write $f_{t}\approx g_{t}$ if $\log(f_{t})\sim \log(g_{t})$ (see also Appendix \ref{app:ldp}).
		
		\begin{theorem}\label{th:ivol:logmod}
			Let us assume that {\bf(A1)\rm}, {\bf(K1)\rm}, {\bf (K2)\rm},  {\bf($\bf\Sigma 1$)}, {\bf($\bf\Sigma 2$)} hold.
			If $x<0$ and
			\begin{equation}\label{eq:log:moneyness}
				k_t=xt^{-H+1/2}L(t)^{-1},
			\end{equation}
			the short-time put price satisfies
			\[
			p(t,k_t)=\E[(e^{k_t}-S_t)^+]\approx \exp\{-t^{-2H} L(t)^{-2} \Lambda(x) \}.
			\]
			Let us now assume that the process $S$ is a martingale and there exist $p>1,t>0$ such that $\E [ S_t^p ]<\infty$ (cf. Remark \ref{rm:mart}).
			If $x>0$, $k_t$ is as in \eqref{eq:log:moneyness}, we have
			\[
			c(t,k_t)=\E[(S_t-e^{k_t})^+]\approx \exp\{-t^{-2H} {L(t)^{-2} } \Lambda(x) \}.
			\]
		\end{theorem}
		\begin{proof} 
			We just prove the call asymptotics (the least straightforward).
			From Theorem \ref{th:LDP-log-price} and Theorem \ref{th:small-time-LDP}, following the computations in Section \ref{sec:modulated}, we have that
			the family $ (\ep_n^{H-1/2}L(\ep_n)X_{\ep_n })_{n \in \mathbb{N}}$ satisfies an LDP  with inverse speed $\ep_n^{2H}L(\ep_n)^2$ and good rate function  $I_{X_1}$ given by formula (\ref{eq:rate-fun-finite}).
			Since $\inf_{y\geq x} I_{X_1}(y)=\inf_{y>x} I_{X_1}(y)=I_{X_1}(x)$  (see Remark \ref{rem:I=J})  we have for $x>0$
			$$
			\displaylines{ \lim_{n \to +\infty}\ep_n^{2H}L(\ep_n)^2\log\P(\ep_n^{H-1/2}L(\ep_n)X_{\ep_n }> x)=-\Lambda(x),}
			$$
			for every sequence $\ep_n\downarrow 0$. 
			Therefore, setting $\nu_t=t^{H-1/2}L(t)$, so that $k_t=x/\nu_t$, we have 
			$$
			\displaylines{ \lim_{t\to 0}t^{2H}L(t)^2\log\P(X_{t }> k_t)= \lim_{t\to 0}t^{2H}L(t)^2\log\P(\nu_tX_{t }> x)=-\Lambda(x),}
			$$
			i.e.,
			\begin{equation}\label{eq:ldp:st}
				\P(S_t> e^{k_t})=\P(X_t>k_t)=\P(\nu_tX_{t }> x) \approx \exp(-t^{-2H}L( t)^{-2} \Lambda(x) ).
			\end{equation}
			Let us prove the upper bound. Let  $t>0$ be small enough such that  $\nu_t\geq 1$ and fix $y>x$. 
			We have
			\begin{eqnarray*}
				\E [ (S_t - \exp (k_t))^+]
				&=&  \E [ (\exp (X_t) - \exp (k_t))^+]   \\
				&=&  \E [ (\exp( X_t) - \exp (k_t))^+ 1_{\{ \nu_t X_t \in (x,y]\}}]
				+
				\E [  (\exp( X_t) - \exp (k_t))^+ 1_{\{\nu_t  X_t > y\}}]  \\
				&\leq& (e^{y/\nu_t}-e^{x/\nu_t}) \P ( \nu_t X_t > x ) + \E [ \exp( X_t)^p ]^{1/p} \P( \nu_t X_t > y)^{1/q}
				\\
				&\leq& (e^{y}-e^{x}) \P (\nu_tX_t > x ) + \E [ \exp( X_t)^p ]^{1/p} \P(\nu_t X_t > y)^{1/q}
			\end{eqnarray*}
			where we have used H\"older's inequality and the existence of $p>1,t>0$ such that $\E [ S_t^p ]<\infty$. Moreover, $\E [ S_t^p ]$ is uniformly bounded as $t\to 0$, using Doob's maximal inequality for the martingale $S$. Now from LDP \eqref{eq:ldp:st} it follows
			\[
			\limsup_{t \to 0}  t^{2H}L^2(t) \log \left(\E [ (S_t - \exp (k_t))^+]
			\right) \leq \max \left( - \Lambda(x) , - \frac{\Lambda(y)}{q}\right)
			\]
			and we conclude by taking $y$ large enough (here we also use the goodness of the rate function, which implies that $\Lambda(y) \to \infty$ as $y \to \infty$.)
			
			Now let us look at the lower bound. We have
			\begin{eqnarray*}
				\E [ (S_t - \exp (k_t))^+]
				&\geq&    \E [ (\exp( X_t) - \exp (k_t)) 1_{\{\nu_{t} X_t > y\}}] \\
				&\geq&     (\exp( y / \nu_t) - \exp (x/\nu_t))\P( \nu_{t}X_t > y )\\
				&\geq&     \exp (k_{t}) (\exp( (y-x) / \nu_t) - 1)\P( \nu_{t}X_t > y ) \\
				&\geq&    \frac{y-x}{\nu_t}  \exp (k_{t}) \P( \nu_{t}X_t > y).
			\end{eqnarray*}
			Therefore
			\begin{align*}
				t^{2H}L(t)^{2} \log \E [ (S_t - &\exp (k_t))^+] 
				\geq\\&\geq     t^{2H}L(t)^{2} (k_{t}+ \log (y-x) -  \log \nu_t )+t^{2H}L(t)^{2} \log \P( \nu_{t} X_t > y)
			\end{align*}
			and the first summand goes to $0$ as $t\to 0$. Therefore, for any $y>x$,
			\[
			\liminf_{{t\to 0}}
			t^{2H}L(t)^{2} \log \E [ (S_t - \exp (k_t))^+]
			\geq
			\liminf_{{t\to 0}}
			t^{2H} L(t)^{2} \log \P( \nu_{t}X_t > y)\geq - \Lambda(y).
			\]
			By continuity of $\Lambda$  \cite[Corollary 4.10]{FZ17} and the fact that the rate function is the same as for the self-similar process, this holds 
			for $\Lambda(x)$ as well and the lower bound is proved.
			\cvd
			
		\end{proof}
	
		\noindent The following implied volatility asymptotics is a consequence of the previous result and an application of \cite{gaolee}. Let us denote with $\sim$ asymptotic equivalence ($f_t\sim g_t$ iff $f_t/ g_t\to 1$).
		\begin{corollary}
			For model \eqref{eq:price-SDE}, 
			let us assume that {\bf(A1)\rm}, {\bf(K1)\rm}, {\bf (K2)\rm},  {\bf($\bf\Sigma 1$)}, {\bf($\bf\Sigma 2$)} hold, that $S$ is a martingale and there exist $p>1,t>0$ such that $\E [ S_t^p ]<\infty$. Then, with log-moneyness as in \eqref{eq:log:moneyness} and $x\neq 0$,  the short-time asymptotics for  implied volatility
			\begin{equation}
				\label{eq:ivol:as}
				\sigma_{BS}(t,k_t)\to \frac{x}{\sqrt{2\Lambda(x)}}=:\Sigma_{LM}(x) \mbox{ as }  t\to 0
			\end{equation}
			holds.
			As a consequence, with $k_t'=x' t^{-H+1/2}L(t)^{-1}$, the finite difference implied volatility skew satisfies
			\begin{equation}
				\label{eq:skew:as}
				\frac{\sigma_{BS}(t,k_t)-\sigma_{BS}(t,k'_t)}
				{k_t-k'_t}
				\sim\frac{\Sigma_{LM}(x)-\Sigma_{LM}(x')}{x-x'}
				t^{H-1/2}L(t)
			\end{equation}
		\end{corollary}
		
		\begin{remark}\rm\label{rem:lmfbm}
			
			When taking the kernel in \eqref{eq:kernel:logmod}
			with $0< H \leq 1/2$ we have
			\begin{equation*}\label{eq:Lt}
				L(t)\sim  (-\log t)^{-p}
			\end{equation*}
			in \eqref{eq:skew:as}, and the finite difference skew at the LDP regime explodes as $t^{H-1/2}(-\log t)^{-p}$. We prove this for $H>0$, because {(\bf K2)\rm} fails for $H=0$. However, even for $H=0$ the process is defined and the skew asymptotics \eqref{eq:skew:as} can be computed and is consistent with the  ``Gaussian'' result at the Edgeworth regime in \cite{BHP20}.
			It is also clear that
			\[
			\frac{\Sigma_{LM}(x)-\Sigma_{LM}(x')}{x-x'}
			\]
			is an approximation of $\partial_{x}\Sigma_{{LM}}(0)$ for $x,x'$ close to $0$. Assuming $\Lambda$ smooth and, as one expects, $\Lambda(0)=0$ and $\Lambda'(0)=0$, we have
			\[
			\frac{x}{\sqrt{2\Lambda(x)}}=\frac{1}{\sqrt{\Lambda''(0)+\frac{\Lambda'''(0)x}{3}+O(x^2)}}=
			\frac{1}{\sqrt{\Lambda''(0)}}\left(1-\frac{\Lambda'''(0)}{6\Lambda''(0)}x+O(x^2)\right)
			\]
			so that we can approximate the implied skew as
			\[
			\frac{\sigma_{BS}(t,k_t)-\sigma_{BS}(t,k'_t)}
			{k_t-k'_t}
			\approx
			\Sigma_{{LM}}'(0)
			t^{H-1/2}L(t)
			= -\frac{\Lambda'''(0)}{6\Lambda''(0)^{3/2}}  t^{H-1/2}L(t).
			\]
			Note, however, that \eqref{eq:skew:as} and the asymptotics in \cite{BHP20} are different mathematical results. In addition, besides providing the at-the-money behaviour, result \eqref{eq:ivol:as} can also be used to compute the whole short-dated smile, including the wings, so it can be used for calibration and, for example, for tail risk hedging.
			Since, as noted at the end of Section \ref{sec:modulated}, the rate function is the same as for the self-similar process and does not depend on the modulating function $L$, it can be computed as explained in Section \ref{sec:numerics} for fOU.

		\end{remark}
		\begin{proof} We first prove equation \eqref{eq:ivol:as}. Apply  Corollary 7.1 - Equation (7.2)  in \cite{gaolee}, along the lines of Appendix D in \cite{FGP21} or  Corollary 4.15 in \cite{FZ17}.
			Then
			\[
			\sigma_{BS}^2(t,k_t)\sim -\frac{1}{t}\frac{k_t^2}{2\log c(t,k_t)}
			\sim \frac{x^2}{2\Lambda(x)}
			\]
			and taking the square root we get the result. Equation
			\eqref{eq:skew:as} follows easily from equation
			\eqref{eq:ivol:as}.
			\cvd
		\end{proof}

		\subsection{Large deviation pricing under fractional Ornstein-Uhlenbeck volatility }
		As consequence of  Theorem \ref{th:LDP-log-price} and Theorem \ref{th:small-time-LDP} and the computations in Section \ref{sec:fOU}, we can derive asymptotic pricing formulas for European put and call options under the price dynamics in \eqref{eq:price-SDE}, with volatility driven by the process given in \eqref{eq:FOU}. In this case, we are considering the stochastic volatility dynamics
		\begin{equation}\label{eq:stoch:vol:fOU}
			\begin{cases}
				dS_t &=S_t\sigma(V_t)d(\bar{\rho}\bar B_t+\rho B_t),\\
				d V_{t}& = -a V_{t } dt + dB^{H}_{t}, 
			\end{cases}
		\end{equation}
		with $S_{0}=1,\,V_{0}=0$.
		Notice that this is written in differential form but $V$ could also be written explicitly as in \eqref{eq:FOU}.
		With the same arguments used in the proof of Theorem \ref{th:ivol:logmod}, we have
		$$
		\P(S_{t}>  e^{x t^{1/2-H}})=\P(X_{t}> x t^{1/2-H})
		\approx \exp\{-t^{-2H} J(x) \},
		$$
		where  $J(x)=I_{X_1}(x)$. More explicitly, \eqref{eq:rate-fun-finite} reads
		\begin{equation}\label{J}
			J(x)=\inf_{f\in H_0^1[0,1]} \left[\frac12 \lVert f\rVert_{H_0^1[0,1]}^2+\frac12 \frac{\big(x-\int_0^1 \rho \sigma(\hat{f}(t))\dot{f}(t)\,dt\big)^2}{\int_0^1 \bar{\rho}^2\sigma^2(\hat{f}(t)) dt}\right].
		\end{equation}
		
		We have the following theorem.
		\begin{theorem}\label{thm:main:smalltime}
			Suppose {\bf($\bf\Sigma 1$)}, {\bf($\bf\Sigma 2$)} hold.
			If $x<0$ and $k_t=x t^{{H-1/2}}$, the put price in short-time satisfies
			\[
			p(t,k_t)=\E[(e^{k_t}-S_t)^+]\approx \exp\{-t^{-2H}  J(x) \}.
			\]
			In addition, we now assume that the process $S$ is a martingale and there exist $p>1,t>0$ such that $E [ S_t^p ]<\infty$.
			If $x>0$ and $k_t=x t^{{H-1/2}}$, we have
			\[
			c(t,k_t)=\E[(S_t-e^{k_t})^+]\approx \exp\{-t^{-2H} J(x) \}.
			\]
		\end{theorem}
		
		\begin{remark}\label{rm:mart}
			In both Theorems \ref{th:ivol:logmod} and \ref{thm:main:smalltime} the call price asymptotics holds under the assumption that the price process $S$ is a martingale, along with a moment condition. In the diffusive case ($H=1/2$) several related results are available. In particular, martingality holds if $\sigma$ has exponential growth and $\rho <0$ \cite{sin1998complications, jourdain2004loss, lions2007correlations}. Note that the assumption of negative correlation is justified from a financial perspective.
			
			In the rough case, martingality is known to hold when $\sigma$ has linear growth and the driving process is the fBM \cite{FZ17}. In \cite{gassiat2018martingale}, it is 
			shown that for a class of rough volatility models with $\sigma$ of exponential growth (that includes the rough 
			Bergomi model) the stock price is a true martingale if and only if $\rho \leq 0$, while $\E[S_{t}^{p}]=+\infty$ for $p>1/(1-\rho^{2})$, for any $t>0$. 
			
			Models where the volatility is a function $\sigma$ of a Gaussian process are considered in \cite{Gu2}. If $\sigma$ grows faster than linearly, conditions for the explosion of moments are given both in the correlated and uncorrelated case.
			
			For models \eqref{eq:Volterra-modulated} and \eqref{eq:stoch:vol:fOU}, these are open questions. We expect the conditions for the call asymptotics in Theorems \ref{th:ivol:logmod} and \ref{thm:main:smalltime} to hold in case $\rho <0$ and $\sigma$ with exponential growth. In particular, martingality should definitely hold in the cases analogous to \cite{gassiat2018martingale}, but with fOU driver. Indeed, the distribution of the fOU process is more concentrated than the one of the fBM, because of the mean reversion property.
		\end{remark}

		\begin{proof} This follows from the classic argument that we spelled out in the proof of Theorem \ref{th:ivol:logmod}. The proof follows as in Appendix C, Proof of Corollary 4.13 in \cite{FZ17}.
			\cvd
		\end{proof}
		
		Again, from this call and put price asymptotics, an application of Corollary 7.1 in \cite{gaolee}  gives the following result.
		
		\begin{corollary}
			\label{corollary:ivol} 
			Under the assumptions of Theorem \ref{thm:main:smalltime}, writing $k_t=xt^{1/2-H}$, we have, for $x\in \R \setminus \{0\}$, 
			\begin{equation}\label{ivol:expansion}
				\sigma_{BS}(t, k_t)
				\to\frac{|x|}{\sqrt{2J(x)}}=:
				\Sigma_{fOU}(x), \mbox { as } t\to 0 
			\end{equation}
		\end{corollary}
		
		As a consequence, the behavior of the implied skew at the large deviations regime under fOU-driven volatility is as follows.
		\begin{corollary}
			\label{corollary:skew} 
			Under the assumptions of Theorem \ref{thm:main:smalltime}, writing $k_t=xt^{1/2-H}$, we have, for $x>0$, 
			\begin{equation}\label{skew:expansion}
				\frac{\sigma_{BS}(t, k_t)-\sigma_{BS}(t, -k_t)}{2 k_{t}}
				\sim \frac{\Sigma_{fOU}(x)-\Sigma_{fOU}(-x)}{2x} t^{H-1/2}, \mbox { as } t\to 0. 
			\end{equation}
		\end{corollary}

		\begin{remark}[On moderate deviations]\label{th:moderate_deviations}
			
			Model \eqref{eq:stoch:vol:fOU} should satisfy a moderate deviation result analogous to the ones in \cite{BFGHS} and Theorem 3.13 in \cite{FGP22}. 
			Let $c(\cdot,\cdot)$ be as in \eqref{eq:putcall}, the price process $S$ given in \eqref{eq:stoch:vol:fOU}.
			Assume that $J$ is $n\in \N$ times continuously differentiable.
			Let $H\in (0,1/2)$, $\beta>0$ and $n\in \N$ such that 
			$ \beta \in (\frac{2H}{n+1}, \frac{2H}{n}]$.
			Set $\ell_t=x t^{1/2-H+\beta}$. Then, we can formally compute the call asymptotics from 
			Theorem \ref{thm:main:smalltime}, plugging $\ell_{t}$ as log price instead of $k_{t}$, so that we substitute $x_{t}=xt^{\beta}$ to $x$ in a Taylor expansion of $J$ at $0$ and get
			\[
			J(x_t)=\sum_{i=2}^{n} \frac{J^{(i)}(0)}{i!} x^i t^{i\beta} +O(t^{(n+1)\beta}).
			\]
			Now, consider the speed $t^{{2H}}$ in Theorem \ref{thm:main:smalltime} and that $t^{(n+1)\beta-2H}\rightarrow 0$ if $ \beta \in(\frac{2H}{n+1}, \frac{2H}{n}]$, 
			recall from  \cite{FZ17} and \cite{BFGHS} that $J(0)=J'(0)=0,\,J''(0)=1/\sigma(0)^{2}$ and we find that the call price should satisfy the following moderate deviations asymptotics, as $t\to 0$,
			\[
			\log 
			c(t,\ell_t)= -
			\sum_{i=2}^{n} \frac{J^{(i)}(0)}{i!} x^i t^{i\beta-2H}+O(t^{(n+1)\beta-2H}).
			\]
			We expect that a complete proof of this fact could be adapted from Proof of Theorem 3.13 in \cite{FGP22} or Proof of Theorem 3.4 in \cite{BFGHS}.
			Assuming this call price asymptotics holds true, the following implied volatility asymptotics can be derived using Corollary 7.1, Equation (7.2) in \cite{gaolee} and that $J''(0)=1/\sigma(0)^{2}$
			\begin{equation}\label{exp:imp:vol}
				\begin{split}
					\sigma_{BS}^2(t, \ell_t)=
					\sum_{j=0}^{n-2}
					(-1)^j 2^j \sigma(0)^{2(j+1)}
					\left(
					\sum_{i=3}^{n} \frac{J^{(i)}(0)}{ i!} x^{i-2} t^{(i-2)\beta}
					\right)^j
					+O(t^{2H-2\beta}).
				\end{split}
			\end{equation}
		\end{remark}
		
		\begin{remark}[On related results]
			\label{rm:related} 
			A pathwise small-noise LDP under fOU volatility has been proved in \cite{horvath_jacquier_lacombe_2019}, with different hypothesis in particular on the function $\sigma$. From this LDP, a short-time result for a suitably renormalized process is also derived, with a time-scaling different from ours. 
			
			In \cite{jacquier2020volterra} asymptotic results are given for Volterra driven volatility models, including large and moderate deviations, also in small-time. Hypothesis on the models are different from ours, for example $\sigma^{2}(x)$ is of linear growth, or alternatively a moment condition of type $\E[\sigma(V)^{2p}]<\infty$ for any $p\geq 1$ holds. The rate function  is given as an expression involving fractional derivatives of the minimiser.  In particular, in \cite[Section 4.2.1]{jacquier2020volterra} these results are applied to the rough Stein-Stein model, which is similar to \eqref{eq:stoch:vol:fOU}, with the RLp instead of the fBM, and with the specific choice of volatility function $\sigma^{2}(x)=x^{2}$. Analogous results should also hold with the fBM instead of the RLp as driver of the volatility.
		\end{remark}
		
		\begin{remark}[On applications]
			As mentioned in the introduction, short-time asymptotic approximations to the implied volatility surface are used for model calibration, pricing and other applications. They give information on option prices with short maturity, with low computational burden. This helps for example in the creation of delta-hedging strategies that are sensitive to short-term moves in the underlying and in general in trading and risk management.
			Efficient and accurate methods for calibrating fOU-driven volatility models are relevant, for example, because these volatility models are used for computing option prices and implied volatilities \cite{GS17,GS20} and for hedging \cite{GS20hedging}.  Furthermore, \cite{GS18} compare the price impact of fast mean-reverting Markov stochastic volatility models with the price impact of mean reverting rough volatility models (see also \cite{GS19}). In \cite{FH18}, a model with both return and volatility driven by a fast mean reverting fOU process are used for portfolio optimization, in the $H>1/2$ regime. 
		\end{remark}

		\section{Numerical experiments}\label{sec:numerics}
		
		In this section we test the accuracy of short-time pricing formulas \eqref{ivol:expansion}, and \eqref{exp:imp:vol} and of the implied skew asymptotics \eqref{skew:expansion}. We do so for a stochastic volatility model with asset price dynamics given by \eqref{eq:price-SDE}, with both fBM-driven volatility (i.e., $V=B^H$ is the fBM)
		and fOU-driven volatility (i.e., $V=V^H$  is the fOU process, as in \eqref{eq:stoch:vol:fOU}).
		Recall that both fBM and fOU models lead to the same rate function. 
		
		For numerical experiments with log-fBM volatility, we refer to \cite{BHP20}. In particular, the discussion in Remark \ref{rem:lmfbm} on the at-the-money implied skew for log-modulated models is consistent with the numerical evaluations of at-the-money skews in \cite[Section 7]{BHP20}.

		From Section \ref{sec:fOU}, we have the Volterra representation of the fBM 
		$$
		B^H_r = \int_0^t K_H(r,s) dB_s,
		$$
		where $K_H$ is the kernel in \eqref{eqn:kernelfbm}, and the Volterra representation of the fOU process
		\begin{equation}\label{FOU}
			V^H_r= \int_0^r K(r,s) ds = \int_0^r \Big(K_H(r,s)-a\int_s^r e^{-a(r-u)} K_H(u,s)du       \Big) dB_s.
		\end{equation}
		To evaluate the quality of approximations \eqref{ivol:expansion}, \eqref{skew:expansion} and \eqref{exp:imp:vol}, we first simulate Monte Carlo call prices under both these models, from which we then recover Black-Scholes implied volatilities. In both cases we consider a volatility function $\sigma(\cdot)$, depending on positive parameters $\sigma_{0}, \eta$ given by
		\begin{equation}\label{volFun}
			\sigma(x)=\sigma_0 \exp\Big({\frac{\eta}{2} \,x}\Big).
		\end{equation}
		To compute these prices under our stochastic volatility dynamics, we need to simulate the asset price at the fixed time horizon $t>0$. Hence we consider a time-grid $t_k= k\frac{t}{N}$, $k=0,\ldots, N$, and on this grid the random vector $(V_{t_1},\ldots, V_{t_N}, B_{t_1}, \ldots, B_{t_N})$, first with $V=B^H$ and then $V=V^H$. In both cases, it is a multivariate Gaussian vector with zero mean and known covariance matrix, that can be computed from the Volterra representation of the processes. The whole vector can be simulated using a Cholesky factorization of this covariance matrix. We then use this vector to construct an approximate sample of the log-asset price 
		\[
		X_t=-\frac 12\int_0^t\sigma^2(V_s) ds  + \rho\int_0^t \sigma(V_s) dB_s +\bar{\rho}\int_0^t\sigma(V_s) d\bar B_s
		\]
		by using a forward Euler scheme on the same time-grid
		$$
		X_t^N= -\frac{t}{2N}\sum_{k=0}^{N-1} \sigma^2(V_{t_k})+\sum_{k=0}^{N-1} \sigma(V_{t_k})\Big(\rho(B_{t_{k+1}}-B_{t_k} )+ \bar \rho(\bar B_{t_{k+1}}-B_{t_k}) \Big).
		$$
		We produce $M$ i.i.d. approximate Monte Carlo samples $(X_t^{N,m}, V_T^{m})_{1\leq m\leq M}$, that we use to evaluate call option prices by standard sample average. Then, we compute the corresponding implied volatilities $\sigma_{BS}(t, k)$ by Brent's method (see \cite{At08}, \cite{Pr08}), where $t$ is the maturity and $k$ the log-moneyness.
		
		Note that Theorem \ref{thm:main:smalltime}, Corollary \ref{corollary:ivol} and  Corollary \ref{corollary:skew} do not apply to the model above, because
		$\sigma(\cdot)$ does not satisfy the polynomial growth condition {\bf($\bf\Sigma 2$)}. However, also in in the self-similar case, large deviations pricing results were first obtained under  linear growth conditions in \cite{FZ17} and then the conditions were weakened  in \cite{bayer_regularity, Gu1} to include exponential growth. Therefore, we chose here to test our result on the exponential volatility in \eqref{volFun}, for which our result should hold as well. This choice is more realistic, being analogous to the rough Bergomi model, and being the volatility function considered e.g. in \cite{GS20hedging}.

		To evaluate the accuracy of large deviations approximation \eqref{ivol:expansion}, we follow the choice in \cite{FGP22} and use as model parameters $H=0.3, \rho=-0.7, \sigma_0=0.2, \eta=1.5$, and as mean reversion parameter in fOU we take  $\alpha=1$ or  $\alpha=2$. These parameters are similar to the ones estimated on empirical volatility surfaces, as for example in \cite{bayer2016pricing}. We take a rough, but not ``extremely rough'' ($0.3$ instead of $0.1$) Hurst parameter $H$, motivated by the recent study \cite{GuyonSkew}.
		
		We simulate $M=10^6$ Monte Carlo samples using $N=500$ discretization points. 
		We estimate call option prices $\E[(S_0e^{X_t}- S_0e^{k_t})^+ ]$, where $k_t= x t^{1/2-H}$, and the corresponding implied volatility $\sigma_{BS}(t, k_t)$.
		
		Then, we need to compute $\Sigma_{{fOU}}$. The rate function $J$ in \eqref{J} can be approximated numerically using the Ritz method, as described in detail in \cite[Section 40]{gelfand_fomin}, \cite{FZ17}, and
		\cite[Remark 4.3 and Section 5.1]{FGP22}. The rate function is obtained trough numerical optimization on a fixed, finite number of coefficients associated to a basis of the Cameron-Martin space $H_0^1$. We take as basis the Fourier basis, i.e. $\{\dot e_i\}_{i\in \N}$ with
		$$
		\dot e_1(s)= 1, \qquad \dot e_{2n}(s)=\sqrt{2} \cos(2\pi n s),\qquad \dot e_{2n+1}=\sqrt{2}\sin(2\pi n s), \quad n \in \N\setminus \{0\},
		$$
		that we truncate to $N=5$ (larger values of $N$ did not seem to improve the computation) and use the more explicit representation of the rate function $J$ in \eqref{J} given in
		\cite{FZ17}, \cite[Proposition 5.1]{BFGHS}, \cite[Section 5.1]{FGP22}.
		
		In Figure \ref{fig:1} we show for each model how, as the maturity $t$ becomes smaller, $\sigma_{BS}(t, k_t)$ gets closer to the asymptotic limit in \eqref{ivol:expansion}, where $k_t=x t^{1/2-H}$. We recall that $t$ is the option maturity and we numerically evaluate $\sigma_{BS}(t, k_t)$ for $t\in \{0.05, 0.1, 0.2, 0.3, 0.5 \}$ and $x \in [-0.2, 0.2]$ for $50$ equidistant points. The fact that, even for very small maturities, the short-time limit is not reached, can be explained by the fact that the error is of order $t^{2H}$ (as shown in the self-similar case in \cite{FGP22}), which vanishes as $t\to 0$, albeit slowly, since $H<1/2$.
		
		In Figure \ref{fig:2}, for each fixed maturity, we compare the implied volatility smiles produced by each model (fOU vs fBM-driven volatilities), in order to observe the influence of the magnitude of the mean reversion parameter $a$ on the volatility smiles. In particular, we note that implied volatilities generated by fOU-driven models seem to fall between the implied volatilities generated by fBM-driven models and the asymptotic smile, indicating convergence also if polynomial growth of $\sigma(\cdot)$ is not satisfied in this example. 
		\begin{figure}[t]\centering
			\includegraphics[width=0.49\linewidth]{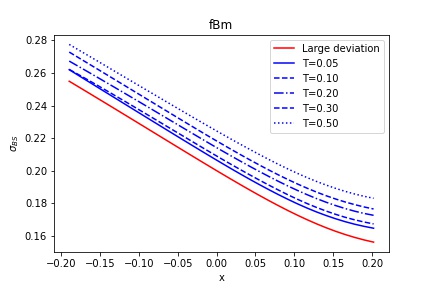}\\
			\includegraphics[width=0.49\linewidth]{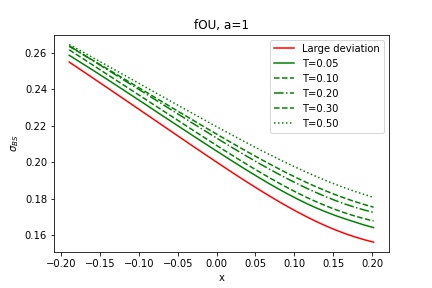}
			\includegraphics[width=0.49\linewidth]{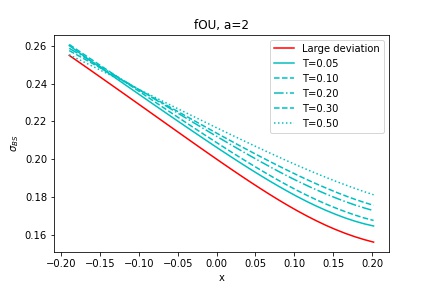}
			\caption{Implied volatility smile with fBM-driven stochastic volatility, fOU-driven stochastic volatility with $a=1$, fOU-driven stochastic volatility with $a=2$ and large deviation approximation \eqref{ivol:expansion}. Model parameters: $H=0.3, \rho=-0.7, \sigma_0=0.2$ and $\eta=1.5$. Monte Carlo parameters: $10^6$ trajectories and $500$ time-steps. Recall that $k_t=x t^{1/2-H}$. The rate function is computed with the Ritz method with $N=5$ Fourier basis function.  \label{fig:1}
			}

		\end{figure}
		
		\begin{figure}[t]
			
			\centering
			\includegraphics[width=0.49\linewidth]{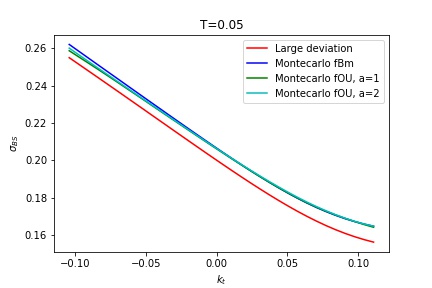}
			\includegraphics[width=0.49\linewidth]{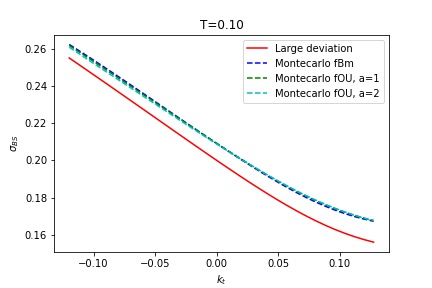}
			
			\includegraphics[width=0.49\linewidth]{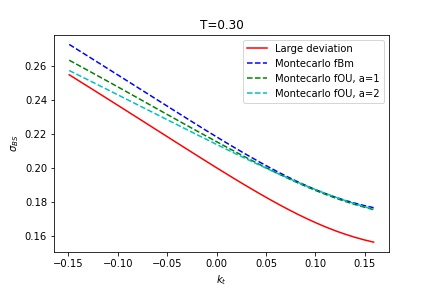}
			\includegraphics[width=0.49\linewidth]{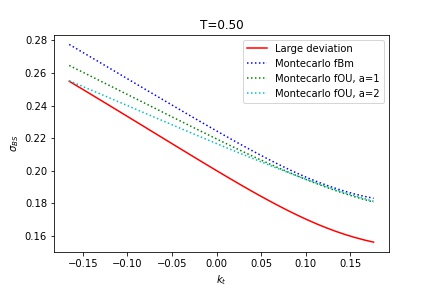}
			
			\caption{Implied volatility smile with fBM-driven stochastic volatility, fOU-driven stochastic volatility with $a=1$, fOU-driven stochastic volatility with $a=2$ and large deviation approximation \eqref{ivol:expansion}. Model parameters: $H=0.3,\, \rho=-0.7, \sigma_0=0.2$ and $\eta=1.5$. Monte Carlo parameters: $10^6$ trajectories and $500$ time-steps. We plot each model fixing the time horizon and varying $x$, where $k_t=xt^{1/2-H}$. Rate function is computed with the Ritz method with $N=5$ Fourier basis function.
				\label{fig:2}}
		\end{figure}
		
		
		\medskip
		We test now the moderate deviation asymptotics in Remark \ref{th:moderate_deviations}. In order to do so, let us recall an expansion to the fourth order of the rate function that allows us to use the second order moderate deviation.\footnote{This expansion is given in \cite[Lemma 6.1]{FGP22}, where the kernel $C(t-s)^{H-1/2}$ is used. However, in the proof of this result the specific shape of the kernel is not used, but only self-similarity, and therefore it holds for $K(t,s)$ in \eqref{eqn:kernelfbm} as well.} We denote now $K_{H}f(t)=\int_0^t K_{H}(t,s)f(s)ds$ and with $\overline{K_{H}}$ the adjoint of $K_{H}$ in $L^2[0,1]$, so that $\overline{K_{H}}f(u)=\int_u^1 K_{H}(t,u) f(t) dt$, where again $K_{H}$ is the fBM kernel in \eqref{eqn:kernelfbm}.

		\begin{lemma}[Fourth order energy expansion]
			\label{expansion:J}
			Let us assume that $\sigma(\cdot)$ is countinuously differentiable two times.
			Let $J(x)$ be the energy function in \eqref{J}. Then
			\begin{equation}\label{expanJ}
				\begin{split}
					J(x)&= \frac{J''(0)}{2} x^2 +
					\frac{J'''(0)}{3!} x^3+\frac{J^{(4)}(0)}{4!} x^4+O(x^5)\,\\
				\end{split}
			\end{equation}
			where
			$$
			J''(0)=\frac{1}{\sigma(0)^2}
			,\quad J'''(0)=- 6 \frac{ \rho \sigma'(0)}{\sigma(0)^4} \langle K_{H}1,1 \rangle, 
			$$
			and
			\[
			\begin{split}
				J^{(4)}(0)
				& =12  \frac{\sigma'(0)^2}{\sigma(0)^6}
				\left\{
				9 \rho^2 
				\langle K_{H}1,1\rangle^2
				- \rho^2
				\langle (K_{H}1)^2 ,1\rangle  
				-
				\langle (\overline{K_{H}}1)^2 ,1\rangle  
				-
				2 \rho^2 
				\langle K_{H}1,\overline{K_{H}}1\rangle
				\right\}+\\
				& -12 \frac{\sigma''(0)}{\sigma(0)^5}   \rho^2
				\langle (K_{H}1)^2 ,1\rangle. \end{split}
			\]
			
			Plugging \eqref{expanJ} into \eqref{exp:imp:vol} and fixing $n=4$, from straightforward computations, we obtain the equivalent asymptotic formula
			\begin{equation}\label{equiFor}
				\sigma(t,\ell_t)=\Sigma_{fOU}(0)+\Sigma_{fOU}'(0)x t^\beta + \frac{\Sigma_{fOU}''(0)}{2}x^2 t^{2\beta}+o(t^{2H-2\beta})
			\end{equation}
			where
			\[
			\begin{split}
				\Sigma_{fOU}(0)&= \sigma(0), \quad
				\Sigma_{fOU}'(0)=\frac{\rho \sigma'(0)\langle K_{H}1,1\rangle}{\sigma(0)}, \\
				\frac{\Sigma_{fOU}''(0)}{2}
				&=
				\frac{\sigma'(0)^2}{\sigma(0)^3}
				\left\{
				- 3 \rho^2 
				\langle K_{H}1,1\rangle^2
				+ \frac{\rho^2}{2}
				\langle (K_{H}1)^2 ,1\rangle  
				+
				\frac{ 1}{2}
				\langle (\overline{K_{H}}1)^2 ,1\rangle  
				+
				\rho^2 
				\langle K1,\overline{K_{H}}1\rangle
				\right\}\\
				&+ \frac{\sigma''(0)}{\sigma(0)^2}  \frac{ \rho^2 }{2}
				\langle (K_{H}1)^2 ,1\rangle.  
			\end{split}
			\]
		\end{lemma}
		We plot in Figure \ref{fig:3} implied volatilities computed via Monte Carlo simulations and the corresponding approximation given in \eqref{equiFor}.
		We take again $\sigma(\cdot)$ as in\eqref{volFun}, with parameters $H=0.3, \rho=-0.7, \sigma_0=0.2, \eta=0.2$, and $\beta=0.125$. We first note that we fix $n=4$ in \eqref{exp:imp:vol}, and so we choose $\beta \in (\frac{2H}{n+1}, \frac{2H}{n}]$, that is the interval $(0.12, 0.15]$. 
		With respect to our previous experiments, we also take the smaller vol of vol parameter $\eta=0.2$, which is in line with the choices in \cite{BFGHS,FGP22}. Indeed, the quality of the approximation deteriorates as $\eta$ grows, and for larger $\eta$ the asymptotic formula \eqref{equiFor} is accurate on a smaller time interval.   
		
		\begin{figure}[t]
			\centering
			\includegraphics[width=0.7\linewidth]{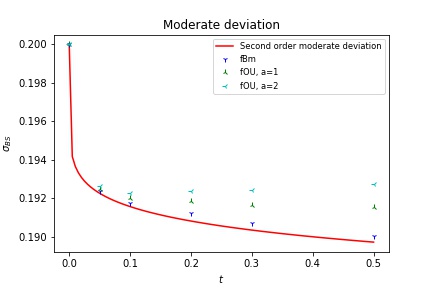}
			\caption{Moderate deviation implied volatilities with fBM-driven stochastic volatility (blue), with fOU-driven stochastic volatility with $a=1$ (green) and fOU-driven stochastic volatility with $a=2$ (light blue), $\ell_t=x t^{1/2-H+\beta}$ with $x=0.3$ and $\beta=0.125$. Model parameters: $H=0.3, \rho=-0.7, \sigma_0=0.2$ and $\eta=0.2$. Monte Carlo parameters: $10^7$ trajectories, $500$ time-steps. We plot each model with fixed $x$ and varying maturity.
				\label{fig:3}}
		\end{figure}

		In Figure \ref{fig:skew} we compare, with $k_{t}=x t^{H-1/2}$, $x>0$, the absolute value of the large deviations finite difference implied skew
		\begin{equation}\label{eq:abs:skew}
			\Psi_{t}:=\frac{|\sigma_{BS}(t, k_t)-\sigma_{BS}(t, -k_t)|}{2 k_{t}}
		\end{equation}
		computed on fBm-driven and fOU-driven stochastic volatility models, with the asymptotic skew expected from Corollary \ref{corollary:skew}, where we also use the approximation, as $x\to 0$,
		\[
		\frac{\Sigma_{fOU}(x)-\Sigma_{fOU}(-x)}{2x}
		\sim
		\Sigma_{fOU}'(0)=\frac{\rho \sigma'(0)\langle K_{H}1,1\rangle}{\sigma(0)}.
		\]
		We observe that, consistently with the smile slopes observed in Figure \ref{fig:2},
		larger mean-reversion parameters $a$ correspond to flatter smiles and smaller skews (in absolute value), 
		smaller mean-reversion parameters $a$ correspond to steeper smiles and larger skews (in absolute value),
		fBm has a larger skew than fOU, and the asymptotic skew is even larger than the one generated from fBm, although very close to it. As maturity $t\to 0$, the difference between all these skews vanishes.
		
		This could reflect the fact that larger mean-reversion parameters $a$ give more concentrated volatility trajectories, with $V$ in \eqref{FOU} staying closer to $0$ and therefore the stochastic volatility path $(\sigma_0 \exp(\frac{\eta}{2} \,V_{t}))_{t>0}$ staying closer to the spot-vol $\sigma_0$. This may produce flatter implied volatility surfaces and explain smiles and skews observed in Figures \ref{fig:2}, \ref{fig:3} and \ref{fig:skew} corresponding to larger $a$'s. On the short end of the surface, however, all of these have to coincide due to our asymptotic results.
		Let us also mention that the discrepancy observed in Figure \ref{fig:2} on the level of the smile (regardless of the skew), between the asymptotic red line and all the simulated ``positive maturity'' lines is likely due to a term-structure term of order $t^{2H}$, for which we refer the reader to \cite{FGP22} (large deviations setting) and \cite{euch2018short} (central limit setting).

		\begin{figure}[t]
			\centering
			\includegraphics[width=0.495\linewidth]{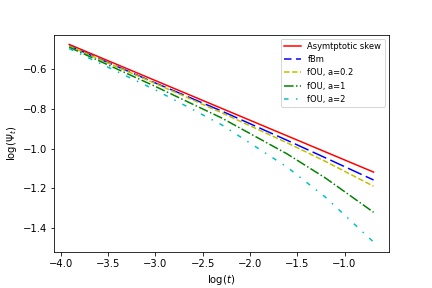}
			\includegraphics[width=0.495\linewidth]{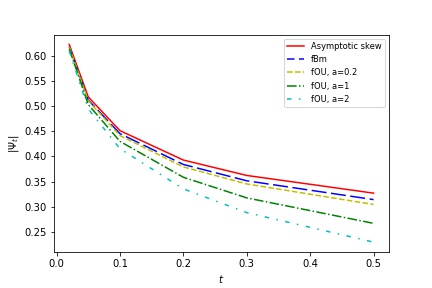}
			\caption{ Implied skew in \eqref{eq:abs:skew} with fBm-driven stochastic volatility, fOU-driven stochastic volatility with $a=0.2$, $a=1$, $a=2$ and asymptotic skew from \eqref{skew:expansion}. Model parameters: $H=0.3,  \rho=-0.7, \sigma_0=0.2$ and $\eta=1.5$. Monte Carlo parameters: $10^7$ trajectories and $500$ time-steps. We take $x=0.01$ (recall that $k_t=x t^{1/2-H}$). Log-plot on the left hand side, linear plot on the right hand side.
				\label{fig:skew}}
		\end{figure}

%
%
%
%
%

		\addcontentsline{toc}{chapter}{Bibliography}
		\bibliographystyle{abbrv}
		\bibliography{Bibliografia_dottorato}

	\end{document}